\documentclass[10pt]{amsart}  % eventually amsart

\usepackage{epsf,latexsym,amsmath,amssymb,amscd,datetime}
\usepackage{amsmath,amsthm,amssymb,enumerate,eucal,url,calligra,mathrsfs}
\usepackage{subcaption}
\usepackage{graphicx}
\usepackage{color}

\newcommand{\indicateDgeA}{{\cI_{\mec d\ge\mec a}}}

\DeclareMathOperator{\SHom}{\mathscr{H}\text{\kern -3pt {\calligra\large om}}\,}
\DeclareMathOperator{\SExt}{\mathscr{E}\text{\kern -2pt {\calligra\large xt}}\,\,}

\IfFileExists{my_xrefs}{\input my_xrefs}{}
\newcommand{\naturals}{{\mathbb N}}

\newcommand{\mec}[1]{{\bf #1}}	% for vector 𝐤 and 𝐦  with 𝐤 ⋅𝐦  = k 
\newcommand{\bec}[1]{{\boldsymbol #1}}	% for greek letters

\usepackage{mathrsfs}
\usepackage{amssymb}
\usepackage{dsfont}
\usepackage{verbatim}
\usepackage{url}

\theoremstyle{plain}
\newtheorem{theorem}{Theorem}[section]
\newtheorem{lemma}[theorem]{Lemma}
\newtheorem{proposition}[theorem]{Proposition}
\newtheorem{corollary}[theorem]{Corollary}

\theoremstyle{definition}
\newtheorem{definition}[theorem]{Definition}

\newtheorem{xca}{Exercise}[section]
\newtheorem{example}[theorem]{Example}

\newtheorem{convention}[theorem]{Convention}

\newcommand{\isom}{\simeq} % later this should, perhaps, be changed to =
\newcommand{\ignore}[1]{}

\newcommand{\Hom}{{\rm Hom}}
\newcommand{\field}{{\mathbb F}}

\newcommand{\integers}{{\mathbb Z}}

\DeclareMathAlphabet{\mathcal}{OMS}{cmsy}{m}{n}

\newcommand\cC{\mathcal{C}}
\newcommand\cD{\mathcal{D}}

\newcommand\cF{\mathcal{F}}
\newcommand\cG{\mathcal{G}}

\newcommand\cI{\mathcal{I}}

\newcommand\cK{\mathcal{K}}
\newcommand\cL{\mathcal{L}}
\newcommand\cM{\mathcal{M}}
\newcommand\cN{\mathcal{N}}
\newcommand\cO{\mathcal{O}}
\newcommand\cP{\mathcal{P}}

\newcommand\frakm{\mathfrak{m}}

\newcommand\fraks{\mathfrak{s}}
\newcommand\frakt{\mathfrak{t}}

\def\from{\colon}

\def\isom{\simeq}

\def\eqdef{\overset{\text{def}}{=}}

\DeclareMathOperator{\coker}{coker}

\DeclareMathOperator{\Ext}{Ext}

\def\Hom{\qopname\relax o{Hom}}
\DeclareMathOperator{\id}{id}

\def\implies{\Rightarrow}

\DeclareRobustCommand
  \rddots{\mathinner{\mkern1mu\raise\p@
    \vbox{\kern7\p@\hbox{.}}\mkern2mu
    \raise4\p@\hbox{.}\mkern2mu\raise7\p@\hbox{.}\mkern1mu}}

\newcommand\xhookrightarrow[2][]{\ext@arrow 0062{\hookrightarrowfill@}{#1}{#2}}
\def\hookrightarrowfill@{\arrowfill@\lhook\relbar\rightarrow}

\usepackage{relsize}
\usepackage{tikz}
\usetikzlibrary{matrix,arrows,decorations.pathmorphing}
\usepackage{tikz-cd}
\usetikzlibrary{cd}

\usepackage[pdftex,colorlinks,linkcolor=blue,citecolor=brown]{hyperref}
\usepackage{blkarray}
\usepackage{array}
\usetikzlibrary{shapes.misc}

\tikzset{cross/.style={cross out, draw=black, minimum size=2*(#1-\pgflinewidth), inner sep=0pt, outer sep=0pt},
cross/.default={1pt}}

\begin{document}

\title[Duality in Riemann Functions]
{Euler Characteristics and Duality in Riemann Functions and the
Graph Riemann-Roch Rank}

\author{Nicolas Folinsbee}
\address{Department of Mathematics, University of British Columbia,
        Vancouver, BC\ \ V6T 1Z2, CANADA. }
\curraddr{}
\email{{\tt nicolasfolinsbee@gmail.com}}
\thanks{Research supported in part by an NSERC grant.}

\author{Joel Friedman}
\address{Department of Computer Science, 
        University of British Columbia, Vancouver, BC\ \ V6T 1Z4, CANADA. }
\curraddr{}
\email{{\tt jf@cs.ubc.ca}}
\thanks{Research supported in part by an NSERC grant.}

% \date{\today} % , at \currenttime  (get rid of time in final version)}
\date{July 14, 2021} % , at \currenttime  (get rid of time in final version)}

% \subjclass[2010]{Primary: 05C38, 14H55.  Secondary: 55N30} from Nick's thesis


\begin{abstract}    
By a {\em Riemann function} we mean a function
$f\from\integers^n\to\integers$ such that
$f(\mec d)=f(d_1,\ldots,d_n)$ is equals $0$
for $\deg(\mec d)=d_1+\cdots+d_n$ sufficiently
small, and equals $d_1+\cdots+d_n+C$ for a constant, $C$---the {\em offset
of $f$}---for
$\deg(\mec d)$ sufficiently large.
By adding $1$ to the Baker-Norine rank function of a graph, one gets
an equivalent Riemann function, and similarly for related rank functions.
For such an $f$,
for any $\mec K\in\integers^n$ there is a unique Riemann function
$f^\wedge_\mec K$ such that for all $\mec d\in\integers^n$ we have
$$
f(\mec d) - f^\wedge_\mec K(\mec K-\mec d) = \deg(\mec d)+C 
$$
which we call a {\em generalized Riemann-Roch formula}.
We show that any such equation can be interpreted as an Euler
characteristic computation involving sheaves of vector 
spaces on a fixed topological space (fixed for the entire article)
consisting of five points.

This article does not assume any prior knowledge of sheaf theory.
Our models are diagrams built on five vector spaces and some
linear transformations between them.
We work entirely with these diagrams; at the end of this article we explain
that these diagrams and the methods we use
are borrowed from sheaf theoery.

To a Riemann functions $f$ we associate 
a related function $W\from\integers^n\to\integers$ via M\"obius inversion
that we call
the {\em weight} of the Riemann function.
If $f\from\integers^2\to\integers$ whose weight $W$ is non-negative,
then there is a simple family of diagrams (over any fixed
field)
$\{\cM_{W,\mec d}\}_{\mec d\in\integers^2}$ such that
$f(\mec d)=b^0(\cM_{W,\mec d})$ and
$f^\wedge_{\mec K}(\mec K-\mec d)=b^1(\cM_{W,\mec d})$.
Furthermore, the family $\{\cM_{W,\mec d}\}$ are related 
in a way that mimics the modern
formulation of the Riemann-Roch theorem for curves, making it
immediate that
$$
\chi( \cM_{W,\mec d} ) = \deg(\mec d)+ 
\chi( \cM_{W,\mec 0} )
$$
for all $\mec d$.
Furthermore we show that there is a diagram that
represents the functor $\cF\mapsto H^1(\cF)^*$, for which we prove
a full analog of
Serre duality; this also leads to an
isomorphism
$$
H^1(\cM_{W,\mec d})^* \to  H^0(\cM_{W',\mec K-\mec d})
$$
where $W'$ is the weight of $f^\wedge_\mec K$.

For Riemann functions $\integers^2\to\integers$ with a weight, $W$,
such that $W$ attains negative values, we similarly model the generalized
Riemann-Roch formulas by a {\em virtual diagram}, meaning a formal
difference of diagrams.  For Riemann functions
$\integers^n\to\integers$, we model such formulas by piecing together
restrictions of these functions to two of their variables.
The constructions of the virtual diagrams and the two-variable
restrictions require some ad hoc choices, although the equivalence
class of virtual diagrams obtained is independent of the ad hoc
choices.
The duality theorems extend to Riemann functions
$\integers^n\to\integers$.
\end{abstract}

\maketitle
\setcounter{tocdepth}{3}
\tableofcontents

% This is based on a copy of Nick's thesis

\newcommand{\axiscubism}{
\begin{center}
\begin{tikzpicture}[scale=0.5]

    \coordinate (Origin)   at (0,0);
    \coordinate (XAxisMin) at (-5,0);
    \coordinate (XAxisMax) at (5,0);
    \coordinate (YAxisMin) at (0,-5);
    \coordinate (YAxisMax) at (0,5);
    \draw [thin, gray,-latex] (XAxisMin) -- (XAxisMax);% Draw x axis
    \draw [thin, gray,-latex] (YAxisMin) -- (YAxisMax);% Draw y axis

   % Clips the picture...
    %\pgftransformcm{1}{0.6}{0.7}{1}{\pgfpoint{0cm}{0cm}}

    \foreach \x in {-5,...,5}{
      \foreach \y in {-5,-4,...,5}{
        \node[draw,circle,inner sep=0.8pt,fill] at (1*\x,1*\y) {};
            
      }
    }

\node[draw=none,fill=none] at (0.5,.5) {$1$};
\node[draw=none,fill=none] at (-0.5,.5) {$1$};
\node[draw=none,fill=none] at (0.5,-.5) {$1$};
\node[draw=none,fill=none] at (-0.5,-.5) {$1$};

\node[draw=none,fill=none] at (1.5,.5) {$2$};
\node[draw=none,fill=none] at (.5,1.5) {$2$};
\node[draw=none,fill=none] at (-.5,1.5) {$2$};
\node[draw=none,fill=none] at (-1.5,.5) {$2$};
\node[draw=none,fill=none] at (.5,-1.5) {$2$};
\node[draw=none,fill=none] at (1.5,-.5) {$2$};
\node[draw=none,fill=none] at (-.5,-1.5) {$2$};
\node[draw=none,fill=none] at (-1.5,-.5) {$2$};

\node[draw=none,fill=none] at (2.5,.5) {$3$};
\node[draw=none,fill=none] at (1.5,1.5) {$3$};
\node[draw=none,fill=none] at (.5,2.5) {$3$};
\node[draw=none,fill=none] at (-2.5,.5) {$3$};
\node[draw=none,fill=none] at (-1.5,1.5) {$3$};
\node[draw=none,fill=none] at (-.5,2.5) {$3$};
\node[draw=none,fill=none] at (2.5,-.5) {$3$};
\node[draw=none,fill=none] at (1.5,-1.5) {$3$};
\node[draw=none,fill=none] at (.5,-2.5) {$3$};
\node[draw=none,fill=none] at (-2.5,-.5) {$3$};
\node[draw=none,fill=none] at (-1.5,-1.5) {$3$};
\node[draw=none,fill=none] at (-.5,-2.5) {$3$};

\draw[blue,thick] (-3,-1) -- (3,-1);
\draw[blue,thick] (-3,0) -- (3,0);
\draw[blue,thick] (-3,1) -- (3,1);
\draw[blue,thick] (-2,2) -- (2,2);
\draw[blue,thick] (-2,-2) -- (2,-2);
\draw[blue,thick] (-1,3) -- (1,3);
\draw[blue,thick] (-1,-3) -- (1,-3);

\draw[blue,thick] (-1,-3) -- (-1,3);
\draw[blue,thick] (0,-3) -- (0,3);
\draw[blue,thick] (1,-3) -- (1,3);
\draw[blue,thick] (2,-2) -- (2,2);
\draw[blue,thick] (-2,-2) -- (-2,2);
\draw[blue,thick] (-3,1) -- (-3,-1);
\draw[blue,thick] (3,1) -- (3,-1);

\end{tikzpicture}
\end{center}
}

\newcommand{\degreecubism}{
\begin{center}
\begin{tikzpicture}[scale=0.5]

    \coordinate (Origin)   at (0,0);
    \coordinate (XAxisMin) at (-5,0);
    \coordinate (XAxisMax) at (5,0);
    \coordinate (YAxisMin) at (0,-5);
    \coordinate (YAxisMax) at (0,5);
    \draw [thin, gray,-latex] (XAxisMin) -- (XAxisMax);% Draw x axis
    \draw [thin, gray,-latex] (YAxisMin) -- (YAxisMax);% Draw y axis

   % Clips the picture...
    %\pgftransformcm{1}{0.6}{0.7}{1}{\pgfpoint{0cm}{0cm}}

    \foreach \x in {-5,...,5}{
      \foreach \y in {-5,-4,...,5}{
        \node[draw,circle,inner sep=0.8pt,fill] at (1*\x,1*\y) {};
            
      }
    }
   
\node[draw=none,fill=none] at (0.5,.5) {$1$};
\node[draw=none,fill=none] at (-0.5,.5) {$1$};
\node[draw=none,fill=none] at (0.5,-.5) {$1$};
\node[draw=none,fill=none] at (-0.5,-.5) {$1$};

\node[draw=none,fill=none] at (1.5,-1.5) {$2$};
\node[draw=none,fill=none] at (.5,-1.5) {$2$};
\node[draw=none,fill=none] at (1.5,-0.5) {$2$};
\node[draw=none,fill=none] at (-1.5,1.5) {$2$};
\node[draw=none,fill=none] at (-.5,1.5) {$2$};
\node[draw=none,fill=none] at (-1.5,0.5) {$2$};

\node[draw=none,fill=none] at (-2.5,2.5) {$3$};
\node[draw=none,fill=none] at (-1.5,2.5) {$3$};
\node[draw=none,fill=none] at (-2.5,1.5) {$3$};

\node[draw=none,fill=none] at (1.5,.5) {$3$};
\node[draw=none,fill=none] at (.5,1.5) {$3$};
\node[draw=none,fill=none] at (-1.5,-.5) {$3$};
\node[draw=none,fill=none] at (-.5,-1.5) {$3$};
\node[draw=none,fill=none] at (2.5,-2.5) {$3$};
\node[draw=none,fill=none] at (1.5,-2.5) {$3$};
\node[draw=none,fill=none] at (2.5,-1.5) {$3$};

\draw[blue,thick] (-3,3) -- (-1,3);
\draw[blue,thick] (-3,2) -- (1,2);
\draw[blue,thick] (-3,1) -- (2,1);
\draw[blue,thick] (-2,0) -- (2,0);
\draw[blue,thick] (-2,-1) -- (3,-1);
\draw[blue,thick] (-1,-2) -- (3,-2);
\draw[blue,thick] (1,-3) -- (3,-3);

\draw[blue,thick] (3,-3) -- (3,-1);
\draw[blue,thick] (2,-3) -- (2,1);
\draw[blue,thick] (1,-3) -- (1,2);
\draw[blue,thick] (0,-2) -- (0,2);
\draw[blue,thick] (-1,-2) -- (-1,3);
\draw[blue,thick] (-2,-1) -- (-2,3);
\draw[blue,thick] (-3,1) -- (-3,3);
\end{tikzpicture}
\end{center}
}

\newcommand{\PicCubeZero}{
\begin{tikzpicture}[scale=0.5]

    \coordinate (Origin)   at (0,0);
    \coordinate (XAxisMin) at (-5,0);
    \coordinate (XAxisMax) at (5,0);
    \coordinate (YAxisMin) at (0,-5);
    \coordinate (YAxisMax) at (0,5);
    \draw [thin, gray,-latex] (XAxisMin) -- (XAxisMax);% Draw x axis
    \draw [thin, gray,-latex] (YAxisMin) -- (YAxisMax);% Draw y axis

   % Clips the picture...
    %\pgftransformcm{1}{0.6}{0.7}{1}{\pgfpoint{0cm}{0cm}}

    \foreach \x in {-5,...,5}{
      \foreach \y in {-5,-4,...,5}{
        \node[draw,circle,inner sep=0.8pt,fill] at (1*\x,1*\y) {};
            
      }
    }

\fill[red] (-5,0) circle (6pt);
\fill[red] (-4,0) circle (6pt);
\fill[red] (-3,0) circle (6pt);
\fill[red] (-2,0) circle (6pt);
\fill[red] (-1,0) circle (6pt);
\fill[red] (0,0) circle (6pt);
\fill[red] (1,0) circle (6pt);
\fill[red] (2,0) circle (6pt);
\fill[red] (3,0) circle (6pt);
\fill[red] (4,0) circle (6pt);
\fill[red] (5,0) circle (6pt);

\fill[red] (0,-5) circle (6pt);
\fill[red] (0,-4) circle (6pt);
\fill[red] (0,-3) circle (6pt);
\fill[red] (0,-2) circle (6pt);
\fill[red] (0,-1) circle (6pt);
\fill[red] (0,0) circle (6pt);
\fill[red] (0,1) circle (6pt);
\fill[red] (0,2) circle (6pt);
\fill[red] (0,3) circle (6pt);
\fill[red] (0,4) circle (6pt);
\fill[red] (0,5) circle (6pt);

\end{tikzpicture}
}

\newcommand{\PicCubeOne}{
\begin{tikzpicture}[scale=0.5]

    \coordinate (Origin)   at (0,0);
    \coordinate (XAxisMin) at (-5,0);
    \coordinate (XAxisMax) at (5,0);
    \coordinate (YAxisMin) at (0,-5);
    \coordinate (YAxisMax) at (0,5);
    \draw [thin, gray,-latex] (XAxisMin) -- (XAxisMax);% Draw x axis
    \draw [thin, gray,-latex] (YAxisMin) -- (YAxisMax);% Draw y axis

   % Clips the picture...
    %\pgftransformcm{1}{0.6}{0.7}{1}{\pgfpoint{0cm}{0cm}}

    \foreach \x in {-5,...,5}{
      \foreach \y in {-5,-4,...,5}{
        \node[draw,circle,inner sep=0.8pt,fill] at (1*\x,1*\y) {};
            
      }
    }

\fill[blue] (-5,0) circle (6pt);
\fill[blue] (-4,0) circle (6pt);
\fill[blue] (-3,0) circle (6pt);
\fill[blue] (-2,0) circle (6pt);
\fill[blue] (-1,0) circle (6pt);
\fill[blue] (0,0) circle (6pt);
\fill[blue] (1,0) circle (6pt);
\fill[blue] (2,0) circle (6pt);
\fill[blue] (3,0) circle (6pt);
\fill[blue] (4,0) circle (6pt);
\fill[blue] (5,0) circle (6pt);

\fill[blue] (0,-5) circle (6pt);
\fill[blue] (0,-4) circle (6pt);
\fill[blue] (0,-3) circle (6pt);
\fill[blue] (0,-2) circle (6pt);
\fill[blue] (0,-1) circle (6pt);
\fill[blue] (0,0) circle (6pt);
\fill[blue] (0,1) circle (6pt);
\fill[blue] (0,2) circle (6pt);
\fill[blue] (0,3) circle (6pt);
\fill[blue] (0,4) circle (6pt);
\fill[blue] (0,5) circle (6pt);

\fill[red] (1,1) circle (6pt);
\fill[red] (1,-1) circle (6pt);
\fill[red] (-1,1) circle (6pt);
\fill[red] (-1,-1) circle (6pt);

\end{tikzpicture}
}

\newcommand{\PicCubeTwo}{
\begin{tikzpicture}[scale=0.5]

    \coordinate (Origin)   at (0,0);
    \coordinate (XAxisMin) at (-5,0);
    \coordinate (XAxisMax) at (5,0);
    \coordinate (YAxisMin) at (0,-5);
    \coordinate (YAxisMax) at (0,5);
    \draw [thin, gray,-latex] (XAxisMin) -- (XAxisMax);% Draw x axis
    \draw [thin, gray,-latex] (YAxisMin) -- (YAxisMax);% Draw y axis

   % Clips the picture...
    %\pgftransformcm{1}{0.6}{0.7}{1}{\pgfpoint{0cm}{0cm}}

    \foreach \x in {-5,...,5}{
      \foreach \y in {-5,-4,...,5}{
        \node[draw,circle,inner sep=0.8pt,fill] at (1*\x,1*\y) {};
            
      }
    }

\fill[blue] (-5,0) circle (6pt);
\fill[blue] (-4,0) circle (6pt);
\fill[blue] (-3,0) circle (6pt);
\fill[blue] (-2,0) circle (6pt);
\fill[blue] (-1,0) circle (6pt);
\fill[blue] (0,0) circle (6pt);
\fill[blue] (1,0) circle (6pt);
\fill[blue] (2,0) circle (6pt);
\fill[blue] (3,0) circle (6pt);
\fill[blue] (4,0) circle (6pt);
\fill[blue] (5,0) circle (6pt);

\fill[blue] (0,-5) circle (6pt);
\fill[blue] (0,-4) circle (6pt);
\fill[blue] (0,-3) circle (6pt);
\fill[blue] (0,-2) circle (6pt);
\fill[blue] (0,-1) circle (6pt);
\fill[blue] (0,0) circle (6pt);
\fill[blue] (0,1) circle (6pt);
\fill[blue] (0,2) circle (6pt);
\fill[blue] (0,3) circle (6pt);
\fill[blue] (0,4) circle (6pt);
\fill[blue] (0,5) circle (6pt);

\fill[blue] (1,1) circle (6pt);
\fill[blue] (1,-1) circle (6pt);
\fill[blue] (-1,1) circle (6pt);
\fill[blue] (-1,-1) circle (6pt);

\fill[red] (2,1) circle (6pt);
\fill[red] (1,2) circle (6pt);
\fill[red] (2,-1) circle (6pt);
\fill[red] (1,-2) circle (6pt);
\fill[red] (-2,1) circle (6pt);
\fill[red] (-1,2) circle (6pt);
\fill[red] (-2,-1) circle (6pt);
\fill[red] (-1,-2) circle (6pt);

\end{tikzpicture}
}

\newcommand{\PicCubeThree}{
\begin{tikzpicture}[scale=0.5]

    \coordinate (Origin)   at (0,0);
    \coordinate (XAxisMin) at (-5,0);
    \coordinate (XAxisMax) at (5,0);
    \coordinate (YAxisMin) at (0,-5);
    \coordinate (YAxisMax) at (0,5);
    \draw [thin, gray,-latex] (XAxisMin) -- (XAxisMax);% Draw x axis
    \draw [thin, gray,-latex] (YAxisMin) -- (YAxisMax);% Draw y axis

   % Clips the picture...
    %\pgftransformcm{1}{0.6}{0.7}{1}{\pgfpoint{0cm}{0cm}}

    \foreach \x in {-5,...,5}{
      \foreach \y in {-5,-4,...,5}{
        \node[draw,circle,inner sep=0.8pt,fill] at (1*\x,1*\y) {};
            
      }
    }

\fill[blue] (-5,0) circle (6pt);
\fill[blue] (-4,0) circle (6pt);
\fill[blue] (-3,0) circle (6pt);
\fill[blue] (-2,0) circle (6pt);
\fill[blue] (-1,0) circle (6pt);
\fill[blue] (0,0) circle (6pt);
\fill[blue] (1,0) circle (6pt);
\fill[blue] (2,0) circle (6pt);
\fill[blue] (3,0) circle (6pt);
\fill[blue] (4,0) circle (6pt);
\fill[blue] (5,0) circle (6pt);

\fill[blue] (0,-5) circle (6pt);
\fill[blue] (0,-4) circle (6pt);
\fill[blue] (0,-3) circle (6pt);
\fill[blue] (0,-2) circle (6pt);
\fill[blue] (0,-1) circle (6pt);
\fill[blue] (0,0) circle (6pt);
\fill[blue] (0,1) circle (6pt);
\fill[blue] (0,2) circle (6pt);
\fill[blue] (0,3) circle (6pt);
\fill[blue] (0,4) circle (6pt);
\fill[blue] (0,5) circle (6pt);

\fill[blue] (1,1) circle (6pt);
\fill[blue] (1,-1) circle (6pt);
\fill[blue] (-1,1) circle (6pt);
\fill[blue] (-1,-1) circle (6pt);

\fill[blue] (2,1) circle (6pt);
\fill[blue] (1,2) circle (6pt);
\fill[blue] (2,-1) circle (6pt);
\fill[blue] (1,-2) circle (6pt);
\fill[blue] (-2,1) circle (6pt);
\fill[blue] (-1,2) circle (6pt);
\fill[blue] (-2,-1) circle (6pt);
\fill[blue] (-1,-2) circle (6pt);

\fill[red] (3,1) circle (6pt);
\fill[red] (2,2) circle (6pt);
\fill[red] (1,3) circle (6pt);
\fill[red] (3,-1) circle (6pt);
\fill[red] (2,-2) circle (6pt);
\fill[red] (1,-3) circle (6pt);
\fill[red] (-3,1) circle (6pt);
\fill[red] (-2,2) circle (6pt);
\fill[red] (-1,3) circle (6pt);
\fill[red] (-3,-1) circle (6pt);
\fill[red] (-2,-2) circle (6pt);
\fill[red] (-1,-3) circle (6pt);

\end{tikzpicture}
}

\newcommand{\PicDegCubeZero}{
\begin{tikzpicture}[scale=0.5]

    \coordinate (Origin)   at (0,0);
    \coordinate (XAxisMin) at (-5,0);
    \coordinate (XAxisMax) at (5,0);
    \coordinate (YAxisMin) at (0,-5);
    \coordinate (YAxisMax) at (0,5);
    \draw [thin, gray,-latex] (XAxisMin) -- (XAxisMax);% Draw x axis
    \draw [thin, gray,-latex] (YAxisMin) -- (YAxisMax);% Draw y axis

   % Clips the picture...
    %\pgftransformcm{1}{0.6}{0.7}{1}{\pgfpoint{0cm}{0cm}}

    \foreach \x in {-5,...,5}{
      \foreach \y in {-5,-4,...,5}{
        \node[draw,circle,inner sep=0.8pt,fill] at (1*\x,1*\y) {};
            
      }
    }

\fill[red] (0,0) circle (6pt);

\fill[red] (-4,5) circle (6pt);
\fill[red] (-3,4) circle (6pt);
\fill[red] (-2,3) circle (6pt);
\fill[red] (-1,2) circle (6pt);
\fill[red] (0,1) circle (6pt);
\fill[red] (1,0) circle (6pt);
\fill[red] (2,-1) circle (6pt);
\fill[red] (3,-2) circle (6pt);
\fill[red] (4,-3) circle (6pt);
\fill[red] (5,-4) circle (6pt);

\fill[red] (-5,4) circle (6pt);
\fill[red] (-4,3) circle (6pt);
\fill[red] (-3,2) circle (6pt);
\fill[red] (-2,1) circle (6pt);
\fill[red] (-1,0) circle (6pt);
\fill[red] (0,-1) circle (6pt);
\fill[red] (1,-2) circle (6pt);
\fill[red] (2,-3) circle (6pt);
\fill[red] (3,-4) circle (6pt);
\fill[red] (4,-5) circle (6pt);

\end{tikzpicture}
}

\newcommand{\PicDegCubeOne}{
\begin{tikzpicture}[scale=0.5]

    \coordinate (Origin)   at (0,0);
    \coordinate (XAxisMin) at (-5,0);
    \coordinate (XAxisMax) at (5,0);
    \coordinate (YAxisMin) at (0,-5);
    \coordinate (YAxisMax) at (0,5);
    \draw [thin, gray,-latex] (XAxisMin) -- (XAxisMax);% Draw x axis
    \draw [thin, gray,-latex] (YAxisMin) -- (YAxisMax);% Draw y axis

   % Clips the picture...
    %\pgftransformcm{1}{0.6}{0.7}{1}{\pgfpoint{0cm}{0cm}}

    \foreach \x in {-5,...,5}{
      \foreach \y in {-5,-4,...,5}{
        \node[draw,circle,inner sep=0.8pt,fill] at (1*\x,1*\y) {};
            
      }
    }

\fill[blue] (0,0) circle (6pt);

\fill[blue] (-4,5) circle (6pt);
\fill[blue] (-3,4) circle (6pt);
\fill[blue] (-2,3) circle (6pt);
\fill[blue] (-1,2) circle (6pt);
\fill[blue] (0,1) circle (6pt);
\fill[blue] (1,0) circle (6pt);
\fill[blue] (2,-1) circle (6pt);
\fill[blue] (3,-2) circle (6pt);
\fill[blue] (4,-3) circle (6pt);
\fill[blue] (5,-4) circle (6pt);

\fill[blue] (-5,4) circle (6pt);
\fill[blue] (-4,3) circle (6pt);
\fill[blue] (-3,2) circle (6pt);
\fill[blue] (-2,1) circle (6pt);
\fill[blue] (-1,0) circle (6pt);
\fill[blue] (0,-1) circle (6pt);
\fill[blue] (1,-2) circle (6pt);
\fill[blue] (2,-3) circle (6pt);
\fill[blue] (3,-4) circle (6pt);
\fill[blue] (4,-5) circle (6pt);

\fill[red] (-1,1) circle (6pt);
\fill[red] (1,-1) circle (6pt);
\fill[red] (1,1) circle (6pt);
\fill[red] (-1,-1) circle (6pt);

\end{tikzpicture}
}

\newcommand{\PicDegCubeTwo}{
\begin{tikzpicture}[scale=0.5]

    \coordinate (Origin)   at (0,0);
    \coordinate (XAxisMin) at (-5,0);
    \coordinate (XAxisMax) at (5,0);
    \coordinate (YAxisMin) at (0,-5);
    \coordinate (YAxisMax) at (0,5);
    \draw [thin, gray,-latex] (XAxisMin) -- (XAxisMax);% Draw x axis
    \draw [thin, gray,-latex] (YAxisMin) -- (YAxisMax);% Draw y axis

   % Clips the picture...
    %\pgftransformcm{1}{0.6}{0.7}{1}{\pgfpoint{0cm}{0cm}}

    \foreach \x in {-5,...,5}{
      \foreach \y in {-5,-4,...,5}{
        \node[draw,circle,inner sep=0.8pt,fill] at (1*\x,1*\y) {};
            
      }
    }

\fill[blue] (0,0) circle (6pt);

\fill[blue] (-4,5) circle (6pt);
\fill[blue] (-3,4) circle (6pt);
\fill[blue] (-2,3) circle (6pt);
\fill[blue] (-1,2) circle (6pt);
\fill[blue] (0,1) circle (6pt);
\fill[blue] (1,0) circle (6pt);
\fill[blue] (2,-1) circle (6pt);
\fill[blue] (3,-2) circle (6pt);
\fill[blue] (4,-3) circle (6pt);
\fill[blue] (5,-4) circle (6pt);

\fill[blue] (-5,4) circle (6pt);
\fill[blue] (-4,3) circle (6pt);
\fill[blue] (-3,2) circle (6pt);
\fill[blue] (-2,1) circle (6pt);
\fill[blue] (-1,0) circle (6pt);
\fill[blue] (0,-1) circle (6pt);
\fill[blue] (1,-2) circle (6pt);
\fill[blue] (2,-3) circle (6pt);
\fill[blue] (3,-4) circle (6pt);
\fill[blue] (4,-5) circle (6pt);

\fill[blue] (-1,1) circle (6pt);
\fill[blue] (1,-1) circle (6pt);
\fill[blue] (1,1) circle (6pt);
\fill[blue] (-1,-1) circle (6pt);

\fill[red] (-2,2) circle (6pt);
\fill[red] (2,-2) circle (6pt);
\fill[red] (0,2) circle (6pt);
\fill[red] (2,0) circle (6pt);
\fill[red] (0,-2) circle (6pt);
\fill[red] (-2,0) circle (6pt);

\end{tikzpicture}
}

\newcommand{\PicDegCubeThree}{
\begin{tikzpicture}[scale=0.5]

    \coordinate (Origin)   at (0,0);
    \coordinate (XAxisMin) at (-5,0);
    \coordinate (XAxisMax) at (5,0);
    \coordinate (YAxisMin) at (0,-5);
    \coordinate (YAxisMax) at (0,5);
    \draw [thin, gray,-latex] (XAxisMin) -- (XAxisMax);% Draw x axis
    \draw [thin, gray,-latex] (YAxisMin) -- (YAxisMax);% Draw y axis

   % Clips the picture...
    %\pgftransformcm{1}{0.6}{0.7}{1}{\pgfpoint{0cm}{0cm}}

    \foreach \x in {-5,...,5}{
      \foreach \y in {-5,-4,...,5}{
        \node[draw,circle,inner sep=0.8pt,fill] at (1*\x,1*\y) {};
            
      }
    }

\fill[blue] (0,0) circle (6pt);

\fill[blue] (-4,5) circle (6pt);
\fill[blue] (-3,4) circle (6pt);
\fill[blue] (-2,3) circle (6pt);
\fill[blue] (-1,2) circle (6pt);
\fill[blue] (0,1) circle (6pt);
\fill[blue] (1,0) circle (6pt);
\fill[blue] (2,-1) circle (6pt);
\fill[blue] (3,-2) circle (6pt);
\fill[blue] (4,-3) circle (6pt);
\fill[blue] (5,-4) circle (6pt);

\fill[blue] (-5,4) circle (6pt);
\fill[blue] (-4,3) circle (6pt);
\fill[blue] (-3,2) circle (6pt);
\fill[blue] (-2,1) circle (6pt);
\fill[blue] (-1,0) circle (6pt);
\fill[blue] (0,-1) circle (6pt);
\fill[blue] (1,-2) circle (6pt);
\fill[blue] (2,-3) circle (6pt);
\fill[blue] (3,-4) circle (6pt);
\fill[blue] (4,-5) circle (6pt);

\fill[blue] (-1,1) circle (6pt);
\fill[blue] (1,-1) circle (6pt);
\fill[blue] (1,1) circle (6pt);
\fill[blue] (-1,-1) circle (6pt);

\fill[blue] (-2,2) circle (6pt);
\fill[blue] (2,-2) circle (6pt);

\fill[blue] (-2,0) circle (6pt);
\fill[blue] (2,0) circle (6pt);
\fill[blue] (0,2) circle (6pt);
\fill[blue] (0,-2) circle (6pt);

\fill[red] (-3,3) circle (6pt);
\fill[red] (3,-3) circle (6pt);
\fill[red] (-1,3) circle (6pt);
\fill[red] (1,2) circle (6pt);
\fill[red] (2,1) circle (6pt);
\fill[red] (1,-3) circle (6pt);
\fill[red] (-1,-2) circle (6pt);
\fill[red] (-2,-1) circle (6pt);
\fill[red] (-3,1) circle (6pt);
\fill[red] (3,-1) circle (6pt);

\end{tikzpicture}
}

\newcommand{\ronegraph}{
\begin{center}
\begin{tikzpicture}
\tikzset{vertex/.style = {shape = circle,fill=black,minimum size=0.1cm}}
%\tikzset{edge/.style = {-,> = latex'}}
% vertices
\node[vertex] (ta) at  (-3,2) {};
\node[vertex] (tb) at  (-2,2) {};
\node[vertex] (tc) at  (-1,2) {};
\node[vertex] (td) at  (0,2) {};
\node[vertex] (te) at  (1,2) {};
\node[vertex] (tf) at  (2,2) {};
\node[vertex] (tg) at  (3,2) {};

\node[vertex] (ba) at  (-3,0) {};
\node[vertex] (bb) at  (-2,0) {};
\node[vertex] (bc) at  (-1,0) {};
\node[vertex] (bd) at  (0,0) {};
\node[vertex] (be) at  (1,0) {};
\node[vertex] (bf) at  (2,0) {};
\node[vertex] (bg) at  (3,0) {};
%edges
%\draw[edge] (b) to[bend left=10]  (c);
%\draw[edge] (b) to[bend right=10] (c);

%\draw[edge] (a) to[bend left=10] (b);
%\draw[edge] (a) to[bend right=10] (b);

\draw[red] (ta) to (bg);
\draw[red] (tb) to (bf);
\draw[red] (tc) to (be);
\draw[red] (td) to (bd);
\draw[red] (te) to (bc);
\draw[red] (tf) to (bb);
\draw[red] (tg) to (ba);

\node[draw=none,fill=none] at (-3.6,2) {$\boldsymbol{\cdots} $};
\node[draw=none,fill=none] at (-3,2.5) {$-3$};
\node[draw=none,fill=none] at (-2,2.5) {$-2$};
\node[draw=none,fill=none] at (-1,2.5) {$-1$};
\node[draw=none,fill=none] at (0,2.5) {$0$};
\node[draw=none,fill=none] at (1,2.5) {$1$};
\node[draw=none,fill=none] at (2,2.5) {$2$};
\node[draw=none,fill=none] at (3,2.5) {$3$};
\node[draw=none,fill=none] at (3.7,2) {$\boldsymbol{\cdots} $};

\node[draw=none,fill=none] at (-3.6,0) {$\boldsymbol{\cdots} $};
\node[draw=none,fill=none] at (-3,-0.5) {$-3$};
\node[draw=none,fill=none] at (-2,-0.5) {$-2$};
\node[draw=none,fill=none] at (-1,-0.5) {$-1$};
\node[draw=none,fill=none] at (0,-0.5) {$0$};
\node[draw=none,fill=none] at (1,-0.5) {$1$};
\node[draw=none,fill=none] at (2,-0.5) {$2$};
\node[draw=none,fill=none] at (3,-0.5) {$3$};
\node[draw=none,fill=none] at (3.7,0) {$\boldsymbol{\cdots} $};

\end{tikzpicture}
\end{center}

}

\newcommand{\rtwograph}{
\begin{center}
\begin{tikzpicture}
\tikzset{vertex/.style = {shape = circle,fill=black,minimum size=0.1cm}}
%\tikzset{edge/.style = {-,> = latex'}}
% vertices

\node[vertex] (tb) at  (-2,2) {};
\node[vertex] (tc) at  (-1,2) {};
\node[vertex] (td) at  (0,2) {};
\node[vertex] (te) at  (1,2) {};
\node[vertex] (tf) at  (2,2) {};
\node[vertex] (tg) at  (3,2) {};

\node[vertex] (bb) at  (-2,0) {};
\node[vertex] (bc) at  (-1,0) {};
\node[vertex] (bd) at  (0,0) {};
\node[vertex] (be) at  (1,0) {};
\node[vertex] (bf) at  (2,0) {};
\node[vertex] (bg) at  (3,0) {};
%edges
%\draw[edge] (b) to[bend left=10]  (c);
%\draw[edge] (b) to[bend right=10] (c);

%\draw[edge] (a) to[bend left=10] (b);
%\draw[edge] (a) to[bend right=10] (b);

\draw[red] (tb) to (bg);
\draw[red] (tc) to (bf);
\draw[red] (td) to (bd);
\draw[red] (te) to (be);
\draw[red] (tf) to (bc);
\draw[red] (tg) to (bb);

\node[draw=none,fill=none] at (-2.6,2) {$\boldsymbol{\cdots} $};
%\node[draw=none,fill=none] at (-3,2.5) {$-3$};
\node[draw=none,fill=none] at (-2,2.5) {$-2$};
\node[draw=none,fill=none] at (-1,2.5) {$-1$};
\node[draw=none,fill=none] at (0,2.5) {$0$};
\node[draw=none,fill=none] at (1,2.5) {$1$};
\node[draw=none,fill=none] at (2,2.5) {$2$};
\node[draw=none,fill=none] at (3,2.5) {$3$};
\node[draw=none,fill=none] at (3.7,2) {$\boldsymbol{\cdots} $};

\node[draw=none,fill=none] at (-2.6,0) {$\boldsymbol{\cdots} $};
%\node[draw=none,fill=none] at (-3,-0.5) {$-3$};
\node[draw=none,fill=none] at (-2,-0.5) {$-2$};
\node[draw=none,fill=none] at (-1,-0.5) {$-1$};
\node[draw=none,fill=none] at (0,-0.5) {$0$};
\node[draw=none,fill=none] at (1,-0.5) {$1$};
\node[draw=none,fill=none] at (2,-0.5) {$2$};
\node[draw=none,fill=none] at (3,-0.5) {$3$};
\node[draw=none,fill=none] at (3.7,0) {$\boldsymbol{\cdots} $};

\end{tikzpicture}
\end{center}

}

\newcommand{\diagramone}
{

\begin{tikzpicture}[scale=0.5]
	\node (A) at (0,0) {A};
    \node (B) at (0,3) {B};
    \node (C) at (3,3) {C};
    \node (D) at (3,0) {D};
    
    \node (E) at (-6,-5) {E};
    \node (F) at (-6,-2) {F};
    \node (G) at (-3,-2) {G};
    \node (H) at (-3,-5) {H};
    
    \node (I) at (6,-5) {I};
    \node (J) at (6,-2) {J};
    \node (K) at (9,-2) {K};
    \node (L) at (9,-5) {L};

    \path [->,blue] (A) edge node[left,black] {$e_j$} (B);
    \path [->,red] (B) edge node[above,black] {$e_i$} (C);
    \path [->,red] (A) edge node {$.$} (D);
    \path [->,blue] (D) edge node {$.$} (C);
    
    \path [->,blue] (E) edge node {$.$} (F);
    \path [->,red] (F) edge node {$.$} (G);
    \path [->,red] (E) edge node {$.$} (H);
    \path [->,blue] (H) edge node {$.$} (G);

    \path [->,blue] (I) edge node {$.$} (J);
    \path [->,red] (J) edge node {$.$} (K);
    \path [->,red] (I) edge node {$.$} (L);
    \path [->,blue] (L) edge node {$.$} (K);

     \path [->,teal] (E) edge node {$.$} (A);
    \path [->,teal] (F) edge node[above,black] {$e_k$} (B);
    \path [->,teal] (G) edge node {$.$} (C);
    \path [->,teal] (H) edge node {$.$} (D);
    
    \path [->,orange] (I) edge node {$.$} (A);
    \path [->,orange] (J) edge node {$.$} (B);
    \path [->,orange] (K) edge node[above,black] {$e_{k'}$} (C);
    \path [->,orange] (L) edge node {$.$} (D);

\end{tikzpicture}

}

%diagramtwo
\newcommand{\diagramtwo}
{
\begin{tikzpicture}[scale=0.5]
	\node (A) at (0,0) {a};
    \node (B) at (0,3) {a+1};
    \node (C) at (3,3) {a+1};
    \node (D) at (3,0) {a+1};
    
    \node (E) at (-6,-5) {a-1};
    \node[text=red] (F) at (-6,-2) {a};
    \node (G) at (-3,-2) {*};
    \node[text=red] (H) at (-3,-5) {a};
    
    \node (I) at (6,-5) {**};
    \node[text=red] (J) at (6,-2) {a};
    \node (K) at (9,-2) {a};
    \node[text=red] (L) at (9,-5) {a};

    \path [->,blue] (A) edge node {$.$} (B);
    \path [->,red] (B) edge node {$.$} (C);
    \path [->,red] (A) edge node {$.$} (D);
    \path [->,blue] (D) edge node {$.$} (C);
    
    \path [->,blue] (E) edge node {$.$} (F);
    \path [->,red] (F) edge node {$.$} (G);
    \path [->,red] (E) edge node {$.$} (H);
    \path [->,blue] (H) edge node {$.$} (G);

    \path [->,blue] (I) edge node {$.$} (J);
    \path [->,red] (J) edge node {$.$} (K);
    \path [->,red] (I) edge node {$.$} (L);
    \path [->,blue] (L) edge node {$.$} (K);

     \path [->,teal] (E) edge node {$.$} (A);
    \path [->,teal] (F) edge node {$.$} (B);
    \path [->,teal] (G) edge node {$.$} (C);
    \path [->,teal] (H) edge node {$.$} (D);
    
    \path [->,orange] (I) edge node {$.$} (A);
    \path [->,orange] (J) edge node {$.$} (B);
    \path [->,orange] (K) edge node {$.$} (C);
    \path [->,orange] (L) edge node {$.$} (D);

\end{tikzpicture}
}

%diagramthree
\newcommand{\diagramthree}
{
\begin{tikzpicture}[scale=0.5]
	\node (A) at (0,0) {a};
    \node (B) at (0,3) {a+1};
    \node (C) at (3,3) {a+1};
    \node (D) at (3,0) {a+1};
    
    \node (E) at (-6,-5) {a-1};
    \node (F) at (-6,-2) {a};
    \node (G) at (-3,-2) {a};
    \node (H) at (-3,-5) {a+1};
    
    \node (I) at (6,-5) {a};
    \node (J) at (6,-2) {a};
    \node (K) at (9,-2) {a};
    \node (L) at (9,-5) {a};
	
	\node[text=red] (M) at (0,-10) {a-1};
    \node (N) at (0,-7) {*};
    \node[text=red] (O) at (3,-7) {a};
    \node (P) at (3,-10) {**};

    \path [->,blue] (A) edge node {$.$} (B);
    \path [->,red] (B) edge node {$.$} (C);
    \path [->,red] (A) edge node {$.$} (D);
    \path [->,blue] (D) edge node {$.$} (C);
    
    \path [->,blue] (E) edge node {$.$} (F);
    \path [->,red] (F) edge node {$.$} (G);
    \path [->,red] (E) edge node {$.$} (H);
    \path [->,blue] (H) edge node {$.$} (G);

    \path [->,blue] (I) edge node {$.$} (J);
    \path [->,red] (J) edge node {$.$} (K);
    \path [->,red] (I) edge node {$.$} (L);
    \path [->,blue] (L) edge node {$.$} (K);
    
    \path [->,blue] (M) edge node {$.$} (N);
    \path [->,red] (N) edge node {$.$} (O);
    \path [->,red] (M) edge node {$.$} (P);
    \path [->,blue] (P) edge node {$.$} (O);

    \path [->,teal] (E) edge node {$.$} (A);
    \path [->,teal] (F) edge node {$.$} (B);
    \path [->,teal] (G) edge node {$.$} (C);
    \path [->,teal] (H) edge node {$.$} (D);
    
    \path [->,orange] (I) edge node {$.$} (A);
    \path [->,orange] (J) edge node {$.$} (B);
    \path [->,orange] (K) edge node {$.$} (C);
    \path [->,orange] (L) edge node {$.$} (D);

    \path [->,orange] (M) edge node {$.$} (E);
    \path [->,orange] (N) edge node {$.$} (F);
    \path [->,orange] (O) edge node {$.$} (G);
    \path [->,orange] (P) edge node {$.$} (H);
    
    \path [->,teal] (M) edge node {$.$} (I);
    \path [->,teal] (N) edge node {$.$} (J);
    \path [->,teal] (O) edge node {$.$} (K);
    \path [->,teal] (P) edge node {$.$} (L);

\end{tikzpicture}

}

\newcommand{\DiagramCDOne}{
 \begin{tikzpicture}[scale=0.5]
    \coordinate (Origin)   at (0,0);
    \coordinate (XAxisMin) at (-4,0);
    \coordinate (XAxisMax) at (6,0);
    \coordinate (YAxisMin) at (0,-4);
    \coordinate (YAxisMax) at (0,6);
    \draw [thin, black,-latex] (XAxisMin) -- (XAxisMax);% Draw x axis
    \draw [thin, black,-latex] (YAxisMin) -- (YAxisMax);% Draw y axis

    \clip (-5,-5) rectangle (10cm,10cm); 
    \foreach \x in {-4,-3,...,6}{% Two indices running over each
      \foreach \y in {-4,-3,...,6}{% node on the grid we have drawn 
        \node[draw,circle,inner sep=1.2pt,fill] at (1*\x,1*\y) {};
            % Places a dot at those points
      }
    }

\fill[red] (-2,4) circle (7pt);
\fill[red] (1,1) circle (7pt);
\fill[red] (4,-2) circle (7pt);

\fill[red] (-3,6) circle (7pt);
\fill[red] (0,3) circle (7pt);
\fill[red] (3,0) circle (7pt);
\fill[red] (6,-3) circle (7pt);

\fill[red] (-4,5) circle (7pt);
\fill[red] (-1,2) circle (7pt);
\fill[red] (2,-1) circle (7pt);
\fill[red] (5,-4) circle (7pt);

\end{tikzpicture}
}

\newcommand{\DiagramCDTwo}{
\begin{tikzpicture}[scale=0.5]
    \coordinate (Origin)   at (0,0);
    \coordinate (XAxisMin) at (-4,0);
    \coordinate (XAxisMax) at (6,0);
    \coordinate (YAxisMin) at (0,-4);
    \coordinate (YAxisMax) at (0,6);
    \draw [thin, black,-latex] (XAxisMin) -- (XAxisMax);% Draw x axis
    \draw [thin, black,-latex] (YAxisMin) -- (YAxisMax);% Draw y axis

    \clip (-5,-5) rectangle (10cm,10cm); 
    \foreach \x in {-4,-3,...,6}{% Two indices running over each
      \foreach \y in {-4,-3,...,6}{% node on the grid we have drawn 
        \node[draw,circle,inner sep=1.2pt,fill] at (1*\x,1*\y) {};
            % Places a dot at those points
      }
    }

\fill[red] (-2,4) circle (7pt);
\fill[red] (1,1) circle (7pt);
\fill[red] (4,-2) circle (7pt);

\fill[red] (-3,6) circle (7pt);
\fill[red] (0,3) circle (7pt);
\fill[red] (3,0) circle (7pt);
\fill[red] (6,-3) circle (7pt);

\fill[red] (-4,5) circle (7pt);
\fill[red] (-1,2) circle (7pt);
\fill[red] (2,-1) circle (7pt);
\fill[red] (5,-4) circle (7pt);

\end{tikzpicture}
}

\newcommand{\ThreeVertex}{
\begin{center}
\begin{tikzpicture}
\tikzset{vertex/.style = {shape = circle,fill=black,minimum size=0.1cm}}
\tikzset{edge/.style = {-,> = latex'}}
% vertices
\node[vertex] (b) at  (1.4,2) {};
\node[vertex] (a) at (-1.4,2) {};
\node[vertex] (c) at (0,0) {};
%edges
\draw[edge] (b) to[bend left=10]  (c);
\draw[edge] (b) to[bend left=20] node[below right] {t}  (c);
\draw[edge] (b) to[bend right=10] node[below,rotate=50] {$\mathellipsis$} (c);
\draw[edge] (b) to[bend right=20] (c);

\draw[edge] (a) to[bend left=10] (b);
\draw[edge] (a) to[bend left=20] node[above] {r}  (b);
\draw[edge] (a) to[bend right=10] node[above] {$\mathellipsis$} (b);
\draw[edge] (a) to[bend right=20]  (b);

\draw[edge] (a) to[bend left=10] (c);
\draw[edge] (a) to[bend left=20]  (c);
\draw[edge] (a) to[bend right=10] node[above,rotate=-50] {$\mathellipsis$} (c);
\draw[edge] (a) to[bend right=20] node[below left]{s} (c);

\node[draw=none,fill=none] at (0.3,-0.3) {$v_3$};
\node[draw=none,fill=none] at (-1.7,2.3) {$v_1$};
\node[draw=none,fill=none] at (1.7,2.3) {$v_2$};
\end{tikzpicture}
\end{center}

}

\newcommand{\DiagramCDThree}{
\begin{tikzpicture}[scale=0.65]
  
  \draw[fill=blue!15!white,blue!15!white] (5,3)--(-4,3)--(-4,-10)--(5,-10)--cycle;   
  \draw[fill=green!15!white,green!15!white] (2,5)--(-4,5)--(-4,-10)--(2,-10)--cycle; 
  \draw[fill=teal!15!white,teal!15!white] (2,3)--(-4,3)--(-4,-10)--(2,-10)--cycle;

    \coordinate (Origin)   at (0,0);
    \coordinate (XAxisMin) at (-4,0);
    \coordinate (XAxisMax) at (10,0);
    \coordinate (YAxisMin) at (0,-10);
    \coordinate (YAxisMax) at (0,10);
    \draw [thin, gray,-latex] (XAxisMin) -- (XAxisMax);% Draw x axis
    \draw [thin, gray,-latex] (YAxisMin) -- (YAxisMax);% Draw y axis

   % Clips the picture...
    %\pgftransformcm{1}{0.6}{0.7}{1}{\pgfpoint{0cm}{0cm}}

    \foreach \x in {-4,-3,...,10}{
      \foreach \y in {-10,-9,...,10}{
        \node[draw,circle,inner sep=0.8pt,fill] at (1*\x,1*\y) {};
            
      }
    }

\fill[red] (0,0) circle (6pt);
\fill[red] (1,2) circle (6pt);
\fill[red] (2,3) circle (6pt);
\fill[red] (3,-2) circle (6pt);
\fill[red] (4,1) circle (6pt);
\fill[red] (5,-1) circle (6pt);
\fill[red] (6,-4) circle (6pt);
\fill[red] (7,-3) circle (6pt);

\fill[red] (8,-8) circle (6pt);
\fill[red] (9,-6) circle (6pt);
\fill[red] (10,-5) circle (6pt);

\fill[red] (-1,5) circle (6pt);
\fill[red] (-2,4) circle (6pt);
\fill[red] (-3,7) circle (6pt);
\fill[red] (-4,9) circle (6pt);

\node[draw=none,fill=none] at (2,5.3) {\footnotesize $f(2,5)$};
\node[draw=none,fill=none] at (5,3.3) {\footnotesize $g(2,5)$};

\end{tikzpicture}
}

\newcommand{\ThreeVertexTwo}{
\begin{center}
\begin{tikzpicture}
\tikzset{vertex/.style = {shape = circle,fill=black,minimum size=0.1cm}}
\tikzset{edge/.style = {-,> = latex'}}
% vertices
\node[vertex] (b) at  (1.4,2) {};
\node[vertex] (a) at (-1.4,2) {};
\node[vertex] (c) at (0,0) {};
%edges
\draw[edge] (b) to[bend left=10]  (c);
\draw[edge] (b) to[bend right=10] (c);

\draw[edge] (a) to[bend left=10] (b);
\draw[edge] (a) to[bend right=10] (b);

\draw[edge] (a) to (c);

\node[draw=none,fill=none] at (0.3,-0.3) {$v_3$};
\node[draw=none,fill=none] at (-1.7,2.3) {$v_1$};
\node[draw=none,fill=none] at (1.7,2.3) {$v_2$};
\end{tikzpicture}
\end{center}
}

\newcommand{\FourVertex}{
\begin{center}
\begin{tikzpicture}
\tikzset{vertex/.style = {shape = circle,fill=black,minimum size=0.1cm}}
\tikzset{edge/.style = {-,> = latex'}}
% vertices
\node[vertex] (a) at  (0,2) {};
\node[vertex] (b) at (0,0) {};
\node[vertex] (c) at (2,2) {};
\node[vertex] (d) at (2,0) {};
%edges
\draw[edge] (a) to  (c);
\draw[edge] (a) to  (b);
\draw[edge] (a) to  (d);

\draw[edge] (b) to (c);
\draw[edge] (b) to[bend left=10]  (d);
\draw[edge] (b) to[bend right=10]  (d);

\node[draw=none,fill=none] at (-0.3,2.3) {$v_1$};
\node[draw=none,fill=none] at (-0.3,-0.3) {$v_2$};
\node[draw=none,fill=none] at (2.3,2.3) {$v_3$};
\node[draw=none,fill=none] at (2.3,-0.3) {$v_4$};
\end{tikzpicture}
\end{center}

}

\newcommand{\DiagramCDFour}{
\begin{tikzpicture}[scale=0.65]

    \coordinate (Origin)   at (0,0);
    \coordinate (XAxisMin) at (-4,0);
    \coordinate (XAxisMax) at (10,0);
    \coordinate (YAxisMin) at (0,-10);
    \coordinate (YAxisMax) at (0,10);
    \draw [thin, gray,-latex] (XAxisMin) -- (XAxisMax);% Draw x axis
    \draw [thin, gray,-latex] (YAxisMin) -- (YAxisMax);% Draw y axis

   % Clips the picture...
    %\pgftransformcm{1}{0.6}{0.7}{1}{\pgfpoint{0cm}{0cm}}

    \foreach \x in {-4,-3,...,10}{
      \foreach \y in {-10,-9,...,10}{
        \node[draw,circle,inner sep=0.8pt,fill] at (1*\x,1*\y) {};
            
      }
    }

\fill[red] (0,0) circle (6pt);
\fill[red] (1,2) circle (6pt);
\fill[red] (2,3) circle (6pt);
\fill[red] (3,-2) circle (6pt);
\fill[red] (5,1) circle (6pt);
\fill[red] (4,-1) circle (6pt);
\fill[red] (6,-4) circle (6pt);
\fill[red] (7,-3) circle (6pt);

\fill[red] (8,-8) circle (6pt);
\fill[red] (9,-6) circle (6pt);
\fill[red] (10,-5) circle (6pt);

\fill[red] (-1,5) circle (6pt);
\fill[red] (-2,4) circle (6pt);
\fill[red] (-3,7) circle (6pt);
\fill[red] (-4,9) circle (6pt);

\end{tikzpicture}
}

\newcommand{\DiagramCDFive}{
\begin{tikzpicture}[scale=0.4]

    \coordinate (Origin)   at (0,0);
    \coordinate (XAxisMin) at (-3,0);
    \coordinate (XAxisMax) at (12,0);
    \coordinate (YAxisMin) at (0,-10);
    \coordinate (YAxisMax) at (0,12);
    \draw [thin, gray,-latex] (XAxisMin) -- (XAxisMax);% Draw x axis
    \draw [thin, gray,-latex] (YAxisMin) -- (YAxisMax);% Draw y axis

   % Clips the picture...
    %\pgftransformcm{1}{0.6}{0.7}{1}{\pgfpoint{0cm}{0cm}}

    \foreach \x in {-3,...,11}{
      \foreach \y in {-10,-9,...,11}{
        \node[draw,circle,inner sep=0.8pt,fill] at (1*\x,1*\y) {};
            
      }
    }

\fill[red] (-3,6) circle (6pt);
\fill[red] (-2,7) circle (6pt);
\fill[red] (-1,8) circle (6pt);

\fill[red] (0,0) circle (6pt);
\fill[red] (1,10) circle (6pt);
\fill[red] (2,11) circle (6pt);

\fill[red] (3,3) circle (6pt);
\fill[red] (4,4) circle (6pt);
\fill[red] (5,5) circle (6pt);

\fill[red] (6,-3) circle (6pt);
\fill[red] (7,-2) circle (6pt);
\fill[red] (8,-1) circle (6pt);
\fill[red] (9,-9) circle (6pt);

\fill[red] (10,1) circle (6pt);
\fill[red] (11,2) circle (6pt);

\end{tikzpicture}
}

\newcommand{\DiagramCDEight}{
\begin{tikzpicture}[scale=0.4]

    \coordinate (Origin)   at (0,0);
    \coordinate (XAxisMin) at (-3,0);
    \coordinate (XAxisMax) at (12,0);
    \coordinate (YAxisMin) at (0,-10);
    \coordinate (YAxisMax) at (0,12);
    \draw [thin, gray,-latex] (XAxisMin) -- (XAxisMax);% Draw x axis
    \draw [thin, gray,-latex] (YAxisMin) -- (YAxisMax);% Draw y axis

   % Clips the picture...
    %\pgftransformcm{1}{0.6}{0.7}{1}{\pgfpoint{0cm}{0cm}}

    \foreach \x in {-3,...,11}{
      \foreach \y in {-10,-9,...,11}{
        \node[draw,circle,inner sep=0.8pt,fill] at (1*\x,1*\y) {};
            
      }
    }

\fill[red] (-3,6) circle (6pt);
\fill[red] (-2,7) circle (6pt);
\fill[red] (-1,8) circle (6pt);

\fill[red] (0,0) circle (6pt);
\fill[red] (1,1) circle (6pt);
\fill[red] (2,11) circle (6pt);

\fill[red] (3,3) circle (6pt);
\fill[red] (4,4) circle (6pt);
\fill[red] (5,5) circle (6pt);

\fill[red] (6,-3) circle (6pt);
\fill[red] (7,-2) circle (6pt);
\fill[red] (8,-1) circle (6pt);
\fill[red] (9,-9) circle (6pt);

\fill[red] (10,-8) circle (6pt);
\fill[red] (11,2) circle (6pt);

\end{tikzpicture}
}

\newcommand{\DiagramCDNine}{
\begin{tikzpicture}[scale=0.4]

    \coordinate (Origin)   at (0,0);
    \coordinate (XAxisMin) at (-3,0);
    \coordinate (XAxisMax) at (12,0);
    \coordinate (YAxisMin) at (0,-10);
    \coordinate (YAxisMax) at (0,12);
    \draw [thin, gray,-latex] (XAxisMin) -- (XAxisMax);% Draw x axis
    \draw [thin, gray,-latex] (YAxisMin) -- (YAxisMax);% Draw y axis

   % Clips the picture...
    %\pgftransformcm{1}{0.6}{0.7}{1}{\pgfpoint{0cm}{0cm}}

    \foreach \x in {-3,...,11}{
      \foreach \y in {-10,-9,...,11}{
        \node[draw,circle,inner sep=0.8pt,fill] at (1*\x,1*\y) {};
            
      }
    }

\fill[red] (-3,6) circle (6pt);
\fill[red] (-2,7) circle (6pt);
\fill[red] (-1,8) circle (6pt);

\fill[red] (0,0) circle (6pt);
\fill[red] (1,1) circle (6pt);
\fill[red] (2,2) circle (6pt);
\fill[red] (3,3) circle (6pt);
\fill[red] (4,4) circle (6pt);
\fill[red] (5,5) circle (6pt);

\fill[red] (6,-3) circle (6pt);
\fill[red] (7,-2) circle (6pt);
\fill[red] (8,-1) circle (6pt);
\fill[red] (9,-9) circle (6pt);

\fill[red] (10,-8) circle (6pt);
\fill[red] (11,-7) circle (6pt);

\end{tikzpicture}
}

\newcommand{\DiagramCDSeven}{
\begin{tikzpicture}[scale=0.65]

    \coordinate (Origin)   at (0,0);
    \coordinate (XAxisMin) at (-2,0);
    \coordinate (XAxisMax) at (9,0);
    \coordinate (YAxisMin) at (0,-6);
    \coordinate (YAxisMax) at (0,5);
    \draw [thin, gray,-latex] (XAxisMin) -- (XAxisMax);% Draw x axis
    \draw [thin, gray,-latex] (YAxisMin) -- (YAxisMax);% Draw y axis

   % Clips the picture...
    %\pgftransformcm{1}{0.6}{0.7}{1}{\pgfpoint{0cm}{0cm}}

    \foreach \x in {-2,...,9}{
      \foreach \y in {-6,-5,...,5}{
        \node[draw,circle,inner sep=0.8pt,fill] at (1*\x,1*\y) {};
            
      }
    }

\fill[red] (-2,4) circle (6pt);
\fill[red] (-1,5) circle (6pt);
\fill[red] (0,0) circle (6pt);
\fill[red] (1,1) circle (6pt);
\fill[red] (2,2) circle (6pt);
\fill[red] (3,3) circle (6pt);

\fill[red] (4,-2) circle (6pt);
\fill[red] (5,-1) circle (6pt);
\fill[red] (6,-6) circle (6pt);
\fill[red] (7,-5) circle (6pt);
\fill[red] (8,-4) circle (6pt);

\end{tikzpicture}
}

\newcommand{\DiagramCDTen}{
\begin{tikzpicture}[scale=0.7]
  
  \draw[fill=blue!15!white,green!15!white] (3,2)--(-3,2)--(-3,-10)--(3,-10)--cycle;   
  \draw[fill=green!15!white,blue!15!white] (4,3)--(11,3)--(11,11)--(4,11)--cycle; 
  \draw[fill=green!15!white,gray!15!white] (3,3)--(3,11)--(-3,11)--(-3,3)--cycle; 
  \draw[fill=green!15!white,gray!15!white] (4,2)--(11,2)--(11,-10)--(4,-10)--cycle;

    \coordinate (Origin)   at (0,0);
    \coordinate (XAxisMin) at (-3,0);
    \coordinate (XAxisMax) at (12,0);
    \coordinate (YAxisMin) at (0,-10);
    \coordinate (YAxisMax) at (0,12);
    \draw [thin, gray,-latex] (XAxisMin) -- (XAxisMax);% Draw x axis
    \draw [thin, gray,-latex] (YAxisMin) -- (YAxisMax);% Draw y axis

   % Clips the picture...
    %\pgftransformcm{1}{0.6}{0.7}{1}{\pgfpoint{0cm}{0cm}}

    \foreach \x in {-3,...,11}{
      \foreach \y in {-10,-9,...,11}{
        \node[draw,circle,inner sep=0.8pt,fill] at (1*\x,1*\y) {};
            
      }
    }

\fill[red] (-3,6) circle (6pt);
\fill[red] (-2,7) circle (6pt);
\fill[red] (-1,8) circle (6pt);

\fill[red] (0,0) circle (6pt);
\fill[red] (1,1) circle (6pt);
\fill[red] (2,2) circle (6pt);
\fill[red] (3,3) circle (6pt);
\fill[red] (4,4) circle (6pt);
\fill[red] (5,5) circle (6pt);

\fill[red] (6,-3) circle (6pt);
\fill[red] (7,-2) circle (6pt);
\fill[red] (8,-1) circle (6pt);
\fill[red] (9,-9) circle (6pt);

\fill[red] (10,-8) circle (6pt);
\fill[red] (11,-7) circle (6pt);

\end{tikzpicture}
}
                                % contains a number of diagrams

% Note: for arxiv version need this style of input:

\section{Introduction}  % This comes from se_pub_intro
			% Edited to just have Euler char stuff

The main goal of this article is to develop a way understand
a large class of what we call {\em generalized Riemann-Roch formulas} as
formulas that expresses an Euler characteristic of a certain
{\em sheaves} of vector spaces; we later show that such
sheaves satisfy a property akin to Serre duality for line bundles on curves.

This article does not assume any prior knowledge of sheaf theory.
In fact, we mostly speak of {\em $k$-diagrams}, where $k$ is an
arbitrary field, which is a
structure of five $k$-vector spaces with some linear transformations
between them.  We mention sheaves only in the last section of this article,
where 
explain the connection of $k$-diagrams and their invariants to sheaf theory.

This article was motivated by the question of Baker-Norine
\cite{baker_norine} as to whether their
``graph Riemann-Roch formula'' 
can be viewed as such an Euler characteristic formula.
However, our main results apply to any such formula
that arises from a much wider and simpler class of
functions that we call {\em Riemann functions}.
Roughly speaking, our main results says that any such formula
can be modeled as such, provided that (1) one is willing to work
with ``formal differences'' of $k$-diagrams (or sheaves), and
(2) one is willing to make a number of ad hoc choices in
building the model (which we will prove do not change
the equivalence class of the formal difference of $k$-diagrams).
We therefore view this article as a first step in modeling
Riemann-Roch formulas, that we hope will ultimately lead to better---meaning
simpler and less ad hoc---models
of Riemann-Roch formulas.
Beyond this, the foundations we develop to construct our models 
have a number of interesting byproducts.

We emphasize that the main results in this article do not assume any prior
knowledge beyond some basic combinatorics and linear algebra.  We do not assume
the reader is familiar with the Baker-Norine formula for any of our main
results. 
However, some examples we use to illustrate our theorems---which are
not essential
to their statements or proofs---are chosen from
the Baker-Norine formula for graphs and related formulas; hence we
briefly describe the Baker-Norine formula and similar formulas.
We do not assume any familiarity with sheaf theory (either on graphs, as
in \cite{friedman_memoirs_hnc}, or in the classical setting) or with the
Riemann-Roch formula; however, our techniques mimic ideas from there,
and we briefly discuss these connections
in Section~\ref{se_duality_second}.

At this point let us summarize our main results, using notation
that is common in the literature and
made precise starting in the next section.

\subsection{Riemann Functions}

We use $\integers$ to denote the integers, and
$\naturals=\integers_{\ge 1}=\{1,2,\ldots\}$
for the natural numbers.
For $n\in\naturals$ we use $[n]$ to denote $\{1,\ldots,n\}$.
For 
$\mec d=(d_1,\ldots,d_n)\in\integers^n$,
the {\em degree} of $\mec d$ is defined as
$\deg(\mec d)=d_1+\cdots+d_n$, and endow $\integers^n$ with its usual
partial order, writing $\mec d'\le\mec d$ to mean $d_i'\le d_i$ for
all $i\in[n]$.

By a {\em Riemann function} we mean a function
$f\from\integers^n\to\integers$ such that:
\begin{enumerate}
\item 
$f(\mec d)=0$ for $\deg(\mec d)$ sufficiently small, and
\item
for some $C\in\integers$---called the {\em offset of $f$}---we have
$f(\mec d)=\deg(\mec d)+C$ for $\deg(\mec d)$ sufficiently large.
\end{enumerate}
If so, then for each
$\mec K\in\integers^n$ the function 
$f^\wedge_{\mec K}\from\integers^n\to\integers$ given by
$$
f_\mec K^\wedge(\mec d) = f(\mec K-\mec d)+h(\mec K-\mec d)
$$
satisfies
\begin{equation}\label{eq_Riemann_Roch_formula}
\forall\mec d\in\integers^n,\quad
f(\mec d)-f_\mec K^\wedge(\mec K-\mec d) = \deg(\mec d)+C,
\end{equation} 
and we easily see that $f_\mec K^\wedge$ is also a Riemann function.
We refer to the above formula as a {\em generalized
Riemann-Roch formula for $f$}.
We say that \eqref{eq_Riemann_Roch_formula} or $f$ is
{\em self-dual} if $f^\wedge_{\mec K}=f$.

% In the graph Riemann-Roch and related literature,
% a ``Riemann-Roch formula'' requires self-duality.
% Roughly speaking, the goal of this work is to model any 
% formula \eqref{eq_Riemann_Roch_formula}
% as expressing the Euler characteristic of a simple kind of 
% sheaf, in a way that the sheaves involved may shed light
% on aspects of this formula, including when self-duality holds.
% Again, in this article we avoid any essential reference to sheaf theory,
% and instead refer to certain {\em diagrams of vector spaces} (and later
% {\em diagram of rings}).
% Moreover, grossly speaking we can view each such diagram as
% producing merely a
% linear transformations of vector spaces
% (similar to \cite{friedman_linear}), which technically arises
% as a cohomology computation from a projective resolution
% of the constant diagram (see \cite{friedman_memoirs_hnc,friedman_linear,
% friedman_cohomology}). 

The point of articles such as
\cite{baker_norine,amini_manjunath} is to study certain Riemann functions
of interest,
$f$, and determine if such $f=f^\wedge_{\mec K}$ for some $\mec K$.
Our approach may seem a bit ``happy-go-lucky,'' in that we develop 
combinatorics and models for any Riemann-Roch formula, 
whether or not self-duality holds.
However, as we explain below, self-duality is not preserved under
{\em restrictions}---which is how we build our models---and hence we
will be forced to consider Riemann-Roch formulas without self-duality.

The motivating example for us is that if $G=(V,E)$ is a graph with an
ordered vertex set $V=\{v_1,\ldots,v_n\}$,
then Baker-Norine \cite{baker_norine} defined the {\em rank}, a function 
$r_{{\rm BN},G}\from\integers^n\to\integers$, and $1+r_{{\rm BN},G}$ is
a Riemann function.
There is a large literature on these and related functions
\cite{baker_norine,amini_manjunath,an_baker_et_al,cori_le_borgne,backman}, 
which is
strongly related to {\em chip firing games} and the {\em sandpile model};
see \cite{cori_le_borgne,backman} and the references there for 
more historical context.
Although we have organized this article primarily for the reader interested in
the Baker-Norine rank and related functions, our results
do not require any knowledge of such functions.
Our motivation
for the term {\em Riemann function} is the classical {\em Riemann's theorem}
for curves.

\subsection{Weights and 
Models for Riemann functions $\integers^2\to\integers$ that are 
Perfect Matchings}

We model Riemann functions $\integers^2\to\integers$ by starting
with a particularly simple case of functions, related to what we call
{\em perfect matchings}.
To describe this case, we note that for each Riemann function 
$f\from\integers^2\to\integers$ there is a unique function
$W\from\integers^2\to\integers$ such that for all $\mec d\in\integers^2$
we have
$$
f(\mec d) = \sum_{\mec d'\le\mec d} W(\mec d);
$$
furthermore, $W(\mec d)\ge 0$ for all $\mec d$ holds iff
there is a bijection $\pi\from\integers\to\integers$ such that
$W(i,j)=1$ if $j=\pi(i)$, and otherwise $W(i,j)=0$.
In this case we call $W$ a {\em perfect matching}.

If $W$ is a perfect matching, then the formula
\eqref{eq_Riemann_Roch_formula} can be viewed as an {\em Euler characteristic}
in a natural way: namely,
we will define a family of $k$-diagrams (which are essentially
sheaves of $k$-vector spaces on a fixed diagram) 
$\{\cM_{W,\mec d}\}_{\mec d\in\integers}$ indexed on
$\mec d\in\integers^2$ such that: 
\begin{enumerate}
\item 
for all $\mec d\in\integers^2$,
$\cM_{W,\mec d}$ has {\em Betti numbers}, $b^i(\cM_{W,\mec d})$,
which vanish except for $i=0,1$;
\item 
for all $\mec d\in\integers^2$,
\begin{equation}
\label{eq_cM_models_zeroth_Betti_number}
f(\mec d)=b^0(\cM_{W,\mec d}),
\end{equation} 
and
\item 
for all $\mec d\in\integers^2$, and any $\mec K\in\integers^2$,
\begin{equation}
\label{eq_cM_models_first_Betti_number}
f^\wedge_{\mec K}(\mec K-\mec d) = b^1(\cM_{W,\mec d})
\end{equation} 
(the left-hand-side is independent of $\mec K$).
\end{enumerate}
We will explain what this means, and we assume no knowledge of 
sheaf theory or Betti numbers.
Defining the {\em Euler characteristic}, $\chi(\cM)$ of a 
$k$-diagram (or sheaf), $\cM$, as usual, i.e.,
as the alternating sum of its Betti numbers,
\eqref{eq_Riemann_Roch_formula} is equivalent to
$$
\chi(\cM_{W,\mec d}) = \deg(\mec d)+C.
$$

The construction of the $W_{\mec d}$ have two additional important
properties: first, for $j=1,2$, one has a simple relationship
between $\cM_{W,\mec d+\mec e_j}$ and $\cM_{W,\mec d}$ 
(involving a {\em skyscraper} $k$-diagram) that immediately implies
\begin{equation}\label{eq_chi_relations_from_skyscrapers}
\chi(\cM_{W,\mec d+\mec e_j}) = \chi(\cM_{W,\mec d})+1;
\end{equation} 
hence as soon
as one verifies that $\chi(\cM_{W,\mec d})=\deg(\mec d)+C$
for some $C\in\integers$
and for a single $\mec d\in\integers^2$, it immediately follows that
this holds for all $\mec d\in\integers^2$.
Second,
as part of our discussion of weights, it will turn out that 
for any $\mec K\in\integers^2$, setting $\mec L=\mec K+(1,1)$,
the weight of $f^\wedge_{\mec K}$ is the function $W^*_{\mec L}$
given by $W^*_{\mec L}(\mec d)=W(\mec L-\mec d)$.
It will follow that one has, for $i=0,1$, 
$$
b^i(\cM_{W^*_{\mec L},\mec K-\mec d})=b^{1-i}(\cM_{W,\mec d}) .
$$
We will show that this equality of integers actually arises from
an isomorphism
\begin{equation}\label{eq_duality_cohomology_intro}
H^i(\cM_{W^*_{\mec L},\mec K-\mec d})^* \to
H^{1-i}(\cM_{W,\mec d}) ,
\end{equation} 
which in turn arises from a statement
akin to Serre duality, that states
that for any $k$-diagram, $\cF$, 
and for $i=0,1$, there is an isomorphism 
\begin{equation}\label{eq_Serre_duality_analog_intro}
H^i(\cF)^* \to \Ext^{1-i}(\cF,\underline k_{/B_1,B_2}),
\end{equation} 
where $\underline k_{/B_1,B_2}$ is a $k$-diagram that therefore
plays the role of the
{\em canonical sheaf} in Serre duality.

\subsection{Models for General Riemann Functions $\integers^2\to\integers$}

If $f\from\integers^2\to\integers$ is a general Riemann function,
its weight, $W$, may have negative values.  In this case one can 
model $f$ by Euler characteristics provided that one passes to
{\em virtual $k$-diagrams} and {\em virtual Euler characteristics}
in the following sense:
by a {\em virtual $k$-diagram} or
{\em formal difference of $k$-diagrams} we mean a pair of $k$-diagrams
$(\cF_1,\cF_2)$, where we consider
$(\cF_1,\cF_2)$ to be equivalent $(\cF_1',\cF_2')$ if
for some $k$-diagram $\cG$ we have
$$
\cF_1 \oplus \cF_2' \oplus \cG \isom
\cF_2 \oplus \cF_1' \oplus \cG 
$$
(one often calls this the {\rm Grothendieck group} arising from
a commutative monoid);
assuming that we work over the category of $k$-diagrams
with finite Betti numbers,
we define
$$
b^i(\cF_1,\cF_2) = b^i(\cF_1)-b^i(\cF_2),
$$
which is independent of the equivalence class of $(\cF_1,\cF_2)$.
We will prove that for any weight $W\from\integers^2\to\integers$
of a Riemann function $\integers^2\to\integers$ can be written as
\begin{equation}\label{eq_W_as_alternating_sum_of_perfect_intro}
W = W_1 + \cdots + W_k - \tilde W_1 - \cdots - \tilde W_{k-1}
\end{equation} 
for some $k\ge 1$, where each $W_i$ and $\tilde W_i$ are 
perfect matchings;  we then define the formal difference
$$
\cM_{W,\mec d} = \bigl(\cF_{\mec d},\widetilde\cF_{\mec d}\bigr),
$$
where 
$$
\cF_{\mec d} =  \cM_{W_1,\mec d} \oplus \cdots \oplus  \cM_{W_k,\mec d} 
,\quad 
\tilde\cF_{\mec d} 
= \cM_{\tilde W_1,\mec d} \oplus\cdots\oplus \cM_{\tilde W_{k-1},\mec d};
$$
it is easy to verify that, up to equivalence,
$(\cF_{\mec d},\widetilde\cF_{\mec d})$ is independent of the 
way one writes $W$ in
in \eqref{eq_W_as_alternating_sum_of_perfect_intro}.
Then the formal difference $\cM_{W,\mec d}$, or really
the equivalence class $[\cM_{W,\mec d}]$, models $f(\mec d)$
in the sense 
that~\eqref{eq_cM_models_zeroth_Betti_number},
\eqref{eq_cM_models_first_Betti_number},
and~\eqref{eq_chi_relations_from_skyscrapers}
hold with $[\cM_{W,\mec d}]$ replacing $\cM_{W,\mec d}$.

\subsection{Modeling Riemann functions $\integers^n\to\integers$ for
$n\ge 3$}

To model a Riemann function $f\from\integers^n\to\integers$, we
piece together various {\em two-variable restrictions of $f$} in the
following sense:
for any distinct $i,j\in [n]$ and $\mec d\in\integers^n$, we define the
{\em two-variable restriction of $f$ at $i,j,\mec d$}
to be the function $f_{i,j,\mec d}\from\integers^2\to\integers$
given by
$$
f_{i,j,\mec d}(a_i,a_j)\eqdef
f\bigl(\mec d+a_i\mec e_i + a_j\mec e_j \bigr)
$$
(and hence $f_{i,j,\mec d}(0,0)=f(\mec d)$).
If $W$ is the weight of $f_{i,j,\mec d}$,
we use $[\cM_{f;i,j,\mec d}]$ to denote the virtual $k$-diagram
$[\cM_{W,\mec 0}]$.
It follows that 
$[\cM_{f;i,j,\mec d}]$ is a family of $k$-diagrams that
satisfies
$$
b^0([\cM_{f;i,j,\mec d}]) = f(\mec d),
$$
$$
\chi([\cM_{f;i,j,\mec d}]) = \deg(\mec d)+C
$$
where $C$ is the offset of $f$, and for any $\mec K\in \integers^2$
we have
$$
b^1([\cM_{f;i,j,\mec d}])
=  f^\wedge_\mec K(\mec K-\mec d).
$$
Of course, 
$\cM_{f;i,j,\mec d}$ appears to depend on the choice
of $i,j$; 
however, we will prove that for any 
$j'\in[n]$ distinct from $i,j$, 
$[\cM_{f;i,j,\mec d}]$ is equivalent to $[\cM_{f;i,j',\mec d}]$,
and hence the equivalence class
$[\cM_{f;i,j,\mec d}]$
is independent of the
choice of $i,j$.
It easily follows that $[\cM_{f;i,j,\mec d}]$ is independent of
the choice of $i,j$.

Moreover, in case
$\cM_{f;i,j,\mec d}$ and
$\cM_{f;i,j',\mec d}$ are $k$-diagrams, not just virtual $k$-diagrams
(i.e., the weights of $f_{i,j,\mec d}$ and $f_{i,j',\mec d}$ are
perfect matchings), we will prove that
$\cM_{f;i,j,\mec d}$ and
$\cM_{f;i,j',\mec d}$ are isomorphic as $k$-diagrams.

We therefore use the notation $[\cM_{f{\rm \;at\;} \mec d}]$
to denote the equivalence class of $[\cM_{f;i,j,\mec d}]$, which is
independent of $i,j$.
We will show that 
\eqref{eq_duality_cohomology_intro} gives rise, for $i=0,1$ to an isomorphism
\begin{equation}\label{eq_first_duality_theorem_n_var_intro}
H^i\bigl([\cM_{f {\rm \;at\;}\mec d}]\bigr)^* \to
H^{1-i}\bigl([\cM_{f^\wedge_{\mec K} {\rm \;at\;}\mec K- \mec d}]\bigr) .
\end{equation} 
This involves the following fundamental fact:  if
$f\from\integers^n\to\integers$ is a Riemann function, and if we
fix $d_3,\ldots,d_n$ and $\mec K\in\integers^n$, and we set
$g(d_1,d_2)=f(\mec d)$ viewing $d_1,d_2$ as variables, then the 
resulting generalized Riemann-Roch formula for $g$ is
\begin{equation}\label{eq_generalized_rr_two_var_intro}
g(d_1,d_2) - g^\wedge_{(K_1,K_2)}(K_1-d_1,K_2) = d_1+d_2+C_g
\end{equation} 
where $C_g$ is the offset of $g$.
It is not hard to see that formula is the restriction of
\eqref{eq_Riemann_Roch_formula}, in the sense that
$$
d_1+d_2+C_g = \deg(\mec d)+C_f
$$
where $C_f$ is the offset of $f$, and
\begin{equation}\label{eq_g_two_dim_restrict_intro}
g(d_1,d_2)=f(\mec d),
\quad
g^\wedge_{(K_1,K_2)}(K_1-d_1,K_2-d_2) = f^\wedge_\mec K(\mec K-\mec d).
\end{equation} 
It follows that the generalized Riemann-Roch formulas
\eqref{eq_Riemann_Roch_formula} restricts to two-variable generalized
Riemann-Roch formulas, and that all the two-variable formulas (i.e., fixing
some $n-2$ variables and varying the two remaining variables) determine
the all the $n$-variable formulas.

The articles \cite{baker_norine,amini_manjunath} focus on proving
that the Riemann functions there are self-dual.  We remark that the
notion of self-duality is not well-behaved under two-variable restrictions:
indeed, if $f\from\integers^n\to\integers$ satisfies $f^\wedge_\mec K=f$
for some $K\in\integers^n$, then 
\eqref{eq_g_two_dim_restrict_intro} implies that for $d_3,\ldots,d_n$
and $\mec K$ fixed we have
$$
g(d_1,d_2)=f(d_1,\ldots,d_n)
\quad
g^\wedge_{(K_1,K_2)}(K_1-d_1,K_2-d_2)
=
f(K_1-d_1,\ldots, K_n-d_n).
$$
Hence $g$ is not generally self-dual.  Hence if $f$ is self-dual, the
two-variable restrictions in a single generalized Riemann-Roch formula
still come in pairs, $g$ and $g^\wedge_{(K_1,K_2)}$.

\subsection{Additional Remarks and Future Work}

The invariants of $k$-diagrams that we compute---such as their
Betti numbers, and Euler characteristics---all arise from their
cohomology groups, which to each $k$-diagram, $\cM$, are computed
as the kernel and cokernel of an associated linear transformation
$\tau_\cM$.  Therefore the reader who prefers can translate our
entire discussion and use of $k$-diagrams into equivalent statements
regarding the kernel and cokernel of the associated linear transformations.

We also remark our duality theorems as stated above may seem trivial:
for example, given that for a perfect matching $W$ we have
$$
b^i(\cM_{W^*_{\mec L},\mec K-\mec d})=b^{1-i}(\cM_{W,\mec d}) ,
$$
it is immediate that the spaces
$$
H^i(\cM_{W^*_{\mec L},\mec K-\mec d})
\quad\mbox{and}\quad 
H^{1-i}(\cM_{W,\mec d}) 
$$
and their duals are all isomorphic, since these
are all $k$-vector spaces of this same dimension.  Hence in our theorems
and their proofs, it is also important to note the way we construct
these duality isomorphisms; often we make this explicit in the
statement of the theorem (see, e.g., Theorem~\ref{th_first_duality_cM}).

The Riemann functions $f\from\integers^n\to\integers$
associated to the Baker-Norine rank
\cite{baker_norine} and its generalizations studied by
Amini and Manjunath \cite{amini_manjunath}
are {\em periodic} (in a sense described in 
Subsection~\ref{su_periodicity_trans}).  In this case the $k$-diagrams
$\cM_{W,\mec d}$ associated the two-variable restrictions of $f$
have a much stronger structure: namely, they are $\cO$-modules,
where $\cO$ is a $k$-diagram of rings.
In this case we believe that the
diagrams $\cM_{W,\mec d}$ themselves act as canonical $k$-diagrams
in a form of Serre duality; we explain this at the end of this
article, and plan to address this in a future work.
For this reason, when $W\from\integers^2\to\integers$ is periodic,
when writing $W$ as a difference of a sum of perfect matchings
\eqref{eq_W_as_alternating_sum_of_perfect_intro}, we will be
interested in showing that the $W_i$ and $\tilde W_i$ can be chosen
with the same periodicity.

A good challenge for future work is to develop models that
explain generalized Riemann-Roch formulas as a type of sheaf
or diagram of $k$-vector spaces that does not have all the ad hoc
choices we make, and that does not need to pass to virtual
diagrams or virtual sheaves.

Another---perhaps independent challenge---is to use the theory
of diagrams or sheaves to give proofs of self-duality, such as
in the Baker-Norine formula \cite{baker_norine} and some more 
general situations, such as those studied by
Amini and Manjunath \cite{amini_manjunath}.

\subsection{Organization of the Rest of this Article}

In Section~\ref{se_basic_weights} we introduce some basic notation
and state some theorems about the {\em weight} of a Riemann function,
referring the reader to \cite{folinsbee_friedman_weights} for the
proofs.
In Section~\ref{se_weight_decomp}
we will prove some theorems regarding the weights of 
Riemann functions $\integers^2\to\integers$ that we will use.
In Section~\ref{se_diagrams_Betti_perfect}
we give some conventions regarding the sheaves we build---that
we call $k$-diagrams---and
show use them to model
a Riemann function $\integers^2\to\integers$ whose weight is
non-negative, i.e., is a {\em perfect matching}.
In Section~\ref{se_morphisms_etc} we discuss morphisms between
$k$-diagrams, and a number of related ideas needed later on;
in particular, to define virtual $k$-diagrams we need to know 
some facts about direct sums and isomorphisms of $k$-diagrams.
In Section~\ref{se_indicator}
we introduce {\em indicator} $k$-diagrams that gives an
alternate way to view the $k$-diagrams that we use to model
Riemann functions $\integers^2\to\integers$; we will need them when
we prove duality theorems later on.
In Section~\ref{se_virtual_two_dim_Riemann}
we describe our conventions about virtual $k$-diagrams and
show that any Riemann function $\integers^2\to\integers$
can be modeled by a single equivalence class of virtual $k$-diagrams.
In Section~\ref{se_higher_Riemann}
we model any Riemann function $\integers^n\to\integers$ by
diagrams obtained by fixing any $n-2$ of its variables
and modeling the resulting Riemann function $\integers^2\to\integers$.
In Section~\ref{se_duality_first}
we will prove the $i=1$ case of
\eqref{eq_Serre_duality_analog_intro},
\eqref{eq_duality_cohomology_intro}, and
\eqref{eq_first_duality_theorem_n_var_intro}.
In Section~\ref{se_duality_second}
we prove 
the $i=0$ case of \eqref{eq_Serre_duality_analog_intro}
and explain the connection
of $k$-diagrams to sheaf theory; we tie up a few other loose ends,
including a discussion of periodic Riemann functions and a
{\em Serre functor} computation that yields 
$\underline k_{/B_1,B_2}$.

\section{Basic Terminology and Weights}
\label{se_basic_weights}

In this section we introduce some basic terminology used throughout
this paper, including the definition of a {\em Riemann function} and
its {\em weight function}.
Then we derive some combinatorial results about the weights of Riemann
functions that we will need to construct our models.

The weight function of a Riemann function is quite interesting for
its own sake, and we refer to 
\cite{folinsbee_friedman_weights} for a fuller discussion of weights
of Riemann functions.

\subsection{Basic Notation}

We use $\integers,\naturals$ to denote the integers and positive
integers; for $a\in\integers$, we use $\integers_{\le a}$ to denote
the integers less than or equal to $a$, and similarly for the
subscript $\ge a$.
For $n\in\naturals$ we use $[n]$ to denote $\{1,\ldots,n\}$.
We use bold face $\mec d=(d_1,\ldots,d_n)$ to denote elements of $\integers^n$, 
using plain face for the components of $\mec d$;
by the {\em degree} of $\mec d$,
denoted $\deg(\mec d)$ or at times $|\mec d|$, we mean
$d_1+\ldots+d_n$.

We set
$$
\integers^n_{\deg 0} = \{ \mec d\in\integers^n \ | \ \deg(\mec d)=0 \},
$$
and for $a\in \integers$ we similarly set
$$
\integers^n_{\deg a} = \{ \mec d\in\integers^n \ | \ \deg(\mec d)=a \},
\quad
\integers^n_{\deg \le a} = 
\{ \mec d\in\integers^n \ | \ \deg(\mec d)\le a \}.
$$

We use $\mec e_i\in\integers^n$ (with $n$ understood) be the $i$-th 
standard basis vector (i.e., whose $j$-th component is $1$ if $j=i$ and
$0$ otherwise), and for $I\subset [n]$ (with $n$ understood) we set
\begin{equation}\label{eq_e_I_notation}
\mec e_I = \sum_{i\in I} \mec e_i;
\end{equation} 
hence in case $I=\emptyset$ is the empty set, then 
$\mec e_\emptyset=\mec 0=(0,\ldots,0)$,
and similarly $e_{[n]}=\mec 1=(1,\ldots,1)$.

For $n\in\naturals$,
we endow $\integers^n$ with the usual partial order, that is
$$
\mec d'\le \mec d \quad\mbox{iff}\quad d'_i\le d_i\ \forall i\in[n],
$$
where $[n]=\{1,2,\ldots,n\}$.

\subsection{Riemann Functions, Generalized
Riemann-Roch Formulas, and Self-Duality}
\label{su_riemann_functions}

In this section we define {\em Riemann functions} and 
{\em generalized Riemann-Roch formulas} and give some examples.

\begin{definition}
We say that a function $f\from\integers^n\to\integers$ is 
a Riemann function if for some $C,a,b\in\integers$ we have
\begin{enumerate}
\item $f(\mec d)=0$ if $\deg(\mec d)\le a$; and
\item $f(\mec d)=\deg(\mec d)+C$ if $\deg(\mec d)\ge b$;
\end{enumerate}
we refer to $C$ as the {\em offset} of $f$.
\end{definition}

In our study of Riemann functions,
it will be useful to introduce the following terminology.

\begin{definition}
If $f,g$ are functions $\integers^n\to\integers$, we say that
{\em $f$ equals $g$ initially} (respectively, {\em eventually})
if $f(\mec d)=g(\mec d)$ for $\deg(\mec d)$ sufficiently small
(respectively, sufficiently large); similarly, we say that
that $f$ is {\em initially zero}
(respectively {\em eventually zero})
if $f(\mec d)=0$ for $\deg(\mec d)$ sufficiently small
(respectively, sufficiently large).
\end{definition}
Therefore $f\from \integers^n\to\integers$ 
is a Riemann function iff it is initially zero
and it eventually equals the function $\deg(\mec d)+C$
for a constant $C\in\integers$ that we call the {\em offset} of $f$.

In particular, Riemann's theorem, which is a precursor to the
classical Riemann-Roch theorem, gives examples of Riemann functions.
In the next subsection we will give a number of examples 
of Riemann functions, including those associated to the
Baker-Norine rank function of a graph \cite{baker_norine}
and related functions.
Before doing so, we give some of the basic properties 
of Riemann functions.

\begin{definition}\label{de_generalized_Riemann_Roch_formula}
Let $f\from\integers^n\to\integers$ be a Riemann function
with offset $C$,
and $\mec K\in\integers^n$.  The
{\em $\mec K$-dual of $f$}, denoted $f^\wedge_{\mec K}$,
refers to the function $\integers^n\to\integers$ given by
\begin{equation}\label{eq_first_dual_formulation}
f^{\wedge}_{\mec K}(\mec d)=f(\mec K-\mec d)-\deg(\mec K-\mec d)-C.
\end{equation}
Replacing $\mec d$ with $\mec K-\mec d$ we equivalently write
\begin{equation}\label{eq_generalized_riemann_roch}
f(\mec d) - f^{\wedge}_{\mec K}(\mec K-\mec d) = \deg(\mec d) + C
\end{equation}
and refer to this equation as a {\em generalized Riemann-Roch formula}.
We say that {\em $f$ is self-dual at $\mec K$} if $f^\wedge_\mec K=f$.
\end{definition}

If $f$ is a Riemann functions that is self-dual at $\mec K$, then
\eqref{eq_generalized_riemann_roch} reads
\begin{equation}\label{eq_usual_Riemann_Roch}
f(\mec d) - f(\mec K-\mec d) = \deg(\mec d) + C,
\end{equation} 
which resembles the classical Riemann-Roch formula
and the Baker-Norine analog \cite{baker_norine} and related formulas.
We remark that in \eqref{eq_generalized_riemann_roch},
$f^{\wedge}_{\mec K}(\mec K-\mec d)$ equals
$f(\mec d)-\deg(\mec d)-C$, which is independent of $\mec K$.

\begin{proposition}\label{pr_properties_of_wedge_sub_mec_K}
Let $f\from\integers^n\to\integers$ be a Riemann function
with offset $C$,
and $\mec K\in\integers^n$.  Then:
\begin{enumerate}
\item
$f^\wedge_{\mec K}$ is a Riemann function with offset $-\deg(\mec K)-C$;
\item
$(f^\wedge_{\mec K})^\wedge_\mec K=f$;
\item
for any other Riemann function, $g\from\integers^n\to\integers$,
$f=g$ iff for some (equivalently any) $\mec K\in\integers^n$,
$f^\wedge_\mec K=g^\wedge_\mec K$.
\end{enumerate}
\end{proposition}
\begin{proof}
This proof is a straightforward calculation.
For $\mec d$ sufficiently small we have
$$
f(\mec K-\mec d)=\deg(\mec K-\mec d)+C,
$$
which by \eqref{eq_first_dual_formulation} implies that $f^\wedge_\mec K$
is initially zero.
For $\mec d$ with $\deg(\mec d)$ sufficiently large we have
$f(\mec K-\mec d)=0$, and hence
\eqref{eq_first_dual_formulation} implies that for $\deg(\mec d)$
sufficiently large
$$
f^\wedge_{\mec K}(\mec d)=\deg(\mec K-\mec d) -C
=\deg(\mec d)-\deg(\mec K)-C.
$$
Hence $f^\wedge_\mec K$ is a Riemann function with offset 
$-\deg(\mec K)-C$.

To see claim~(2), since $f^\wedge_\mec K$ is a Riemann function with
offset $-\deg(\mec K)-C$, by
\eqref{eq_first_dual_formulation} we have
$$
(f^\wedge_\mec K)^\wedge_\mec K (\mec d)=
f^\wedge_\mec K(\mec K-\mec d)-\deg(\mec K - \mec d)-\bigl(\deg(\mec K)-C\bigr)
=f^\wedge_\mec K(\mec K-\mec d) + \deg(\mec d)+C,
$$
which \eqref{eq_first_dual_formulation} implies
$$
=f(\mec d)-\deg(\mec d)-C +\deg(\mec d)+C = f(\mec d).
$$

To prove claim~(3), if $f=g$ then we may apply $\,^\wedge_\mec K$
to both to conclude that $f^\wedge_\mec K=g^\wedge_\mec K$.
Conversely, if $f^\wedge_\mec K=g^\wedge_\mec K$, then applying
$\,^\wedge_\mec K$ and using claim~(2) we get $f=g$.
\end{proof}

\subsection{Examples of Riemann functions}

We briefly give some examples of Riemann functions.
This section is not essential to the rest of this paper,
although these examples are helpful for intuition;
we will refer to some of these examples to
illustrate some of our results.

\begin{example}\label{ex_baker_norine}
Let $G=(V,E)$ be a connected graph with $V=\{v_1,\ldots,v_n\}$.
Let $L={\rm Image}(\Delta_G)$ be the image of the Laplacian of
$G$, $\Delta_G$.  
Say that $\mec d\in\integers^n$ is 
{\em effective} if there is a $\mec d'\ge\mec 0$
(i.e., $d'_i\ge 0$ for all $i\in[n]$) such that
$\mec d-\mec d'\in L$, and otherwise say that $\mec d$ is
{\em not effective}.
Let $\cN\subset\integers^n$ be the subset of elements that are
not effective.  Let
\begin{equation}\label{eq_f_is_distance_to_cN}
f(\mec d) = {\rm Distance}_{L^1}(\mec d,\cN)
= \min_{\mec d'\in\cN} \|\mec d-\mec d'\|_{L^1},
\end{equation} 
where $L^1$ is the usual $L^1$-norm,
$$
\|\mec x \|_{L^1} = |x_1|+\cdots+|x_n|.
$$
Then $r_{{\rm BN},G}=-1+f$ 
is the usual {\em Baker-Norine rank} \cite{baker_norine} of $G$,
and the Baker-Norine Graph Riemann-Roch formula 
\cite{baker_norine}
asserts that
$$
f(\mec d)-f(\mec K-\mec d) = \deg(\mec d)+1-g
$$
where $g=1+|E|-|V|$ (which is non-negative since $G$ is connected), and
where
\begin{equation}\label{eq_baker_norine_canonical}
\mec K = \bigl( \deg_G(v_1)-2,\ldots,\deg_G(v_n)-2 \bigr) .
\end{equation} 
Since $\cN$ contains all elements of $\integers^n$ of negative degree,
it follows that $f$ is initially zero; the Baker-Norine formula
implies that $f$ is a Riemann function with offset $1-g$ that is
self-dual at $K$.
\end{example}

\begin{example}\label{ex_amini_manjunath}
Let $L$ be, more generally, any lattice of rank $n-1$ in
$\integers^n_0$.  Then the same definitions work---effective, not effective,
$\cN$, and furthermore
Amini and Manjurath \cite{amini_manjunath} show that $f$ as in
\eqref{eq_f_is_distance_to_cN} is a Riemann function with
offset $1-g_{\rm max}$ defined on page~5 there.
They give conditions---which hold sometimes, but not always---for
$f$ to be self-dual at some $\mec K\in\integers^n$.
\end{example}

One can slightly generalize this construction of Riemann functions, $f$, in
\eqref{eq_f_is_distance_to_cN}
by allowing $\cN\subset\integers^n$ to satisfy some weaker conditions;
see \cite{folinsbee_friedman_weights}.

\begin{example}\label{ex_classical_Riemann_function}
Let $P_1,\ldots,P_n$ be $n$ points of an algebraic curve over an
algebraically closed field, $k$, and let
$K$ denote the function field of the curve.  Let
\begin{equation}\label{eq_classical_Riemann_function}
f(\mec d) =  
\dim_k\{ g\in K \ | \ (g)\ge -d_1P_1-\cdots-d_n P_n \}
\end{equation} 
where $(g)$ is the (Weil) divisor associated with $g$
(and we view $(0)$ as larger than any divisor).  Then the classical
{\em Riemann theorem} states that $f$ is a Riemann function, and
that its offset equals
$1-g$, where $g$ is the genus of the curve.
\end{example}
The above example was our motivation for the name {\em Riemann function}.

\subsection{Restrictions of Riemann functions and Alternating Sums}

In this subsection we give examples of obtaining Riemann functions
and constructing new Riemann functions.
Both ideas are fundamental to the way we construct the models
in this article.

\begin{example}\label{ex_Riemann_function_restriction}
Let $f\from\integers^n\to\integers$ be any Riemann function with
$f(\mec d)=\deg(\mec d)+C$ for $\deg(\mec d)$ sufficiently large.
Then for any distinct $i,j\in[n]$ and $\mec d\in\integers^n$,
the function $f_{i,j,\mec d}\from\integers^2\to\integers$ given as
\begin{equation}\label{eq_two_variable_restriction}
f_{i,j,\mec d}(a_i,a_j) =
f\bigl(\mec d + a_i\mec e_i + a_j\mec e_j \bigr)
\end{equation}
is a Riemann function $\integers^2\to\integers$,
and for $a_i+a_j$ large we have
\begin{equation}\label{eq_two_variable_restriction_constant}
f_{i,j,\mec d}(a_i,a_j) = a_i+a_j+ C',\quad\mbox{where}\quad
C'=\deg(\mec d)+ C.
\end{equation}
We call $f_{i,j,\mec d}$ a {\em two-variable restriction} of $f$;
we may similarly restrict $f$ to one variable or three or more
variables; clearly any restriction of a Riemann is again
a Riemann function.
(We write $a_i,a_j$ as the arguments for $f_{i,j,\mec d}$
instead of, say, $a_1,a_2$, to stress that
$a_i$ corresponds to adding $a_i\mec e_i$ in
\eqref{eq_two_variable_restriction}, and similarly for $a_j$).
\end{example}
In the above example, it will be crucial to us that
$C'$ depends only on $\mec d$ and
not on $i,j$.

\begin{example}\label{ex_alternating_sum_of_Riemann_functions}
If for some $s,n\in\naturals$, 
$f_1,\ldots,f_s$ and $\tilde f_1,\ldots,\tilde f_{s-1}$
are Riemann
functions $\integers^n\to\integers$, then so is
$$
% f_1 - f_2 + f_3 - \cdots - f_{2\ell}+f_{2\ell+1}.
f = f_1 + \cdots + f_s - (\tilde f_1+\cdots+\tilde f_{s-1}) .
$$
Moreover, the offset, $C$, of $f$ is given as
\begin{equation}\label{eq_offset_of_difference_of_sums}
C = (C_1+\cdots+C_s) - (\tilde C_1+\cdots+\tilde C_{s-1}),
\end{equation} 
where $C_i$ is the offset of $f_i$ and $\tilde C_i$ is the offset
of $\tilde f_i$.
\end{example}

\subsection{The Weight of a Riemann function}

Our main technique of modeling Riemann functions involves their
weights.  In this article we are concerned with weights of
Riemann functions $\integers^2\to\integers$, but the foundations
of weights apply to Riemann functions in any number of 
variables; see \cite{folinsbee_friedman_weights}.

If $f\from\integers^n\to\integers$ is initially zero, then there is a unique
initially zero $W\from\integers^n\to\integers$ for which
\begin{equation}\label{eq_define_weight}
\forall\mec d\in\integers^n,
\quad
f(\mec d)=\sum_{\mec d'\le\mec d} W(\mec d'),
\end{equation} 
since we can determine $W(\mec d)$ inductively on $\deg(\mec d)$,
setting $W$ initially zero 
(in degrees where $f$ is initially zero), and using the equation
\begin{equation}\label{eq_inductively_define_W_from_f}
W(\mec d) = f(\mec d)-\sum_{\mec d'\le\mec d,\ \mec d'\ne \mec d} W(\mec d').
\end{equation} 

\begin{definition}
Let $f\from\integers^n\to\integers$ be initially zero.
By the {\em weight of $f$} we mean the unique initially zero
function $W\from\integers^n\to\integers$ satisfying
\eqref{eq_define_weight}.
\end{definition}

Recall from \eqref{eq_e_I_notation}
the notation $\mec e_I$ for $I\subset [n]$.

\begin{proposition}\label{pr_Mobius_inversion}
Consider the operator $\frakm$ on functions $f\from\integers^n\to\integers$
defined via 
\begin{equation}\label{eq_define_mu}
(\frakm f)(\mec d) = \sum_{I\subset [n]} (-1)^{|I|} f(\mec d-\mec e_I),
\end{equation} 
and the operator on functions $W\from\integers^n\to\integers$ that
are initially zero given by
\begin{equation}\label{eq_define_s}
(\fraks W)(\mec d) = \sum_{\mec d'\le\mec d} W(\mec d').
\end{equation} 
If $f$ is any initially zero function, and $W$ is the weight of $f$,
then we have
$f=\fraks W$ and $W=\frakm f$.
\end{proposition}
The proof is an easy computation; see \cite{folinsbee_friedman_weights}
for details.
One may also write 
$$
(\frakm f) = (1-\frakt_1)\ldots(1-\frakt_n) f,
$$
where $\frakt_i$ is the operator taking $f$ to the function
$$
(\frakt_i f)(\mec d) = f(\mec d-\mec e_i);
$$
it easily
follows that for $n\ge 2$, if $W\from\integers^n\to\integers$ is the weight
of any Riemann function
$f\integers^n\to\integers$, then $W$ is eventually zero.

\subsection{Weights and the Riemann-Roch formulas}

\begin{definition}
If $W\from\integers^n\to\integers$ is any function and
$\mec L\in\integers^n$,
the {\em $\mec L$-dual weight of $W$}, denoted
$W^*_{\mec L}$ refers to the function given by
$$
W^*_{\mec L}(\mec d)=W(\mec L-\mec d).
$$
\end{definition}
It is immediate that $(W^*_{\mec L})^*_{\mec L}=W$.

\begin{theorem}\label{th_easy_dual_functions_theorem}
Let $f\from\integers^n\to\integers$ be a generalized Riemann function, and
$W=\frakm f$.
Let $\mec K\in\integers^n$ and let $\mec L = \mec K + \mec 1$.
\begin{enumerate}
\item We have
\begin{equation}\label{eq_dual_weight_equation}
\frakm\bigl(f^\wedge_{\mec K}\bigr)  = (-1)^n W^*_\mec L
= (-1)^n (\frakm f)^*_{\mec L}.
\end{equation}
\item
$f^\wedge_{\mec K}=f$ iff
$W^*_{\mec L}=(-1)^n W$.
\end{enumerate}
\end{theorem}
The proof is a straightforward computation; see
\cite{folinsbee_friedman_weights} for details.

We remark that it is immediate that the map
$W\mapsto W^*_\mec L$ is an involution,
i.e., applying it twice gives the same function; furthermore,
the map $W\mapsto W^*_\mec L$ 
is defined on all functions $\integers^n\to\integers$.
By contrast, the fact that $\,^\wedge_\mec K$ is an involution
requires a bit more computation (in the proof of
Proposition~\ref{pr_properties_of_wedge_sub_mec_K}), and
it is only defined on Riemann functions
(at least as we have defined it)
since the definition of $f^\wedge_\mec K$ 
in \eqref{eq_first_dual_formulation} requires us know the
offset, $C$, of $f$.
[In \cite{folinsbee_friedman_weights},
$f^\wedge_\mec K$ is defined on a more
general class of functions, $f$, namely $f$ that are initially zero
and whose weight, $W$, is eventually zero.]
This gives two indications that working with the weight of a Riemann
function has advantages over working with the Riemann function itself.

\subsection{Translation Invariance and Periodicity}
\label{su_periodicity_trans}

This subsection has two goals: first, in case $f=f^\wedge_\mec K$,
then we study the uniqueness of such a $\mec K$.
Second, we will introduce the related notation of {\em periodicity}
which will be useful in a future work to show that certain
generalized Riemann-Roch formulas have a second type of Serre
duality beyond what we cover in this article; 
we will briefly explain this in 
Section~\ref{se_duality_second}.

\begin{definition}
Let $f\from\integers^n\to\integers$ be any function.  We say that
$f$ is {\em invariant by translation under $\mec t\in\integers^n$}
if 
% for all $\mec d\in\integers^n$ we have
$$
\forall \mec d\in\integers^n, \quad
f(\mec d+\mec t)=f(\mec d).
$$
We define the {\em set of invariant translations of $f$} to be
the set of all such $\mec t$.
\end{definition}
In the above definition, we easily that the set, $T$, of invariant
translations of a function, $f$, is a lattice, i.e., $T$ is closed
under addition, and if $\mec t\in T$ then also $-\mec t\in T$.
Furthermore, if $f$ is translation invariant by $\mec t$,
then for any $\mec d\in\integers^n$ we have
$f(\mec d)=f(\mec d+m\mec t)$ for any $m\in\integers$; it
follows that if $f$ is non-zero but initially zero,
then any such $\mec t$ must lie in $\integers^n_{\deg 0}$.

\begin{proposition}
Let $f\from\integers^n\to\integers$ be any Riemann function,
and $W=\frakm f$ its weight.  
Let $T$ be the set of invariant translations of $f$.  Then
\begin{enumerate}
\item 
$T$ equals the set of invariant translations of $W$;
\item 
for any $\mec L_1,\mec L_2\in\integers^n$,
$W^*_{\mec L_1}=W^*_{\mec L_2}$ iff 
$\mec L_1-\mec L_2\in T$;
\item 
for any $\mec K_1,\mec K_2\in\integers^n$,
$f^\wedge_{\mec K_1}=f^\wedge_{\mec K_2}$ iff 
$\mec K_1-\mec K_2\in T$;
\item
if for some $\mec K\in\integers^n$ we have $f=f^\wedge_{\mec K}$, then
for any $\mec K'\in\integers^n$,
$f=f^\wedge_{\mec K'}$ iff $\mec K'-\mec K\in T$.
\item
if for some $\mec L\in\integers^n$ we have $W=W^*_{\mec L}$, then
for any $\mec L'\in\integers^n$,
$W=W^*_{\mec L'}$ iff $\mec L'-\mec L\in T$.
\end{enumerate}
\end{proposition}
\begin{proof}
Claim~(1) follows by observing that $\frakm$ and $\fraks$ commute
with translation by $\mec t$, (i.e., the operator taking $f$ to $\tilde f$
given by
$\tilde f(\mec d)=f(\mec d+\mec t)$).

To prove claim~(2), we see that $W^*_{\mec L_1}=W^*_{\mec L_2}$
iff 
$$
(W^*_{\mec L_1})^*_{\mec L_2}= (W^*_{\mec L_2})^*_{\mec L_2}=W
$$
iff, for all $\mec d\in\integers^n$ we have
$$
W(\mec d) = (W^*_{\mec L_1})^*_{\mec L_2}(\mec d)
=(W^*_{\mec L_1})(\mec L_2-\mec d)=
W\bigl(\mec L_1 - (\mec L_2-\mec d) \bigr)
=
W(\mec d+\mec L_1-\mec L_2),
$$
i.e., $\mec L_1-\mec L_2\in T$.

Claim~(3) follows from claims~(1) and~(2).
Claim~(4) follows from claim~(3), and claim~(5) from claim~(2).
\end{proof}

\begin{definition}
We say that a function $f\from\integers^n\to\integers$ is
{\em $r$-periodic} for an $r\in\integers$ if for all
$i,j\in[n]$ we have that $f$ is invariant under
translation by $r(\mec e_i-\mec e_j)$.
\end{definition}

\begin{example}
In Examples~\ref{ex_baker_norine} 
and~\ref{ex_amini_manjunath}, $f$ as in
\eqref{eq_f_is_distance_to_cN} is translation invariant by
$L\subset\integers^n_{\deg 0}$ where
$\integers^n_{\deg 0}/L$ is finite (i.e., $L$ is of rank $n-1$).
If $p = |\integers^n_{\deg 0}/L|$, then any element of
$\integers^n_{\deg 0}/L$ is of order divisible by $p$, and
hence $f$ is $p$-periodic.
\end{example}

\begin{example}
If in Example~\ref{ex_classical_Riemann_function} we take an 
elliptic curve, and $P_1$ is any point, then $P_2-P_1$ has
finite order for only countably many $P_2$; hence
$f$ is $r$-periodic for some $r\ge 1$ for only countably many $P_2$.
\end{example}

\subsection{Weights of Riemann Functions $\integers^2\to\integers$}

We will be especially interested in 
Riemann functions $\integers^2\to\integers$ and their weights
$W=\frakm f$.  It is useful to notice that for such functions we that
that for any fixed $d_1$ and $d_2'$ sufficiently large,
$$
f(d_1,d_2')-f(d_1-1,d_2') = 1,
$$
and that for any $d_1,d_2'\in\integers$ we have
$$
f(d_1,d_2')-f(d_1-1,d_2') 
= \sum_{d_2=-\infty}^{d_2'}W(d_1,d_2).
$$
Hence,
for fixed $d_1$,
\begin{equation}\label{eq_two_dim_row_sums}
\sum_{d_2=-\infty}^\infty W(d_1,d_2) = 1,
\end{equation} 
and similarly, for fixed $d_2$ we have
\begin{equation}\label{eq_two_dim_col_sums}
\sum_{d_1=-\infty}^\infty W(d_1,d_2) = 1.
\end{equation} 
We easily check the converse, i.e., if 
$W\from\integers^2\to\integers$ is initially and eventually
zero and satisfies~\eqref{eq_two_dim_row_sums}
and~\eqref{eq_two_dim_col_sums}, then for $d_1'+d_2'$ fixed and sufficiently
large we have
$$
f(d_1',d_2')+1
=
f(d_1'+1,d_2')
=
f(d_1',d_2'+1),
$$
and we conclude that $f$ is a Riemann function.
Viewing $W$ as a two-dimensional infinite array of numbers indexed in
$\integers\times\integers$, one can therefore say that
$W\from\integers^2\to\integers$ is the weight of a Riemann function iff
all its ``row sums'' \eqref{eq_two_dim_row_sums}
and all its ``column sums'' \eqref{eq_two_dim_col_sums}
equal $1$.

\subsection{Weights of Slowly Growing 
Riemann Functions $\integers^2\to\integers$ and Perfect Matchings}

In this subsection we make some remarks on weights that we call
``perfect matchings.''
In \cite{folinsbee_friedman_weights}, these ideas were used to
compute the weight of the Baker-Norine rank on graphs of two
vertices (jointed by some number of edges).
Here we will just state the definitions and an easy proposition.

\begin{definition}\label{de_slowly_growing}
We say that a function $f\from\integers^n\to\integers$ is 
{\em slowly growing} if for all $\mec d\in\integers^n$ and $i\in[n]$
we have
$$
f(\mec d) \le f(\mec d+\mec e_i) \le f(\mec d)+1.
$$
% $$
% f(\mec d)-1 \le f(\mec d-\mec e_j) \le f(\mec d).
% $$
\end{definition}

\begin{definition}\label{de_perfect_matching}
Let $W$ be a function $\integers^2\to\integers$ that is initially
and eventually zero.  We say that $W$ is a {\em perfect matching}
if there exists a permutation (i.e., a bijection)
$\pi\from\integers\to\integers$ such that
\begin{equation}\label{eq_W_perfect_and_pi}
W(i,j) = \left\{ \begin{array}{ll}
1 & \mbox{if $j=\pi(i)$, and} \\
0 & \mbox{otherwise.}
\end{array}
\right.
\end{equation}
\end{definition}
It follows that for $\pi$ as above,
$\pi(i)+i$ is bounded above and below, since $W$ is initially
and eventually $0$.
Conversely, if $\pi\from\integers\to\integers$ is a bijection
with $\pi(i)+i$ bounded independently of $i$, then
\eqref{eq_W_perfect_and_pi} is a perfect matching.

\begin{proposition}\label{pr_W_either_zero_one_minus_one}
Let $f\from\integers^2\to\integers$ be a slowly
growing Riemann function.
Let $W=\frakm f$ be the weight of $f$.  Then $W$ takes only the
values $0$ and $\pm 1$.  Furthermore, for any $\mec d\in\integers^2$,
let $a=f(\mec d)$; then
\begin{equation}\label{eq_W_is_one}
W(\mec d)=1 \iff
f(\mec d-\mec e_1)=f(\mec d-\mec e_2)=f(\mec d - \mec e_1 - \mec e_2)=a-1,
\end{equation}
and
\begin{equation}\label{eq_W_is_minus_one}
W(\mec d)=-1 \iff
f(\mec d-\mec e_1)=f(\mec d-\mec e_2)=a=f(\mec d - \mec e_1 - \mec e_2)+1.
\end{equation}
If $W$ is everywhere non-negative, i.e., 
$W(\mec d)\ge 0$ for all $0$,
then $W$ is a perfect matching.
\end{proposition}
\begin{proof}
For the proof, see 
\cite{folinsbee_friedman_weights}; for ease for reading,
we give the main idea: namely,
since $f$ is slowly growing, if $a=f(\mec d)$,
then
$$
a-2 \le f(\mec d-\mec e_1-\mec e_2) \le a.
$$
In case $f(\mec d-\mec e_1-\mec e_2)=a-2$, then for $j=1,2$,
$f(\mec d-\mec e_j)$ must equal $a-1$ (since $f$ is slowly growing),
in which case $W(\mec d)=0$.
Similarly if $f(\mec d-\mec e_1-\mec e_2)=a$ then
$f(\mec d-\mec e_j)$ must equal $a$, and again $W(\mec d)=0$.
This leaves the case $f(\mec d-\mec e_1-\mec e_2)=a-1$, whereupon for $j=1,2$
each $f(\mec d-\mec e_j)$ is either $a$ or $a-1$; 
this gives four cases to check, and after checking them we see that
$W(\mec d)=1$ iff for $j=1,2$ both $f(\mec d-\mec e_j)$ are $a-1$,
and
$W(\mec d)=-1$ iff for $j=1,2$ both $f(\mec d-\mec e_j)$ are $a$.
\end{proof}

Of course, if $W$ is $r$-periodic, i.e., for all $\mec d\in\integers^2$,
$W(\mec d)=W(\mec d+(r,-r))$, then $\pi$ is {\em skew-periodic}
in the sense that $\pi(i+r)=\pi(i)-r$ for all $i\in\integers$.

\subsection{Riemann Functions of Genus $1$ and their Two-Variable
Restrictions}

There is a collection of Riemann functions that will be 
especially helpful to give concrete examples
in Section~\ref{se_higher_Riemann}.
This section can be skipped until then; however, the reader may
want to read this now to get some more concrete examples of
Riemann functions.

\begin{definition}
Let $b\in\integers$.
We say that a Riemann function $f\from\integers^n\to\integers$
is {\em of generalized genus $1$ of shift $b$} if 
\begin{enumerate}
\item for all $\mec d\in\integers^n$ with $\deg(\mec d)\ne b$,
$$
f(\mec d) = \max\bigl( 0 , \deg(\mec d) -b \bigr),
$$
\item
for all $\mec d\in\integers^n$ with $\deg(\mec d)=b$,
$f(\mec d)\in\{0,1\}$;
\end{enumerate}
if so, we say that $f$ is {\em of genus $1$} if $b=0$.
\end{definition}

\begin{example}\label{ex_classical_genus_one}
If in Example~\ref{ex_classical_Riemann_function},
the curve is of genus $1$, and $P_1,\ldots,P_n$
are any points on the curve, then 
$f$ in \eqref{eq_classical_Riemann_function} is a Riemann function
of genus $1$.  If $\mec d\in\integers$ with $\deg(\mec d)=0$, then
$$
f(\mec d)=1 \iff
d_1 P_1 + \cdots + d_n P_n = 0
$$
where $+$ is with respect to the group law in the curve with respect
to a some point (i.e., one chooses a point $P_\infty$ and one defines
$P_3=P_1+P_2$ if $P_3$ is linearly equivalent to
$P_1+P_2-P_\infty$), and $0$ is
the identity element in the group law (i.e., the point $P_\infty$).
\end{example}

\begin{example}\label{ex_cycle_Baker_Norine}
Let $G$ be a cycle of length $n\ge 2$, whose vertices in are, in cyclic order
(of which there are $2n$ choices for $n\ge 3$)
$v_1,\ldots,v_n$.  If $r_{\rm BN}$ is the Baker-Norine rank, then
$f(\mec d)=1+r_{\rm BN}(\mec d)$ is a Riemann function of genus $g$.
If $\deg(\mec d)=0$, then we easily see that 
$$
f(\mec d)=0 \quad\iff\quad
\mbox{$d_1+2d_2+\cdots+(n-1)d_{n-1}$ is divisible by $n$}
$$
(since $d_1+\cdots+d_n=0$, the above divisibility by $n$ is independent
of which cyclic order we choose for $v_1,\ldots,v_n$).
Hence $f$ is the same $f$ as in
Example~\ref{ex_classical_genus_one} where
$P_1$ is any point of order $n$ on the curve, and
$P_i = i P_1$ in the group law on the curve
(hence we require the elliptic curve to have a point of order $n$,
which is always the case in the classical case over the complex numbers).
\end{example}

\begin{example}
The two-variable restriction of any function $\integers^n\to\integers$
that is of generalized genus $1$ is again a function
$\integers^2\to\integers$ that is of generalized genus $1$.
(Similarly for a restriction to any number of variables.)
\end{example}

\begin{proposition}
Let $f\from\integers^2\to\integers$ be a Riemann function of
genus $1$.  Then $f$ is slowly growing,
and the weight, $W$, of $f$ attains a negative somewhere iff for some
$d_1\in\integers$ we have
$$
f(-d_1,d_1)=f(-d_1-1,d_1+1)=1.
$$
\end{proposition}
\begin{proof}
We easily see that $f$ is slowly growing.
To see that its weight, $W$, is non-negative,
assume that $W(\mec d)<0$; then
then for some $a\in\integers$ 
\eqref{eq_W_is_minus_one} holds;
but in this case, we have $a\ge 1$ (since $f$ cannot attain the value
$-1$, and we have $a<2$, because if $f(\mec d-\mec e_1)\ge 2$,
then $f(\mec d)=f(\mec d-\mec e_1)+1$.
Hence we necessarily have $a=1$; since
$$
f(\mec d)=f(\mec d-\mec e_1)=f(\mec d-\mec e_2)=a=1,
$$
we have that $\deg(\mec d)=1$.
\end{proof}

\begin{corollary}
\label{co_non_negative_weights_restrict_genus_one}
Let $f\from\integers^n\to\integers$ be of genus $1$ such that for
for some (necessarily distinct)
$i,j\in[n]$ and all $\mec d\in\integers^n$ of degree $0$ we have
$$
f(\mec d)=1 \quad\implies\quad f(\mec d+\mec e_i-\mec e_j)=0.
$$
Then any restriction $f_{i,j,\mec d}$ has non-negative weight, i.e.,
its weight is a perfect matching.
\end{corollary}

\begin{example}\label{ex_weights_genus_one_examples}
Let $n=4$ in Example~\ref{ex_cycle_Baker_Norine} of a cycle of length $4$,
or the equivalent special case of
Example~\ref{ex_classical_genus_one}, where $P_1$ is order $4$, and
$P_i=iP_1$ for $i=2,3,4$.
We easily see that 
$f$ satisfies the hypothesis of
Corollary~\ref{co_non_negative_weights_restrict_genus_one}
for all distinct $i,j\in [n]$.
Then the two-variable restriction of $f$, $f_{1,2,\mec 0}$, is a Riemann
function, and we easily see that for $a_1\in\integers$ we have
$$
f_{1,2,\mec 0}(a_1,-a_1)=1
\iff
a_1 \bmod 4 = 0.
$$
Similarly for $f_{1,3,\mec 0}$, except that
$$
f_{1,3,\mec 0}(a_1,-a_1)=1
\iff
a_1 \bmod 2 = 0.
$$
We easily see that the weight, $W$, of $f_{1,2,\mec 0}$ has
period $4$ and $W(0,0)=W(1,1)=W(-1,2)=W(-2,1)=1$
(which determines $W$ everywhere).
By contrast, the weight, $W'$ of $f_{1,3,\mec 0}$ has period
$2$ and satisfies $W'(0,0)=W'(1,1)=1$.
Hence
$W'(-1,3)=W'(-2,2)=1$, and so $W'\ne W$.
\end{example}

\section{Weight Decomposition Theorems}
\label{se_weight_decomp}

In this section we prove two theorems about decomposing
the weight of Riemann functions $\integers^2\to\integers$ into
an alternating sum of perfect matchings.  These are fundamental
steps in modeling an arbitrary Riemann function $\integers^2\to\integers$.

\subsection{$r$-fold Matchings and Infinite Versions of Hall's Theorem}

In this subsection we define $r$-fold matchings, which feature
prominently in our models; we also 
give two infinite versions of Hall's Theorem
which help understand $r$-fold matchings
but are not essential to the rest of this article.

\begin{definition}
Let $W\from\integers^2\to\integers$ 
be initially and eventually zero.  For $i\in\integers$, the
{\em $i$-th row sum of $W$} (respectively, {\em column sum})
is $\sum_{j\in\integers} W(i,j)$ (respectively, 
$\sum_{j\in\integers} W(j,i)$).
For any $r\in\naturals$, we say that
$W$ is an {\em $r$-fold matching} 
all values of $W$ are non-negative and all the row sums and all the
column sums of $W$ equal $r$.
\end{definition}
Of course, a perfect matching (Definition~\ref{de_perfect_matching})
is a $1$-fold matching, and the sum of $r$ perfect matchings
is an $r$-fold matching.  The rest of this subsection is devoted
to proving the converse, both for general $r$-fold matchings
and for $p$-periodic matchings (as equaling a sum of $r$ $p$-periodic 
perfect matchings).

\begin{definition}
If $W$ is a perfect matching, we refer to the $\pi$ satisfying
\eqref{eq_W_perfect_and_pi} as the {\em bijection associated to $W$};
if $\pi\from\integers\to\integers$ is a bijection with $\pi(i)+i$
bounded independent of $i$, we say that $W$ in
\eqref{eq_W_perfect_and_pi} is the {\em perfect matching associated
to $\pi$}.
\end{definition}

Of course, if $W$ is a perfect matching with associated bijection
$\pi$, then for any $p\ge 1$, $W$ is $p$-periodic iff
for all $i\in\integers$ we have $\pi(i+p)=\pi(i)-p$.

% It will be useful to prove two infinite versions of Hall's theorem.
% The first one involves $p$-periodic $r$-fold matchings.

\begin{theorem}\label{th_halls_theorem_infinite_periodic}
Let $W$ be an $r$-fold matching that is $p$-periodic.  
Then there exist $p$-periodic perfect matchings
$W_1,\ldots,W_r$ whose sum is $W$.
\end{theorem}
\begin{proof}
We will first prove that
there is a perfect matching $W_1$ such that
$W_1(i,j)\le W(i,j)$ for all $i,j$; if so, then $W-W_1$ is an
$(r-1)$-fold matching, and hence we prove the theorem
by induction on $r$.

Let $G=(V,E)$ be the bipartite graph, where $V=V_L\amalg V_R$
and all edges run from left to right, where 
$V_L=V_R=\{0,1,\ldots,p-1\}$, and the number of edges running from
$i\in V_L$ to $j\in V_R$ is just 
$$
e(i,j)=\sum_{m\equiv j \bmod p} W(i,j) .
$$
Then $G$ is a finite bipartite graph that is $r$-regular on both sides,
i.e., each vertex is incident upon exactly $r$ edges.
It then follows that if $V'\subset V_L$, and
$\Gamma(V')$ denotes the set of neighbours of $V_L$, i.e., of vertices
(in $V_R$) adjacent to some vertex of $V'$, then
$|\Gamma(V')|\ge |V'|$ (since $V'$ is incident upon $r|V'|$ edges, whose
right endpoints span at least $|V'|$ vertices).
Similarly if $V'\subset V_R$, then also $|\Gamma(V')|\ge |V'|$.
Then Hall's theorem implies that $G$ has a perfect matching, i.e.,
a subgraph $G'=(V,E')$ where each vertex is adjacent to
exactly one vertex.
This gives us a bijection $\pi\from V_L\to V_R$ such that 
$e(i,\pi(i))\ge 1$.  For each $i=0,\ldots,p-1$ 
choose a $j=\pi_1(i)$ such that
$j\equiv \pi(i)\bmod p$ such that
$W(i,j)\ge 1$.
Now extend $\pi_1(i): \{0,\ldots,p-1\}\to\integers$ as a function
$\integers\to\integers$ by setting for all $m\in\integers$ and 
$i\in\{0,\ldots,p-1\}$
$$
\pi_i(i+pm) = \pi(i)-pm.
$$
It follows that $\pi$ is a bijection, and that its associated
weight, $W_1$, satisfies $W(i,j)\ge 1$ whenever $W_1(i,j)=1$.
Hence $W_1\le W$ and we have our desired $W_1$.
\end{proof}

The next case we prove the same theorem without the assumption
of periodicity.

\begin{theorem}\label{th_halls_theorem_infinite_non_periodic}
Let $W$ be an $r$-fold matching.  Then there exist perfect matchings
$W_1,\ldots,W_r$ whose sum is $W$.
\end{theorem}
The proof is well-known, based on the
a general principle that Philip Hall's ``marriage theorem''
holds on a bipartite
graph with countably many vertices on each side, provided that
each vertex has finitely many neighbours
(namely, in our case, at most $r$ neighbours);
see Marshall Hall's textbook
(e.g., Theorem~5.1.2 of \cite{marshall_hall_textbook}); we give a
proof---in terms of our language---for ease of reading this article.
\begin{proof}
Again, it suffices to show that there is a perfect matching $W_1$ such that
$W_1(i,j)\le W(i,j)$ for all $i,j$, and then to prove the above theorem
by induction on $r$.

Consider the bipartite graph, $G=(V,E)$, $V=V_L\amalg V_R$
with $V_L=V_R=\integers$, and where the number of edges from $i\in V_L$
to $j\in V_R$ is $W(i,j)$.
Let $V$ be enumerated as $v_1,v_2,\ldots$ .
Again, $|\Gamma(V')|\ge |V'|$ since $G$ is $r$-regular on both sides,
and hence the augmenting path technique to prove Hall's theorem shows
that for any $m$ there is a matching $G_m=(V_m,E_m)\subset G$,
meaning that each vertex of $V_m$ is adjacent via $E_m$ to exactly
one other vertex of $V_m$, with the property that
$\{v_1,\ldots,v_m\}\subset V_m$.

Now we use $G_1,G_2,\ldots$ to build a perfect matching
$G'\subset G$.
Namely, since $v_1\subset V_m$ for $m\ge 2$, and since $v_1$ is adjacent
to at most $r$-vertices in $G$, there is an (infinite) subsequence of
$G_1,G_2,\ldots$
in which $v_1$ is adjacent to some fixed vertex of $V$;
next we choose a further infinite subsequence
in which $v_2$ is adjacent to some fixed vertex of $V$;
we similarly apply this process to $v_3,v_4,\ldots$.
This gives a fixed matching defined on $v_1,v_2,\ldots$,
and therefore on all of $V$.
Hence we get a bijection $\pi_1\from\integers\to\integers$
such that for all $i\in\integers$, $W(i,\pi_1(i))\ge 1$.
Hence we take $W$ to be the weight associated to $\pi_1$.
\end{proof}

\subsection{Main Lemmas about Weights
$\integers^2\to\integers$ as an Alternating Sum of Perfect Matchings}

In this subsection we prove that any slowly growing Riemann function
has a weight that can be written as the alternating sum of perfect
matchings.  It will be convenient to prove a more general result,
namely Lemma~\ref{le_row_col_sums_weight_as_perfect} below.

\begin{definition}
Let $r\in\integers$.  By an {\em $r$-regular weight} we mean a
function $W\from\integers^2\to\integers$ such that:
\begin{enumerate}
\item $W$ is initially and eventually zero;
\item 
each row sum and each column sum of $W$ equals $r$;
\item 
for some $C\in\naturals$, for all $\mec d\in\integers^2$ we have
$|W(\mec d)|\le C$.
\end{enumerate}
\end{definition}

The rest of this subsection is devoted to proving the following lemma.

\begin{lemma}\label{le_row_col_sums_weight_as_perfect}
Let $r\in\integers$,
and let $W$ be a $r$-regular weight.
Then
for some $\ell\in\naturals$ we may write
\begin{equation}\label{eq_W_alternating_sum}
W = (W_1+W_2 + \cdots + W_\ell)  - (\tilde{W_1} + \cdots + \tilde{W}_{\ell-r})
\end{equation}
where $W_1,\ldots,W_{\ell}$ and $\tilde{W_1}, \ldots, \tilde{W}_{\ell-r}$  
are perfect matchings.
Moreover, if $W$ is $p$-periodic, then we may take each
$W_i$ and $\tilde{W_i}$ to be $p$-periodic.
\end{lemma}

Lemma~\ref{le_row_col_sums_weight_as_perfect} is immediate if
there exists an $a\in\integers$,
such that $W(\mec d)=0$ whenever $\deg(\mec d)\ne a$, for then
we must have $W(\mec d)=r$ whenever $\deg(\mec d)=a$.
It will be helpful to introduce some definitions and notation
related to this simple observation.

\begin{definition}
For $b\in\integers$, the
{\em perfect matching in degree $b$}, denoted $W_b$,
refers to the perfect matching given by
\begin{equation}\label{eq_W_b_one_periodic}
W_b(\mec d)=
\left\{
\begin{array}{ll}
1 & \mbox{if $\deg(\mec d)=b$, and} \\
0 & \mbox{otherwise.} \\
\end{array}\right.
\end{equation}
\end{definition}
Hence each $W_b$ is $1$-periodic (and any perfect matching that is
$1$-periodic is of this form).

\begin{definition}
Let $W\from\integers^2\to\integers$ be any function.  The {\em support of $W$}
is the set of $\mec d\in\integers^2$ such that $W(\mec d)\ne 0$.
If $a,b\in\integers$ with $a<b$,
we say that $W$ {\em is supported in degrees $a$ through $b$}
if the support of $W$ is a subset of those $\mec d\in\integers^2$
with $a\le \deg(\mec d)\le b$.
\end{definition}

\begin{lemma}\label{le_support_degree_within_one}
Let $W$ be an $r$-regular weight for some $r\in\integers$, 
and say that
there exists an $a\in\integers$ such that $W$ is supported in
degrees $a$ and $a+1$.
Then for some $c\in\integers$ we have
$$
W(\mec d) =
\left\{
\begin{array}{ll}
c & \mbox{if $\deg(\mec d)=a$,} \\
r-c & \mbox{if $\deg(\mec d)=a+1$, and} \\
0 & \mbox{otherwise.}
\end{array}\right.
$$
Moreover, $W$ can be written as a difference of a 
sum of perfect matchings, each of which equals either $W_a$ or 
$W_{a+1}$ with notation as in \eqref{eq_W_b_one_periodic}.
\end{lemma}
\begin{proof}
Let $W(0,a)=c$.  Then $W(0,a+1)=W(1,a)=r-c$, given that the $0$-th column
sum and $a$-th row sum of $W$ both equal $r$.
Similarly, $W(1,a-1)=W(-1,a+1)=c$.  It then follows by induction on
$m=2,3,\ldots$ that $W(m,a-m)=W(-m,a+m)=c$ and $W(m,a-m+1)=W(-m,a+m+1)=r-c$.

For the second claim, we have
\begin{equation}\label{eq_W_as_degree_a_a_plus_one_perfect}
W = c W_a + (r-c) W_{a+1};
\end{equation} 
if $r,c-r$ are both non-negative, 
then \eqref{eq_W_as_degree_a_a_plus_one_perfect}
expresses $W$ as a sum
of $r$ perfect matchings; 
otherwise 
\eqref{eq_W_as_degree_a_a_plus_one_perfect} expresses $W$
as a difference of two sums
of perfect matchings (since $r\ge 0$, and hence at least one of $r,c-r$
is non-negative).
\end{proof}

To prove
Lemma~\ref{le_row_col_sums_weight_as_perfect}
we will use
induction on $b-a$ where $W$ is supported on elements of 
degrees between $a$ and $b$.  The discussion above deals with the
cases where $b=a$ or $b=a+1$.  Let us explain the inductive step.
For this it will be helpful to introduce the following notation:
first, let $U\from\integers^2\to\integers$ be the function
$$
U(\mec d) =
\left\{
\begin{array}{ll}
1 & \mbox{if $\mec d=(0,0),(1,1)$,} \\
-1 & \mbox{if $\mec d=(1,0),(0,1)$, and} \\
0 & \mbox{otherwise.} \\
\end{array}\right. 
$$
[For intuition, note that all row and columns sums of $U$ equal $0$.]
For a doubly-infinite sequence $S=\{\ldots,s_{-1},s_0,s_1,\ldots\}$ of
integers we use the notation $U_S$ to denote
the function $\integers^2\to\integers$ given by
$$
U_S(\mec d) = 
\sum_{i\in\integers} U\bigl(\mec d + (i,-i) \bigr) s_i 
$$
[for intuition, it may help to observe that
$U_S$ is the convolution of $U$ with the function supported in
degree $0$ taking $(i,-i)$ to
$s_{-i}$].
Hence we have $U_S$ is supported on $\mec d$ of degrees $0,1,2$, and
for all $a\in\integers$, $U_S(a,-a)=s_a$.
% [Intuitively $U_S$ is the convolution of $U$ with the function
% in degree zero whose value at $a$ is $s_a$.]
Clearly all row and column sums of $U$ are zero, and hence the same
holds of $U_S$.

Here is the essential ingredient in our inductive step.

\begin{lemma}\label{le_inductive_step}
Let $S=\{\ldots,s_{-1},s_0,s_1,\ldots\}$ be any doubly-infinite sequence
of elements of $\{0,1\}$.
Then $U_S$ can be written as the difference of a sum of perfect matchings
supported in degrees $0$ through $2$.
Furthermore if for some $p\in\naturals$ we have 
$s_{a+p}=s_a$ for all $a\in\integers$, then the perfect matchings in
the sums can be taken to be $p$-periodic.
\end{lemma}
\begin{proof}
First consider the lemma in case where $S$ is not assumed to be
periodic, in the special case where
for all $a\in\integers$ with $a$ odd we have $s_a=0$;
let us prove the lemma in this situation.
Let $W_1$ be as 
\eqref{eq_W_b_one_periodic} and let $W$ be given as follows:
for all $t\in\integers$
\begin{enumerate}
\item
if $s_{2t}=1$, 
$W(2t,-2t)=W(2t+1,-2t+1) = 1$;
\item
if $s_{2t}=0$,
$W(2t+1,-2t)=W(2t,-2t+1) = 1$;
\item
all other values of $W$ not specified above are zero.
\end{enumerate}
We easily check that:
\begin{enumerate}
\item
$W$ and $W_1$ are perfect matchings;
\item 
$W,W_1$ are supported
in degrees $0$ through $2$; 
\item 
we have $W-W_1=U_S$.
\end{enumerate}
Hence this proves the lemma in this special case of $S$.

We may similarly show the case where for all $a\in\integers$ with $a$ even we
have $s_a=0$ (i.e., by translating
the construction in the last paragraph by $(1,-1)$).

In general we can write $S=S_{\rm even}+S_{\rm odd}$ (where $+$ means
adding the sequences element-by-element), where
\begin{align*}
S_{\rm even} = (\ldots,s_{-2}, 0,& s_0,0,s_2,0\ldots) \\
S_{\rm odd} = (\ldots,0,s_{-1}, & 0,s_1,0,s_3,\ldots) 
\end{align*}
As such we have
$$
U_S = U_{S_{\rm even}} + U_{S_{\rm odd}},
$$
and now we can write $U_{S_{\rm even}}$ and
$U_{S_{\rm odd}}$ each as the difference of two
perfect matchings.
This proves the lemma in the non-periodic case.

Next consider the case when $S$ is $p$-periodic.  

If $p=1$, then
$s_a$ are all $1$ or $0$; if they are all $0$ then $U_S$ is identically
zero, and otherwise $U_S=W_0$ with notation as in
\eqref{eq_W_b_one_periodic}.

Hence we may assume $p\ge 2$.  
If $p$ is even, then we can write $U_S$ as above, and notice that
$S_{\rm even}$ and $S_{\rm odd}$ are $p$-periodic.
Then follows that when we write 
$U_{S_{\rm even}}=W-W_1$ as above, both $W$ and $W_1$ are $p$-periodic
($W_1$ is $1$-periodic),
and similarly for $U_{S_{\rm odd}}$.
This solves the lemma in this case.

The only case that remains is when $S$ is $p$-periodic
when $p\ge 3$ is odd (in which case
$S_{\rm even},S_{\rm odd}$ are not $p$-periodic).
In this case we take a similar
approach, being careful to have the matchings all $p$-periodic as follows:
first consider the special
case of $S$ for which $s_a=0$ whenever $a\in\integers$ is not divisible by $p$.
For each $t\in\integers$ we let $W$ be the following perfect matching:
\begin{enumerate}
\item
if $s_{pt}=1$,
$W(pt,-pt)=W(pt+1,-pt+1) = 1$;
\item
if $s_{pt}=0$,
$W(pt+1,-pt)=W(pt,-pt+1) = 1$;
\item
for all $a$ with $a\bmod p\ne 0,1$, $W(a,1-a)=1$;
\item
all other values of $W$ not specified above are zero.
\end{enumerate}
We then have that $W$ is $p$-periodic, and $U_S=W-W_1$.
This solves the lemma in this case.

Similarly, we solve the lemma in case for some $i=1,\ldots,p-1$
we have $s_a=0$ for all $a\in\integers$ with $a \bmod p\ne i$.

For the general case of $S$ $p$-periodic with $p\ge 3$ odd,
we write
$$
S = S_0 + \cdots + S_{p-1},
$$
where for $i=0,\ldots,p-1$, $(S_i)_a=0$ if $a\bmod p\ne i$.
Then we write each $U_{S_i}$ as a difference of periodic matchings
supported in degrees $0$ through $2$, and use
$$
U_S = U_{S_0} + \cdots + U_{S_{p-1}} 
$$
to write $U_S$ as the difference of
sum of $p$ perfect matchings, each $p$-periodic.
\end{proof}

\begin{proof}[Proof of Lemma~\ref{le_row_col_sums_weight_as_perfect}]
First let us prove the lemma in the case where $W$ attains only
non-negative values.

Let us prove the lemma in this case by induction on $m=0,1,\ldots$
for all $W$ supported in degrees $a$ through $a+m$.
The cases $m=0,1$ are given in
Lemma~\ref{le_support_degree_within_one}.
Now consider the inductive step, where the lemma holds for $m-1$
for some $m\ge 2$ and we wish to prove it for $m$.
By translating $W$ we may assume that it is supported in
degrees $0$ through $m$.
Let 
$$
C=\max_{t\in\integers} W(t,-t).
$$
Now let us prove our desired inductive step, i.e., that the 
lemma holds for $W$ supported in degrees $0$ through $m$,
by using induction on $C=0,1,\ldots$
For $C=0$, it follows that $W$ is supported in degrees $1$ through $m$,
and hence by translation we can reduce the theorem to the case $m-1$.

Now say the claim holds for some value of $m$ and $C\ge 0$,
and say that $W(t,-t)\le C+1$.  Let
$S=\{\ldots,s_{-1},s_0,s_1,\ldots\}$ be given by
$$
s_t = \max\bigl( 1, W(t,-t) \bigr) .
$$
Then $s_t\in\{0,1\}$ for all $t$, and if $W$ is $p$-periodic then
$s_{t+p}=s_t$ for all $t\in\integers$.
According to Lemma~\ref{le_inductive_step} we can find
a difference of sums of perfect matchings supported in degrees $0$
through $2$---all $p$-periodic if
$W$ is $p$-periodic---whose value at $(a,-a)$ equal $s_a$.
Subtracting this difference of sums from $W$ we get $W'$ where
$W'(t,-t)=W(t,-t)-s_t$, to which we can apply the inductive claim.

This proves the lemma assuming $W(\mec d)\ge 0$ for all $\mec d$.
If $W$ is supported in degrees $a$ and $b$ and is bounded in absolute
value by $C\in\naturals$, then with
notation in \eqref{eq_W_b_one_periodic} we
have 
$$
W' = W+ C(W_a+\cdots+W_b)
$$
attains only non-negative values for some $C$ sufficiently large,
and is an $r'$-fold matching for $r'=r+C(b-a+1)$.
Hence we apply the lemma to $W'$, and then subtract $C(W_a+\cdots+W_b)$.
\end{proof}

\section{Diagrams, Betti Numbers, and Models for Riemann Functions Whose 
Weight is a Perfect Matchings}
\label{se_diagrams_Betti_perfect}

In this section we introduce our basic models and develop some of
their properties.
We will especially study those related to Riemann functions
$f\from\integers^2\to\integers$ whose weight is a perfect matching;
such functions have a number of especially remarkable
properties.

\subsection{Conventions Regarding Linear Algebra: Cohomology,
Betti Numbers, Fredholm Maps, and Direct Sums}
% \subsection{Betti Numbers, Fredholm Maps, and Direct Sums}

In this subsection we some basic concepts in linear algebra that
we will need to compute the sheaf invariants of interest to us.
Our motivation is that the invariants of the sheaves that we use
can be computed as the kernel and the cokernel\footnote{
  If $\tau\from B\to A$ is a linear map of vector spaces, the
  {\em cokernel} of $\tau$ is $A/{\rm Image(\tau)}$.
  }
of an associated linear map.

Let $k$ be a field and $\tau\from B\to A$ a linear map
of $k$-vector spaces $B,A$.
For $i=0,1$
we define the {\em $i$-th cohomology group
of $\tau$} to be, respectively
$$
H^0(\tau)\eqdef \ker(\tau), \quad
H^1(\tau)\eqdef \coker(\tau),
$$
and the {\em $i$-th Betti number} of $\tau$ to be
$b^i(\tau)=\dim_k H^i(\tau)$;
we say that $\tau$ is a {\em $k$-Fredholm map},
or simply {\em Fredholm}, if both Betti numbers are finite,
and if so we define
the {\em Euler characteristic of $\tau$} (also known is its {\em index})
to be
$$
\chi(\tau)=b^0(\tau)-b^1(\tau) ;
$$
if exactly one of $b^0(\tau)$ and $b^1(\tau)$ is infinite, we
may also define $\chi(\tau)$ as $\pm\infty$ accordingly.

[Hence the $i$-th cohomology group and $i$-th Betti number
of $\tau$ is the usual notion when we view $\tau$ as a chain
$$
\cdots \to 0 \to B \to A \to 0 \to \cdots
$$
with $B$ positioned in degree $0$.]

If $\{B_i\}_{i\in I}$ is a family of $k$-vector spaces indexed on a
set $I$, we define its {\em direct sum}, denoted
$\oplus_{i\in I} B_i$ as usual, i.e., the vector
space of tuples $\{b_i\}_{i\in I}$ such that each $b_i\in B_i$ and
all but finitely many of the $b_i$ are zero.
If $\tau_i\from B_i\to A_i$ is a family of $k$-linear maps of vector
spaces indexed on $i\in I$, we define the {\em direct sum} of
$\{\tau_i\}_{i\in I}$ as usual, i.e., as the map
$$
\bigoplus_{i\in I} \tau_i
\from
\bigoplus_{i\in I} B_i
\to
\bigoplus_{i\in I} A_i;
$$
we easily check that for $j=0,1$ we have a simple isomororphism
$$
H^j\biggl( \bigoplus_{i\in I}\tau_i \biggr) \isom
\bigoplus_{i\in I} H^j(\tau_i),
$$
and hence Betti numbers
$$
b^j\biggl( \bigoplus_{i\in I}\tau_i \biggr) =
\sum_{i\in I} b^j(\tau_i),
$$
so that if all the $\tau_i$ are Fredholm maps, where all but finitely
many of the $\tau_i$ have both Betti numbers equal to zero, we get a
finite and well-defined Euler characteristic
$$
\chi \biggl( \bigoplus_{i\in I}\tau_i \biggr) =
\sum_{i\in I} \chi(\tau_i).
$$

\subsection{$k$-Diagrams: Diagrams of $k$-Vector Spaces}
\label{su_k_diagram_basics}

Our models of Riemann functions will be $k$-linear maps
$\tau\from B\to A$ which are built from one fixed type 
of ``diagram'' of vector spaces, depicted in
Figure~\ref{fi_our_diagrams_new_again}
and which we now make precise.
 \begin{figure}
\hspace*{2cm}\begin{tikzpicture}[scale=0.75]
% \node at (8,2) {$B=B_1\oplus B_3\oplus B_2$};
% \node at (8,0) {$A=A_1\oplus A_2$};
\node (B1) at (0,2) {$\cF(B_1)$};
\node (B2) at (0,-2) {$\cF(B_2)$};
\node (B3) at (0,0) {$\cF(B_3)$};
\node (A1) at (8,1) {$\cF(A_1)$};
\node (A2) at (8,-1) {$\cF(A_2)$};
\draw [->] (B1) -- (A1) node [midway,above] {$\cF(\rho_{1,1})$} ;
\draw [->] (B2) -- (A2) node [midway,below] {$\cF(\rho_{2,2})$} ;
\draw [->] (B3) -- (A1) node [midway,above] {$\cF(\rho_{3,1})$} ;
\draw [->] (B3) -- (A2) node [midway,below] {$\cF(\rho_{3,2})$} ;
\end{tikzpicture}
\caption{Our Diagrams}
\label{fi_our_diagrams_new_again}
\end{figure}
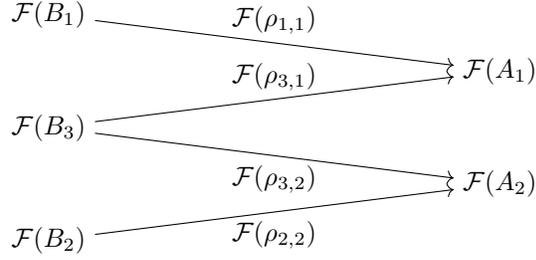
% \begin{figure}
%
% % OLD PIC
% % \begin{tikzpicture}
% % \node at (-3,0) {$B=B^1\oplus B^3\oplus B^2$};
% % \node at (6,0) {$A=A^1\oplus A^2$};
% % \node (B1) at (0,2) {$B^1$};
% % \node (B2) at (0,-2) {$B^2$};
% % \node (B3) at (0,0) {$B^3$};
% % \node (A1) at (4,1) {$A^1$};
% % \node (A2) at (4,-1) {$A^2$};
% % \draw [->] (B1) -- (A1) node [midway,above] {$\phi^{1,1}$} ;
% % \draw [->] (B2) -- (A2) node [midway,below] {$\phi^{2,2}$} ;
% % \draw [->] (B3) -- (A1) node [midway,above] {$\phi^{3,1}$} ;
% % \draw [->] (B3) -- (A2) node [midway,below] {$\phi^{3,2}$} ;
% % \end{tikzpicture}
%
% \begin{tikzpicture}[scale=0.6]
% \node (B1) at (-2,2) {$B^1$};
% \node (B2) at (-2,-2) {$B^2$};
% \node (B3) at (-2,0) {$B^3$};
% \node (A1) at (4,1) {$A^1$};
% \node (A2) at (4,-1) {$A^2$};
% \draw [->] (B1) -- (A1) node [midway,above] {$\phi^{1,1}$} ;
% \draw [->] (B2) -- (A2) node [midway,below] {$\phi^{2,2}$} ;
% \draw [->] (B3) -- (A1) node [midway,above] {$\phi^{3,1}$} ;
% \draw [->] (B3) -- (A2) node [midway,below] {$\phi^{3,2}$} ;
% \node at (10,1) {$B=B^1\oplus B^3\oplus B^2$};
% \node at (10,0) {$A=A^1\oplus A^2$};
% \node at (10,-1) {$\phi(\partial)\from B\to A$};
% \end{tikzpicture}
% \caption{Our Diagrams}
% \label{fi_our_diagrams}
% \end{figure}

\begin{definition}\label{de_diagram_k_vs}  % vs = vector space
Let $k$ be a field.
By a {\em diagram of $k$-vector spaces}, or simply a
{\em $k$-diagram} we mean a collection, $\cF$,
of data consisting of:
\begin{enumerate}
\item
five $k$-vector spaces, 
$$
\cF(B_1),\cF(B_2),\cF(B_3),\cF(A_1),\cF(A_2)
$$
called the {\em values} of $\cF$;
and
\item
$k$-linear maps $\cF(\rho_{i,j})\from \cF(B_i)\to \cF(A_j)$ for the pairs 
$(i,j)$ where $(i,j) \in \{ (1,1),(2,2),(3,1),(3,2) \}$  
(i.e., $\cF(\rho_{1,2})$ and $\cF(\rho_{2,1})$ don't exist); we call the
$\cF(\rho_{ij})$ the {\em restriction maps} of $\cF$.
\end{enumerate}
To this diagram we associate the vector spaces
$$
\cF(B)=\cF(B_1)\oplus\cF(B_2)\oplus\cF(B_3),\quad
\cF(A)=\cF(A_1)\oplus\cF(A_2),
$$
and the
linear transformation
$\cF(\partial)\from \cF(B)\to \cF(A)$, called the
{\em differential of $\cF$}, given by
\begin{equation}\label{eq_formula_for_cF_partial}
\cF(\partial)(b_1,b_2,b_3) =
\bigl(
\cF(\rho_{1,1})(b_1)-\cF(\rho_{3,1})(b_3) ,
% \cF(\rho_{3,2})(b_3)-\cF(\rho_{2,2})(b_2)
\cF(\rho_{2,2})(b_2) - \cF(\rho_{3,2})(b_3)
\bigr),
\end{equation} 
and define the {\em zeroth and first cohomology groups of $\cF$}
to be, respectively
$$
H^0(\cF) \eqdef \ker\bigl( \cF(\partial) \bigr), 
\quad
H^1(\cF) \eqdef \coker\bigl( \cF(\partial) \bigr),
$$
i.e.,  
the kernel and cokernel of
$\cF(\partial)$.
If $(b_1,b_2,b_3)\in H^0(\cF)$, and $a_j=\cF(\rho_{jj})b_j$ for 
$j=1,2$, 
then the tuple $(b_1,b_2,b_3,a_1,a_2)$ satisfies
$\cF(\rho_{ij})b_i=a_j$ whenever $\cF(\rho_{ij})$ is defined,
and we refer to $(b_1,b_2,b_3,a_1,a_2)$ as a
{\em global section of $\cF$}.
\end{definition}

\begin{convention}
When we speak of a {\em $k$-vector space}
or a {\em $k$-diagram} without prior reference to $k$, we understand
$k$ to be an arbitrary field.
\end{convention}

Note that if $(b_1,b_2,b_3,a_1,a_2)$ is a global section, then
$(b_1,b_2,b_3)\in H^0(\cF)$; this therefore gives a 
bijection between global sections of $\cF$ and $H^0(\cF)$, i.e.,
the kernel of the differential $\cF(\partial)$ of $\cF$.
[Global sections tend to be conceptually more useful, 
but equivalent
descriptions are useful in certain computations; we will
later give another equivalent description of global sections
as elements of $\Hom(\underline k,\cF)$.]

In Section~\ref{se_duality_second} we will explain that our choice of
$\cF(\partial)$ and $H^1(\cF)$ are not canonical, 
but involve a
choice of basis for each of two one-dimensional vector spaces;
see the proof of
Lemma~\ref{le_projective_resolution_of_underline_k} and the remark
after its proof;
however, this choice does not affect $H^0(\cF)$, i.e.,
the kernel of $\cF(\partial)$.

\subsection{Conventions on Sets, Multisets, Induced Vector Spaces,
and $\cM_{W,\mec d}$ for
Non-Negative Weights}

\begin{definition}
Let $k$ be a field.  If $S$ is a set, we use $k^{\oplus S}$ to
denote the $k$-vector space that is direct sum of one copy of $k$ for
each element of $S$, i.e., whose elements are collections
$\{v_s\}_{s\in S}$ with $v_s\ne 0$ for at most finitely many
values of $s$; for $s\in S$, we use $\mec e_s$ to denote the
vector that is $1$ in component $s$ and $0$ elsewhere.
If $T$ is another set and $\alpha\from S\to T$ a map of sets, then
$\alpha$ gives rise to a unique
$k$-linear transformation, denoted $k^{\oplus\alpha}$,
from $k^{\oplus S}\to k^{\oplus T}$
taking $\mec e_s$ to $\mec e_{\alpha(s)}$.
If $S\subset T$, then the inclusion map $\iota\from S\to T$
gives an injection $k^{\oplus\iota}$ which
we call the {\em inclusion map (of $k^{\oplus S}$ to $k^{\oplus T}$)}.
\end{definition}
In the above one easily checks that if 
$\alpha$ is an injection, surjection, or bijection, then 
the same is true of $k^{\oplus\alpha}$.
Next we fix a convention for multisets (any reasonable 
convention would suffice).

\begin{definition}
Let $S_1,S_2$ be sets, and $W\from S_1\times S_2\to 
\integers_{\ge 0}\cup\{\infty\}$.
The {\em multiset on $S_1\times S_2$ with multiplicities $W$}
refers to the set
\begin{equation}\label{eq_multiset_convention}
{\rm Multi}(W)
= \{ (s_1,s_2,i) \in S_1\times S_2\times\naturals \ | \ i\le W(s_1,s_2) \} ,
\end{equation} 
where if $W(s_1,s_2)=\infty$, then we view all $i$ as satisfying 
$i\le W(s_1,s_2)$.
We refer to the maps 
${\rm Multi}(W)\to S_1$ and 
${\rm Multi}(W) \to S_2$
taking $(s_1,s_2,i)$ to, respectively, $s_1$ and $s_2$,
as, respectively, the {\em first and second projections}.
We use the notation $k^{\oplus W}$ to denote $k^{\oplus{\rm Multi}(W)}$,
which comes with maps 
\begin{equation}\label{eq_oplus_W_notation}
{\rm proj}_i \from k^{\oplus W}\to k^{\oplus S_i}
\end{equation} 
induced by the first and second projections.
The {\em support of $W$} is the set of $(s_1,s_2)\in S_1\times S_2$
such that $W(s_1,s_2)\ge 1$.
When $W$ takes on only the values $\{0,1\}$, then with mild abuse of notation
we may identify ${\rm Multi}(W)$ with its support, which is a subset of
$S_1\times S_2$, since in this case $W$ is determined by its support.
\end{definition}

\begin{example}
If $W\from\integers^2\to\integers$ is a perfect matching,
and $\pi\from\integers\to\integers$ is its associated bijection,
then $k^{\oplus W}$ has one copy of $k$ for each pair
$(a_1,\pi(a_1))\in\integers^2$ varying over all $a_1\in\integers$.
In this case we may identify $k^{\oplus W}$ with $k^\integers$,
where the first projection is the identity map on $k^\integers$,
and the second projection is the map $k^\integers\to k^\integers$
takes $\mec e_{a_1}$ to $\mec e_{\pi(a_1)}$.
Hence both maps
\eqref{eq_oplus_W_notation} are isomorphisms.
\end{example}

\begin{definition}\label{de_cM_W_d}
Let $k$ be a field,
$W\from\integers^2\to\integers_{\ge 0}\cup\{\infty\}$, 
and $\mec d\in\integers^2$.  
We use $\cM_{W,\mec d}$ to denote the following $k$-diagram
(Definition~\ref{de_diagram_k_vs}): 
\begin{enumerate}
\item 
for $i=1,2$,
$\cM_{W,\mec d}(B_i)=k^{\oplus\integers_{\le d_i}}$, 
$\cM_{W,\mec d}(A_i)=k^{\oplus\integers}$,
$\rho_{i,i}$ is the inclusion,
\item 
$B_3=k^{\oplus W}$, and for $j=1,2$, $\rho_{3,j}$ are
the projection maps
(as in \eqref{eq_oplus_W_notation}).
\end{enumerate}
\end{definition}
We depict these $k$-diagrams in Figure~\ref{fi_cM_W_d}.

 \begin{figure}
\hspace*{2cm}
\begin{tikzpicture}[scale=0.60]
\node (B1) at (0,2) {$\cM_{W,\mec d}(B_1)=k^{\oplus\integers_{\le d_1}}$};
\node (B2) at (0,-2) {$\cM_{W,\mec d}(B_2)=k^{\oplus\integers_{\le d_2}}$};
\node (B3) at (0,0) {$\cM_{W,\mec d}(B_3)=k^{\oplus W}$};
\node (A1) at (10,1) {$k^{\oplus\integers}=\cM_{W,\mec d}(A_1$)};
\node (A2) at (10,-1) {$k^{\oplus\integers}=\cM_{W,\mec d}(A_2$)};
\draw [->] (B1) -- (A1) node [midway,above] {$\rho_{1,1}={\rm inclusion}$} ;
\draw [->] (B2) -- (A2) node [midway,below] {$\rho_{2,2}={\rm inclusion}$} ;
\draw [->] (B3) -- (A1) node [midway,above] {$\rho_{3,1}$} ;
\draw [->] (B3) -- (A2) node [midway,below] {$\rho_{3,2}$} ;
\end{tikzpicture}
\caption{The $k$-Diagram $\cM_{W,\mec d}$.}
\label{fi_cM_W_d}
\end{figure}

The cohomology groups of $\cM_{W,\mec d}$ are therefore the kernel and
cokernel of the maps
$$
\tau_{W,\mec d}= \cM_{W,\mec d}(\partial)
\from
k^{\oplus\integers_{\le d_1}}
\oplus
k^{\oplus W}
\oplus
k^{\oplus\integers_{\le d_2}}
\to
k^{\oplus\integers}
\oplus
k^{\oplus\integers} 
$$
given as the map
$$
(b_1,b_3,b_2) \mapsto 
\bigl(b_1- k^{{\rm pr}_1}(b_3),b_2- k^{{\rm pr}_2}(b_3) \bigr),
$$
where ${\rm pr}_i$ denotes the $i$-th projection 
$k^{\oplus W}\to k^{\oplus\integers}$.

\subsection{The Euler Characteristic of
$\cM_{W,\mec d}$ as a Function of $\mec d$ and Riemann Functions}

Before computing the Betti numbers of $\cM_{W,\mec d}$ for specific
$W$ of interest, we wish to point out some general properties
of their Betti numbers and Euler characteristics.
In particular,
we will prove that if for some $\mec d$ and $W$
we have that $\chi(\cM_{W,\mec d})$ is well-defined,
i.e., at least one of the Betti numbers of $\cM_{W,\mec d}$ is finite, then
\begin{equation}\label{eq_Euler_characteristic_fixed_W_varying_mec_d}
\chi(\cM_{W,\mec d+\mec e_1})=
\chi(\cM_{W,\mec d+\mec e_2})=
\chi(\cM_{W,\mec d}) + 1 .
\end{equation} 
This is an easy consequence of the following lemma.

\begin{lemma}\label{le_codimension_one_cons}
Let $\tau\from B\to A$ be a linear map of $k$-vector spaces, 
and let $B'\subset B$ be a subspace of codimension one, and 
let $\tau'=\tau|_{B'}$, i.e., the restriction of $\tau$ to $B'$.
Then if either $\chi(\tau)$ or $\chi(\tau')$ is well defined, then so
is the other, and 
\begin{equation}\label{eq_chi_of_tau_and_one_dim_restriction}
\chi(\tau)=\chi(\tau')+1.
\end{equation} 
In more detail, 
either
\begin{equation}\label{eq_first_case_Betti_number_shift}
b^0(\tau)=
b^0(\tau')+1 \quad\mbox{and}\quad
b^1(\tau)=
b^1(\tau')
\end{equation} 
or
\begin{equation}\label{eq_second_case_Betti_number_shift}
b^0(\tau)=
b^0(\tau') \quad\mbox{and}\quad
b^1(\tau)=
b^1(\tau')-1 ,
\end{equation} 
where we allow for these Betti numbers to equal $\infty$, in which case
$\infty\pm 1$ is taken to $\infty$.
\end{lemma}
\begin{proof}
The proof of this lemma
is straightforward: since $B'$ has codimension $1$ in $B$,
$\ker(\tau')=\ker(\tau)\cap B'$ has either codimension $1$ or $0$ in
$\ker(\tau)$; we easily verify that codimension $1$ implies
\eqref{eq_first_case_Betti_number_shift} and codimension $0$ implies
\eqref{eq_second_case_Betti_number_shift}.
Both cases~\eqref{eq_first_case_Betti_number_shift}
and~\eqref{eq_second_case_Betti_number_shift} imply
\eqref{eq_chi_of_tau_and_one_dim_restriction}.
\end{proof}

\begin{corollary}\label{co_cM_Euler_char}
Let $W\from\integers^2\to\integers_{\ge 0}$.
Then for any $\mec d\in\integers^2$ and $i=1,2$, we have
that the conclusions of Lemma~\ref{le_codimension_one_cons}
hold for $\tau=\cM_{W,\mec d+\mec e_i}(\partial)$ and
$\tau'=\cM_{W,\mec d}(\partial)$.
In particular, for any $\mec d$ and $i=1,2$ we have
\begin{equation}\label{eq_Betti_zero_cM_slowly_growing}
b^0(\cM_{W,\mec d}) \le 
b^0(\cM_{W,\mec d+\mec e_i}) \le 
b^0(\cM_{W,\mec d}) + 1
\end{equation} 
(which makes sense if these Betti numbers equal $+\infty$, in which
case the above reads $+\infty\le+\infty\le+\infty$) and
\begin{equation}\label{eq_Betti_one_cM_slowly_decreasing}
b^1(\cM_{W,\mec d}) + 1 \ge 
b^1(\cM_{W,\mec d+\mec e_i}) \ge 
b^1(\cM_{W,\mec d}) .
\end{equation} 
In particular, if at least one of the Betti numbers of $\cM_{W,\mec d}$
is finite, or one of 
$\cM_{W,\mec d+\mec e_1}$ or 
$\cM_{W,\mec d+\mec e_2}$,
then \eqref{eq_Euler_characteristic_fixed_W_varying_mec_d}
holds.
\end{corollary}

Applying this corollary repeatedly we get the following result.

\begin{theorem}\label{th_cM_Euler_char_repeated_Riemann_functions}
Let $W\from\integers^2\to\integers_{\ge 0}$ be a function such that 
$\chi(\cM_{W,\mec d})$ is well defined
for some $\mec d\in\integers^2$. 
Then for all $\mec d'\in\integers^2$,
$\chi(\cM_{W,\mec d'})$ is well defined, and 
\begin{equation}\label{eq_chi_cM_mec_d_mec_d_prime}
\chi(\cM_{W,\mec d+\mec d'}) = \chi(\cM_{W,\mec d})+\deg(\mec d');
\end{equation} 
equivalently, for all $\mec d\in\integers^2$ we have
\begin{equation}\label{eq_cM_mec_d_alternating_Betti_is_Riemann}
b^0(\cM_{W,\mec d}) - b^1(\cM_{W,\mec d}) 
=\ \chi(\cM_{W,\mec d}) = \deg(\mec d)+C,
\quad\mbox{where}\quad
C=\chi(\cM_{W,\mec 0}).
\end{equation} 
Furthermore, if
for $\deg(\mec d)$ sufficiently small we have
$b^0(\cM_{W,\mec d})=0$,
and
for $\deg(\mec d)$ sufficiently large we have
$b^1(\cM_{W,\mec d})=0$,
then $f(\mec d)=b^0(\cM_{W,\mec d})$ is a 
slowly growing Riemann function, and
for any $\mec K$ and $\mec L=\mec K+\mec 1$ we have
$$
b^1(\cM_{W,\mec d}) = 
f^\wedge_{\mec K}(\mec K-\mec d).
% we don't know this yet:   =b^0(\cM_{W^*_\mec L,\mec K-\mec d})
$$
\end{theorem}
\begin{proof}
Applying Corollary~\ref{co_cM_Euler_char} repeatedly we see that
if $\chi(\cM_{W,\mec d})$ is well defined, then 
\eqref{eq_chi_cM_mec_d_mec_d_prime} holds for all
$\mec d'\ge\mec d$, and 
in particular for $a_1,a_2$ sufficiently large we have
\begin{equation}\label{eq_chi_cM_W_mec_a}
\chi(\cM_{W,a_1\mec e_1+a_2\mec e_2})=
\chi(\cM_{W,\mec d}) + a_1+a_2-d_1-d_2 .
\end{equation} 
Then applying Corollary~\ref{co_cM_Euler_char} repeatedly to
$\mec a=(a_1,a_2)$ we have that for all $\mec d'\le\mec a$
$$
\chi(\cM_{W,\mec d'}) = \chi(\cM_{W,a_1\mec e_1+a_2\mec e_2}) 
+d'_1+d'_2-a_1-a_2 ,
$$
which, in view of \eqref{eq_chi_cM_W_mec_a}, equals
the left-hand-side of \eqref{eq_chi_cM_mec_d_mec_d_prime}.
Hence \eqref{eq_chi_cM_mec_d_mec_d_prime} holds for all $\mec d'$.
Applying this with $\mec d'$ replaced with an arbitrary 
$\mec d''\in\integers^2$ and subtracting yields
$$
\chi(\cM_{W,\mec d''}) - \chi(\cM_{W,\mec d'}) 
= \deg(\mec d''-\mec d'),
$$
for all $\mec d'',\mec d'$, and setting $\mec d'=\mec 0$ and
$\mec d''=\mec d$ yields
\eqref{eq_cM_mec_d_alternating_Betti_is_Riemann}.

Setting $f(\mec d)=b^0(\cM_{W,\mec d})$, we have that if
$f(\mec d)=0$ for $\deg(\mec d)$ sufficiently small, then
$f$ is intially zero; according to
Corollary~\ref{co_cM_Euler_char}, $f(\mec d)$ is finite for all
$\mec d$ and,
by \eqref{eq_Betti_zero_cM_slowly_growing},
it is slowly growing; if $b^1(\cM_{W,\mec d})=0$ for 
$\deg(\mec d)$ large,
then for such $\mec d$ we have
$$
f(\mec d)=b^0(\cM_{W,\mec d}) = b^0(\cM_{W,\mec d})-
b^1(\cM_{W,\mec d})=\deg(\mec d)+C,
$$
and hence $f(\mec d)=b^0(\cM_{W,\mec d})$ is a Riemann function.
It is slowly growing by \eqref{eq_Betti_zero_cM_slowly_growing}.
In view of, \eqref{eq_generalized_riemann_roch},
for any $\mec K\in\integers^2$ we have
$$
f(\mec d)-f^\wedge_{\mec K}(\mec K-\mec d)=\deg(\mec d)+C
$$
and so
$$
f^\wedge_{\mec K}(\mec K-\mec d) = b^1(\cM_{W,\mec d}).
$$
\end{proof}

[In terms of sheaf theory, the above lemmas and corollaries
express the fact that the two $k$-diagrams~$\cM_{W,\mec d}$
and~$\cM_{W,\mec d+\mec e_i}$
fit into a short exact
sequence with a {\em skyscraper sheaf} supported at $B_i$ whose value is
$k$; see Subsection~\ref{su_skyscraper}.]

\subsection{Simple Examples of $\cM_{W,\mec d}$ Betti Number Bounds}

We remark that without assumptions on $W$, the cohomology groups
and Betti numbers of $\cM_{W,\mec d}$ may not be finite (or 
particularly interesting).

\begin{example}
Let $W=0$.  Then the kernel of 
$\tau_{W,\mec d}=\cM_{W,\mec d}(\partial)$ is zero, and
its cokernel can be identified with
$$
k^{\oplus\integers_{\ge d_1+1}}
\oplus
k^{\oplus\integers_{\ge d_2+1}},
$$
which is infinite dimensional.
Hence $b^0(\cM_{W,\mec d})=0$, $b^1(\cM_{W,\mec d})=+\infty$.
\end{example}

% The next two examples are similar; we leave the details to the reader.
% 
% \begin{example}
% Let $W$ satisfy $W(d,-d)=2$ for all $d\in\integers$ (and elsewhere
% $W$'s values are arbitrary).
% Then we easily see that $b^0(\cM_{W,\mec d})=+\infty$ 
% and $b^1(\cM_{W,\mec d})=0$.
% \end{example}
% 
% \begin{example}
% Let $W$ be given as $W(d_1,d_2)$ is $1$ if $d_1=0$, and $0$ if $d_1\ne 0$.
% Then we easily see that
% $b^0(\cM_{W,\mec d})=b^1(\cM_{W,\mec d})=+\infty$.
% \end{example}

\begin{example}
We easily see that if $W(\mec d)=2$ for some $\mec d\in\integers^2$,
then ${\rm Multi}(W)$ contains some elements of the form $(d_1,d_2,1),(d_1,d_2,2)$,
and if $b_3=\mec e_{(d_1,d_2,2)}-\mec e_{(d_1,d_2,1)}$, then
$(0,b_3,0)\in\ker(\tau_{W,\mec d})$.  Similarly if $W(d_1,d_2)\ge m$
for some $m\ge 3$, with $b_3=\mec e_{(d_1,d_2,m)} - \mec e_{(d_1,d_2,1)}$ we see that
\begin{equation}\label{eq_betti_zero_cM_W_lower_bound}
b^0(\cM_{W,\mec d}) \ge \sum_{\mec d\in\integers^2} \bigl( W(\mec d)-1\bigr).
\end{equation} 
\end{example}

\begin{example}
We say that $s_1\in\integers$ is {\em isolated in the first component of $W$}
if 
$W(s_1,s_2)=0$ for all $s_2$;
we similarly define when an $s_2\in\integers$ is
isolated from the second component of $W$. 
If $s_1\ge d_1+1$, then all elements of the image of $\tau_{W,\mec d}$
have a zero coefficient in the $\mec e_{s_1}$ component
in the $\cM_{W,\mec d}(A_1)$ summand of the codomain (or range) of
$\tau_{W,\mec d}$; similarly if $s_2\ge d_2+1$ is missing form
the second component of $W$.
It follows that
\begin{equation}\label{eq_betti_one_cM_W_lower_bound}
b^1(\cM_{W,\mec d}) \ge 
\bigl|{\rm Iso}_{1,\ge d_1+1}\bigr|
+ \bigl| {\rm Iso}_{2,\ge d_2+1}  \bigr|
\end{equation} 
where ${\rm Iso}_{1,\ge d_1+1}$ is the set of isolated $s_1$ in the first
component of $W$ with
$s_1\ge d_1+1$, and similarly
for ${\rm Iso}_{2,\ge d_2+1}$.
\end{example}

\begin{example}
Let $W$ be given as $W(d_1,d_2)=2$ if $d_1=0$, and $W(d_1,d_2)=0$ if $d_1\ne 0$.
Then the 
bounds~\eqref{eq_betti_zero_cM_W_lower_bound}
and~\eqref{eq_betti_one_cM_W_lower_bound}
show that
$b^0(\cM_{W,\mec d})=b^1(\cM_{W,\mec d})=+\infty$.
\end{example}

\subsection{The Betti Numbers of $\cM_{W,\mec d}$ for Perfect Matchings}

In the case $W\from\integers^2\to\{0,1\}$ is a perfect matching, it
is easy to determine its Betti numbers.

\begin{theorem}\label{th_perfect_matching}
Let $W\from\integers^2\to\integers$ be a perfect matching.  Then
\begin{enumerate}
\item
$b^0(\cM_{W,\mec d})$ equals the number of $\mec a\in\integers^2$ such
that $W(\mec a)=1$ and $\mec a\le \mec d$, and hence
$$
f(\mec d)\eqdef b^0(\cM_{W,\mec d}) = (\fraks W)(\mec d);
$$
\item
more precisely,
$H^0(\cM_{W,\mec d})$ has a basis consisting of
\begin{equation}\label{eq_basis_for_H_zero_perfect_matching}
\bigl(\mec e_{a_1},\mec e_{(a_1,a_2)},\mec e_{a_2}\bigr)\in 
\ker\bigl(\cM_{W,\mec d}(\partial) \bigr)
\quad\mbox{s.t.}\quad
\mec a\le\mec d;
\end{equation} 
\item
$b^1(\cM_{W,\mec d})$ equals the number of $\mec a\in\integers^2$
such that $W(\mec a)=1$
and $\mec a\ge \mec d+\mec 1$; and
\item
more precisely,
if $\pi$ is the bijection associated to $W$, then
$H^1(\cM_{W,\mec d})$ has a basis consisting of
the images in $H^1(\cM_{W,\mec d})$ of
\begin{equation}\label{eq_basis_for_H_one_perfect_matching}
\mec e_{a_1}\in\cM_{W,\mec d}(A_1) 
\quad\mbox{s.t.}\quad
a_1\ge d_1+1,\ a_2\ge d_2+1.
\end{equation} 
\end{enumerate}
In particular 
$b^i(\cM_{W,\mec d})$ is finite for all $i=0,1$ and all $\mec d\in\integers^2$,
and is zero when $i=0$ and $\deg(\mec d)$ is sufficiently small
or when $i=1$ and $\deg(\mec d)$ is sufficiently large; furthermore
for some $C\in\integers$ we have
$$
\chi(\cM_{W,\mec d})=\deg(\mec d) +C ,
$$
and, moreover, $C=\chi(\cM_{W,\mec 0})$.
Hence for any $\mec K\in\integers^2$
and $\mec L=\mec K+\mec 1$ we have
$$
b^1(\cM_{W,\mec d}) = f^\wedge_{\mec K}(\mec K-\mec d)
=(\fraks W^*_{\mec L})(\mec K-\mec d).
$$
\end{theorem}
\begin{proof}
Let us begin by proving claims~(1)---(4) above.
Note that~(2) implies~(1), and (4)~implies~(3),
so it suffices to prove~(2) and~(4).
The proofs of~(2) and~(4) straightforward; let us begin with~(2).

To prove~(2), we note that $\cM_{W,\mec d}(B_3)=k^{\oplus W}$,
so the vectors 
$$
\bigl(\mec e_{a_1},\mec e_{(a_1,a_2)},\mec e_{a_2}\bigr)
\in \cM_{W,\mec d}(B)=
\cM_{W,\mec d}(B_1)\oplus\cM_{W,\mec d}(B_3)\oplus\cM_{W,\mec d}(B_2)
$$
ranging over all $a_1,a_2$ such that $W(a_1,a_2)=1$
are linearly independent in $\cM_{W,\mec d}(B)$
by considering merely their $\cM_{W,\mec d}(B_3)$ component.
Consider an element
$$
(b_1,b_3,b_2)\in \ker\bigl( \cM_{W,\mec d}(\partial) \bigr);
$$
then
$$
b_3=\sum_{W(a_1,a_2)=1}  \mec e_{(a_1,a_2)} c_{a_1,a_2}
$$
for some $c_{a_1,a_2}\in k$; the condition that $(b_1,b_3,b_2)$
lies in the kernel is equivalent to
$$
b_1 = \sum_{W(a_1,a_2)=1} \mec e_{a_1} c_{a_1,a_2}\in\cM_{W,\mec d}(B_1)=
k^{\oplus \integers_{\le d_1}},
\quad
b_2= \sum_{W(a_1,a_2)=1} \mec e_{a_2} c_{a_1,a_2}\in
\cM_{W,\mec d}(B_2)=
k^{\oplus \integers_{\le d_2}},
$$
which holds iff $a_1\le d_1$ and $a_2\le d_2$ whenever
$c_{a_1,a_2}\ne 0$.  Hence each such triple $(b_1,b_3,b_2)$ is
a unique linear combination of the vectors in
\eqref{eq_basis_for_H_zero_perfect_matching}.

To prove~(4), since $\cM_{W,\mec d}(\rho_{31}),\cM_{W,\mec d}(\rho_{32})$
are isomorphisms, it follows that 
$$
V = \bigl( \cM_{W,\mec d}(A_1) \oplus \cM_{W,\mec d}(A_2)  \bigr) / 
{\rm Image}\bigl(\cM_{W,\mec d}(B_3) \bigr)
$$
has $(\mec e_{a_1},0)$ as a basis, where $a_1$ ranges over all of $\integers$.
The image of $\cM_{W,\mec d}(B_1)$ in $V$ is precisely the span
of all $(\mec e_{a_1},0)$ with $a_1\le d_1$, and hence
$$
V' = V / {\rm Image}\bigl(\cM_{W,\mec d}(B_1) \bigr)
$$
has a basis consisting of all the $(\mec e_{a_1},0)$ with $a_1\ge d_1+1$;
finally the image of $\cM_{W,\mec d}(B_2)$ in $V'$ is precisely
the span of all $(0,\mec e_{a_2})$ with $a_2\le d_2$,
each of which equals $(-\mec e_{a_1},0)$ for the unique $a_1$
with $W(a_1,a_2)=1$.  Hence
$$
H^1(\cF) = V' / {\rm Image}\bigl(\cM_{W,\mec d}(B_2) \bigr)
$$
has a basis as claimed in 
\eqref{eq_basis_for_H_one_perfect_matching}.

This establishes~(1)--(4) of the theorem.
Next we prove the rest of the theorem.
Since 
$W$ is a perfect matching, by definition it is 
initially and eventually zero; hence the number of $\mec a\le\mec d$
with $W(\mec a)=1$ is zero for
$\deg(\mec d)$ sufficiently small, and for such $f(\mec d)=b^0(\mec d)=0$;
similarly, for $\deg(\mec d)$ sufficiently large the number of
$\mec a\ge \mec d+\mec 1$ with $W(\mec a)=1$ is zero,
and for such $\mec d$ we have $b^1(\mec d)=0$.
The remaining claims follow from
Theorem~\ref{th_cM_Euler_char_repeated_Riemann_functions}
and the fact that 
$\frakm f^\wedge_{\mec K}=(-1)^2 W^*_{\mec L}=W^*_{\mec L}$
(by Theorem~\ref{th_easy_dual_functions_theorem}).
% Theorem~\ref{th_easy_dual_functions_theorem}.
\end{proof}

% Material that was here on June 17,2022 was deleted since it now
% appears earlier.  See backups

\subsection{The Betti Numbers of $\cM_{W,\mec d}$ for General $W$ via
an Associated Graph}

One can give a description of the Betti numbers of $\cM_{W,\mec d}$
for any $W$ in terms of a graph associated to $W$ and $\mec d$.
This formula is foundational and seems interesting, but it is independent
of the rest of this article.

Given $W\from\integers^2\to\integers_{\ge 0}\cup\{\infty\}$ and
$\mec d\in\integers^2$ we associate the following graph, which may have
multiple edges and self-loops,
$G={\rm Graph}(W,\mec d)$, which one can describe in two ways:
first,
\begin{enumerate}
\item one forms the bipartite graph $G'$ whose vertex set is 
$\integers\times\{1,2\}$ and whose edge set has $W(s_1,s_2)$ edges joining
$(s_1,1)$ with $(s_2,2)$;
\item one then takes $G$ to be the graphs obtained from $G'$
by collapsing the vertices $\integers_{\le d_1}\times\{1\}$
and $\integers_{\le d_2}\times\{2\}$ into a single vertex, $v_0$.
\end{enumerate}
In particular, $G$ is not generally bipartite,
and has $r$ self-loops at $v_0$ (and no self-loops about any other vertex), 
where 
\begin{equation}\label{eq_self_loops_at_v_null}
r = \sum_{(s_1,s_2)\le\mec d} W(s_1,s_2) ;
\end{equation} 
each such self-loop adds $1$ to the first Betti number of $G$.

The second way to describe $G$ is more explicit: namely $G$ is the graph
with vertex set $V_G=v_0\amalg V_{\rm first}\amalg V_{\rm second}$\footnote{
  We assume some reasonable convention for the meaning of the {\em disjoint
  union} $\amalg$, which is a limit and hence only defined
  up to unique isomorphism; e.g., for sets $A_1,\ldots,A_s$, the set
  $A_1\amalg \ldots  \amalg A_s$ refers to the union 
  $\bigcup_i A_i\times\{i\}$.
  }
with $V_{\rm first}=\integers_{\ge d_1+1}$
and $V_{\rm second}=\integers_{\ge d_2+1}$,
and whose edge set can be identified with
$E_G={\rm Multi}(W)$ (as 
in~\eqref{eq_multiset_convention}), where each element
$(s_1,s_2,i)\in{\rm Multi}(W)$ creates:
\begin{enumerate}
\item
an edge joining $s_1\in V_{\rm first}$ and $s_2\in V_{\rm second}$
if $s_1\ge d_1+1$ and $s_2\ge d_2+1$;
\item
a self-loop about $v_0$ if $s_1\le d_1$ and $s_2\le d_2$;
\item
an edge joining $s_1\in V_{\rm first}$ and $v_0$ if
$s_1\ge d_1+1$ and $s_2\le d_2$; and
\item
an edge joining $v_0$ and $s_2\in V_{\rm second}$ if
$s_1\le d_1$ and $s_2\ge d_2+1$.
\end{enumerate}

It will be convenient to write $G$ as the union of 
its connected components $G_i=(V_i,E_i)$,
where $i$ ranges over $\{0,\ldots,\ell\}$ if $G$ has finitely
many connected components, or $i$ ranging over 
$\integers_{\ge 0}$ otherwise; we will also set $G_0$ to be
the connected component of $v_0$.

We need to recall some convenient definitions of the Betti numbers of
an infinite graph.  If $G=(V,E)$ is a graph 
(with $V,E$ not necessarily finite),
then one can define its incidence
matrix, $\iota_G$, as usual, by orienting each edge arbitrarily, 
so that $\iota_G$ is a map $k^{\oplus E}\to k^{\oplus V}$, and then
$b^0(G),b^1(G)$ are, respectively, the dimensions of the
cokernel and kernel of $\iota_G$.  
This will be useful to us.
However, it will also be conventient
to define the Betti numbers as follows:
$b^0(G)$ is the number of connected components, i.e., equivalence
classes of vertices where two vertices are equivalent if they are
connected by a walk of finite length.
For each connected component we choose a spanning tree (by fixing a vertex,
$r$
as the root of the tree,
then adding one edge joining $r$ to each of its neighbours, then
one edge for each vertex of distance $2$ to $r$, etc.).
This gives a spanning forest of $G$.
Then $b^1(G)$ is the cardinality of set of edges of $G$ that don't lie in
the spanning forest.

\begin{theorem}\label{th_cM_Betti_in_graph_terms}
Let $W\from\integers^2\to\integers_{\ge 0}\cup\{\infty\}$ and
$\mec d\in\integers^2$.  Let
$G={\rm Graph}(W,\mec d)$.  Then
\begin{align*}
b^1(\cM_{W,\mec d}) & = b^0( G) - 1  \\
b^0(\cM_{W,\mec d}) & = b^1( G) .
\end{align*}
\end{theorem}
We remark that it is interesting how each Betti number
of $\cM_{W,\mec d}$ can be inferred from the opposite Betti
number of $G$.
We also remark that if $W$ is a perfect matching, 
then Theorem~\ref{th_perfect_matching} is a simpler description
of the Betti numbers of $\cM_{W,\mec d}$.
Hence the seeming simplicity of the above theorem is not
necessarily the most practical way to describe the Betti numbers
of $\cM_{W,\mec d}$.
\begin{proof}
Let $\tau=\cM_{W,\mec d}(\partial)$, which one can view as a map
$$
\tau \from 
k^{\oplus\bigl(\integers_{\le d_1}
\amalg\;\integers_{\le d_2}\amalg\;{\rm Multi}(W)\bigr)}
\to
k^{\oplus(\integers\;\amalg\; \integers)}.
$$
In the notion introduced after the definition of $G$, we saw that
$G$ decomposes into its connected components
$G_i=(V_i,E_i)$ where $i$ ranges over $I$, where $I$
equals $\{0,\ldots,\ell\}$ or $\integers_{\ge 0}$, and $G_0$
is the component containing $v_0$.
Then ${\rm Multi}(W)$ is partitioned into multisets $E_i$ with $i\in I$;
setting
$$
E'_0 = \integers_{\le d_1} \amalg \integers_{\le d_2} \amalg E_0,
$$
we have that
$$
\integers_{\le d_1}
\amalg\;\integers_{\le d_2}\amalg\;{\rm Multi}(W)
=
E'_0 \cup \bigcup_{i\in I\setminus\{0\}} E_i
$$
where we identify $E_i \subset{\rm Multi}(W)$ with 
the subset ${\rm Multi}(W)$ as it lies in 
$\integers_{\le d_1} \amalg\;\integers_{\le d_2}\amalg\;{\rm Multi}(W)$.
Similarly setting
$$
V'_0 = 
\bigl( \integers_{\le d_1} \amalg \integers_{\le d_2} \bigr)
\cup (V_0\setminus\{v_0\}) 
\quad
\mbox{(which is a subset of $\integers\amalg\,\integers$)},
$$
we have
$$
\integers\amalg\,\integers =
V'_0 \cup \bigcup_{i\in I\setminus \{0\}} V_i .
$$
It follows that $\tau$ factors as a map
$$
\tau = \bigoplus_{i\in I} \tau_i
$$
where 
$$
\tau_0 \from k^{\oplus E_0'} \to k^{\oplus V_0'}
$$
and for $i\ne 0$,
$$
\tau_i\from k^{\oplus E_i}\to k^{\oplus V_i} .
$$
Note that for $i\ne 0$, $\tau_i$ sends 
$(a_1,a_2,j)\in E_i\subset{\rm Multi}(W)$ to $(\mec e_{a_1},\mec e_{a_2})$;
since $G_i$ is bipartite, 
$b^0(\tau_i),b^1(\tau)$ are the same as $b^0,b^1$ of the map
sending 
$(a_1,a_2,j)$ to $(\mec e_{a_1},-\mec e_{a_2})$, which is an
incidence matrix of $G_i$.
Hence for $j=0,1$,
$$
b^j(\tau) = \sum_{i\in I} b^j(\tau_i) 
= b^j(\tau_0) +  \sum_{i>0} b^{1-j}(G_i).
$$
Since 
$$
b^{1-j}(G) = \sum_{i\in I} b^{1-j}(G_i) ,
$$
it remains to show that $b^1(\tau_0)=0$ and 
$b^0(\tau_0)=b^1(G_0)$.

We claim that $\tau_0$ is surjective; hence for $a\in V_0'$
we need to show that the standard basis vector $\mec e_a$
is in the image of $\tau$; this is clear
for $a\in  \integers_{\le d_1} \amalg \integers_{\le d_2}$.
If the distance from $a$ to $v_0$ is $1$, and $a\in V_{\rm first}$,
then for some $a'\le d_2$ we have $(a,a')\in E_0$, and hence
$\mec e_{a}+\mec e_{a'}\in{\rm Image}(\tau)$, and hence
$\mec e_a\in{\rm Image}(\tau)$; similarly if $a\in V_{\rm second}$.  
We then similarly that
$\mec e_a\in {\rm Image}(\tau)$ if the distance from $a$ to $v_0$
equals $2$, since then
$\mec e_{a}+\mec e_{a'}\in{\rm Image}(\tau)$ for some $a'$ of distance
$1$ to $v_0$.  We then argue the general case by induction on
its distance to $v_0$.
Hence $b^1(\tau_0)=0$.

Finally let us show that $b^0(\tau_0)=b^1(G_0)$.
Let us give an isomorphism $\ker(\tau_0)\to \ker(\iota_{G_0})$
where $\iota_{G_0}$ is an incidence matrix of $G_0$.
To describe such an incidence matrix, we orient each 
$w\in E_0$ arising from $(a_1,a_2,j)\in{\rm Multi}(W)$ as
running from $a_1$ to $a_2$.
Next note that $\tau_0$ is a map
$$
\tau_0 \from 
k^{\oplus (\integers_{\le d_1}
\amalg\;\integers_{\le d_2})\cup E_0} 
\to
k^{\oplus (\integers_{\le d_1}
\amalg\;\integers_{\le d_2})\cup (V_0\setminus\{v_0\})},
$$
so each element of the domain of $\tau_0$ can be viewed as 
a pair $(\alpha,\beta)$ which refers to the element
$$
\sum_{u\in \integers_{\le d_1} \amalg\;\integers_{\le d_2}}
\alpha(u) \mec e_u 
\ +\   %
\sum_{w\in E_0} \beta(w)\mec e_w,
$$
where $\alpha,\beta$ are functions that are zero for all but 
finitely many of their values.
Since $\tau$ takes $(\alpha,0)$ to
$$
\sum_{u \in \integers_{\le d_1} \amalg\;\integers_{\le d_2}} 
\alpha(u)\mec e_u \in k^{\oplus(\integers\;\amalg\;\integers)},
$$
it follows that for each $\beta$ there exists at most one
$\alpha$ with $(\alpha,\beta)\in\ker(\tau_0)$, and such an $\alpha$
exists precisely when
\begin{equation}\label{eq_beta_condition}
\sum_{u=(a_1,a_2,\ell)\in E_0, a_1\ge d_1+1} \beta(u) \mec e_{a_1}  
=0=
\sum_{u=(a_1,a_2,\ell)\in E_0, a_2\ge d_2+1} \beta(u) \mec e_{a_2}  .
\end{equation} 
So we want to prove that the dimension of all $\beta$
satisfying \eqref{eq_beta_condition} equals $b^1(G_0)$.

First consider the case where $G_0=(V_0,E_0)$ is a tree:
we claim that \eqref{eq_beta_condition} forces $\beta=0$,
since if $U\subset E_0$ is the support of $\beta$
and $u'\in U$ is an edge of maximum distance
to $v_0$, one vertex incident upon $u'$ is not incident upon any other
edge of $U'$, which is a contradiction.

Next, consider the case that $G_0$ is a tree $(V_0,T)$ plus 
an edge $u_1$: there exists a $\beta$ as above with $\beta(u_1)$
nonzero, 
using the unique cycle created
by $u_1$; 
since $\beta$ can be increased by at most $1$ with the addition of an
edge, it follows that the dimension of $\beta$ satisfying
\eqref{eq_beta_condition} is exactly $1$.
Similarly, if we add another edge, $u_2$, to $G_0$,
the dimension of $\beta$ satisfying
\eqref{eq_beta_condition} increases by $1$,
using any cycle created by $u_2$; since it can increase by at most $1$,
the dimension of such $\beta$ is exactly $2$.
It similarly
follows by induction on $m$, that if $G_0$ is a tree plus $m$ edges,
then the dimenion of $\beta$ satisfying
\eqref{eq_beta_condition} is exactly $m$.
It follows by taking $m=b^1(G_0)$, or letting $m\to\infty$ if
$b^1(G_0)=\infty$, that 
$$
b^0(\tau) = b^1(G_0) .
$$
\end{proof}

\section{Morphisms, Isomorphisms, and Direct Sum of $k$-Diagrams}
\label{se_morphisms_etc}

For the rest of this article we will need to know when two
$k$-diagrams are {\em isomorphic} and some other basic properties of
(the category of) $k$-diagrams.  

\subsection{Morphisms of $k$-Diagrams}

\begin{definition}
\label{de_morphism_k_diagrams}
Let $\cF,\cG$ be two $k$-diagrams.
By a {\em morphism $\phi\from\cF\to\cG$}
we mean the data, $\phi$, of linear maps from each 
value of $\cF$ to the corresponding
value on $\cG$ in a way that commutes with the restriction maps:
i.e., $\phi$ consists of
$k$-linear maps $\phi(B_i)\from \cF(B_i)\to\cG(B_i)$ for
$i=1,2,3$ and $\phi(A_j)\from\cF(A_j)\to\cG(A_j)$ for $j=1,2$
such that $\cG(\rho_{ij})\phi(B_i)=\phi(A_j)\cF(\rho_{ij})$ whenever
$\cF(\rho_{ij}),\cG(\rho_{ij})$ exist (i.e., $i=j$ or $i=3$ and any $j$).
For $k$-diagrams $\cF,\cG$ we use $\Hom(\cF,\cG)$ to denote the
set of morphism $\cF\to\cG$; if $\phi,\phi'\in\Hom(\cF,\cG)$ and
$\alpha,\alpha'\in\field$, then one can define
$\alpha\phi+\alpha'\phi'\in \Hom(\cF,\cG)$ to be the map that is
the value-by-value linear combination,
i.e., for $P=A_j$ or $P=B_i$,
$$
\bigl( \alpha\phi+\alpha'\phi'\bigr) (P) 
=
\alpha\phi(P)+\alpha' \phi'(P);
$$
this gives $\Hom(\cF,\cG)$ the structure of a $k$-vector space.
\end{definition}
We illustrate a morphism of $k$-diagrams 
in Figure~\ref{fi_morphism_of_diagrams}.
\begin{figure}
\begin{tikzpicture}[scale=0.50,font=\small]
\node (B1) at (0,4) {$\cF(B_1)$};
\node (B2) at (0,-4) {$\cF(B_2)$};
\node (B3) at (0,0) {$\cF(B_3)$};
\node (A1) at (8,2) {$\cF(A_1)$};
\node (A2) at (8,-2) {$\cF(A_2)$};
\draw [->] (B1) -- (A1) node [midway,above] {$\cF(\rho_{1,1})$} ;
\draw [->] (B2) -- (A2) node [midway,below] {$\cF(\rho_{2,2})$} ;
\draw [->] (B3) -- (A1) node [midway,above] {$\cF(\rho_{3,1})$} ;
\draw [->] (B3) -- (A2) node [midway,below] {$\cF(\rho_{3,2})$} ;
\node (BB1) at (15,4) {$\cG(B_1)$};
\node (BB2) at (15,-4) {$\cG(B_2)$};
\node (BB3) at (15,0) {$\cG(B_3)$};
\node (AA1) at (23,2) {$\cG(A_1)$};
\node (AA2) at (23,-2) {$\cG(A_2)$};
\draw [->] (BB1) -- (AA1) node [midway,above] {$\cG(\rho_{1,1})$} ;
\draw [->] (BB2) -- (AA2) node [midway,below] {$\cG(\rho_{2,2})$} ;
\draw [->] (BB3) -- (AA1) node [midway,above] {$\cG(\rho_{3,1})$} ;
\draw [->] (BB3) -- (AA2) node [midway,below] {$\cG(\rho_{3,2})$} ;
\draw [->,ultra thick] (B1) -- (BB1) node [pos=0.75,above] {$\phi(B_1)$};
\draw [->,ultra thick] (B2) -- (BB2) node [pos=0.75,above] {$\phi(B_2)$};
\draw [->,ultra thick] (B3) -- (BB3) node [pos=0.75,above] {$\phi(B_3)$};
\draw [->,ultra thick] (A1) -- (AA1) node [pos=0.2,above] {$\phi(A_1)$};
\draw [->,ultra thick] (A2) -- (AA2) node [pos=0.2,above] {$\phi(A_2)$};
\node (F) at (4,-6) {\Huge$\cF$};
\node (G) at (19,-6) {\Huge$\cG$};
\draw [->,ultra thick] (F) -- (G) node [midway,above] {\Large$\phi$};
\end{tikzpicture}
\caption{A morphism of diagrams $\phi\from\cF\to\cG$, depicted in thick lines}
\label{fi_morphism_of_diagrams}
\end{figure}

\subsection{Constant and the Four Basic $k$-Diagrams}

Let us describe some simple $k$-diagrams that we will use.

\begin{definition}\label{de_constant_diagram_k}
Let $k$ be a field, and $V$ a $k$-vector space.  The
{\em constant $k$-diagram $V$}, denoted $\underline V$,
refers to the diagram whose values are all $V$, and whose
restriction maps are the identity map on $V$.
\end{definition}
In particular, $\underline k$ is the constant diagram whose
values are $k$, viewed as a $k$-vector space.

Let us describe a number of $k$-diagrams closely related to
the diagram $\underline k$ that we will use; we will collectively 
refer to these $k$-diagrams as the {\em four basic $k$-diagrams}.
Before giving the formal definition, we depict these diagrams
in Figure~\ref{fi_four_basic} by their values, all of which are
either $0$ or $k$, and all maps $k\to k$ are the identity maps.

\newcommand{\ficonsttikzpiece}[3]{
\begin{tikzpicture}[scale=0.40]
\node (B1) at (0,2) {$#1$};
\node (B2) at (0,-2) {$#2$};
\node (B3) at (0,0) {$k$};
\node (A1) at (3,1) {$k$};
\node (A2) at (3,-1) {$k$};
\draw [->] (B1) -- (A1) ;
\draw [->] (B2) -- (A2) ;
\draw [->] (B3) -- (A1) ;
\draw [->] (B3) -- (A2) ;
\node at (1.5,-4){$#3$};
\end{tikzpicture}
}

\begin{figure}
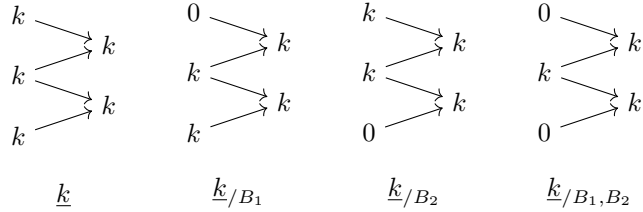

\hspace*{2cm}
\ficonsttikzpiece{k}{k}{\underline{k}}
\quad
\ficonsttikzpiece{0}{k}{\underline{k}_{/B_1}}
\quad
\ficonsttikzpiece{k}{0}{\underline{k}_{/B_2}}
\quad
\ficonsttikzpiece{0}{0}{\underline{k}_{/B_1,B_2}}
\caption{The Diagram $\underline k$ and Related Diagrams}
\label{fi_four_basic}
\end{figure}

\begin{definition}\label{de_four_basic}
Let $k$ be a field.  Consider the four possible diagrams, $\cF$,
such that:
\begin{enumerate}
\item
all values of $\cF$ equal $k$ or $0$;
\item
$\cF(B_3)=\cF(A_1)=\cF(A_2)=k$;
\item
all restriction maps $k\to k$ are the identity map.
\end{enumerate}
We use the notation $\underline{k}$ to refer to $\cF$ as above
with $\cF(B_1)=\cF(B_2)=k$ and call it the {\em constant diagram
$k$}; for $i=1,2$ we use the notation
$\underline{k}_{/B_i}$ to refer the same diagram except
with $\cF(B_i)=0$;
and we use the notation $\underline{k}_{/B_1,B_2}$ to refer to the
remaining such diagram, i.e., with $\cF(B_1)=\cF(B_2)=0$.
We refer to these four diagrams collectively as the
{\em four basic $k$-diagrams}.
\end{definition}

\subsection{Simple Examples of Morphisms with the Four Basic Diagrams}

The reader who has never worked with $k$-diagrams or the related
notions of presheaves and sheaves is encouraged to consider morphisms
between
the four basic diagrams.

\begin{example}\label{ex_morphisms_between_underline_k}
As $k$-vector spaces we have
$$
\Hom( \underline k_{/B_1},\underline k) \isom k,
$$
since for any $\alpha \in k$ there is a unique morphism
$\phi\from \underline k_{/B_1}\to \underline k$ 
such that for $P=A_1,A_2,B_3$, $\phi(P)\from k\to k$ is multiplication
by $\alpha$ (and for $P=B_1,B_2$, since
$\underline k_{/B_1,B_2}(P)=0$, $\phi(P)$ is the trivial
$k$-linear transformation $\{0\}\to k$).
By contrast
$$
\Hom(\underline k,\underline k_{/B_1}) \isom \{0\},
$$
since if
$\phi\from\underline k\to\underline k_{/B_1}$, then 
$$
\phi(B_1)\from \underline k(B_1) \to  \underline k_{/B_1}(B_1)  = 0
$$
must be the zero map, but then
(recall the meaning of $\rho_{i,j}$ from 
Definition~\ref{de_diagram_k_vs})
$$
\underline k(B_1) 
\xrightarrow{\underline k_{/B_1}(\rho_{1,1})\circ \phi(B_1)}
\phi(A_1)
$$
must be the zero map, and since this must equal the map
$$
\underline k(B_1) 
\xrightarrow{\phi(A_1)\circ\underline k(\rho_{1,1})}
\phi(A_1)
$$
this forces $\phi(A_1)$ to be multiplication by $0$.  
Following Figure~\ref{fi_morphism_of_diagrams}
around (with $\cF=\underline k$ and $\cG=\underline k_{/B_1}$)
we see that $\phi$ must be everywhere zero.
\end{example}

[The experts will realize that the above example reflects the fact that
there is a 
canonical inclusion 
$\cF_U\to\cF$ (and the fact there is typically no nonzero morphism
$\cF\to\cF_U$)
where $\cF$ is a sheaf on a topological
space, $U$ an open subset, and $\cF_U$ is the extension by zero
of the restriction of $\cF$ to $U$; we will explain this in
more detail in Subsection~\ref{su_k_diagrams_sheaf_theory}.]

For future use it will be helpful to note the following calculations.

\begin{example}\label{ex_morphisms_of_basic_to_dualizing}
Similar to Example~\ref{ex_morphisms_between_underline_k}, we see that
$$
\Hom(\underline k_{/B_1,B_2},\underline k_{/B_1,B_2}) \isom k,
$$
and $\Hom(\cF,\underline k_{/B_1,B_2})=\{0\}$ for
$\cF=\underline k,\underline k_{/B_1},\underline k_{/B_2}$.
\end{example} 

\begin{example}\label{ex_morphisms_all_basic_four}
More generally, there is a partial order of our four basic
$k$-diagrams: denoting each of these diagrams by $\underline k_{/S}$
where $S$ is some subset of $\{B_1,B_2\}$, where we understand
$\underline k_{/\emptyset}=\underline k$. We have
$$
\Hom(\underline k_{/S},\underline k_{/S'}) 
\isom
\left\{  % }
\begin{array}{ll} k & \mbox{if $S\subset S'$, and} \\
\{0\} & \mbox{otherwise.}
\end{array}
\right.
$$
\end{example}

\subsection{Example: Global Sections}

If $\phi\from\cF\to\cG$, then we easily see that $\phi$ gives maps
$\cF(B)\to \cG(B)$ and $\cF(A)\to\cG(A)$ that induce maps
$H^i(\cF)\to H^i(\cG)$ for $i=0,1$.

If $\cF$ is any $k$-diagram, then if
$$
\phi\in \Hom(\underline k,\cF),
$$
then $\phi(A_1)$ takes the element $1\in k$ to an element $a_1\in A_1$,
and similarly for $\phi(A_2)$ and the $\phi(B_i)$; this 
gives a tuple $(b_1,b_2,b_3,a_1,a_2)$, and the fact that the
restrictions of $\underline k$ are the identity maps implies that
$(b_1,b_2,b_3,a_1,a_2)$ is a global section 
(Definition~\ref{de_diagram_k_vs});
conversely every global section $(b_1,b_2,b_3,a_1,a_2)$
determines a $\phi$
where $\phi(A_i)1=a_i$ and $\phi(B_j)1=b_j$ which we easily check
is an element of $\Hom(\underline k,\cF)$.
Hence we see (as usual in sheaf theory)
$$
H^0(\cF) \isom \Hom(\underline k,\cF) ,
$$
and we easily check that for any morphism $\phi\from\cF\to\cG$ 
the map $H^0(\cF)\to H^0(\cG)$ is the same as the map
$$
\Hom(\underline k,\cF) \to 
\Hom(\underline k,\cG)
$$
given by composition of the morphism $\underline k\to\cF$ with  $\phi$;
this is a standard fact about global sections of $\cF$
in sheaf theory.

\subsection{Isomorphisms and Direct Sums of $k$-Diagrams}

We now give the notion of {\em isomorphisms} and {\em direct sums}
for $k$-diagrams; later, 
in Subsection~\ref{su_intro_homological_algebra} we will
explain that these notions really result once one specifies what is
meant by a $k$-diagram and a morphism of $k$-diagrams.

\begin{definition}
A morphism $\phi\from\cF\to\cG$ of $k$-diagrams is
an {\em isomorphism} if all the $\phi(A_j)$ and $\phi(B_i)$
are isomorphisms.
\end{definition}
This is equivalent to saying that there exists an inverse morphism
$\nu\from\cG\to\cF$ such that $\phi\nu$ is the identity on $\cG$
(i.e., all the $\phi\nu(A_j)$ and $\phi\nu(B_i)$ are the identity maps)
and $\nu\phi$ is the identity morphism on $\cF$.

\begin{definition}
Let $L$ be a set and for each $\ell\in L$, say that we are given a
$k$-diagram, $\cF_\ell$.  The {\em direct sum of $\{\cF_\ell\}_{\ell\in L}$}
refers to the $k$-diagram denoted
$$
\bigoplus_{\ell\in L} \cF_\ell
$$
whose values at the $A_j$ (for $j=1,2$) are
$$
\bigoplus_{\ell\in L} \cF_\ell(A_j),
$$
and similarly for the values at the $B_i$, and similarly the
restriction maps
are the direct sums of those of the $\cF_\ell$.
Similarly, if for each $\ell\in L$ we are given a morphism
$\phi_\ell$ of $k$-diagrams, the morphism $\oplus_{\ell\in L}\phi_\ell$
is the morphism of $k$-diagrams $\oplus_{\ell\in L}\cF_\ell$
to $\oplus_{\ell\in L}\cG_\ell$ given by direct sum of the $\phi_\ell$.
\end{definition}

We easily check that the all constructions in 
Definition~\ref{de_diagram_k_vs}
commute with taking direct sums.
In particular, for any direct sum $\{\cF_\ell\}_{\ell\in L}$ we have
$$
\left( \bigoplus_{\ell\in L} \cF_\ell \right)(\partial)
=
\bigoplus_{\ell\in L} \cF_\ell (\partial),
$$
and
$j=0,1$ we have
$$
H^j \left( \bigoplus_{\ell\in L} \cF_\ell \right) =
\bigoplus_{\ell\in L} H^j(\cF_\ell)
$$
and taking dimensions we have
$$
b^j \left( \bigoplus_{\ell\in L} \cF_\ell \right) =
\sum_{\ell\in L} b^j(\cF_\ell) .
$$

\section{Sums of $\cM_{W,\mec d}$ and Indicator $k$-Diagrams}
\label{se_indicator}

The main point of this section is to prove the following theorem.

\begin{theorem}\label{th_invariance_under_perfect_matching_sum}
Let $W_1,\ldots,W_s$ and $\tilde W_1,\ldots,\tilde W_s$
be perfect matchings $\integers^2\to\integers$ such that
$$
W_1+\cdots+W_s = 
\tilde W_1+\cdots+\tilde W_s .
$$
Then for any $\mec d$ we have
$$
\cM_{W_1,\mec d} \oplus \cdots \oplus \cM_{W_s,\mec d}
\isom
\cM_{\tilde W_1,\mec d} \oplus \cdots \oplus \cM_{\tilde W_s,\mec d}.
$$
\end{theorem}

Moreover, we will prove this theorem is true in a very strong sense:
namely, if $W=W_1+\cdots+W_s$, then the direct sum
$$
\cM_{W_1,\mec d} \oplus \cdots \oplus \cM_{W_s,\mec d}
$$
is isomorphic to a sum,
$\bec\cI_{\mec d}^{\oplus W}$
of what we call {\em indicator diagrams}, that can be
inferred from $W$ alone, without reference to the $W_1,\ldots,W_s$.
This will provide additional intuition regarding 
Theorem~\ref{th_invariance_under_perfect_matching_sum}.

\subsection{Example: $\cM_{W,\mec d}$ as a Direct Sum of Indicator Diagrams}

The $k$-diagrams $\cM_{W,\mec d}$ of the 
last section can be naturally viewed as a direct sum of our four basic
diagrams.  In fact, for $W$ fixed, the family $\cM_{W,\mec d}$
with $\mec d$ varying decomposes as a sum of a family of 
our four basic diagrams indexed on $\mec d$.
This point of view will be useful to understand the virtual $k$-diagrams
that we study later.

\begin{definition}\label{de_indicator_diagrams}
Let $k$ be a field and $\mec a\in\integers^2$.  The
{\em $\ge\mec a$-indicator family of $k$-diagrams}
refers to the family of $k$-diagrams indexed on $\mec d\in\integers^2$, 
denoted
$\{\indicateDgeA\}_{\mec d\in\integers^2}$,
where for each $\mec d\in\integers^2$ we set
\begin{enumerate}
\item
$\indicateDgeA=\underline k$ if 
% $\mec a\le \mec d$,
$\mec d\ge \mec a$,
% \item
% $\indicateDgeA=\underline k_{/B_1,B_2}$ if 
% % $\mec a\ge \mec d+\mec 1$,
% $\mec d+\mec 1\le \mec a$,
\item
$\indicateDgeA=\underline k_{/B_2}$ if 
% $a_1\le d_1$ and $a_2\ge d_2+1$, and
$d_1\ge a_1$ and $d_2< a_2$, 
\item
$\indicateDgeA=\underline k_{/B_1}$ if 
% $a_2\le d_2$ and $a_1\ge d_1+1$.
$d_2\ge a_2$ and $d_1 < a_1$, and 
\item
$\indicateDgeA=\underline k_{/B_1,B_2}$ if 
% $\mec a\ge \mec d+\mec 1$,
% $\mec d+\mec 1\le \mec a$,
$d_1 < a_1$ and $d_2< a_2$.
\end{enumerate}
Equivalently, for each $\mec d$, $\indicateDgeA$
is equal to one of the four basic $k$-diagrams,
where for each $j=1,2$, 
$\indicateDgeA(B_j)=k$ iff $a_j\le d_j$.
\end{definition}

We depict the indicator diagram in Figure~\ref{fi_indicator}.
\begin{figure}
\begin{tikzpicture}[scale=0.75]
% \node (B1) at (0,2) {$\cI(B_1)=k^{\II(a_1\le d_1)}$};
\node (B1) at (0,2) {$\indicateDgeA(B_1)=\left\{\begin{array}{cc}
k & \mbox{if $a_1\le d_1$,} \\ 0 & \mbox{otherwise.}
\end{array}\right.$};
\node (B2) at (0,-2) {$\indicateDgeA(B_2)=\left\{\begin{array}{cc}
k & \mbox{if $a_2\le d_2$,} \\ 0 & \mbox{otherwise.}
\end{array}\right.$};
\node (B3) at (0,0) {$k$};
\node (A1) at (8,1) {$k$};
\node (A2) at (8,-1) {$k$};
\draw [->] (B1) -- (A1) node [midway,above] {${\rm Inclusion}$};
\draw [->] (B2) -- (A2) node [midway,below] {${\rm Inclusion}$};
\draw [->] (B3) -- (A1) node [midway,above] {${\rm Identity}$};
\draw [->] (B3) -- (A2) node [midway,below] {${\rm Identity}$};
\end{tikzpicture}
\caption{The Indicator $k$-Diagram $\indicateDgeA$}
% $\DiracDeltaKAD$}
\label{fi_indicator}
\end{figure}

\begin{definition}
If $W\from\integers^2\to\integers_{\ge 0}\cup\{\infty\}$, the 
{\em $W$-sum indicator $k$-diagram},
denoted $\bec\cI_{\mec d}^{\oplus W}$,
refers to the direct sum
$$
\bec\cI_{\mec d}^{\oplus W} =
\bigoplus_{\mec a\in\integers^2} 
\bigl( \indicateDgeA  \bigr)^{\oplus W(\mec a)},
$$
where if $W(\mec a)=\infty$, then $\oplus W(\mec a)$ refers to
$\oplus\naturals$, i.e., the summand involved
is the direct sum of a countably infinite number of
copies of $\indicateDgeA$.
\end{definition}

The following proposition is immediate, but worth stating.

\begin{proposition}\label{pr_cM_W_decomp}
Let $W\from\integers^2\to\{0,1\}$ be a perfect matching.
Then there is an isomorphism
$$
\iota_{\mec d}\from\cM_{W,\mec d}\to \bec\cI_{\mec d}^{\oplus W}
$$
given by the canonical isomorphisms
for $i=1,2$ 
\begin{equation}\label{eq_cM_W_to_cI_A_B_i}
\cM_{W,\mec d}(A_i)= k^{\oplus\integers} \to \bec\cI_{\mec d}^{\oplus W}(A_i),
\quad
\cM_{W,\mec d}(B_i)= k^{\oplus\integers_{\le d_i}}\to\bec\cI_{\mec d}^{\oplus W}(B_i),
\end{equation} 
and the isomorphisms
\begin{equation}\label{eq_cM_W_to_cI_B_3}
\cM_{W,\mec d}(B_3)= k^{\oplus W}
\to \bec\cI_{\mec d}^{\oplus W}(B_3);
\end{equation} 
moreover,
and for $j=1,2$, the maps $\cM_{W,\mec d}(\rho_{3j})$ are isomorphisms
of $k^{\oplus W}\to k^{\integers}$.
\end{proposition} 
\begin{proof}
The equalities 
in~\eqref{eq_cM_W_to_cI_A_B_i} 
and~\eqref{eq_cM_W_to_cI_B_3} are by definition
(Definition~\ref{de_cM_W_d}).
The fact that $W$ is a perfect matching implies that for each $a_1$
there is a unique $a_2$ with $W(a_1,a_2)=1$;
this gives the
isomorphism $k^{\oplus\integers} \to \bec\cI_{\mec d}^{\oplus W}(A_1)$
and $k^{\oplus\integers_{\le d_1}}\to\bec\cI_{\mec d}^{\oplus W}(B_1)$;
similarly for the subscript $1$ replaced everywhere by $2$.
By Definition~\ref{de_indicator_diagrams} each indicator diagram
$\indicateDgeA$ has $\indicateDgeA(\rho_{3j})$ being an isomorphism,
and hence the same is true for $\bec\cI_{\mec d}^{\oplus W}$, and
hence, 
by~\eqref{eq_cM_W_to_cI_A_B_i} 
and~\eqref{eq_cM_W_to_cI_B_3},
it also holds for $\cM_{W,\mec d}$.
We easily check that the isomorphisms
in~\eqref{eq_cM_W_to_cI_A_B_i}
and~\eqref{eq_cM_W_to_cI_B_3} intertwine with the restriction maps,
and hence gives the desired isomorphism $\iota_{\mec d}$.
\end{proof}

\begin{proposition}
\label{pr_indicator_diagram_Betti_nums}
For any $W\from\integers\times\integers\to\integers_{\ge 0}\cup\{\infty\}$,
\begin{align*}
b^0(\cI_{\mec d}^{\oplus W})  & = \sum_{\mec a\le \mec d} W(\mec a), \\
b^1(\cI_{\mec d}^{\oplus W})  & = \sum_{\mec a\ge \mec d+\mec 1} W(\mec a),
\end{align*}
and hence, if one of these two Betti numbers is finite, we have
$$
\chi(\cI_{\mec d}^{\oplus W}) = 
\left( \sum_{\mec a\le \mec d} W(\mec a) \right) 
-
\left( \sum_{\mec a\ge \mec d+\mec 1} W(\mec a) \right) .
$$
\end{proposition}

Our main interest in Proposition~\ref{pr_indicator_diagram_Betti_nums}
is for $W$ that are $s$-fold perfect matchings in the following sense.

\begin{definition}
We say that a function $W\from\integers^2\to\integers_{\ge 0}$
is an {$s$-fold matching} if it can be written as the
sum of $s$ (bounded) perfect matchings.
\end{definition}

To build models for general Riemann functions $\integers^2\to\integers$
we will require the following strengthening 
of Proposition~\ref{pr_indicator_diagram_Betti_nums}.

\begin{proposition}\label{pr_s_fold_indicator_function}
Let $\{W_i\}_{i\in I}$ be a finite or countably infinite set
of functions $W_i\from\integers^2\to\integers_{\ge 0}\cup\{\infty\}$, and
let $W=\sum_{i\in I}W_i$.  Then
\begin{equation}\label{eq_s_fold_indicator_one}
\bec\cI_{\mec d}^{\oplus W} \isom 
\bigoplus_{i\in I} \bec\cI_{\mec d}^{\oplus W_i}.
\end{equation} 
Let $W\from\integers^2\to\integers_{\ge 0}$ be an $s$-fold matching,
and $W=W_1+\cdots+W_s$ be a decomposition of $W$ into perfect matchings.
Then for any $\mec d\in\integers^2$ there is an isomorphism
\begin{equation}\label{eq_s_fold_indicator_two}
\bec\cI_{\mec d}^{\oplus W} \isom 
\cM_{W_1,\mec d} \oplus \cdots \oplus \cM_{W_s,\mec d}.
\end{equation} 
In particular, if $W=W_1'+\cdots+ W'_s$ is another decomposition of $W$
into perfect matchings, then we have 
\begin{equation}\label{eq_s_fold_indicator_three}
\cM_{W_1,\mec d} \oplus \cdots \oplus \cM_{W_s,\mec d}
\isom
\cM_{W_1',\mec d} \oplus \cdots \oplus \cM_{W_s',\mec d}
\end{equation} 
\end{proposition}
\begin{proof}
\eqref{eq_s_fold_indicator_one} follows from the easily verified 
fact that any direct sum
of direct sums is the direct sum of all summands in the double summation
(this is, more generally, valid in any additive category, since this is
an inductive limit of inductive limits, see e.g.,
\cite{sga4.1}, Section~I.2.5.0).
\eqref{eq_s_fold_indicator_two} follows from \eqref{eq_s_fold_indicator_one}
and Proposition~\ref{pr_indicator_diagram_Betti_nums}.
\eqref{eq_s_fold_indicator_three} follows from
\eqref{eq_s_fold_indicator_two}.
\end{proof}

\section{Virtual Fredholm $k$-Diagrams and Riemann Functions for $n=2$}
\label{se_virtual_two_dim_Riemann}

If $f\from\integers^2\to\integers$ is a general Riemann function,
its weight, $W$, can attain negative values.
In this case we don't know of a good way to model $f$ as the zeroth
Betti number of a family of $k$-diagrams; however, one can do so
if we work with {\em virtual $k$-diagrams}.
Our strategy is to use Lemma~\ref{le_row_col_sums_weight_as_perfect}
to write
\begin{equation}\label{eq_W_as_W_i_minus_tilde_W_i}
W = (W_1+\cdots+W_s) - (\tilde{W_1}+\cdots+\tilde{W}_{s-1})
\end{equation} 
for some $s$, where the $W_i$ and $\tilde W_i$ are perfect matchings.
We then model $f$ as $b^0$ of the {\em virtual $k$-diagram}
\begin{equation}\label{eq_virtual_difference_W_i_tilde_W_i}
\left( 
\bigoplus_{i } \cM_{W_i,\mec d} ,
\bigoplus_{i } \cM_{\tilde{W}_i,\mec d} 
\right),
\end{equation} 
which is a ``formal difference'' of the first $k$-diagram ``minus''
the second.

Most of the work in this section is to iron out the notion
of virtual vector spaces, virtual Fredholm maps, and
virtual $k$-diagrams.
[We borrow the term {\em virtual} from
{\em virtual characters} in group theory.]
We then show that the ``virtual'' or ``formal difference of'' 
$k$-diagrams
\eqref{eq_virtual_difference_W_i_tilde_W_i} has the desired
Betti numbers (this is immediate from our discussion
of formal differences), and---more notably---is
independent, up to equivalence, of the decomposition
\eqref{eq_W_as_W_i_minus_tilde_W_i}.

We caution the reader that by our conventions below,
a {\em virtual $k$-vector space} will refer to 
a formal difference of finite dimensional vector spaces
(Convention~\ref{co_virtual_k_vector_space}) and a
{\em virtual $k$-diagram}
will refer to a formal difference of
Fredholm $k$-diagrams, i.e., of $k$-diagrams whose Betti numbers
are finite (Convention~\ref{co_virtual_k_diagram}).
These conventions are needed to get a well defined notion of
Betti numbers, as we explain below.

Our main modeling result is
Proposition~\ref{pr_cM_W_poss_neg_single_virt_class}, which builds
a virtual (Fredholm) $k$-digram that expresses any generalized
Riemann-Roch formula
\eqref{eq_generalized_riemann_roch} 
as an Euler characteristic formula.

We remark that any function $W\from\integers^2\to\integers$ can be
canonically expressed as
$$
W=W^+-W^-,
\quad\mbox{where}\quad
W^+=\max(W,0),\ %
W^-=\max(-W,0).
$$
Although we can see (for example, 
Example~\ref{ex_period_three_non_canonical_example} below) that
$\cM_{W^+,\mec d},\cM_{W^-,\mec d}$ are not generally Fredholm,
even if $\fraks W$ is a Riemann function,
by contrast
$\bec\cI_{\mec d}^{\oplus W^+},
\bec\cI_{\mec d}^{\oplus W^-}$
are always Fredholm, and this gives a sort of ``canonical'' or
``minimal'' way to express
$\cM_{W,\mec d}$ as a formal difference
of (Fredholm) indicator $k$-diagrams
(see Proposition~\ref{pr_cM_W_poss_neg_indicator_canonical}).

\subsection{Formal Definition of Virtual $k$-Diagrams
(and Virtual Vector Spaces, Etc.)}

We now introduce a group of formal differences---either of $k$-vector spaces,
$k$-Fredholm maps, and $k$-diagrams---that
one sees in, say, $K$-theory, or constructing the integers from
the natural numbers.  This general idea is
often technically
called the {\em Grothendieck completion} or the {\em Grothendieck
group} (of a commutative monoid).

\begin{definition}
Let $k$ be a field.
By a {\em virtual $k$-diagram} (respectively, {\em virtual $k$-Fredholm 
map}, {\em virtual $k$-vector space}, etc.)
we mean a
pair $(\cF_1,\cF_2)$ of $k$-diagrams (respectively, $k$-Fredholm map,
$k$-vector space, etc.);
we write 
$(\cF_1,\cF_2)\sim(\cF_1',\cF_2')$
if there is an isomorphism $\phi$
$$
\phi\from 
\cF_1\oplus \cF_2' \oplus\cF_0 
\to
\cF_1'\oplus \cF_2\oplus \cF_0
$$
for some $k$-diagram $\cF_0$ (respectively, $k$-Fredholm map, etc.).
At times we use the notation $\cF_1\ominus\cF_2$ to denote $(\cF_1,\cF_2)$.
We also view a $k$-diagram ($k$-Fredholm map, etc.) $\cF$ as 
the virtual diagram $\cF\ominus \underline 0$, 
where $\underline 0$ is the zero
$k$-diagram (and similarly for Fredholm $k$-map, etc.).
\end{definition}
We easily see that $\sim$ is an equivalence relation, and that
$\oplus$ and $\ominus$ can be extended to 
act on virtual $k$-diagrams (respectively, Fredholm $k$-diagram, etc.)
taken to behave like $+$ and $-$ regarding
parenthesis, e.g., 
$$
(\cF_1\ominus\cF_2)\ominus (\cF_3 \ominus \cF_4) 
\quad\mbox{refers to}\quad
(\cF_1\oplus \cF_4)\ominus (\cF_2 \oplus \cF_3).
$$

\subsection{Virtual Vector Spaces}

It is important to understand the difference between virtual $k$-vector
spaces and virtual $k$-vector spaces of finite dimension.

As virtual $k$-vector spaces,
$k\sim 0$ since there is an isomorphism 
$k\oplus k^{\naturals}\to k^{\naturals}$; however, we will
easily prove that for virtual
finite-dimensional $k$-vector spaces, we have $k^b\ominus k^a \sim
k^{b'}\ominus k^{a'}$ iff $b-a=b'-a'$.  To prove this, one notices that:
\begin{enumerate}
\item 
if $V,V'$ are isomorphic finite dimensional $k$-vector spaces, then
(clearly) $\dim(V)=\dim(V')$; and
\item 
therefore if $V_1\ominus V_2$ is equivalent to $V_3\ominus V_4$
as finite dimensional $k$-vector spaces, then for some finite
dimensional $k$-vector space $V_0$ we have that
$$
V_1\oplus V_4\oplus V_0 \isom
V_2\oplus V_3\oplus V_0 ,
$$
and taking dimensions we have an equality of finite integers
$$
\dim(V_1)+\dim(V_4)+\dim(V_0) 
=
\dim(V_2)+\dim(V_3)+\dim(V_0) ,
$$
and hence
$$
\dim(V_1)-\dim(V_2)=\dim(V_3)-\dim(V_4).
$$
\item Hence we can define
$$
\dim( V_1 \ominus V_2) = \dim(V_1) - \dim(V_2) \in \integers,
$$
which is well defined in the equivalence class of $V_1\ominus V_2$.
\end{enumerate}

\begin{convention}\label{co_virtual_k_vector_space}
By a {\em virtual $k$-vector space} we always mean a
virtual $k$-vector space of finite dimension, unless we specify otherwise.
\end{convention}

\subsection{Virtual $k$-Fredholm Maps}

Let $f\from B\to A$ and $f'\from B'\to A'$ be morphisms 
(i.e., linear transformations) of
$k$-vector spaces.
By a {\em morphism from $f$ to $f'$} we mean a pair
$\phi=(\phi_B,\phi_A)$ of morphisms
$\phi_B\from B\to B'$ and $\phi_A\from A\to A'$ that commute
in the evident fashion, i.e., $\phi_A f = g \phi_B$; see
Figure~\ref{fi_morphism_of_linear_transformations}.
\begin{figure}
\begin{tikzpicture}[scale=0.75]
\node at (0,2)(1){$B$};
\node at (2,2)(2){$A$};
\draw[->] (1) -- (2) node [midway,above] {$f$};
\node at (0,0)(3){$B'$};
\node at (2,0)(4){$A'$};
\draw[->] (3) -- (4) node [midway,above] {$f'$};
\draw[->] (1) -- (3) node [midway,left] {$\phi_B$};
\draw[->] (2) -- (4) node [midway,right] {$\phi_A$};
\end{tikzpicture}
\caption{A Morphism $f\to f'$.}
\label{fi_morphism_of_linear_transformations}
\end{figure}
A morphism $\phi=(\phi_B,\phi_A)\from f\to f'$ is an
{\em isomorphism} if $\phi_B$ and $\phi_A$ are isomorphisms\footnote{
  We remark that elsewhere one considers morphisms up to 
  homotopy or localizes at quasi-isomorphisms; in this article
  we have no need for this.
  }.
We easily verify the following facts.
\begin{enumerate}
\item 
If 
$\phi=(\phi_B,\phi_A)\from f\to f'$ is an isomorphism, then
$\phi_B$ restricts to an isomorphism $\ker(f)\to\ker(f')$,
and similarly $\phi_A$ restricts to an isomorphism
$\coker(f)\to\coker(f')$.
\item 
It follows that if we define for a virtual Fredholm map
$f=f_1\ominus f_2$ its cohomology groups for $i=0,1$ to
be the virtual (finite dimensional) $k$-vector spaces
\begin{equation}\label{eq_cohomology_groups_virtual_Fredholm_map}
H^i\bigl( f_1\ominus f_2 \bigr) = H^i( f_1) \ominus H^i( f_2 ),
\end{equation} 
then if $f_1\ominus f_2$ is equivalent to $f_3\ominus f_4$
as virtual $k$-Fredholm maps, we have
$$
H^i(f_1)\ominus H^i(f_2)
\sim
H^i(f_3)\ominus H^i(f_4)
$$
as virtual finite dimensional $k$-vector spaces.
\item
Hence \eqref{eq_cohomology_groups_virtual_Fredholm_map}
for $i=0,1$ are well-defined virtual finite dimensional $k$-vector
spaces, and hence setting
$$
b^i\bigl( f_1\ominus f_2 \bigr) = 
\dim H^i\bigl( f_1\ominus f_2 \bigr) 
=
\dim H^i(f_1)-\dim H^i(f_2) \in \integers
$$
gives well-defined Betti numbers of a virtual $k$-Fredholm map.
\end{enumerate}
Note the if we work with virtual $k$-linear transformations,
without insisting that they are Fredholm maps, then there seems to
be no good way to define their
Betti numbers (and hence cohomology groups), since
dimensions are not well defined for virtual $k$-vector spaces
when we allow the spaces to be of infinite dimension.

\subsection{Virtual $k$-Diagrams with Finite Betti Numbers}

If $\cF$ is a $k$-diagram, then
$\cF(\partial)$ is a $k$-Fredholm map iff
$b^i(\cF)$ is finite for both $i=0,1$.
Hence in the category of ``$k$-diagrams with finite Betti
numbers,'' the notion of a (virtual) cohomology group and
(virtual) Betti numbers are well defined, by associating to
$\cF_1\ominus \cF_2$ the virtual $k$-Fredholm map
$$
(\cF_1\ominus \cF_2)(\partial)
\eqdef \cF_1(\partial) \ominus \cF_2(\partial).
$$

\begin{definition}
Let $k$ be a field.  By a {\em Fredholm $k$-diagram}
we mean a $k$-diagram with both Betti numbers finite.
\end{definition}
It follows that virtual Fredholm $k$-diagrams have 
well-defined virtual cohomology groups and virtual Betti numbers.

\begin{convention}\label{co_virtual_k_diagram}
By a {\em virtual $k$-diagram} we mean a virtual Fredholm $k$-diagram
unless we specify otherwise,
i.e., we are working with $k$-diagrams with both Betti numbers finite.
Hence a virtual $k$-diagram has well-defined cohomology groups 
(which are virtual $k$-vector spaces of finite dimension) and
therefore well-defined Betti numbers.
\end{convention}

\subsection{Riemann Functions as Virtual Direct Sums}

Our convention is that a virtual $k$-diagram refers to $k$-diagrams
that are Fredholm; this is necessary to get well-defined Betti numbers.
Hence we need the following easy lemma.

\begin{lemma}\label{le_cM_W_d_are_Fredholm}
Let $W\from\integers^2\to\integers$ be any perfect matching.
Then for any $\mec d\in\integers^2$, 
$\cM_{W,\mec d}$ has finite Betti numbers.
Similarly, for any $\mec d\in\integers^2$,
the number of $\mec a$ with $W(\mec a)=1$
and $\indicateDgeA$ having at least one non-zero Betti number is
finite.
\end{lemma}
\begin{proof}
Let $W$ be a perfect matching.
Theorem~\ref{th_perfect_matching} shows that $\cM_{W,\mec d}$ has finite
Betti numbers.
The claim about $\indicateDgeA$ follows since
$$
\cM_{W,\mec d} = \bec\cI_{\mec d}^{\oplus W}
= \bigoplus_{W(\mec a)=1} \indicateDgeA,
$$
and
hence for $i=0,1$ we have
$$
b^i(\cM_{W,\mec d})
=
\sum_{W(\mec a)=1} b^i(\indicateDgeA).
$$
\end{proof}
One can alternatively prove the claim about $\indicateDgeA$ above by
noting that
the only non-zero Betti numbers of our four basic
diagrams are $b^0$ of $\underline k$ and $b^1$
of $\underline k_{/B_1,B_2}$; furthermore 
we have $\indicateDgeA=\underline k$ iff $\mec d\le \mec a$,
and $\indicateDgeA=\underline k_{/B_1,B_2}$ iff $\mec d+\mec 1\ge \mec a$,
and each of these conditions on $\mec a$
occurs for only finitely many $\mec a$ for which $W(\mec a)=1$,
since $W$ is supported 
in degree bounded from above and below.

\begin{definition}
Let $f\from\integers^2\to\integers$ be a Riemann function,
and $W=\frakm f$.  
For any way of writing 
\begin{equation}\label{eq_W_as_difference_definition}
W = (W_1+ \cdots + W_s) - (\tilde{W}_1 + \cdots + \tilde{W}_{s-1}) .
\end{equation}
we use $\cM_{W,\mec d}$ (with \eqref{eq_W_as_difference_definition}
understood) to denote
to the 
virtual (Fredholm) $k$-diagram
$$
\bigl(\cM_{W_1,\mec d} \oplus\cdots \oplus\cM_{W_{s,\mec d}},
\cM_{\tilde{W}_1,\mec d}\oplus\cdots\oplus\cM_{\tilde{W}_{s-1},\mec d} \bigr).
$$
\end{definition}

The following proposition immediately implies the main result of 
this section.
\begin{proposition}\label{pr_cM_W_poss_neg_single_virt_class}
Let $f\from\integers^2\to\integers$ be a Riemann function, and $W$ its
weight.
For any way of writing $W$ as
\eqref{eq_W_as_difference_definition},
the equivalence class $[\cM_{W,\mec d}]$, of the virtual $k$-diagram
$\cM_{W,\mec d}$
is independent of the way we write $W$ in
\eqref{eq_W_as_difference_definition}.
Furthermore, 
for any $\mec d\in\integers^2$,
\begin{enumerate}
\item
$
f(\mec d) = b^0([\cM_{W,\mec d}]),
$
\item
for any $\mec K\in\integers^2$, 
$
f^\wedge_{\mec K}(\mec K-\mec d) = b^1([\cM_{W,\mec d}]),
$
and 
\item
$
\chi([\cM_{W,\mec d}]) = \deg(\mec d) + C
$
where $C$ is the offset of $f$.
\end{enumerate}
\end{proposition}
\begin{proof}
Say that we write $W$ as 
$$
W = W_1 + \cdots + W_s - \bigl( \tilde W_1 + \cdots + \tilde W_{s-1} \bigr)
$$
with all $W_i$ and $\tilde W_i$ as perfect matchings, and also as
a difference of sums of perfect matchings
$$
W = W'_1 + \cdots + W'_{s'} - 
\bigl( \tilde W'_1 + \cdots + \tilde W'_{s'-1} \bigr).
$$
Then we have
$$
W_1 + \cdots + W_s + \tilde W'_1 + \cdots + \tilde W'_{s'-1}
=
\tilde W_1+\cdots+\tilde W_{s-1} + W'_1 + \cdots + W'_{s'} .
$$
It follows from Proposition~\ref{pr_s_fold_indicator_function}, 
specifically \eqref{eq_s_fold_indicator_three}, that for 
any $\mec d\in\integers^2$
$$
\left(\bigoplus_{i=1}^s \cM_{W_i,\mec d} \right)
\oplus
\left(\bigoplus_{i=1}^{s'-1} \cM_{\tilde W'_i,\mec d} \right)
\isom
\left(\bigoplus_{i=1}^{s-1} \cM_{\tilde W_i,\mec d} \right)
\oplus
\left(\bigoplus_{i=1}^{s'} \cM_{W'_i,\mec d} \right)
$$
and hence
$$
\left(\bigoplus_{i=1}^s \cM_{W_i,\mec d} \right)
\ominus
\left(\bigoplus_{i=1}^{s-1} \cM_{\tilde W_i,\mec d} \right)
\isom
\left(\bigoplus_{i=1}^{s'} \cM_{W'_i,\mec d} \right)
\ominus
\left(\bigoplus_{i=1}^{s'-1} \cM_{\tilde W'_i,\mec d} \right)
$$
as virtual $k$-diagrams.  Hence the class $[\cM_{W,\mec d}]$ is independent
of how we write $W$ as a difference of (finite) sums of perfect matchings.

For the second part of the proposition, we write $W$ as in
\eqref{eq_W_as_difference_definition}, and note that for any $i$,
$$
b^0(\cM_{W_i,\mec d}) = (\fraks W_i)(\mec d),
$$
and similarly for $\tilde W_i$, and
hence
$$
b^0(\cM_{W,\mec d}) = 
\bigl( \fraks(W_1+\cdots+W_s)\bigr) (\mec d)
-
\bigl( \fraks(\tilde W_1+\cdots+\tilde W_{s-1})\bigr) (\mec d)
=
(\fraks W) (\mec d) ,
$$
where the last equality holds
by applying $\fraks$ to both sides of 
\eqref{eq_W_as_difference_definition}
and using the linearity of $\fraks$.
We reason similarly with $b^0$ replaced 
% with $b^1$ and $\chi$.
with $\chi$, in view of \eqref{eq_offset_of_difference_of_sums};
finially we use $b^1=b^0-\chi$ to reason about $b^1$.
\end{proof}

\begin{example}\label{ex_period_three_non_canonical_example}
Let $W\from\integers^2\to\integers$ be $3$-periodic and satisfy
$W(1,0)=W(1,2)=W(0,1)=W(2,1)=1$ and $W(1,1)=-1$.
Therefore for all $m\in\integers$,
\begin{equation}\label{eq_W_non_canonical_as_pm_difference}
W(3m+1,-3m)=W(3m+1,-3m+2)=W(3m,-3m+1)=W(3m+2,-3m+1)=1,
\end{equation} 
and
\begin{equation}\label{eq_W_non_canonical_as_pm_difference_minus}
W(3m+1,-3m+1)=-1.
\end{equation} 
We easily see 
that the values of $\mec a=(a_1,a_2)$ at which
$W(\mec a)\ne 0$ and $0\le a_1\le 2$ or $0\le a_2\le 2$
correspond to the values in
\eqref{eq_W_non_canonical_as_pm_difference} 
and~\eqref{eq_W_non_canonical_as_pm_difference_minus}
with $m=0$.
It follows that the $0,1,2$-th row sums and the $0,1,2$-th column sums
equal $1$, and hence all row sums and all column sums of $W$ equal $1$,
and hence
$\sigma W$ is a Riemann function (one can also check that $W$
{\em slowly growing}, as in Definition~\ref{de_slowly_growing}, by
examining the values of $W$ in rows $0,1,2$ and columns $0,1,2$)\footnote{
  For any $W\from\integers^2\to\integers$ we easily see that $\fraks W$ is
  slowly growing iff the sign pattern in each row and each column
  is an alternating sequence of $+$'s and $-$'s, beginning and
  ending in $+$'s.  For the $W$ in this example, we easily check this
  to be the case.
  }.  One can write $W=W_1-W_2+W_3$ in a number of ways, even where
each $W_i$ is 3-periodic, but there seems to be no canonical way
of doing so.  One can also write ``canonically'' write
$W=W^+-W^-$ with $W^+=\max(W,0)$
and $W^-=\max(-W,0)$, but then we easily see that 
$b^0(\cM_{W^+,\mec d})=+\infty$\footnote{To see that
  $b^0(\cM_{W^+,\mec d})=+\infty$, note that 
  for each $m$ large, note that $G={\rm Graph}(W^+,\mec d)$ as in
  Theorem~\ref{th_cM_Betti_in_graph_terms} has a cycle
  created by the $W$-values $W(3m+1,-3m)=W(3m+1,-3m+2)=1$ that give a
  multiple edge from $v_0$ to the vertex $3m\in V_{\rm first}$.
  } for any $\mec d$ and
$b^1(\cM_{W^-,\mec d})=+\infty$ for any $\mec d$\footnote{
  since for any $m\in\integers$
  the vertices $3m,3m+2\in V_{\rm first}$ are isolated
  in the graph ${\rm Graph}(W^-,\mec d)$ as in
  Theorem~\ref{th_cM_Betti_in_graph_terms}.
  }.
Hence the virtual diagram $(\cM_{W^+,\mec d},\cM_{W^-,\mec d})$
doesn't have 
% ``the right'' virtual Betti numbers.
a well-defined Euler characteristic
(which should equal $\deg(\mec d)+C$ for some $C$).
\end{example}

\begin{example}
\label{ex_uncountably_equivalent}
We remark that the there are virtual
$k$-diagram $[\cM_{W,\mec d}]$ that can be realized as a formal difference
of Fredholm $k$-diagrams in
uncountably many ways.
Indeed, let $W(i,j)=1$ iff
for some $t\in\integers$
we have $i\in\{2t,2t+1\}$ and $j\in\{-2t,-2t+1\}$.
Hence $W$ is a $2$-fold matching.
We claim there are uncountably many perfect matchings $W_1$
such that $W_2=W-W_1$ is also a perfect matching:
indeed, consider any perfect matching, $W_1$,
such that for each $t\in\integers$
either $W_1(2t,-2t)=W_1(2t+1,-2t+1)=1$
or $W_1(2t+1,-2t)=W(2t,-2t+1)=1$;
we easily see that $W_2=W-W_1$ is also a perfect matching,
and there are uncountably many such $W_1$
(and $W_1,W_2$ are supported in degrees 0,1,2, so they are, indeed,
perfect matchings).
By contrast, if we insist that $W_1$ is $r$-periodic for some $r\ge 1$,
then there are only finitely many possible $W_1$ as above.
Hence for fixed perfect matchings
$\tilde W_1,\tilde W_2$ with $\tilde W_1+\tilde W_2=W$, we have
$$
\cM_{\tilde W_1,\mec d}
\oplus
\cM_{\tilde W_2,\mec d}
\isom
\cM_{W_1,\mec d}
\oplus
\cM_{W_2,\mec d}
$$
whenever $W_1+W_2=W$.
\end{example}

Example~\ref{ex_uncountably_equivalent}
shows that unless we assume periodicity,
the virtual
$k$-diagram $[\cM_{W,\mec d}]$ can be realized as a virtual $k$-diagram in
uncountably many ways.
It similarly follows (unless we make some periodicity assumptions) that whenever \eqref{eq_W_as_W_i_minus_tilde_W_i}
holds,
there are uncountably many ways of writing $W$ as
\eqref{eq_W_as_W_i_minus_tilde_W_i},
if we replace $s$ with $s+2$ there, and let
$W_{s+1},W_{s+2},\tilde W_{s},\tilde W_{s+1}$
be, respectively,
$W_1,W_2,\tilde W_1,\tilde W_2$ in Example~\ref{ex_uncountably_equivalent}
above.

One can generalize the above proposition to indicator $k$-diagram sums.

\begin{proposition} 
Let $W_1,\ldots,W_4$ be functions $\integers^2\to\integers_{\ge 0}$ such
that $W_1-W_2=W_3-W_4$.
Then for all $\mec d$ we have
$$
\bec\cI_{\mec d}^{W_1}\ominus \bec\cI_{\mec d}^{W_2}
\sim
\bec\cI_{\mec d}^{W_3}\ominus \bec\cI_{\mec d}^{W_4}.
$$
Moreover, if $\bec\cI_{\mec d}^{W_i}$ is a Fredholm $k$-diagram for
all $i$, then this equivalence holds as virtual Fredholm $k$-diagrams.
\end{proposition}
\begin{proof}
Indeed, if $W=W_1+W_4=W_2+W_3$, then clearly
$$
\bec\cI_{\mec d}^{W_1}\oplus  \bec\cI_{\mec d}^{W_4}
\isom \bec\cI_{\mec d}^{W}
\isom \bec\cI_{\mec d}^{W_2}\oplus  \bec\cI_{\mec d}^{W_3} .
$$
\end{proof}

% \iffalse
% 
% \begin{example}
% Let $W\from\integers^2\to\integers$ be $3$-periodic with
% $W(1,0)=W(1,2)=W(0,1)=W(2,1)=1$ and $W(1,1)=-1$.
% Then $\sigma W$ is a Riemann function (it is also
% {\em slowly growing}, as in Definition~\ref{de_slowly_growing}).
% One can write $W=W_1-W_2+W_3$ in a number of ways, even where
% each $W_i$ is 3-periodic, but there seems to be no canonical way
% of doing so.  One can also write ``canonically'' write
% $W=W^+-W^-$ with $W^+=\max(W,0)$
% and $W^-=\max(-W,0)$, but then we easily see that 
% $b^0(\cM_{W^+,\mec d})=+\infty$ for any $\mec d$ and
% $b^1(\cM_{W^-,\mec d})=+\infty$ for any $\mec d$.
% Hence the virtual diagram $(\cM_{W^+,\mec d},\cM_{W^-,\mec d})$
% doesn't have ``the right'' virtual Betti numbers.
% \end{example}
% 
% \fi

\subsection{A Canonical Virtual $k$-Diagram of a Riemann Function}

In Example~\ref{ex_period_three_non_canonical_example} we remarked
that one can canonically write
\begin{equation}\label{eq_W_as_plus_part_minus_part}
W=W^+-W^-,
\quad\mbox{where}\quad
W^+=\max(W,0),\ %
W^-=\max(-W,0),
\end{equation} 
but this doesn't generally express $W$ as a difference of perfect 
matchings;
furthermore, as
Example~\ref{ex_period_three_non_canonical_example} shows,
even if $\fraks W$ is a Riemann function,
$(\cM_{W^+,\mec d},\cM_{W^-,\mec d})$ may not have a well defined
Euler characteristic.
In this subsection we remark that the
\eqref{eq_W_as_plus_part_minus_part} does lead to a canonical
way to write $[\cM_{W,\mec d}]$ as a virtual $k$-diagram composed
of indicator $k$-diagrams.

\begin{proposition}\label{pr_cM_W_poss_neg_indicator_canonical}
Let $W\from\integers^2\to\integers$ be initially and eventually zero.
With notation as in \eqref{eq_W_as_plus_part_minus_part}, for each
$\mec d\in\integers^2$, the formal difference
$$
\bec\cI_{\mec d}^{\oplus W} 
\eqdef
\ \ %
\bec\cI_{\mec d}^{\oplus W^+}\hskip-0.15cm
\ominus
\bec\cI_{\mec d}^{\oplus W^-} 
$$
is a virtual $k$-diagram, i.e., both
$\bec\cI_{\mec d}^{\oplus W^+}\hskip-0.15cm$ and
$\bec\cI_{\mec d}^{\oplus W^-}\hskip-0.15cm$ are Fredholm $k$-diagrams.
Furthermore, if $\fraks W$ is a Riemann function,
then as virtual (Fredholm) $k$-diagrams 
\begin{equation}\label{eq_indicator_W_non_pos_and_cM}
[ \bec\cI_{\mec d}^{\oplus W} ] = [ \cM_{W,\mec d} ].
\end{equation} 
\end{proposition}
\begin{proof}
Since $W$ is initially and eventually zero, for any $\mec d$, the
number of $\mec a$ with $\mec d\le \mec a$ and $W(\mec a)\ne 0$
is finite, and similarly for the number with $\mec d+\mec 1\ge\mec a$.
Hence both
$\bec\cI_{\mec d}^{\oplus W^+}$ and
$\bec\cI_{\mec d}^{\oplus W^-}$ are Fredholm $k$-diagrams.
To see \eqref{eq_indicator_W_non_pos_and_cM}, we see that
for any equality
\eqref{eq_W_as_W_i_minus_tilde_W_i} we have
$$
(W_1+\cdots+W_s) - (\tilde{W_1}+\cdots+\tilde{W}_{s-1})
= W = W^+ - W^-
$$
and hence
$$
W^- + (W_1+\cdots+W_s)
=
W^+ + (\tilde{W_1}+\cdots+\tilde{W}_{s-1})
$$
and hence for all $\mec d\in\integers$ we have
$$
\bec\cI_{\mec d}^{\oplus W^-} \oplus 
\bec\cI_{\mec d}^{\oplus W_1} \oplus 
\cdots\oplus
\bec\cI_{\mec d}^{\oplus W_s}
\isom
\bec\cI_{\mec d}^{\oplus W^+} \oplus 
\bec\cI_{\mec d}^{\oplus \tilde W_1} \oplus 
\cdots\oplus
\bec\cI_{\mec d}^{\oplus \tilde W_{s-1}} ,
$$
and hence
$$
\bec\cI_{\mec d}^{\oplus W^-} \oplus 
\cM_{W_1,\mec d} \oplus \cdots \oplus \cM_{W_s,\mec d}
\isom
\bec\cI_{\mec d}^{\oplus W^+} \oplus 
\cM_{\tilde W_1,\mec d} \oplus \cdots \oplus \cM_{\tilde W_{s-1},\mec d} ,
$$
and since all $k$-diagrams are Fredholm, this
implies \eqref{eq_indicator_W_non_pos_and_cM} as virtual
(Fredholm) $k$-diagrams.
\end{proof}

\section{Modeling General Riemann Functions}
\label{se_higher_Riemann}

In this section we model any Riemann function
$f\from\integers^n\to\integers$ by gluing together the models
we have developed for $n=2$.
We begin by stating the main results, leaving the proofs of 
the more difficult theorems for later subsections.

\subsection{Main Modeling Results}

\begin{definition}\label{de_virtual_k_diagram_of_two_var_restriction}
Let $f\from \integers^n\to\integers$ be a Riemann function.
For any $i,j\in[n]$ with $i\ne j$ and any $\mec d\in\integers^n$, let
$f_{i,j,\mec d}=f_{i,j,\mec d}(a_i,a_j)\from \integers^2\to\integers$ be the
two-variable restriction
\eqref{eq_two_variable_restriction}.
(We write $a_i,a_j$ as the arguments for $f_{i,j,\mec d}$
instead of, say, $a_1,a_2$, to stress that
$a_i$ corresponds to adding $a_i\mec e_i$ in
\eqref{eq_two_variable_restriction}, and similarly for $a_j$).
We set $W=W_{f;i,j,\mec d}$ to be the weight of $f_{i,j,\mec d}$, 
and define the {\em virtual $k$-diagram associated to $f$ and $\mec d$
at coordinates $i,j$} to be the class
$$
[\cM_{f;i,j,\mec d}]\eqdef [\cM_{W,\mec 0}]=[\cM_{W_{f;i,j,\mec d},\mec 0}]
$$
(which we know is a single equivalence class of virtual (Fredholm)
$k$-diagrams).
\end{definition}

The merit of the above definition is described in the following 
theorem, that is really a
straightforward consequence of
Proposition~\ref{pr_cM_W_poss_neg_single_virt_class}.

% The following theorem is a straightforward consequence of
% Proposition~\ref{pr_cM_W_poss_neg_single_virt_class}.

\begin{theorem}
\label{th_pre_main_modeling_theorem_general_Riemann_function}
Let $f\from \integers^n\to\integers$ be a Riemann function.
Then for any distinct $i,j\in[n]$ and $\mec d\in\integers^n$ we have
\begin{equation}\label{eq_betti_zero_two_var_restriction}
b^0([\cM_{f;i,j,\mec d}])=f_{i,j,\mec d}(\mec 0)=f(\mec d),
\end{equation}
\begin{equation}\label{eq_chi_two_var_restriction}
\chi([\cM_{f;i,j,\mec d}])=\chi(\cM_{W,\mec 0}) = \deg(\mec d)+C
\end{equation} 
where $C$ is the offset of $f$, and for every $\mec K\in\integers^n$ we have
\begin{equation}\label{eq_betti_one_two_var_restriction}
b^1([\cM_{f;i,j,\mec d}])=f^\wedge_\mec K(\mec K-\mec d).
\end{equation}
In particular, it follows that for any distinct $i',j'\in[n]$ we have
\begin{equation}\label{eq_betti_duality_two_var_restriction}
b^1([\cM_{f;i,j,\mec d}])=f^\wedge_\mec K(\mec K-\mec d)
= b^0([\cM_{f^\wedge_\mec K;i',j',\mec K-\mec d}]).
\end{equation} 
\end{theorem}
\begin{proof}
By definition, $\cM_{f;i,j,\mec d}=\cM_{W,\mec 0}$ with
$W=W_{f;i,j,\mec d}$ equal to the weight of $f_{i,j,\mec d}$.
Hence Proposition~\ref{pr_cM_W_poss_neg_single_virt_class} implies that
$$
b^0[\cM_{W,\mec 0}] = (\fraks W)(\mec 0)=f_{i,j,\mec d}(\mec 0)=f(\mec d)
$$
and hence \eqref{eq_betti_zero_two_var_restriction} holds.
If $C$ is the offset of $f$, then for sufficiently large $a_i+a_j$ we have
$$
f_{i,j,\mec d}(a_i,a_j)=f(\mec d+a_i\mec e_i+a_j\mec e_j)
= \deg(\mec d+a_i\mec e_i+a_j\mec e_j) + C=
a_i+a_j+\deg(\mec d)+C
$$
and it follows that the offset of $f_{i,j,\mec d}$ is $C'=\deg(\mec d)+C$.
So Proposition~\ref{pr_cM_W_poss_neg_single_virt_class} implies that
$$
\chi([\cM_{W,\mec 0}]) = \deg(\mec 0) + C' = C' = \deg(\mec d)+C
$$
and \eqref{eq_chi_two_var_restriction} follows.
It follows that
$$
b^1[\cM_{W,\mec 0}] =
b^0[\cM_{W,\mec 0}] - \chi[\cM_{W,\mec 0}] 
=
f(\mec d)-\deg(\mec d)-C,
$$
and Proposition~\ref{pr_cM_W_poss_neg_single_virt_class}
and \eqref{eq_generalized_riemann_roch} implies
\eqref{eq_betti_one_two_var_restriction}.

Finally, \eqref{eq_betti_duality_two_var_restriction} follows from
\eqref{eq_betti_one_two_var_restriction} and from
\eqref{eq_betti_zero_two_var_restriction} with $f$ replaced with
$f^\wedge_\mec K$ and with $\mec d$ replaced with $\mec K-\mec d$.
\end{proof}

% We remark that if $f$ is a Riemann function, then
% by writing 
% $W=W_{f;i,j,\mec d}$ as a difference of sums of perfect matchings
% \eqref{eq_W_as_W_i_minus_tilde_W_i},
% and applying
% Lemma~\ref{le_codimension_one_cons} to each perfect matching
% (i.e., each $W_i$ and $\tilde W_i$), we see that 
% $$
% \chi(\cM_{f;i,j,\mec d+\mec e_i}) 
% =\chi(\cM_{f;i,j,\mec d+\mec e_j}) 
% =\chi(\cM_{f;i,j,\mec d+\mec e_i})  + 1.
% $$

The main goal of this section is to prove the following two theorems
that state that the equivalence class of the
virtual Fredholm $k$-diagram
$[\cM_{f;i,j,\mec d}]$ is independent
of the choice of $i$ and $j$.

\begin{theorem}\label{th_glue_i_j_j_prime}
Let $n\in\naturals$ with $n\ge 3$, $\mec d\in\integers^n$,
$f\from \integers^n\to\integers$ be a Riemann function, and 
$i,j,j'\in [n]$ be three distinct integers.  
Then
$[\cM_{f;i,j,\mec d}]=[\cM_{f;i,j',\mec d}]$.
% , and more generally
% for any $a\in\integers$ we have
% $$
% [\cM_{W_{f;i,j,\mec d},(a,0)}] = [\cM_{W_{f;i,j',\mec d},(a,0)}] .
% $$
\end{theorem}
This theorem is more technical, and will be proven in
Subsection~\ref{su_proof_th_glue_i_j_j_prime}.
The above theorem easily yields one of the main results
in the section.

\begin{corollary}\label{co_glue_i_j_j_prime}
Let $n\in\naturals$ with $n\ge 2$, $\mec d\in\integers^n$,
and $f\from \integers^n\to\integers$ be a Riemann function.
Then the equivalence class
$[\cM_{f;i,j,\mec d}]$ of virtual (Fredholm) $k$-diagrams
is independent of the choice of distinct $i,j\in[n]$.
\end{corollary}
\begin{proof}
For distinct $i,j\in[n]$ we have
$\cM_{f;i,j,\mec d}\isom \cM_{f;j,i,\mec d}$ 
by the evident morphism that exchanges $B_1,A_1$ values respectively with
$B_2,A_2$ values.  Hence for any $i\ne j$ we have
\begin{equation}\label{eq_M_two_var_exchange_indices}
[\cM_{f;i,j,\mec d}] 
= [\cM_{f;j,i,\mec d}].
\end{equation} 

If $n=2$, then the only choices of distinct $i,j\in[2]$ are
$(i,j)$ equal to either $(1,2)$ or $(2,1)$.  Since
\eqref{eq_M_two_var_exchange_indices} shows that
$$
[\cM_{f;1,2,\mec d}] 
= [\cM_{f;2,1,\mec d}],
$$
this proves the corollary in the case $n=2$.

Hence it suffices to prove the corollary when $n\ge 3$.

According to Theorem~\ref{th_glue_i_j_j_prime} we have
$[\cM_{f;1,2,\mec d}]=[\cM_{f;1,j,\mec d}]$ for any $j\ge 3$,
and similarly
for any $i\ne j$ we have
$[\cM_{f;j,1,\mec d}]=[\cM_{f;j,i,\mec d}]$.
Combining these two equalities with 
\eqref{eq_M_two_var_exchange_indices} we have
$$
[\cM_{f;1,2,\mec d}]=[\cM_{f;1,j,\mec d}] =
[\cM_{f;j,1,\mec d}]=[\cM_{f;j,i,\mec d}]=
[\cM_{f;i,j,\mec d}].
$$
\end{proof}

The above corollary makes the following definition well defined.

\begin{definition}\label{de_cM_f_at_d}
For any Riemann function $f\from\integers^n\to\integers$, and any
$\mec d\in\integers^n$, we define
the {\em virtual $k$-diagram of $f$ at $\mec d$}, denoted
$[\cM_{f {\rm \;at\;}\mec d}]$
to be the class of virtual $k$-diagram $[\cM_{f;i,j,\mec d}]$
for any distinct $i,j\in[n]$
(which is a single equivalence class of virtual $k$-diagrams
in view of Corollary~\ref{co_glue_i_j_j_prime}.
\end{definition}

Stating 
Theorem~\ref{th_pre_main_modeling_theorem_general_Riemann_function} in
terms of Definition~\ref{de_cM_f_at_d} immediately
implies the following theorem.

\begin{theorem}\label{th_main_modeling_theorem_general_Riemann_function}
Let $f\from\integers^n\to\integers$ be any Riemann function
with offset $C$,
and let $\mec d,\mec K\in\integers$.  Then we have
$$
b^0 \bigl( [\cM_{f {\rm \;at\;}\mec d}] \bigr) = f(\mec d),
$$
$$
\chi  \bigl( [\cM_{f {\rm \;at\;}\mec d}] \bigr) = \deg(\mec d)+C,
$$
and
$$
b^1 \bigl( [\cM_{f {\rm \;at\;}\mec d}] \bigr)
=
f^\wedge_\mec K(\mec K-\mec d)
=
b^0 \bigl( [\cM_{f^\wedge_\mec K {\rm \;at\;}\mec K-\mec d}] \bigr).
$$
\end{theorem}

In the following special case of 
Theorem~\ref{th_glue_i_j_j_prime} one can prove
a much stronger result.

\begin{theorem}
\label{th_glue_i_j_j_prime_not_virtual}
Let $n\in\naturals$ with $n\ge 3$, $\mec d\in\integers^n$,
$f\from \integers^n\to\integers$ be a Riemann function, and 
$i,j,j'\in [n]$ be three distinct integers.  
Say that the weights of $f_{i,j,\mec d}$ and $f_{i,j',\mec d}$,
respectively
$W_{f;i,j,\mec d}$ and $W_{f;i,j',\mec d}$, are non-negative,
and hence both perfect matchings.
Then we have
$$
\cM_{f;i,j,\mec d} \isom \cM_{f;i,j',\mec d} .
% \cM_{W_{f;i,j,\mec d},(a,0)} \isom \cM_{W_{f;i,j',\mec d},(a,0)} .
$$
\end{theorem}
The proof of this theorem will be given in
Subsection~\ref{su_proof_th_glue_i_j_j_prime_not_virtual}.

\begin{corollary}\label{co_glue_i_j_j_prime_not_virtual}
Let $n\in\naturals$ with $n\ge 3$, $\mec d\in\integers^n$,
and $f\from \integers^n\to\integers$ be a Riemann function.
Say that for some $I\subset[n]$ we have that 
$W_{f;i,j,\mec d}=\frakm f_{i,j,\mec d}$ is everywhere non-negative,
and therefore a perfect matching.
Then all the $k$-diagrams $\cM_{f;i,j,\mec d}$ varying over
distinct $i,j\in I$ are isomorphic (as $k$-diagrams).
\end{corollary}
The proof is the same as that
of Corollary~\ref{co_glue_i_j_j_prime}.

In particular, if $I=[n]$ in the above corollary, then for fixed
$f$ and all $\mec d\in\integers^n$, the
$k$-diagrams $\cM_{f;i,j,\mec d}$ are all isomorphic,
and one can define 
$[\cM_{f {\rm \;at\;}\mec d}]$ as an equivalence class of 
$k$-diagrams.
Then Corollary~\ref{co_glue_i_j_j_prime_not_virtual}
yields the
following result, which
gives a stronger conclusion
than Theorem~\ref{th_main_modeling_theorem_general_Riemann_function}
in a special case thereof.

\begin{theorem}\label{th_main_modeling_theorem_when_non_virtual}
Let $f\from\integers^n\to\integers$ be any Riemann function
with offset $C$, and let $\mec K\in\integers$.
Assume that for all distinct $i,j\in[n]$ and all $\mec d\in\integers^n$,
$\frakm f_{i,j,\mec d}$ and $\frakm (f^\wedge_\mec K)_{i,j,\mec d}$
are perfect matchings.
Then the conclusions of
Theorem~\ref{th_main_modeling_theorem_general_Riemann_function} hold
where we understand that $[\cM_{f {\rm \;at\;}\mec d}]$
and $[\cM_{f^\wedge_\mec K {\rm \;at\;}\mec K-\mec d}]$
refer to equivalence classes of $k$-diagrams.
\end{theorem}
We remark that 
\eqref{eq_n_variate_equality_two_var_duals}
of Section~\ref{se_duality_first} shows that if
$\frakm f_{i,j,\mec d}$ are non-negative for
all $i,j,\mec d$, then so all the $\frakm (f^\wedge_\mec K)_{i,j,\mec d}$,
and conversely.

It is helpful to first prove
Theorem~\ref{th_glue_i_j_j_prime_not_virtual}
first, as it is simpler to prove but illustrates the main idea in
the proof of Theorem~\ref{th_glue_i_j_j_prime}

The rest of this section is dedicated to proving these two theorems.

One crucial ingredient of the proofs of both theorems is the equality
\begin{equation}\label{eq_pasting_relation}
\forall a\in\integers, \quad
f_{i,j,\mec d}(a,0) = f(\mec d+ a\mec e_i) =f_{i,j',\mec d}(a,0).
\end{equation} 
The other idea in both proofs is to look for an isomorphism 
that is very simple
along the $A_1$ values of $\cM_{W,\mec d}$, 
and to see what conditions this requires
elsewhere; it turns out that it is only along the $B_2$ value
that one needs some conditions, and those conditions turn out to
be exactly \eqref{eq_pasting_relation}.
Let us give the details.

\subsection{Isomorphisms That Are Simple Along the $A_1$ Values}

To prove Theorem~\ref{th_glue_i_j_j_prime_not_virtual}, we will use
the following lemma.

\begin{lemma}\label{le_zipper_matchings}
Let $W,W'$ be perfect matchings, and $\pi,\pi'$ their associated
bijections.  Then for any $\mec d\in\integers^2$ the following are equivalent:
\begin{enumerate}
\item
there exists an isomorphism of $k$-diagrams
$\phi\from\cM_{W,\mec d}\to\cM_{W',\mec d}$
such that $\phi(A_1)$ is the identity; and
\item
\begin{equation}
\label{eq_zipper_condition_from_W_to_W_prime} 
\forall a\in\integers_{\le d_1}, \quad
\pi(a)\le d_2
\iff
\pi'(a)\le d_2.
\end{equation}
\end{enumerate}
\end{lemma}
We remark that to prove Theorem~\ref{th_glue_i_j_j_prime}, 
we need know only that
(2) $\implies$ (1).
\begin{proof}
Consider a morphism $\phi$ with $\phi(A_1)$ the identity.  
\begin{figure}
\begin{tikzpicture}[scale=0.50,font=\small]
% \tikzstyle{every node}=[font=\tiny]
% \node (B1) at (0,4) {$\cM_{W,\mec d}(B_1)=k^{\oplus\integers_{\le d_1}}$};
% \node (B2) at (0,-4) {$\cM_{W,\mec d}(B_2)=k^{\oplus\integers_{\le d_2}}$};
% \node (B3) at (0,0) {$\cM_{W,\mec d}(B_3)=k^{\oplus W}$};
% \node (A1) at (8,2) {$\cM_{W,\mec d}(A_1)=k^{\oplus\integers}$};
% \node (A2) at (8,-2) {$\cM_{W,\mec d}(A_2)=k^{\oplus\integers}$};
\node (B1) at (0,4) {$k^{\oplus\integers_{\le d_1}}$};
\node (B2) at (0,-4) {$k^{\oplus\integers_{\le d_2}}$};
\node (B3) at (0,0) {$k^{\oplus W}$};
\node (A1) at (8,2) {$k^{\oplus\integers}$};
\node (A2) at (8,-2) {$k^{\oplus\integers}$};
\draw [->] (B1) -- (A1) node [midway,above] 
{$\quad \mec e_{a_1}\mapsto \mec e_{a_1}$} ;
\draw [->] (B2) -- (A2) node [midway,below] 
{$\quad \mec e_{a_2}\mapsto \mec e_{a_2}$} ;
\draw [->] (B3) -- (A1) node [midway,above] 
{$\mec e_{(a_1,\pi(a_1))}\mapsto \mec e_{a_1}\quad\quad\quad\quad\quad$} ;
\draw [->] (B3) -- (A2) node [midway,below] 
{$\mec e_{(a_1,\pi(a_1))}\mapsto \mec e_{\pi(a_1)}\quad\quad\quad\quad\quad\quad\quad$} ;
% \node (BB1) at (15,4) {$k^{\oplus\integers_{\le d_1}}=\cM_{W',\mec d}(B_1)$};
% \node (BB2) at (15,-4) {$k^{\oplus\integers_{\le d_2}}=\cM_{W',\mec d}(B_2)$};
% \node (BB3) at (15,0) {$k^{\oplus W}=\cM_{W',\mec d}(B_3)$};
% \node (AA1) at (23,2) {$k^{\oplus\integers}=\cM_{W',\mec d}(A_1)$};
% \node (AA2) at (23,-2) {$k^{\oplus\integers}=\cM_{W',\mec d}(A_2)$};
\node (BB1) at (15,4) {$k^{\oplus\integers_{\le d_1}}$};
\node (BB2) at (15,-4) {$k^{\oplus\integers_{\le d_2}}$};
\node (BB3) at (15,0) {$k^{\oplus W}$};
\node (AA1) at (23,2) {$k^{\oplus\integers}$};
\node (AA2) at (23,-2) {$k^{\oplus\integers}$};
\draw [->] (BB1) -- (AA1) node [midway,above] 
{$\quad\quad \mec e_{a_1}\mapsto \mec e_{a_1}$} ;
\draw [->] (BB2) -- (AA2) node [midway,below] 
{$\quad\quad \mec e_{a_2}\mapsto \mec e_{a_2}$} ;
\draw [->] (BB3) -- (AA1) node [pos=0.7,below] 
{$\hskip 10 em \mec e_{(a_1,\pi'(a_1))}\mapsto \mec e_{a_1}$} ;
%
% \enskip,\quad,\qquad = 1/2,1,2 "em"s, em=width of M, ex=... of x, mu = 1/18em
% above used to be:
% {$\quad\quad\quad\quad\quad\quad \mec e_{(a_1,\pi'(a_1))}\mapsto \mec e{a_1}$} ;
%
\draw [->] (BB3) -- (AA2) node [pos=0.7,above] 
{\hskip 12 em $\mec e_{(a_1,\pi'(a_1)}\mapsto \mec e_{\pi'(a_1)}$} ;
\draw [->,dashed,ultra thick] (B1) -- (BB1) node [pos=0.5,above] {$\phi(B_1)={\rm identity}$};
\draw [->,dashed,ultra thick] (B2) -- (BB2) node [pos=0.7,above] {Is $\phi(B_2)$ defined?};
\draw [->,dashed,ultra thick] (B3) -- (BB3) node [pos=0.6,above] 
{$\phi(B_3)\mec e_{(a_1,\pi(a))}=\mec e_{(a_1,\pi'(a_1))}$};
\draw [->,ultra thick] (A1) -- (AA1) node [pos=0.5,above] {$\phi(A_1)={\rm identity}$};
\draw [->,dashed,ultra thick] (A2) -- (AA2) node [pos=0.5,above] 
{$\phi(A_2)(\mec e_{\pi(a_1)})=\mec e_{\pi'(a_1)}$};
\node (F) at (4,-6) {\Large$\cM_{W,\mec d}$};
\node (G) at (19,-6) {\Large$\cM_{W',\mec d}$};
\draw [->,thick] (F) -- (G) node [midway,above] {\Large$\phi$};
\end{tikzpicture}
\caption{Implications of $\phi(A_1)={\rm identity}$, for a  
$\phi\from\cM_{W,\mec d}\to \cM_{W',\mec d}$ (where we omit
the quantifiers $\forall a_1\in\integers$ and $\forall a_2\in\integers$)}
\label{fi_W_one_W_two_zipper}
\end{figure}
In Figure~\ref{fi_W_one_W_two_zipper} we depict
$\phi(A_1)$, which is the
identity map, is depicted in a thick line, and the inferences about the
other values of $\phi$ are depicted in dashed lines.
We easily see that (considering Figure~\ref{fi_W_one_W_two_zipper}):
\begin{enumerate}
\item 
$\phi(B_1)$ is forced to be the identity, 
\item 
$\phi(B_3)$ must take $\mec e_{(a_1,\pi(a_1))}$ to
$\mec e_{(a_1,\pi'(a_1))}$ for all $a_1\in\integers$,
\item
$\phi(A_2)$ is forced to take $\mec e_{\pi(a_1)}$ to 
$\mec e_{\pi'(a_1)}$ for all $a_1\in\integers$, and
\item
$\phi(B_2)$ is uniquely determined if it exists, and it exists
iff for all $a_1\in\integers$
with $\pi(a_1)\le d_2$,
the vector the $\mec e_{\pi'(a_1)}$ lies in
$k^{\oplus\integers_{\le d_2}}$.
\end{enumerate}
Hence $\phi(B_2)$ exists iff
$$
\forall a_1\in\integers_{\le d_1},\quad
\pi(a_1)\le d_2 \implies
\pi'(a_1)\le d_2;
$$
and hence $\phi(B_2)$ exists and is an isomorphism iff
\eqref{eq_zipper_condition_from_W_to_W_prime} holds.
\end{proof}

We similarly prove the following generalization that we will
use to prove Theorem~\ref{th_glue_i_j_j_prime}.

\begin{lemma}\label{le_zipper_multiple_matchings}
Let $W_1,\cdots,W_s$ and $W_1',\cdots,W_s'$ be two sequences of
perfect matchings $\integers^2\to\integers$.
Let $\mec d\in\integers^2$, and for each
$a_1\in\integers$ and $j\in[s]$, 
let $\mec e_{a_i,j}$ and $\mec e_{a_i,j}'$ denote, respectively, the
standard basis vector $\mec e_{a_1}\in k^{\integers}$ in, respectively
$\cM_{W_j,\mec d}(A_1)$ and $\cM_{W_j',\mec d}(A_1)$.
Let
$$
W = W_1 + \cdots + W_s, \quad
W' = W_1'+\cdots+ W_s' ,
$$
and
$$
\cM_{\mec d}=\cM_{W_1,\mec d}\oplus \cdots\oplus \cM_{W_s,\mec d}, 
\quad
\cM_{\mec d}'=\cM_{W_1',\mec d}\oplus \cdots\oplus \cM_{W_s',\mec d}.
$$
Then for any $\mec d\in\integers^2$, the following are equivalent:
\begin{enumerate}
\item 
there exists an isomorphism
$\phi\from\cM\to\cM'$ that for each $a_1\in\integers_{\le d_1}$,
such that $\phi(A_1)$
restricted to $\mec e_{a_1,1},\ldots,\mec e_{a_1,s}$ yields a bijection
from this set 
to $\mec e'_{a_1,1},\ldots,\mec e'_{a_1,s}$;
\item
letting $\pi_r,\pi'_r$ for $r\in[s]$ denote
the bijections associated to $W_r,W'_r$,
for all $a_1\le d_1$,
\begin{equation}\label{eq_pi_pi_prime_iff}
\bigl| \{ r \ | \ \pi_r(a_1) \le d_2 \} \bigr|
=
\bigl| \{ r' \ | \ \pi'_{r'}(a_1) \le d_2 \} \bigr|.
\end{equation} 
\end{enumerate}
\end{lemma}
\begin{proof}
We have
$$
\cM_{\mec d}(A_1)=
\bigoplus_{a_1\in\integers}{\rm Span}(\mec e_{a_1,1},\ldots,\mec e_{a_1,s})
$$
and similarly with $\cM'$ and the $\mec e'_{a_1,i}$.
Condition~(1) says that 
for each $a_1\in\integers$ there is a permutation
$\sigma=\sigma_{a_1}$ on $[s]$ such that
$$
% \cM_{\mec d}(A_1)(\mec e_{a_1,i})=\mec e'_{a_1,\sigma_{a_i}(i)}.
\phi(A_1)(\mec e_{a_1,i})=\mec e'_{a_1,\sigma_{a_1}(i)}.
$$
But the isomorphism $\cM_{W_i,\mec d}\isom\bec\cI_{\mec d}^{\oplus W_i}$ 
allows us to
write
$$
\cM_{\mec d}\isom \bigoplus_{a_1\in\integers,\ i\in [s]} 
% \cI_{(a_1,i))\le\mec d?}
\cI_{\mec d\ge (a_1,\pi_i(a_1))}
$$
in a way that $\mec e_{a_1,i}\in \cM_{\mec d}(A_1)$ 
corresponds to the standard basis vector in vector (with the same
indices) $\mec e_{a_1,i}$.
Since the same is true of $\cM'_{\mec d}$ and $\sigma_{a_1}$, that gives that
the desired isomorphism must have that for all $a_1\in\integers$
and $i\in [s]$
$$
% \cI_{(a_1,i)\le \mec d?}
\cI_{\mec d\ge (a_1,i)}
\isom
% \cI_{(a_1,\sigma_{a_1}i)\le\mec d?};
\cI_{\mec d\ge (a_1,\sigma_{a_1}(i))};
$$
similar to the argument in the proof of 
Lemma~\ref{le_zipper_matchings}, this holds automatically
for the $B_1,B_3,A_2$ values, and holds at the $B_2$ value iff
\begin{equation}
\label{eq_size_pi_versus_pi_prime}
\pi_i(a_1)\le d_2 \iff \pi'_{\sigma_{a_1}(i)}(a_1)\le d_2.
\end{equation}
Hence, for each $a_1\in\integers$, such a $\sigma_{a_1}$ exists
iff \eqref{eq_size_pi_versus_pi_prime} holds,
and if so for each $a_1\in\integers$ we can choose
any permutation $\sigma_{a_1}$ on $[s]$ that maps
$$
\{ r \ | \ \pi_r(a_1) \le d_2 \} 
\quad\mbox{to}\quad
 \{ r' \ | \ \pi'_{r'}(a_1) \le d_2 \} 
$$
(and therefore $\sigma_{a_1}$ also maps the same with $\le d_2$ replaced
everywhere with $>d_2$).
\end{proof}

\subsection{Proof of Theorem~\ref{th_glue_i_j_j_prime_not_virtual} and
Examples}
\label{su_proof_th_glue_i_j_j_prime_not_virtual}

In this section we will prove
Theorem~\ref{th_glue_i_j_j_prime_not_virtual}, which follows almost
immediately 
from the lemma below (which adds a third equivalent condition
to Lemma~\ref{le_zipper_matchings}).

\begin{lemma}\label{le_zipper_condition_on_fs_versus_Ws}
Let $W,W'$ be perfect matchings, and $\pi,\pi'$ their associated
bijections.  Then for any $\mec d\in\integers^2$ the following are equivalent:
\begin{enumerate}
\item
there exists an isomorphism of $k$-diagrams
$\phi\from\cM_{W,\mec d}\to\cM_{W',\mec d}$
such that $\phi(A_1)$ is the identity; and
\item
\begin{equation}
\label{eq_zipper_condition_from_W_one_to_W_two} 
\forall a\in\integers_{\le d_1}, \quad
\pi(a)\le d_2
\iff
\pi'(a)\le d_2;
\end{equation}
and
\item
setting $f=\fraks W$ and $f'=\fraks W'$ we have
\begin{equation}
\label{eq_zipper_condition_f_one_to_f_two} 
\forall a\in\integers_{\le d_1}, \quad
f(a,d_2) 
=
f'(a,d_2).
\end{equation}
\end{enumerate}
\end{lemma}
\begin{proof}
The equivalence of~(1) and~(2) is just 
Lemma~\ref{le_zipper_matchings}.

(2) $\implies$ (3): for any $a\in\integers$ we have
\begin{equation}\label{eq_diff_fraks_Ws_perfect_matching}
f(a,d_2)-f(a-1,d_2)
=
\sum_{a_2\le d_2} W(a,a_2) = 
\left\{ \begin{array}{ll}
1 & \mbox{if $\pi(a)\le d_2$, and} \\
0 & \mbox{otherwise,}
\end{array}\right.
\end{equation} 
and similarly 
% with $f',W',\pi'$ replacing $f,W,\pi$.
\begin{equation}\label{eq_diff_fraks_Ws_perfect_matching_prime}
f'(a,d_2)-f'(a-1,d_2)
=
\left\{ \begin{array}{ll}
1 & \mbox{if $\pi'(a)\le d_2$, and} \\
0 & \mbox{otherwise,}
\end{array}\right.
\end{equation} 
Now $f(a,d_2)=f'(a,d_2)=0$ for $a$ sufficiently
small, and hence
\eqref{eq_zipper_condition_f_one_to_f_two} holds for $a\le a'$
for some $a'$.
Assuming (2), 
\eqref{eq_diff_fraks_Ws_perfect_matching} and
\eqref{eq_diff_fraks_Ws_perfect_matching_prime} imply that
for all $a\le d_1$ we have
$$
f(a,d_2)-f(a-1,d_2)
=
f'(a,d_2)-f'(a-1,d_2)
$$
and therefore
$$
f(a,d_2) - f'(a,d_2) = 
f(a-1,d_2)-f'(a-1,d_2).
$$
Hence we can use
induction on $a$ from $a'+1$ to $d_1$ to infer that
\eqref{eq_zipper_condition_from_W_one_to_W_two} holds for all $a\le d_1$.

(3) $\implies$ (2):
\eqref{eq_zipper_condition_from_W_one_to_W_two} implies that
for any $a$ we have
$$
f(a,d_2)-f(a-1,d_2)
=
f'(a,d_2)-f'(a-1,d_2)
$$
and hence
\eqref{eq_diff_fraks_Ws_perfect_matching} and 
% the same for $f',W',\pi'$
\eqref{eq_diff_fraks_Ws_perfect_matching_prime}
imply that for 
each $a\le d_1$, $\pi(a)\le d_2$ iff $\pi'(a)\le d_2$.
\end{proof}

\begin{proof}[Proof of Theorem~\ref{th_glue_i_j_j_prime_not_virtual}]
The equation \eqref{eq_pasting_relation} implies
condition~(3) of 
Lemma~\ref{le_zipper_condition_on_fs_versus_Ws} with
$d_1=d_2=0$.
Hence we conclude condition~(1) of
Lemma~\ref{le_zipper_condition_on_fs_versus_Ws} in this case,
which is just
the assertion of Theorem~\ref{th_glue_i_j_j_prime_not_virtual}.
\end{proof}

We finish this subsection by showing how to generate non-trivial examples,
of Lemma~\ref{le_zipper_condition_on_fs_versus_Ws} and
Theorem~\ref{th_glue_i_j_j_prime_not_virtual}, 
and specify one such example explicitly;
this may serve to illustrate how this lemma and this theorem
work in practice.

\begin{example}
Let $f$ be as in Example~\ref{ex_weights_genus_one_examples}.
Then
$$
f_{1,2,\mec 0}(a_1,0) = f_{1,3,\mec 0}(a_1,0)
$$
for all $a_1\in\integers$.  Hence if $W,W'$ are the respective
weights of $f_{1,2,\mec 0},f_{1,3,\mec 0}$,
then $W\ne W'$, and, in more detail, $W,W'$ are both $4$-periodic, and
their associated bijections $\pi,\pi'$ satisfy
$$
\pi(0)=\pi'(0)=0,
\ 
\pi(1)=\pi'(1)=1,
$$
and
$$
\pi(-1)=\pi'(-2)=2,
\ 
\pi(-2)=\pi'(-1)=3.
$$
Hence $\pi'\ne\pi$,
but it is nonetheless true that
$$
\forall a_1\in\integers_{\le 0},
\quad\quad
\pi(a_1) \le 0
\iff
\pi'(a_1) \le 0
$$
(which moreover holds for all $a_1\in\integers$).
One can similarly generate examples of 
$f_{1,2,\mec d}$ and $f_{1,3,\mec d}$ for any $\mec d\in\integers^4$.
One can also generate examples as in
Example~\ref{ex_cycle_Baker_Norine} with $n\ge 5$.
\end{example}

\begin{example}
Let $G=K_n$ be the complete graph on $n$ vertices, and 
$f=1+r_{\rm BN}$ the Riemann function associated to the Baker-Norine
rank function on $G$.
Then Folinsbee and Friedman \cite{folinsbee_friedman_weights} show
that any two-variable restriction of $f$ has non-negative weight.
Hence one can generate further examples of
$f_{1,2,\mec d},f_{1,3,\mec d}$ for various $\mec d\in\integers^n$.
\end{example}

\subsection{Generalization of 
Lemma~\ref{le_zipper_condition_on_fs_versus_Ws}}

When $\cM_{f;i,j,\mec d}$ and
$\cM_{f;i,j',\mec d}$ are virtual $k$-diagrams,
we will need the following generalization
of Lemma~\ref{le_zipper_condition_on_fs_versus_Ws}.
(To prove Theorem~\ref{th_glue_i_j_j_prime},
we need only that~(4) implies~(1) below.)

\begin{lemma}\label{le_the_big_result_for_virtual_diagrams}
Let $W_1,\cdots,W_s$ and $W_1',\cdots,W_s'$ be two sequences of
perfect matchings $\integers^2\to\integers$.
Let $\mec d\in\integers^2$, and for each
$a_1\in\integers$ and $j\in[s]$, 
let $\mec e_{a_i,j}$ and $\mec e_{a_i,j}'$ denote, respectively, the
standard basis vector $\mec e_{a_1}\in k^{\integers}$ in, respectively
$\cM_{W_j,\mec d}(A_1)$ and $\cM_{W_j',\mec d}(A_1)$.
Let
$$
W = W_1 + \cdots + W_s, \quad
W' = W_1'+\cdots+ W_s' ,
$$
and
$$
\cM_{\mec d}=\cM_{W_1,\mec d}\oplus \cdots\oplus \cM_{W_s,\mec d}, 
\quad
\cM_{\mec d}'=\cM_{W_1',\mec d}\oplus \cdots\oplus \cM_{W_s',\mec d}.
$$
Then for any $\mec d\in\integers^2$, the following are equivalent:
\begin{enumerate}
\item 
there exists an isomorphism
$\phi\from\cM\to\cM'$ that for each $a_1\in\integers_{\le d_1}$,
$\phi(A_1)$
restricts to a bijection from $\mec e_{a_1,1},\ldots,\mec e_{a_1,s}$
to $\mec e'_{a_1,1},\ldots,\mec e'_{a_1,s}$
(and, moreover, in this case one can also take from the subset
of $\mec e_{a_1,r}$ with $\pi_{r}(a_1)\le d_2$
to those $\mec e'_{a_1,r}$ with $\pi'_{r}(a_1)\le d_2$ where $\pi_r,\pi'_r$
are the bijections associated to $W_r,W'_r$); 
\item
letting $\pi_r,\pi'_r$ for $r\in[s]$ denote
the bijections associated to $W_r,W'_r$,
$$
\forall a_1\in\integers_{\le d_1},\quad
\bigl| \{ r \ | \ \pi_r(a_1) \le d_2 \} \bigr|
=
\bigl| \{ r \ | \ \pi'_r(a_1) \le d_2 \} \bigr|;
$$
\item
$$
\forall a_1\in\integers_{\le d_1},\quad
\sum_{a_2\le d_2} W(a_1,a_2) = \sum_{a_2\le d_2} W'(a_1,a_2);
$$
\item
if $f=\fraks W$ and $f'=\fraks W'$, then
$$
\forall a_1\in\integers_{\le d_1},\quad
f(a_1,d_2) = f'(a_1,d_2).
$$
\end{enumerate}
Moreover, if the above conditions hold, 
then all maps $\phi(A_1)$
are determined as those that
for each $a_1\in\integers_{\le d_1}$
restricts to a bijection 
$$
\{ \mec e_{a_1,r} \ | \ \pi_r(a_1)\le d_2 \} 
\to
\{ \mec e'_{a_1,r} \ | \ \pi'_r(a_1)\le d_2 \} 
$$
and to a bijection
$$
\{ \mec e_{a_1,r} \ | \ \pi_r(a_1)\ge d_2+1 \} 
\to
\{ \mec e'_{a_1,r} \ | \ \pi'_r(a_1)\ge d_2+1 \}  .
$$
\end{lemma}
Another way to think of this lemma is to
recall that if $W_1,\ldots,W_s$ are perfect matchings, and
$W=W_1+\cdots+W_s$, then
$\cM_{W_1,\mec d}\oplus\cdots\oplus\cM_{W_s,\mec d}\isom
\bec\cI^W_{\mec d}$.  
Hence this lemma shows that 
$\bec\cI^W_{\mec d}\isom \bec\cI^{W'}_{\mec d}$.
\begin{proof}
(1) $\iff$ (2): this is Lemma~\ref{le_zipper_multiple_matchings}.

(2) $\iff$ (3): this is by definition: if $W$ is any perfect matching,
then $\pi$ is the unique bijection such that 
$W(a,\pi(a))=1$ for all $a\in\integers$.
Hence, for any $d_2$ and $a_1$ we have
$$
\sum_{a_2\le d_2} W(a_1,a_2) = 
\bigl| \{ r \ | \ \pi_r(a_1) \le d_2 \} \bigr| .
$$
Similarly
$$
\sum_{a_2\le d_2} W'(a_1,a_2) =
\bigl| \{ r \ | \ \pi'_r(a_1) \le d_2 \} \bigr|.
$$
So if for some $a_1$, the number of values of $r$ such that
$\pi_r(a_1)\le d_2$ is the number of $r'$ with 
$\pi_{r'}(a_1)\le d_2$, then for this particular value of $a_1$,
$$
\sum_{a_2\le d_2} W(a_1,a_2) 
=\sum_{a_2\le d_2} W'(a_1,a_2)
$$
Now we apply this fact to all $a_1\le d_1$.

(3) $\iff$ (4): similar to 
\eqref{eq_diff_fraks_Ws_perfect_matching}, we have
for all $a\in\integers$,
\begin{equation}\label{eq_diff_fraks_Ws_s_fold_matching}
f(a,d_2)-f(a-1,d_2)
=
\sum_{a_2\le d_2} W(a,a_2) = 
\bigl| \{ r \ | \ \pi_r(a) \le d_2 \} \bigr|,
\end{equation} 
and similarly
\begin{equation}\label{eq_diff_fraks_Ws_s_fold_matching_prime}
f'(a,d_2)-f'(a-1,d_2)
=
\sum_{a_2\le d_2} W'(a,a_2) = 
\bigl| \{ r \ | \ \pi'_r(a) \le d_2 \} \bigr|.
\end{equation} 

The claim after (1)--(4) about $\phi(A_1)$ follows from 
\eqref{eq_size_pi_versus_pi_prime} (and the discussion below it).
(This claim is not needed in what follows, but serves to 
illustrate the way that $\phi(A_1)$---and
therefore all of $\phi$---is constructed.)
\end{proof}

\subsection{Proof of Theorem~\ref{th_glue_i_j_j_prime}}
\label{su_proof_th_glue_i_j_j_prime}

To prove Theorem~\ref{th_glue_i_j_j_prime},
we need only the part of
Lemma~\ref{le_the_big_result_for_virtual_diagrams} 
that asserts condition~(4) there implies condition~(1).
% which we state as a corollary.

% \begin{corollary}\label{co_the_only_thing_we_really_need}
% Let $W_1,\cdots,W_s$ and $W_1',\cdots,W_s'$ be two sequences of
% perfect matchings $\integers^2\to\integers$,
% and let $f_i=\fraks W_i$, $f_i'=\fraks W'_i$.
% Say that for some $d_2\in\integers$ we have
% $$
% \forall a_1\in\integers,\quad
% f(a_1,d_2) = f'(a_1,d_2).
% $$
% Then for all $d_1\in\integers$, the $k$-diagrams
% $$
% \cM_{\mec d}=\cM_{W_1,\mec d}\oplus \cdots\oplus \cM_{W_s,\mec d}, 
% \quad
% \cM_{\mec d}'=\cM_{W_1',\mec d}\oplus \cdots\oplus \cM_{W_s',\mec d},
% $$
% (with $\mec d=(d_1,d_2)$) are isomorphic.
% \end{corollary}

% Notice that in the application below, there is a fixed Riemann function
% of $n$ variables, and $f$ represents a restriction to the components
% $i$ and $j$, and the $f'$ to components $i$ and $j'$ with $i,j,j'$
% distinct.  So, depending on how we define $f,f'$, it may seem more 
% natural to write the condition in the corollary as
% $$
% \forall a_1\in\integers,\quad
% f(a_1,d_2) = f'(a_1,d_3),
% $$
% where $d_2,d_3$ are not necessarily equal.  In this case, one could
% say that for any $d_1$, we have an isomorphism between
% $$
% \cM_{(d_1,d_2)} \quad\mbox{and}\quad
% \cM'_{(d_1,d_3)},
% $$
% simply by introducing $f''(a_1,a_2)$ to be defined as
% $f'(a_1,a_2-d_2+d_3)$ and applying the corollary with $f''$ replacing
% $f'$, so that
% $f(a_1,d_2)=f''(a_1,d_2)=f'(a_1,d_3)$.
% This is just a bit more awkward to state;
% the passing from $f'$ to $f''$ just translates the weight
% from $\frakm f'$ to $\frakm f''$ in the second variable.

\begin{proof}[Proof of Theorem~\ref{th_glue_i_j_j_prime}]
Since $f_{i,j,\mec d}$ is a Riemann function $\integers^2\to\integers$,
we can write
$$
W=\frakm f_{i,j,\mec d} = W_1+\cdots +W_s - \tilde W_1 -\cdots-\tilde W_{s-1}
$$
where the $W_i$ and $\tilde W_i$ are perfect matchings, and we may similarly
write 
$$
W' = \frakm f_{i,j',\mec d} = W'_1+\cdots +W'_{s'} 
- \tilde W'_1 -\cdots-\tilde W'_{s'-1}
$$
To show that $\cM_{W,\mec 0}\isom \cM_{W',\mec 0}$ it suffices to show
that
\begin{equation}\label{eq_what_we_need_for_isom_of_virt_k_diagrams} 
\left( \bigoplus_{i=1}^s \cM_{W_i,\mec 0}  \right) \oplus
\left( \bigoplus_{i=1}^{s'-1} \cM_{\tilde W'_i,\mec 0}  \right) 
\isom
\left( \bigoplus_{i=1}^{s-1} \cM_{\tilde W_i,\mec 0}  \right) \oplus
\left( \bigoplus_{i=1}^{s'} \cM_{W'_i,\mec 0}  \right) 
\end{equation} 
Let 
$$
f_1 = \fraks( W_1+\cdots +W_s ), \quad
f_2 = \fraks( \tilde W_1+\cdots +\tilde W_{s-1} )
$$
so that $f_{i,j,\mec d}=f_1-f_2$ and similarly
$$
f'_1 = \fraks( W'_1+\cdots +W'_{s'} ), \quad
f'_2 = \fraks( \tilde W'_1+\cdots +\tilde W'_{s'-1} )
$$
and so $f_{i,j',\mec d}=f'_1-f'_2$.
By definition
\begin{equation}\label{eq_zipper_condition_for_f_i_j_j_prime}
f_{i,j,\mec d}(a_1,0) = f(\mec d + \mec e_i a_1)
=
f_{i,j',\mec d}(a_1,0)
\end{equation} 
for all $a_1\in\integers$.  It follows for all $a_1\in\integers$ we have
$$
(f_1-f_2)(a_1,0) = (f_1'-f_2')(a_1,0),
$$
and therefore 
$$
\forall a_1\in\integers,\quad
(f_1+f_2')(a_1,0) = (f_2+f_1')(a_1,0).
$$
Hence applying Lemma~\ref{le_the_big_result_for_virtual_diagrams} with 
$f,f'$ there replaced with $f_1+f_2'$ and $f_2+f_1'$ respectively
(and $W_1,\ldots,W_s$ there replaced with
$W_1,\ldots,W_s,\tilde W'_1,\ldots,\tilde W'_{s'-1}$ here, 
and $W'_1,\ldots,W'_s$ there with
$\tilde W_1,\ldots,\tilde W_{s-1},W'_1,\ldots,W'_{s'}$ here),
we have that condition~(4) of this lemma holds, and therefore
condition~(1) holds.
Therefore \eqref{eq_what_we_need_for_isom_of_virt_k_diagrams} holds,
and therefore
$$
\cM_{W,\mec 0}\isom \cM_{W',\mec 0}
$$
as virtual $k$-diagrams.
\end{proof}

\section{The First Duality Theorems}
\label{se_duality_first}

In this section we show that 
the $k$-diagram $\underline k_{/B_1,B_2}$
is a ``dualizing'' $k$-diagram, in that for any 
$k$-diagram $\cF$ we have that there is an isomorphism
$$
H^1(\cF)^* \to
\Hom(\cF,\underline k_{/B_1,B_2})
$$
which is ``natural'' or ``functorial'' in $\cF$.

We then prove that any perfect matching, $W$, and any
$\mec K,\mec L\in\integers^2$ with $\mec L=\mec K+\mec 1$,
for any $\mec d$ there is a isomorphism
\begin{equation}\label{eq_duality_first_summary}
H^1(\cM_{W,\mec d})^* \to H^0(\cM_{W^*_{\mec L},\mec K-\mec d}).
\end{equation} 
Since these are finite dimensional vector spaces, by
replacing $W$ and $\mec d$ with, respectively,
$W^*_\mec L$ and $\mec K-\mec d$, we moreover get isomorphisms
\begin{equation}\label{eq_duality_first_summary_better}
H^i(\cM_{W,\mec d})^* \to H^{1-i}(\cM_{W^*_{\mec L},\mec K-\mec d})
\quad\mbox{for $i=0,1$.}
\end{equation} 
We use this to infer that if $W$ is the weight of any
Riemann function $\integers^2\to\integers$, there are isomorphism
of virtual $k$-vector spaces
\begin{equation}\label{eq_duality_first_virtual}
H^i\bigl([\cM_{W,\mec d}]\bigr)^* \to 
H^{1-i}\bigl([\cM_{W^*_{\mec L},\mec K-\mec d}]\bigr)
\quad\mbox{for $i=0,1$,}
\end{equation} 
where $[\cM_{W,\mec d}]$ is the equivalence class of
virtual $k$-diagrams, and the dual of a virtual $k$-vector space
is appropriately defined.
We use this to infer that
for any
Riemann function $f\from\integers^n\to\integers$
and any $\mec d,\mec K\in\integers^n$ there is an isomorphism
\begin{equation}\label{eq_duality_first_general}
H^i\bigl([\cM_{f {\rm \; at \;}\mec d}]\bigr)^* \to
H^{1-i}\bigl([\cM_{f^\wedge_\mec K{\rm \; at \;}\mec K-\mec d}] \bigr)
\quad\mbox{for $i=0,1$.}
\end{equation} 

\subsection{Representing $H^1(\cF)^*$}

\begin{theorem}\label{th_representability_of_H_one_dual}
For any $k$-diagram $\cF$ there is a {\em natural} isomorphism
\begin{equation}\label{eq_H_one_star_representable}
H^1(\cF)^* \to 
\Hom(\cF,\underline k_{/B_1,B_2}),
\end{equation} 
where ``natural'' means ``functorial'' in the sense that
if $\mu\from\cF\to\cG$ is any morphism, then the natural
map $H^1(\cG)^*\to H^1(\cF)^*$ obtained by dualizing the
map $\mu$ induces from $H^1(\cF)\to H^1(\cG)$ is the same
as the map
$$
\Hom(\cG,\underline k_{/B_1,B_2}) \to
\Hom(\cF,\underline k_{/B_1,B_2}) .
$$
\end{theorem}
In modern parlance, the functor
$\cF \to H^1(\cF)^*$ 
is represented
by the $k$-diagram $\underline k_{/B_1,B_2}$.
In Section~\ref{se_duality_second}
we will give a conceptually simple proof of this theorem
using standard techniques from homological algebra.
Here we content ourselves to prove this theorem 
directly, which is straightforward,
although a bit tedious.

For the proof below, 
note that if
$\cL\from U\to V$ is any linear map of (possibly infinite dimensional)
$k$-vector spaces, then if $V^*$ denotes the $k$-dual space of $V$,
i.e., the vector space of maps $V\to k$,
then the usual dual map $\cL^*\from V^*\to U^*$ is given by
\begin{equation}\label{eq_dual_map_explicitly}
\forall w\from V\to k, \quad
\cL^*(w) = w\circ \cL;
\end{equation} 
we claim that thre is an isomorphism
\begin{equation}\label{eq_cokernel_to_kernel_map}
\bigl( \coker(\cL) \bigr)^* \to \ker\bigl(\cL^*\bigr)
\end{equation} 
constructed as follows:
any $w\in \bigl( \coker(\cL) \bigr)^*$ is a map
$w$ from $\coker(\cL)=(V/{\rm Image}(\cL))$ to $k$,
and so the quotient map $V\to V/{\rm Image}(\cL)$ followed by $w$
gives a map $\tilde w\from V\to k$ 
(which takes ${\rm Image}(\cL)$ to $0$);
it easily follows that $\tilde w\circ\cL$,
which is an element of $U^*$, lies in the kernel of $\cL^*$.
This gives the map \eqref{eq_cokernel_to_kernel_map}.
Conversely, any $v\in \ker(\cL^*)$ 
is a map $V\to k$ taking ${\rm Image}(\cL)$ to $0$, and hence determines
a $w\from\coker(\cL)\to k$ such that
$v=\tilde w\circ\cL$;
hence \eqref{eq_cokernel_to_kernel_map} is an isomorphism.

We warn the reader that in 
Subsection~\ref{su_linear_algebra_subtlety} we will
see that if $\cL\from U\to V$ is any linear map, there is a natural
\begin{equation}\label{eq_cokernel_of_cL_star_to_star_of_kernel}
\coker(\cL^*) \to \bigl( \ker(\cL) \bigr)^*,
\end{equation} 
however, in contrast to \eqref{eq_cokernel_to_kernel_map},
(to the best of our knowledge)
one needs to assume the axiom
of choice (or Zorn's lemma) to ensure that this is an isomorphism.
Hence it is remarkable that \eqref{eq_cokernel_to_kernel_map} is an
isomorphism without assuming the axiom of choice.

\begin{proof}[Proof of Theorem~\ref{th_representability_of_H_one_dual}]
Consider a morphism $\phi\from\cF\to\underline k_{/B_1,B_2}$;
see Figure~\ref{fi_morphism_to_canonical_diagram}.
\begin{figure}
\begin{tikzpicture}[scale=0.50,font=\small]
\node (B1) at (0,4) {$\cF(B_1)$};
\node (B2) at (0,-4) {$\cF(B_2)$};
\node (B3) at (0,0) {$\cF(B_3)$};
\node (A1) at (8,2) {$\cF(A_1)$};
\node (A2) at (8,-2) {$\cF(A_2)$};
\draw [->] (B1) -- (A1) node [midway,above] {$\cF(\rho_{1,1})$} ;
\draw [->] (B2) -- (A2) node [midway,below] {$\cF(\rho_{2,2})$} ;
\draw [->] (B3) -- (A1) node [midway,above] {$\cF(\rho_{3,1})$} ;
\draw [->] (B3) -- (A2) node [midway,below] {$\cF(\rho_{3,2})$} ;
% \node (BB1) at (15,4) {$\cG(B_1)$};
% \node (BB2) at (15,-4) {$\cG(B_2)$};
% \node (BB3) at (15,0) {$\cG(B_3)$};
% \node (AA1) at (23,2) {$\cG(A_1)$};
% \node (AA2) at (23,-2) {$\cG(A_2)$};
\node (BB1) at (15,4) {$0$};
\node (BB2) at (15,-4) {$0$};
\node (BB3) at (15,0) {$k$};
\node (AA1) at (23,2) {$k$};
\node (AA2) at (23,-2) {$k$};
% \draw [->] (BB1) -- (AA1) node [midway,above] {$\cG(\rho_{1,1})$} ;
% \draw [->] (BB2) -- (AA2) node [midway,below] {$\cG(\rho_{2,2})$} ;
% \draw [->] (BB3) -- (AA1) node [midway,above] {$\cG(\rho_{3,1})$} ;
% \draw [->] (BB3) -- (AA2) node [midway,below] {$\cG(\rho_{3,2})$} ;
\draw [->] (BB1) -- (AA1) node [midway,above] {$0$} ;
\draw [->] (BB2) -- (AA2) node [midway,below] {$0$} ;
\draw [->] (BB3) -- (AA1) node [midway,above] {${\rm Identity}$} ;
\draw [->] (BB3) -- (AA2) node [midway,below] {${\rm Identity}$} ;
\draw [->,ultra thick] (B1) -- (BB1) node [pos=0.75,above] {$\phi(B_1)$};
\draw [->,ultra thick] (B2) -- (BB2) node [pos=0.75,above] {$\phi(B_2)$};
\draw [->,ultra thick] (B3) -- (BB3) node [pos=0.75,above] {$\phi(B_3)$};
\draw [->,ultra thick] (A1) -- (AA1) node [pos=0.2,above] {$\phi(A_1)$};
\draw [->,ultra thick] (A2) -- (AA2) node [pos=0.2,above] {$\phi(A_2)$};
%
% \node (F) at (4,-6) {\Huge$\cF$};
% \node (G) at (19,-6) {\Huge$\underline k_{/B_1,B_2}$};
% \draw [->,ultra thick] (F) -- (G) node [midway,above] {\Large$\phi$};
\end{tikzpicture}
\caption{A morphism $\phi\from\cF\to\underline k_{/B_1,B_2}$}
\label{fi_morphism_to_canonical_diagram}
\end{figure}
Hence we have
$$
\bigl( \phi(A_1),-\phi(A_2) \bigr) \in
\bigl( \cF(A_1) \bigr)^*
\oplus
\bigl( \cF(A_2) \bigr)^* 
$$
(when we view $\phi(A_i)\from\cF(A_i)\to k$ as an element of
the dual space of $\cF(A_i)$);
let us prove that, moreover,
\begin{equation}\label{eq_phi_As_in_the_kernel}
\bigl( \phi(A_1),-\phi(A_2) \bigr) \in 
\ker\Bigl( \bigl(\cF(\partial)\bigr)^*\Bigr).
\end{equation} 
To do so, in view of \eqref{eq_formula_for_cF_partial}, we have
$(\cF(\partial))^*$ is the map
$$
\bigl( \cF(\partial) \bigr)^* 
\from 
\bigl( \cF(A_1) \bigr)^* 
\oplus
\bigl( \cF(A_2) \bigr)^* 
\to
\bigl( \cF(B_1) \bigr)^* 
\oplus
\bigl( \cF(B_2) \bigr)^* 
\oplus
\bigl( \cF(B_3) \bigr)^*
$$
given as the map taking $(w_1,w_2)$ with $w_i\in(\cF(A_i))^*$ as
follows:
$$
(w_1,w_2)\mapsto 
\Bigl(
\ \bigl( \cF(\rho_{1,1})\bigr)^*(w_1) \ ,
\ \bigl( \cF(\rho_{2,2})\bigr)^*(w_2) \ ,
\ -\bigl( \cF(\rho_{3,1}) \bigr)^*(w_1)
-\bigl( \cF(\rho_{3,2}) \bigr)^*(w_2)
\ \Bigr),
$$
which, in view of \eqref{eq_dual_map_explicitly} we may write more simply
as the map
\begin{equation}\label{eq_dual_map_cF_partial}
(w_1,w_2)\mapsto 
\bigl(
\ w_1\cF(\rho_{1,1}) \ ,
\ w_2\cF(\rho_{2,2}) \ ,
\ -w_1 \cF(\rho_{3,1}) 
- w_2\cF(\rho_{3,2})
\ \Bigr)
\end{equation} 
where for brevity we have omitted the composition symbol $\circ$.

So set $w_1=\phi(A_1)$ and $w_2=-\phi(A_2)$.
Since $\underline k_{/B_1,B_2}$ has value $0$ at $B_1$, we have
$\phi(A_1)\cF(\rho_{1,1})$ must be the zero map.  Arguing similarly 
for $B_2$, we have
\begin{align}
\label{eq_first_comp_dual_partial}
 w_1\cF(\rho_{1,1}) & = 0,  \\
\label{eq_second_comp_dual_partial}
-w_2\cF(\rho_{2,2}) & = 0 .
\end{align} 
Similarly, since for $i=1,2$ the map from the $B_3$ value of 
$\underline k_{/B_1,B_2}$ to the $A_i$ is the identity
map (on $k$), we have
$$
\phi(B_3) = \phi(A_i)\cF(\phi_{3,i}),
$$
and hence
$$
w_1 \cF(\phi_{3,1}) 
=
-w_2 \cF(\phi_{3,2}) 
=
\phi(B_3) ,
$$
and hence
\begin{equation}\label{eq_third_comp_dual_partial}
w_1 \cF(\phi_{3,1}) 
+w_2 \cF(\phi_{3,2}) 
=
0.
\end{equation} 
In view of
\eqref{eq_first_comp_dual_partial}--\eqref{eq_third_comp_dual_partial}
we have
$(w_1,w_2)=(\phi(A_1),-\phi(A_2))$
is taken to $(0,0,0)$ under the map
\eqref{eq_dual_map_cF_partial}.  

Next we claim that, conversely, if $(w_1,w_2)\in\ker(((\cF(\partial))^*)$,
then there exists some $\phi\in\Hom(\cF,\underline k_{/B_1,B_2})$ such that
$\phi(A_1)=w_1$ and $\phi(A_2)=-w_2$.  Namely, this
determines $\phi$ at $A_1,A_2$; this forces the values of $\phi$ at
the $B_i$, namely we set $\phi(B_i)=0$ for $i=1,2$ and we set
$$
\phi(B_3) = w_1 \cF(\rho_{3,1}) = \phi(A_1)\cF(\rho_{3,1})
$$
(so that $\phi$ intertwines with the restrictions from $B_3$ to $A_1$
of $\cF$ and $\underline k_{/B_1,B_2}$).
Now we verify that $\phi$ is actually a morphism: for example,
since $(w_1,w_2)\in\ker(((\cF(\partial))^*)$, 
in view of \eqref{eq_dual_map_cF_partial} we have
$$
w_1  \cF(\rho_{3,1}) + w_2 \cF(\rho_{3,2}) = 0
$$
and it follows that
$$
\phi(B_3) = -w_2\cF(\rho_{3,2}) = \phi(A_2)\cF(\rho_{3,2}),
$$
and therefore $\phi$ intertwines with the restrictions from $B_3$ to $A_2$.
Similarly $\phi$ intertwines with the restrictions from $B_i$ to $A_i$
for $i=1,2$, i.e.,
$\phi(A_i)\cF(\rho_{i,i})=0$, in view of \eqref{eq_dual_map_cF_partial}.

It follows that the map
$$
\phi \mapsto 
\bigl( \phi(A_1), -\phi(A_2) \bigr)
$$
is gives an isomorphism 
\eqref{eq_H_one_star_and_zeroth_cohomology_dual}.

To check the desired functoriality, say that 
$\mu\from\cF\to\cG$ is any morphism.
The map
\eqref{eq_H_one_star_and_zeroth_cohomology_dual} with $\cG$ replacing
$\cF$ is given by
associating to $\phi\in\Hom(\cG,k_{/B_1,B_2})$ the element
\begin{equation}\label{eq_element_of_G_partial_star}
\bigl( \phi(A_1),-\phi(A_2) \bigr) \in 
\ker\Bigl( \bigl(\cG(\partial)\bigr)^*\Bigr).
\end{equation} 
Then to $\phi$ we associate the element
$$
\phi\circ\mu \in \Hom(\cF,k_{/B_1,B_2}),
$$
and therefore to $\phi\circ\mu$ we associate the element
\begin{equation}\label{eq_mu_acts_on_element_of_G_partial_star}
\bigl( \phi\circ\mu(A_1),-\phi\circ\mu(A_2) \bigr) \in 
\ker\Bigl( \bigl(\cF(\partial)\bigr)^*\Bigr).
\end{equation} 
But to an element of $H^1(\cG)$, namely an element
\eqref{eq_element_of_G_partial_star}, the action of $\mu$
taking $H^1(\cG)$ to $H^1(\cF)$ is precisely $\mu$ applied to
each element, which again gives
\eqref{eq_mu_acts_on_element_of_G_partial_star}.
Hence the isomorphism 
\eqref{eq_H_one_star_and_zeroth_cohomology_dual} is functorial
(or natural).
\end{proof}

\subsection{A Duality Theorem for Perfect Matchings}

\begin{theorem}\label{th_first_duality_cM}
Let $W\from\integers^2\to\integers$ be a perfect matching.  Then
for any $\mec K\in\integers^2$ and $\mec L=\mec K+\mec 1$,
for any $\mec d$
there is an isomorphism
\begin{equation}\label{eq_isomorphism_for_perfect_matchings}
\Hom\bigl(\cM_{W,\mec d},\underline k_{/B_1,B_2} \bigr) \to
\Hom\bigl(\underline k,\cM_{W^*_{\mec L},\mec K-\mec d}\bigr)
\end{equation} 
induced by decomposing the above diagrams into indicator diagrams
and taking the isomorphism
\begin{equation}\label{eq_identify_two_hom_sets}
\Hom(\underline k_{/B_1,B_2},\underline k_{/B_1,B_2})
\to
\Hom(\underline k,\underline k)
\end{equation} 
which takes the identity morphism of $k_{/B_1,B_2}$ to the identity
morphism of $\underline k$
(both $\Hom$ sets are isomorphic to $k$, in view 
of Example~\ref{ex_morphisms_all_basic_four}).
Moreover, \eqref{eq_identify_two_hom_sets} gives us an isomorphism
\begin{equation}\label{eq_H_one_star_and_zeroth_cohomology_dual}
H^1(\cM_{W,\mec d})^* \to
\Hom(\underline k,\cM_{W^*_{\mec L},\mec K-\mec d}) ,
\end{equation} 
which gives us isomorphisms
\begin{equation}\label{eq_duality_result_perfect_matchings}
H^i(\cM_{W,\mec d})^* \to
H^{1-i}(\cM_{W^*_{\mec L},\mec K-\mec d})
\quad\mbox{for $i=0,1.$}
\end{equation} 
\end{theorem}
Before giving the proof, let us make two remarks regarding this theorem.

We remark that if $f=\fraks W$ in the above theorem, then
Theorem~\ref{th_perfect_matching} and
\eqref{eq_generalized_riemann_roch} imply that
$$
b^1(\cM_{W,\mec d}) = f^\wedge_\mec K(\mec K-\mec d)
= b^0(\cM_{W^*_{\mec L},\mec K-\mec d}),
$$
applying this equation with $W,\mec d$ respectively replaced with
$W^*_\mec L$ and $\mec K-\mec d$, we therefore get
\begin{equation}\label{eq_betti_number_duality_perfect_matchings}
b^i(\cM_{W,\mec d}) = b^{1-i}(\cM_{W^*_{\mec L},\mec K-\mec d})
\quad\mbox{for $i=0,1$.}
\end{equation} 
Hence \eqref{eq_duality_result_perfect_matchings} strengthens
this formula, by giving the isomorphism of vector spaces
\eqref{eq_duality_result_perfect_matchings} that upon taking dimensions
implies \eqref{eq_betti_number_duality_perfect_matchings}.

% We remark that both $\Hom$ sets in
% \eqref{eq_identify_two_hom_sets}
% are one-dimensional, so it may be unclear why we involve this
% particular identification (other than it seems natural to
% have the two identity maps correspond); likely this identification
% is related to that suggested by the {\em Serre functor}, see 
% \eqref{eq_Serre_functor_on_Id_underline_k} and our discussion
% of the Serre functor in 
% Subsection~\ref{su_Serre_functor}.

\begin{proof}
We remark that for any $\mec a,\mec d\in\integers^2$,
% $\cI_{\mec a\ge \mec d?}$ is one of our four basic diagrams, and
$\indicateDgeA$ is one of our four basic diagrams, and
equals $\underline k_{/B_1,B_2}$ iff $\mec a\ge \mec d+\mec 1$
(recall Definition~\ref{de_indicator_diagrams}); hence
$$
\Hom(\indicateDgeA ,\underline k_{/B_1,B_2})
$$
is $0$ unless $\mec a\ge \mec d+\mec 1$, in which case it equals
$$
\Hom(\underline k_{/B_1,B_2}, \underline k_{/B_1,B_2}).
$$
(recall Example~\ref{ex_morphisms_of_basic_to_dualizing}).
Since
$$
\cM_{W,\mec d} = \bigoplus_{W(\mec a)=1} \indicateDgeA ,
$$
we have
$$
\Hom(\cM_{W,\mec d},\underline k_{/B_1,B_2}) \isom
\prod_{W(\mec a)=1} \Hom( \indicateDgeA ,\underline k_{/B_1,B_2})
=
\bigoplus_{W(\mec a)=1,\ \mec a\ge\mec d+\mec 1}
\Hom(\underline k_{/B_1,B_2},\underline k_{/B_1,B_2}).
$$

On the other hand, we similarly have
$$
\Hom(\underline k, \indicateDgeA )
$$
is $0$ unless $\mec a\le\mec d$, in which case it equals
$\Hom(\underline k,\underline k)$.  Hence
\begin{align*}
\Hom(\underline k,\cM_{W^*_\mec L,\mec K-\mec d})
=&
% Note: this next line should be direct sum, since Hom(k,dir sum of blah)
% = dir sum Hom(k,blah), since a single k is taken to an element of the
% dir sum
% \prod_{W^*_{\mec L}(\mec a)=1,\ \mec a\le\mec K-\mec d} 
\bigoplus_{W^*_{\mec L}(\mec a)=1,\ \mec a\le\mec K-\mec d} 
\Hom(\underline k,\underline k) \\
=&
\bigoplus_{W(\mec L-\mec a)=1,\ \mec a\le\mec K-\mec d}
\Hom(\underline k,\underline k)
\end{align*}
which, upon substituting $\mec a'=\mec L-\mec a$ in the sum,
$$
=\bigoplus_{W(\mec a')=1,\ \mec L-\mec a'\le\mec K-\mec d}
\Hom(\underline k,\underline k)
=\bigoplus_{W(\mec a')=1,\ \mec a'\ge\mec 1+\mec d}
\Hom(\underline k,\underline k) .
$$

Hence in \eqref{eq_isomorphism_for_perfect_matchings}, the 
left-hand-side has one copy of
$\Hom(\underline k_{/B_1,B_2},\underline k_{/B_1,B_2})$ for
each $\mec a$ with $W(\mec a)=1$ and $\mec a\ge \mec d+\mec 1$,
and the right-hand-side one copy of
$\Hom(\underline k,\underline k)$ for each such $\mec a$.
Hence we have an isomorphism
\eqref{eq_isomorphism_for_perfect_matchings}.

Composing the isomorphism in
Theorem~\ref{th_representability_of_H_one_dual} with 
\eqref{eq_isomorphism_for_perfect_matchings} gives us an
isomorphism
\eqref{eq_H_one_star_and_zeroth_cohomology_dual},
which also
proves the $i=1$ case of
\eqref{eq_duality_result_perfect_matchings}.
Since these are finite dimensional vector spaces, this also gives
an isomorphism
$$
H^0(\cM_{W^*_{\mec L},\mec K-\mec d})^* \to
H^1(\cM_{W,\mec d}) .
$$
If we apply this isomorphism when we replace all occurrences of
$W,\mec d$ with, respectively,
$W^*_{\mec L},\mec K-\mec d$ (so that $W^*_{\mec L}$ is therefore 
replaced with $W$, and $\mec K-\mec d$ with $\mec d$), we have
an isomorphism
$$
H^0(\cM_{W,\mec d})^* \to
H^1(\cM_{W^*_{\mec L},\mec K-\mec d}).
$$
This proves the $i=0$ case of
\eqref{eq_duality_result_perfect_matchings}.
Hence both the $i=0$ and $i=1$ cases of
\eqref{eq_duality_result_perfect_matchings} hold.
\end{proof}

\subsection{The Dual of a Virtual $k$-Vector Space}

In this subsection we define the dual of a virtual $k$-vector space
and make some brief remarks about this definition.
In Appendix~\ref{ap_dual_virtual_vec_sp} we discuss some more 
foundational aspects about virtual vector spaces that partially
justify our definition.

In Theorem~\ref{th_first_duality_cM}, we have an isomorphism
from the dual of $H^1(\cM_{W,\mec d})$ to
$H^0(\cM_{W_\mec L^*,\mec K-\mec d})$.
In our more general duality theorems 
(namely Theorem~\ref{th_duality_virtual_two_dimensional}
and Theorem~\ref{th_n_variate_duality} below),
$H^1(\cM_{W,\mec d})$ and 
will be 
a virtual $k$-vector space
$(V_1,V_2)$, $H^0(\cM_{W,\mec d})$ a virtual vector space $(V_3,V_4)$;
Theorem~\ref{th_first_duality_cM} then provides an isomorphism
$V_1^*\to V_3$ and $V_2^*\to V_4$.  For this reason, we will be pretty much
forced to make the following definition.

\begin{definition}\label{de_dual_virtual_vector_space}
Let $(V_1,V_2)$ be a virtual $k$-vector space (which, by convention, means
$V_1,V_2$ are finite dimensional).
We define the {\em dual} of $(V_1,V_2)$ to be the virtual $k$-vector space
$(V_1^*,V_2^*)$.
\end{definition}

\begin{proposition}
Let $(V_1,V_2)\sim (V_3,V_4)$ be equivalent (finite dimensional)
virtual $k$-vector spaces.
Then any equivalence, i.e., any isomorphism 
$$
\mu\from V_0\oplus V_1\oplus V_4 \to V_0\oplus V_2 \oplus V_3
$$
for some $V_0$, gives rise to an equivalence
$(V_1,V_2)^*\sim (V_3,V_4)^*$.
\end{proposition}
\begin{proof}
Since $\mu$ is an isomorphism, the dual map of $\mu$,
$$
\mu^*\from (V_0\oplus V_2 \oplus V_3)^* \to
(V_0\oplus V_1\oplus V_4)^*
$$
yields an isomorphism
$$
V_0^*\oplus V_2^*\oplus V_3^* \to V_0^*\oplus V_1^* \oplus V_4^* ;
$$
since all the $V_i$ are finite dimensional, so are all the $V_i^*$.
\end{proof}

We also remark that in
Theorem~\ref{th_duality_virtual_two_dimensional} 
one could avoid references to the dual space of a virtual vector space
provided that one has a good notion of ``pairing'' and ``perfect pairing''
of virtual vector spaces.
See Appendix~\ref{ap_dual_virtual_vec_sp} for further remarks
about pairings, as well as other remarks that motivate
Definition~\ref{de_dual_virtual_vector_space} as possibly
(i.e., in future work) fitting into a broader
notion of ``morphisms'' of virtual vector spaces,
rather than just being a definition of necessity.

\subsection{Duality for Weights of Riemann Functions $\integers^2\to\integers$}
\label{su_duality_general_riemann_functions}

\begin{theorem}\label{th_duality_virtual_two_dimensional}
Let $W$ be the weight of an arbitrary Riemann function 
$\integers^2\to\integers$.  
Then for any $\mec d,\mec K,\mec L\in\integers^2$ there is an
isomorphism of virtual $k$-vector spaces
\begin{equation}\label{eq_duality_two_dim_weight_functions}
H^i([\cM_{W,\mec d}])^* \to
H^{1-i}([\cM_{W^*_{\mec L},\mec K-\mec d}])
\quad\mbox{for $i=0,1$.}
\end{equation} 
\end{theorem}
The proof of this theorem is mostly a matter of unwinding
the various definitions
and applying Theorem~\ref{th_first_duality_cM}
\begin{proof}
Write $W$ as a difference of sums of
perfect matchings
\begin{equation}\label{eq_W_as_difference_of_perf_matchings_co_virt}
W = (W_1 + \cdots + W_s) - (\tilde W_1+\cdots+\tilde W_{s-1}),
\end{equation} 
whereupon $H^i([\cM_{W,\mec d}])$ refers to the kernel and cokernel
of the virtual Fredholm map
$$
\bigl( \cM_{W_1,\mec d}(\partial)\oplus\cdots\oplus 
\cM_{W_s,\mec d}(\partial) \bigr)
\ominus
\bigl( \cM_{\tilde W_1,\mec d}(\partial)\oplus\cdots\oplus 
\cM_{\tilde W_{s-1},\mec d}(\partial) \bigr) .
$$
In view of
\eqref{eq_W_as_difference_of_perf_matchings_co_virt} we have
$$
W^*_{\mec L} = 
\bigl( (W_1)^*_{\mec L} + \cdots + (W_s)^*_{\mec L}) - 
\bigl( (\tilde W_1) ^*_\mec L+\cdots+\tilde (W_{s-1})^*_\mec L \bigr),
$$
and for any $\mec d\in\integers^2$ we have that
$H^i(\cM_{W^*_{\mec L},\mec K-\mec d})$ refers to the equivalence class
of the kernel and cokernel of
$$
\bigl( \cM_{(W_1)^*_{\mec L},\mec K-\mec d}(\partial)\oplus\cdots\oplus 
\cM_{(W_s)^*_{\mec L},\mec K-\mec d}(\partial) \bigr)
\ominus
\bigl( \cM_{(\tilde W_1),\mec K-\mec d}(\partial)\oplus\cdots\oplus 
\cM_{(\tilde W_{s-1})^*_{\mec L},\mec K-\mec d}(\partial) \bigr) .
$$
In view of \eqref{eq_duality_result_perfect_matchings} we have
for each $\mec d\in\integers^2$,
$$
H^1(\cM_{W_i,\mec d})^* \to
H^0(\cM_{(W_i)^*_{\mec L},\mec K-\mec d})
$$
for all $i$, i.e., an isomorphism
$$
\Bigl( \coker\bigl( \cM_{W_i,\mec d}(\partial) \bigr) \Bigr) ^* 
\to
\ker\bigl( \cM_{(W_i)^*_{\mec L},\mec K-\mec d}(\partial) \bigr) ,
$$
and similarly with $\widetilde W_i$ replacing $W_i$.
Hence, in view of 
Definition~\ref{de_dual_virtual_vector_space}, we get an
isomorphism of virtual vector spaces
$$
H^1(\cM_{W,\mec d})^* \to 
H^0(\cM_{W^*_{\mec L},\mec K-\mec d}) .
$$
This proves the $i=1$ case of 
\eqref{eq_duality_two_dim_weight_functions}.
Replacing $W,\mec d$ everywhere in the above equation with
$W^*_\mec L,\mec K-\mec d$ and taking duals gives the
$i=0$ case of
\eqref{eq_duality_two_dim_weight_functions}.
Hence both the $i=0$ and $i=1$ cases of
\eqref{eq_duality_two_dim_weight_functions}  hold.
\end{proof}

\subsection{Duality for All Riemann Functions}

In this subsection we establish
\eqref{eq_duality_first_general}
for an arbitrary Riemann 
function $f\from\integers^n\to\integers$. 
The first lemma below is the essential insight; roughly speaking,
it says that the two-variable restriction of a generalized
Riemann-Roch formula yields a generalized Riemann-Roch
formula for the two-variable restriction.
In other words, if
$f\from\integers^n\to\integers$ is a Riemann function, then
for all $\mec d,\mec K\in\integers^n$ we have
\begin{equation}\label{eq_generalized_rr_duality_all_intro}
f(\mec d)-f^\wedge_\mec K(\mec K-\mec d)=\deg(\mec d)+C_f
\end{equation}
where $C_f$ is the offset of $f$.  Now consider
$\mec K$ and $d_3,\ldots,d_n$ to be fixed, and $(d_1,d_2)$ varying over
all of $\integers^2$; then
$$
g(d_1,d_2)=f(\mec d)
$$
is a Riemann function of two variables $d_1,d_2$, and $g$ therefore satisfies
\begin{equation}\label{eq_two_var_restrict_duality_all_intro}
g(d_1,d_2)-g^\wedge_{(K_1,K_2)}(K_1-d_1,K_2-d_2) = d_1+d_2+C_g 
\end{equation} 
where $C_g$ is the offset of $g$.
The first lemma shows (easily) that the right-hand-sides of
\eqref{eq_two_var_restrict_duality_all_intro} and
\eqref{eq_generalized_rr_duality_all_intro} are equal, and therefore
the negative term of the left-hand-sides of
\eqref{eq_two_var_restrict_duality_all_intro} and
\eqref{eq_generalized_rr_duality_all_intro} are equal, i.e.,
$$
f^\wedge_\mec K(\mec K-\mec d) = g^\wedge_{(K_1,K_2)}(K_1-d_1,K_2-d_2)
$$
where $d_1,d_2$ are varying.
The first lemma expresses this equality of functions in a way that is
useful to establish \eqref{eq_duality_first_general},
namely as \eqref{eq_n_variate_equality_two_var_duals} below.

\begin{lemma}
Let $n\ge 2$,
$f\from\integers^n\to\integers$ be a Riemann function with
offset $C_f$, let 
$\mec d, \mec K\in\integers^n$, and
let $(\tilde K_1,\tilde K_2)\in\integers^2$.
Set
$$
\mec d'=\mec d-d_1\mec e_1-d_2\mec e_2
=(0,0,d_3,\ldots,d_n),
$$
and let $g=f_{1,2,\mec d'}$.
Then
\begin{enumerate}
\item
for all $a_1,a_2\in\integers$,
\begin{equation}\label{eq_g_the_two_var_restrict_of_f}
g(a_1,a_2)=
f_{1,2,\mec d'}(a_1,a_2)=f(a_1,a_2,d_3,\ldots,d_n) ;
\end{equation} 
\item 
the offset of $g=f_{1,2,\mec d'}$ is
\begin{equation}\label{eq_offset_two_var_restriction}
C_g = d_3+\cdots+d_n+C_f ;
\end{equation} 
\item 
\begin{equation}\label{eq_f_dual_and_g_dual_values}
f^\wedge_\mec K(\mec K-\mec d)= 
g^\wedge_{\tilde K_1,\tilde K_2}
(\tilde K_1-d_1,\tilde K_2-d_2);
\end{equation} 
and
\item
setting $\mec K'=(0,0,K_3,\ldots,K_n)$, we have
\begin{equation}\label{eq_n_variate_two_var_pre_equation}
\forall a_1,a_2\in\integers,\quad
(f^\wedge_\mec K)_{1,2,\mec K'-\mec d'}(a_1,a_2)
= g^\wedge_{(K_1,K_2)}(a_1,a_2)
= (f_{1,2,\mec d'})^\wedge_{(K_1,K_2)}(a_1,a_2) .
\end{equation} 
\end{enumerate}
Moreover, for any $\mec d,\mec K$, and $\mec d',\mec K'$ as above, we have
an equality of functions
\item
\begin{equation}\label{eq_n_variate_equality_two_var_duals}
(f^\wedge_\mec K)_{1,2,\mec K'-\mec d'}=(f_{1,2,\mec d'})^\wedge_{(K_1,K_2)} .
\end{equation} 
\end{lemma}
\begin{proof}
(1)~is immediate.  (2)~follows from the fact that with $d_3,\ldots,d_n$
fixed and $a_1+a_2$ sufficiently large we have
$$
f(a_1,a_2,d_3,\ldots,d_n) = a_1+a_2+d_3+\cdots+d_n+C_f,
$$
which by~(1) equals
$$
g(a_1,a_2)= a_1+a_2+C_g
$$
for $a_1+a_2$ sufficiently large.
Hence $C_g=d_3+\cdots+d_n+C_f$.

To prove~(3), in view of
\eqref{eq_first_dual_formulation}
we have
\begin{equation}\label{eq_f_wedge_K_at_K_minus_d}
f^\wedge_\mec K(\mec K-\mec d) =
f(\mec d)-\deg(\mec d)-C_f,
\end{equation}
and similarly
$$
g^\wedge_{(\tilde K_1,\tilde K_2)}
(\tilde K_1-d_1,\tilde K_2-d_2)
=
g(d_1,d_2)-(d_1+d_2)-C_g
$$
which, in view of~(1) and~(2),
$$
=f(\mec d)-(d_1+d_2)-(d_3+\cdots+d_n+C_f)
=
f(\mec d)-\deg(\mec d)-C_f
$$
which is just the right-hand-side of 
\eqref{eq_f_wedge_K_at_K_minus_d};
this yields
\eqref{eq_f_dual_and_g_dual_values}.

To prove~(4), taking $\tilde K_i=K_i$ for $i=1,2$,
\eqref{eq_f_dual_and_g_dual_values} gives
$$
f^\wedge_\mec K(\mec K-\mec d)= 
g^\wedge_{(K_1,K_2)}
(K_1-d_1,K_2-d_2).
$$
Hence we have
$$
(f^\wedge_\mec K)_{1,2,\mec K'-\mec d'}(K_1-d_1,K_2-d_2)
=f^\wedge_\mec K(\mec K-\mec d)=
g^\wedge_{(K_1,K_2)}
(K_1-d_1,K_2-d_2),
$$
and therefore
$$
(f^\wedge_\mec K)_{1,2,\mec K'-\mec d'}(K_1-d_1,K_2-d_2)
=
g^\wedge_{(K_1,K_2)}
(K_1-d_1,K_2-d_2)
=
(f_{1,2,\mec d'})^\wedge_{(K_1,K_2)}
(K_1-d_1,K_2-d_2).
$$
But the functions
$$
(f^\wedge_\mec K)_{1,2,\mec K'-\mec d'},
\quad
g=f_{1,2,\mec d'},
\quad\mbox{and}\quad
g^\wedge_{(K_1,K_2)}=(f_{1,2,\mec d'})^\wedge_{(K_1,K_2)}
$$
are, as functions $\integers^2\to\integers$,
independent of $d_1,d_2$, since $\mec d'$ discards the components
$d_1,d_2$.
Hence, for fixed $f,\mec K$ and $d_3,\ldots,d_n$, we have
$$
\forall d_1,d_2\in\integers,\quad
(f^\wedge_\mec K)_{1,2,\mec K'-\mec d'}(K_1-d_1,K_2-d_2)
= g^\wedge_{(K_1,K_2)}(K_1-d_1,K_2-d_2),
$$
which upon setting $a_i=K_i-d_i$ for $i=1,2$
proves \eqref{eq_n_variate_two_var_pre_equation}.

Finally \eqref{eq_n_variate_equality_two_var_duals}
follows from~(4) and the value of $g$.
\end{proof}

The next lemma is straightforward but needs to be stated formally.

\begin{lemma}\label{le_translation_of_function_cM}
Let $W\from\integers^2\to\integers$ be the weight of a Riemann function,
$f\from\integers^2\to\integers$, and for some $\mec t\in\integers^2$
let $g$ be the translation of $f$ by $\mec t$, i.e., 
given by
$$
g(\mec d)=f(\mec d+\mec t).
$$
Then $W'=\frakm g$ satisfies
\begin{equation}\label{eq_cM_class_under_translation}
\forall \mec d\in\integers^2,
\quad
[\cM_{W',\mec d}] = [\cM_{W,\mec d+\mec t}].
\end{equation} 
\end{lemma}
\begin{proof}
Write $W$ as a difference of sums of perfect matchings
$$
W = W_1+\cdots+W_s - \tilde W_1 - \cdots - \tilde W_{s-1} .
$$
Then setting $W'_i$ to be the translation by $\mec t$ of $W_i$,
and similarly for $\tilde W'_i$, we have
$$
W' = W'_1+\cdots+W'_s - \tilde W'_1 - \cdots - \tilde W'_{s-1} .
$$
Since $W'_i$ is translation by $\mec t$ of $W_i$, we have
$\cM_{W'_i,\mec d}=\cM_{W_i,\mec d+\mec t}$,
and similarly for $\tilde W'_i$ and $\tilde W_i$.
Hence
$$
\cM_{W'_1,\mec d} \oplus \cdots \oplus \cM_{W'_s,\mec d}
\isom
\cM_{W_1,\mec d+\mec t} \oplus \cdots \oplus \cM_{W_s,\mec d+\mec t}
$$
and
$$
\cM_{\tilde W_1,\mec d+\mec t}\oplus\cdots\oplus
\cM_{\tilde W_{s-1},\mec d+\mec t}
\isom
\cM_{\tilde W'_1,\mec d}\oplus\cdots\oplus
\cM_{\tilde W'_{s-1},\mec d} .
$$
Taking direct sums of these last two equations yields
the equivalence of classes \eqref{eq_cM_class_under_translation}.
\end{proof}

\begin{theorem}\label{th_n_variate_duality}
Let $f\from\integers^n\to\integers$ be an arbitrary Riemann functions.
Then for any $\mec d,\mec K\in\integers^n$ there are isomorphisms
\begin{equation}\label{eq_main_duality_n_variate_theorem}
H^i\bigl([\cM_{f {\rm \; at \;}\mec d}]\bigr)^* \to
H^{1-i}\bigl([\cM_{f^\wedge_\mec K{\rm \; at \;}\mec K-\mec d}] \bigr)
\quad\mbox{for $i=0,1$,}
\end{equation} 
which specifically arises 
by setting
$\mec d'=(0,0,d_3,d_4,\ldots,d_n)$
taking $W=\frakm f_{1,2,\mec d'}$,
and applying \eqref{eq_duality_two_dim_weight_functions}.
\end{theorem}
\begin{proof}
By Definition~\ref{de_cM_f_at_d} we have
$$
[\cM_{f {\rm \; at \;}\mec d}] = [\cM_{\frakm f_{1,2,\mec d},(0,0)}].
$$
But for all $\mec a\in\integers^2$ we have
$$
f_{1,2,\mec d}(\mec a)=f_{1,2,\mec d'}\bigl( (d_1,d_2)+\mec a \bigr)
$$
since both sides equal $f(\mec d+a_1\mec e_1+a_2\mec e_2)$.
Hence $f_{1,2,\mec d}$ is a translation of $f_{1,2,\mec d'}$ by
$(d_1,d_2)$,
and Lemma~\ref{le_translation_of_function_cM} implies that
$$
[\cM_{\frakm f_{1,2,\mec d},(0,0)}]
=
[\cM_{W,(d_1,d_2)}].
$$
In view of \eqref{eq_duality_two_dim_weight_functions} we have
an isomorphism
\begin{equation}\label{eq_isom_H_i_n_variate_to_two_variate_stuff}
H^i\bigl( [\cM_{f {\rm \; at \;}\mec d}] \bigr)^*
\to
H^{1-i}\bigl( [ \cM_{W^*_{K_1+1,K_2+1},(K_1-d_1,K_2-d_2)} ] \bigr).
\end{equation} 
Since $W=\frakm g$ where $g=f_{1,2,\mec d'}$, 
\eqref{eq_n_variate_equality_two_var_duals}
implies that
$$
W^*_{K_1+1,K_2+1} = \frakm \bigl( g^\wedge_{(K_1,K_2)} \bigr)=
\frakm \bigl( (f^\wedge_K)_{1,2,\mec K'-\mec d'} \bigr) 
$$
with $\mec K'=(0,0,K_3,\ldots,K_n)$.  
Lemma~\ref{le_translation_of_function_cM} similarly implies that
$$
\Bigl[\cM_{\frakm 
\bigl( (f^\wedge_K)_{1,2,\mec K'-\mec d'} \bigr),(K_1-d_1,K_2-d_2)}\Bigr]
=
\Bigl[\cM_{\frakm 
\bigl( (f^\wedge_K)_{1,2,\mec K-\mec d} \bigr),(0,0)}\Bigr]
$$
which, by Definition~\ref{de_cM_f_at_d},
$$
=\bigl[\cM_{f^\wedge_K {\rm\; at\; \mec K-\mec d}} \bigr].
$$
Hence this equality and
\eqref{eq_isom_H_i_n_variate_to_two_variate_stuff}
yields
\eqref{eq_main_duality_n_variate_theorem}.
\end{proof}

\section{Stronger Duality Properties and Further Remarks}
\label{se_duality_second}

In this section we will prove some stronger duality properties of
$\underline k_{/B_1,B_2}$.
To do so, we will develop some foundations of the homological
algebra of $k$-diagrams, including a discussion of
{\em skyscraper $k$-diagrams} and {\em coskyscraper $k$-diagrams},
based on the unpublished results in
\cite{friedman_cohomology}.
We also discuss the connection of $k$-diagrams to Grothendieck's
sheaf theory and classical sheaf theory.
This discussion will tie up a number of 
loose ends: for example, we will show that our definition of the cohomology
groups of $k$-diagrams agree with the usual definition of 
of these groups, both in the context of sheaf theory for
Grothendieck topologies and for classical topological spaces.
We will also comment on periodic Riemann functions and possible future work.

Let us summarize the main results of this section.

The first stronger duality property of $\underline k_{/B_1,B_2}$
is that for any $k$-diagram, $\cF$,
we have
\begin{equation}\label{eq_zeroth_cohomology_and_Ext_one}
H^0(\cF)^*\isom \Ext^1(\cF,\underline k_{/B_1,B_2}),
\end{equation} 
where $\Ext^1$ is the ``first Ext group;'' in order to define this group,
and to make some basic computations,
we will need some results from homological algebra.

When we define Ext groups we will
see that $\Ext^0(\cF,\cG)$ is isomorphic to $\Hom(\cF,\cG)$, and 
$\Ext^i(\underline k,\cF)$ is isomorphic to $H^i(\cF)$ for all $i$;
hence
combining 
\eqref{eq_zeroth_cohomology_and_Ext_one}
with
Theorem~\ref{th_representability_of_H_one_dual}
yields the following theorem.

\begin{theorem}\label{th_strong_duality}
For any $k$-diagram, $\cF$, and for $i=0,1$, there are isomorphisms
\begin{equation}\label{eq_Serre_type_duality}
H^i(\cF)^* \to \Ext^{1-i}(\cF,\underline k_{/B_1,B_2})
\end{equation} 
that are functorial (i.e., natural) in $\cF$.
In other words,
$\Hom(\cF,\underline k_{/B_1,B_2})$ and $\Ext^1(\cF,\underline k_{/B_1,B_2})$
are isomorphic to the kernel and cokernel, respectively, of
the dual map
$$
\cF(\partial)^*\from \cF(A)^* \to \cF(B)^*
$$
defined in
Definition~\ref{de_diagram_k_vs}.
\end{theorem}
The proof is given in Subsection~\ref{su_strong_duality}.
This theorem shows that $\underline k_{/B_1,B_2}$ plays the role
of the canonical sheaf in the statement of Serre duality
(e.g., Section~III.7 of \cite{hartshorne}).

The second duality result, proven at the end of this section,
gives a method for computing the {\em Serre functor}, $S$, on a $k$-diagram,
and we show that
$S(\underline k)\isom\underline k_{/B_1,B_2}[1]$;
the proof is an immediate consequence of 
% the coskyscraper resolution
% of $\underline k$ and the skyscraper resolution of 
% $\underline k_{/B_1,B_2}$ given in 
the proofs of
Theorems~\ref{th_strong_duality}
and~\ref{th_cohomology_agrees_with_Ext_groups} below.

To prove these results, we will develop the notion of
skyscraper and coskyscraper $k$-diagrams.
The reader familiar with sheaf theory in topology or algebraic geometry
(e.g., \cite{hartshorne}) will be able to check that
skyscraper $k$-diagrams are the analog of skyscraper sheaves
in topological sheaf theory
(coskyscraper $k$-diagrams do not generally exist
in topological sheaf theory).
We stress that all $k$-diagrams have a two-term
injective resolution with skyscraper $k$-diagrams,
and a two-term projective resolution with coskyscraper $k$-diagrams.
This
makes working with $k$-diagrams much simpler than with sheaves
over general topological spaces.

In this section will also show that the definition of the cohomology
groups of a $k$-diagram, Definition~\ref{de_diagram_k_vs}, agrees
with the usual definition of cohomology groups.  Namely, we will prove 
the following theorem.

\begin{theorem}\label{th_cohomology_agrees_with_Ext_groups}
Let $\cF$ be a $k$-diagram.  Then the cohomology groups, $H^i(\cF)$, of $\cF$
defined in Definition~\ref{de_diagram_k_vs} are isomorphic to
the groups $\Ext^i(\underline k,\cF)$ defined by viewing the
category of $k$-diagrams as an abelian category.
\end{theorem}

We will briefly review some facts about homological algebra, and
refer the reader to the textbook \cite{weibel} or
\cite{hartshorne}, Section~III.1 for details.

We remark that the modern foundations of homological algebra 
implicitly assume that Zorn's lemma holds, and it follows
that for any two $k$-vector spaces $A\subset B$, there is a 
$B'\subset B$ such that $B$ is the direct sum of $A$ and $B'$.
Hence we assume this here (see Subsection~\ref{su_linear_algebra_subtlety}
for further discussion).

\subsection{Skyscraper and Co-Skyscraper $k$-Diagrams}

To compute Ext groups below, we will use
{\em skyscraper} and {\em coskyscraper}
diagrams that we now describe.

For any $k$-vector space, $V$, consider the $k$-diagram,
denoted ${\rm Sky}_{A_1}(V)$, depicted below:
$$
\begin{tikzpicture}[scale=0.40]
\node (B1) at (0,2) {$V$};
\node (B2) at (0,-2) {$0$};
\node (B3) at (0,0) {$V$};
\node (A1) at (3,1) {$V$};
\node (A2) at (3,-1) {$0$};
\draw [->] (B1) -- (A1) ;
\draw [->] (B2) -- (A2) ;
\draw [->] (B3) -- (A1) ;
\draw [->] (B3) -- (A2) ;
\node at (1.5,-4){${\rm Sky}_{A_1}(V)$};
\end{tikzpicture}
$$
where all maps $V\to V$ are identity maps; we call this the
{\em skyscraper diagram of $V$ at $A_1$}.  We easily check that for 
any $\cF$, and any
$$
\phi\in \Hom(\cF,{\rm Sky}_{A_1}(V)),
$$
the value of $\phi(A_1)\from \cF(A_1)\to V$ determines all of
$\phi$: indeed, the map $\phi(B_1)\from\cF(B_1)\to V$ must
equal $\phi(A_1)\circ \cF(\rho_{1,1})$; conversely, any
map $\cF(A_1)\to V$ extends to such a morphism $\phi$.
Hence the map $\phi\mapsto\phi(A_1)$ sets up an isomorphism
\begin{equation}\label{eq_skyscraper_hom_formula}
\Hom(\cF,{\rm Sky}_{A_1}(V)) \isom \Hom_k(\cF(A_1),V),
\end{equation} 
where $\Hom_k$ denotes the morphisms as $k$-vector spaces.
% From \eqref{eq_skyscraper_hom_formula} one easily sees that
% ${\rm Sky}_{A_1}(V)$ is injective:
% Using the fact that any $k$-vector space, $V$, is an injective $k$-module,
% it easily follows that ${\rm Sky}_{A_1}(V)$ is injective: indeed,
% if $\nu\from\cF\to\cG$ is an injective map of $k$-diagrams, and
% $\phi\from\cF\to {\rm Sky}_{A_1}(V)$ is a morphism, then
% $\phi(A_1)\from\cF(A_1)\to V$ is a morphism of $k$-vector spaces,
% and $\nu(A_1)\from\cF(A_1)\to\cG(A_1)$ is injective, and
% and hence there exists a morphism $L\from\cG(A_1)\to V$ with
% $L\nu(A_1)=\phi(A_1)$. According to \eqref{eq_skyscraper_hom_formula},
% to $L$ there corresponds a
% map $\mu\from\cG\to {\rm Sky}_{A_1}(V)$ such that $\mu(A_1)=L$.
% Since $\mu\nu$ and $\phi$ agree on their $A_1$ values, and they
% are both elements of $\Hom(\cF,{\rm Sky}_{A_1}(V))$, 
% according to \eqref{eq_skyscraper_hom_formula} they are
% equal.

One similarly defines ${\rm Sky}_{A_2}(V)$, 
and for $j=1,2,3$ one defines
${\rm Sky}_{B_j}(V)$ to be the $k$-diagram whose only nonzero value
is $V$, at $B_j$.  We depict these $k$-diagrams below.
$$
\begin{tikzpicture}[scale=0.40]
\node (B1) at (0,2) {$0$};
\node (B2) at (0,-2) {$V$};
\node (B3) at (0,0) {$V$};
\node (A1) at (3,1) {$0$};
\node (A2) at (3,-1) {$V$};
\draw [->] (B1) -- (A1) ;
\draw [->] (B2) -- (A2) ;
\draw [->] (B3) -- (A1) ;
\draw [->] (B3) -- (A2) ;
\node at (1.5,-4){${\rm Sky}_{A_2}(V)$};
\end{tikzpicture}
\quad\quad
\begin{tikzpicture}[scale=0.40]
\node (B1) at (0,2) {$V$};
\node (B2) at (0,-2) {$0$};
\node (B3) at (0,0) {$0$};
\node (A1) at (3,1) {$0$};
\node (A2) at (3,-1) {$0$};
\draw [->] (B1) -- (A1) ;
\draw [->] (B2) -- (A2) ;
\draw [->] (B3) -- (A1) ;
\draw [->] (B3) -- (A2) ;
\node at (1.5,-4){${\rm Sky}_{B_1}(V)$};
\end{tikzpicture}
\quad\quad
\begin{tikzpicture}[scale=0.40]
\node (B1) at (0,2) {$0$};
\node (B2) at (0,-2) {$V$};
\node (B3) at (0,0) {$0$};
\node (A1) at (3,1) {$0$};
\node (A2) at (3,-1) {$0$};
\draw [->] (B1) -- (A1) ;
\draw [->] (B2) -- (A2) ;
\draw [->] (B3) -- (A1) ;
\draw [->] (B3) -- (A2) ;
\node at (1.5,-4){${\rm Sky}_{B_2}(V)$};
\end{tikzpicture}
\quad\quad
\begin{tikzpicture}[scale=0.40]
\node (B1) at (0,2) {$0$};
\node (B2) at (0,-2) {$0$};
\node (B3) at (0,0) {$V$};
\node (A1) at (3,1) {$0$};
\node (A2) at (3,-1) {$0$};
\draw [->] (B1) -- (A1) ;
\draw [->] (B2) -- (A2) ;
\draw [->] (B3) -- (A1) ;
\draw [->] (B3) -- (A2) ;
\node at (1.5,-4){${\rm Sky}_{B_3}(V)$};
\end{tikzpicture}
$$
This gives for any
$P=A_1,A_2,B_1,B_2,B_3$ and any $k$-vector space, $V$,
a diagram, ${\rm Sky}_P(V)$ with an isomorphism
\begin{equation}\label{eq_skyscraper_hom}
\Hom(\cF,{\rm Sky}_{P}(V)) \isom \Hom_k(\cF(P),V),
\end{equation} 
given by 
$$
\mbox{for $\phi\in \Hom(\cF,{\rm Sky}_{P}(V))$, \quad\quad
$\phi\mapsto \phi(P)$}.
$$
Skyscrapers are particularly useful because one can prove that they
are {\em injective} $k$-diagrams (see Proposition~\ref{pr_two_term}
below).

Similarly one defines for a $k$-vector sapce $V$ the {\em coskyscraper
$k$-diagram at $B_j$} to be the diagrams depicted below:
$$
\begin{tikzpicture}[scale=0.40]
\node (B1) at (0,2) {$0$};
\node (B2) at (0,-2) {$0$};
\node (B3) at (0,0) {$0$};
\node (A1) at (3,1) {$V$};
\node (A2) at (3,-1) {$0$};
\draw [->] (B1) -- (A1) ;
\draw [->] (B2) -- (A2) ;
\draw [->] (B3) -- (A1) ;
\draw [->] (B3) -- (A2) ;
\node at (1.5,-4){${\rm CoSky}(A_1)$};
\end{tikzpicture}
\quad\quad
\begin{tikzpicture}[scale=0.40]
\node (B1) at (0,2) {$0$};
\node (B2) at (0,-2) {$0$};
\node (B3) at (0,0) {$0$};
\node (A1) at (3,1) {$0$};
\node (A2) at (3,-1) {$V$};
\draw [->] (B1) -- (A1) ;
\draw [->] (B2) -- (A2) ;
\draw [->] (B3) -- (A1) ;
\draw [->] (B3) -- (A2) ;
\node at (1.5,-4){${\rm CoSky}(A_2)$};
\end{tikzpicture}
\quad\quad
\begin{tikzpicture}[scale=0.40]
\node (B1) at (0,2) {$V$};
\node (B2) at (0,-2) {$0$};
\node (B3) at (0,0) {$0$};
\node (A1) at (3,1) {$V$};
\node (A2) at (3,-1) {$0$};
\draw [->] (B1) -- (A1) ;
\draw [->] (B2) -- (A2) ;
\draw [->] (B3) -- (A1) ;
\draw [->] (B3) -- (A2) ;
\node at (1.5,-4){${\rm CoSky}(B_1)$};
\end{tikzpicture}
\quad\quad
\begin{tikzpicture}[scale=0.40]
\node (B1) at (0,2) {$0$};
\node (B2) at (0,-2) {$V$};
\node (B3) at (0,0) {$0$};
\node (A1) at (3,1) {$0$};
\node (A2) at (3,-1) {$V$};
\draw [->] (B1) -- (A1) ;
\draw [->] (B2) -- (A2) ;
\draw [->] (B3) -- (A1) ;
\draw [->] (B3) -- (A2) ;
\node at (1.5,-4){${\rm CoSky}(B_2)$};
\end{tikzpicture}
\quad\quad
\begin{tikzpicture}[scale=0.40]
\node (B1) at (0,2) {$0$};
\node (B2) at (0,-2) {$0$};
\node (B3) at (0,0) {$V$};
\node (A1) at (3,1) {$V$};
\node (A2) at (3,-1) {$V$};
\draw [->] (B1) -- (A1) ;
\draw [->] (B2) -- (A2) ;
\draw [->] (B3) -- (A1) ;
\draw [->] (B3) -- (A2) ;
\node at (1.5,-4){${\rm CoSky}(B_3)$};
\end{tikzpicture}
$$
One verifies that for
$P=A_1,A_2,B_1,B_2,B_3$ and any $k$-vector space, $V$,
there is an isomorphism
$$
\Hom({\rm CoSky}_P(V),\cG) \isom \Hom_k(V,\cG(P))
$$
taking $\phi$ to $\phi(P)$.
Coskyscrapers are particularly useful because one can prove that they
are {\em projective} $k$-diagrams (see Proposition~\ref{pr_two_term}
below).
% Since each $k$-vector space is a projective $k$-module, it follows 
% that for all $V$, ${\rm CoSky}_P(V)$ is a projective $k$-diagram.

\subsection{A Minimal Introduction to Homological Algebra and Value-by-Value
Evaluation}

\label{su_intro_homological_algebra}

Here we will briefly review some facts about homological algebra, and
refer the reader to the textbook \cite{weibel} or
\cite{hartshorne}, Section~III.1 for details.

In order to use apply homological algebra, we need to know that
$k$-diagrams and their morphisms form an {\em abelian category}
(see \cite{weibel}, Definition~1.2.2 or
\cite{hartshorne}, Section~III.1); in computations, we will need to
know how to compute the kernel, image, and cokernel in the sense
defined for abelian categories.

\begin{definition}\label{de_value_by_value}
If $\phi\from\cF\to\cG$ is a morphism of $k$-diagrams, then the {\em kernel}
(respectively {\em image} and {\em cokernel}) of $\phi$ is
the $k$-diagram whose value at 
$P=A_1,A_2,B_1,B_2,B_3$ equals $\ker(\phi(P))$
(respectively, ${\rm Image}(\phi(P))$ and $\coker(\phi(P))$),
and whose restriction maps are induced by those of $\cF$
(respectively, those of $\cG$ in both cases).
\end{definition}
In other words, we define {\em kernel}, {\em image}, and {\em cokernel}
by evaluating them ``value-by-value.''

The next proposition is well known.
\begin{proposition}\label{pr_value_by_value}
The category of $k$-diagrams is an abelian category.
The definitions of kernel, image, and cokernel in
Definition~\ref{de_value_by_value}
agree with those notions when viewing the category of $k$-diagrams 
as an abelian category.
\end{proposition}
The reader can prove this proposition directly; however,
in Subsection~\ref{su_k_diagrams_sheaf_theory}
we will give two other proofs: one by a direct appeal to
\cite{sga4.1}, and another by appealing to the special nature of
sheaf theory over
finite topological spaces.
Of course, the reader may prefer to just assume this proposition and
carry out the computations below.

A doubly-infinite sequence of morphisms of $k$-diagrams
$$
\cdots \to \cF^{-1} \xrightarrow{d^{-1}}
\cF^0 \xrightarrow{d^0} 
\cF^1 \xrightarrow{d^1} \cdots
$$
is {\em exact in position $i$} if ${\rm Image}(d^{i-1})=\ker(d^i)$,
i.e., for all
$P\in \{A_1,A_2,B_1,B_2,B_3\}$ we have
${\rm Image}(d^{i-1}(P))=\ker(d^i(P))$.
A sequence of morphisms is {\em exact} if it is exact at each position.
The analogous definition holds for a finite sequence of morphisms,
or a one-sided infinite sequence of morphisms.

A {\em short exact sequence} is an exact sequence
\begin{equation}\label{eq_short_exact_cF_sequence}
0\to \cF^1\to \cF^2\to \cF^3 \to 0.
\end{equation} 
For such a sequence, and any $k$-diagram, $\cG$, the morphisms in the
above sequence determine, via composition, a sequence 
\begin{equation}\label{eq_short_exact_cF_sequence_applied_to_G_first}
0 \to \Hom(\cG,\cF^1)\to \Hom(\cG,\cF^2)\to\Hom(\cG,\cF^3)\to 0
\end{equation} 
and a sequence
\begin{equation}\label{eq_short_exact_cF_sequence_applied_to_G_second}
0 \to \Hom(\cF^3,\cG)\to \Hom(\cF^2,\cG)\to\Hom(\cF^1,\cG)\to 0.
\end{equation} 
We say that a $k$-diagram, $\cG$ is {\em injective}
(respectively, {\em projective}), if for any short exact sequence
\eqref{eq_short_exact_cF_sequence}, the resulting sequence
\eqref{eq_short_exact_cF_sequence_applied_to_G_first}
(respectively, 
\eqref{eq_short_exact_cF_sequence_applied_to_G_second})
is exact.
Any finite direct sum of injectives is injective, and any
of projectives is projective.

In contrast with a lot of commonly used abelian categories---such as
the category of sheaves of abelian groups or of $k$-vector spaces on
a topological space---each $k$-diagram has simple projective and
injective resolutions.

\begin{proposition}\label{pr_two_term}
Any skyscraper (respectively, coskyscraper)
$k$-diagram is injective (respectively, projective).
Any $k$-diagram $\cF$ fits into an exact sequence
\begin{equation}\label{eq_cP_resolving_cF}
0\to \cP_1\to\cP_0\to \cF\to 0
\end{equation} 
where $\cP_1,\cP_0$ are projective---actually direct sums of
coskyscraper diagrams; we call \eqref{eq_cP_resolving_cF} a
{\em two-term projective resolution of $\cF$}.
Moreover, if the values of $\cF$ are finite dimensional $k$-vector spaces,
then the coskyscraper sheaves in the projective resolution can be
taken to be are of the form
${\rm CoSky}_P(V)$ where the $V$ are finite dimensional $k$-vector spaces.
Similarly, any $\cG$ fits into an exact sequence,
\begin{equation}\label{eq_cI_resolving_cG}
0\to \cG\to \cI^0\to\cI^1\to  0 ,
\end{equation}
where $\cI^0,\cI^1$ are injective; we call
\eqref{eq_cI_resolving_cG} a
{\em two-term injective resolution of $\cG$}.
Moreover, if the values of $\cG$ are finite dimensional $k$-vector spaces,
then the injective resolution is a sum of $k$-diagrams
${\rm Sky}_P(V)$ where the $V$ are finite dimensional $k$-vector spaces.
\end{proposition}
We will prove this proposition in
Subsection~\ref{su_two_term_resolutions}.

The fact that any $k$-diagram has a two-term projective resolution
and a two-term injective resolution makes the definition of Ext
groups especially simple.  
(For the general definition of Ext groups, see
\cite{weibel,hartshorne}.)

For any $k$-diagrams, $\cF,\cG$, we take a projective resolution
\eqref{eq_cP_resolving_cF} and 
define the group $\Ext^i(\cF,\cG)$ for $i=0,1$, respectively,
as the kernel and cokernel of the resulting maps
$$
\Hom(\cP_0,\cG)\to \Hom(\cP_1,\cG),
$$
which are independent of the projective $k$-diagrams $\cP_1,\cP_0$
and morphism $\cP_1\to\cP_0$ yielding an exact sequence;
$\Ext^0(\cF,\cG)\isom \Hom(\cF,\cG)$.
Furthermore, taking any injective resolution
\eqref{eq_cI_resolving_cG}, 
the resulting kernel and cokernel of the map
$$
\Hom(\cF,\cI^0)\to \Hom(\cF,\cI^1)
$$
are also isomorphic to $\Ext^i(\cF,\cG)$ for, respectively, $i=0,1$
(see, e.g., Theorem~2.7.6\footnote{
  One is also using the Freyd-Mitchell Embedding Theorem, namely
  Theorem~1.6.1 of \cite{weibel} (page 25), 
  so that working with $R$-modules
  implies the same results over any small abelian category.
  } 
  (page 63) of \cite{weibel}).

Consider any short exact sequence
of $k$-diagrams 
$$
0\to \cF_1\to \cF_2\to \cF_3\to 0
$$
(i.e., for each $P=A_1,A_2,B_1,B_2,B_3$, the sequence
$$
0\to \cF_1(P)\to \cF_2(P)\to \cF_3(P) \to 0
$$
is exact).  In this case for any $k$-diagram, $\cG$,
there is a long exact sequence
$$
0\to \Ext^0(\cF_3,\cG)\to \Ext^0(\cF_2,\cG)\to  \Ext^0(\cF_1,\cG) \to
\Ext^1(\cF_3,\cG) \to \cdots 
$$
and this sequence is ``natural'' or ``functorial'' in $\cG$, i.e.,
if $\cG_1\to \cG_2$ is a morphism, then there are morphisms
$\Ext^i(\cF_j,\cG_1)\to\Ext^i(\cF_j,\cG_2)$
that commute with the maps in the two resulting long exact sequences.
Similarly to a short exact sequence $0\to \cG_1\to\cG_2\to\cG_3$,
there is a long
exact sequence
\begin{equation}\label{eq_long_exact_ext_second_var_varying}
0\to \Ext^0(\cF,\cG_1)\to \Ext^0(\cF,\cG_2)\to  \Ext^0(\cF,\cG_3) \to
\Ext^1(\cF,\cG_1) \to \cdots 
\end{equation} 
that is functorial in $\cF$
(see, for example, the discussion of
{\em universality}, since $k$-diagrams have enough
injectives  (see \cite{hartshorne},
Corollary~1.4, Section~III.1).

% If $\cF$ is a $k$-diagram, then a {\em subdiagram of $\cF$} refers
% to any $k$-diagram, $\cF'$, such that
% for all $P\in \{A_1,A_2,B_1,B_2,B_3\}$ we have
% $\cF'(P)\subset\cF(P)$ is a subspace, and each restriction map
% $\cF'(\rho_{ij})$ equals the restriction of $\cF(\rho_{ij})$ to
% $\cF'(A_i)$.

\subsection{Working with Ext Groups}

Our discussion of skyscraper and coskyscraper $k$-diagrams have
some easy consequences that we will need when we make
computations; we state these results in the two propositions below, and
we leave both proofs to the reader.
 
\begin{proposition}\label{pr_special_case_of_k_valued_cosky_and_sky}
Let
$P$ be one of $A_1,A_2,B_1,B_2,B_3$.
For each $k$-diagram, $\cG$, and each
$u\in\cG(P)$, the isomorphisms 
$$
\Hom({\rm CoSky}_P(k),\cG) \isom \Hom_k(k,\cG(P)) \isom \cG(P)
$$
determine a unique 
$$
\iota_{P,u} \in \Hom\bigl({\rm CoSky}_P(k),\cG\bigr)
$$
such that $\iota_{P,u}(P)$ takes $1$ to $u$.
For each $k$-diagram, $\cF$, and each $w\in\cF(P)^*$, the isomorphisms
$$
\Hom(\cF,{\rm Sky}_P(k)) \isom \Hom_k(\cF(P),k) \isom \cF(P)^*
$$
determine a unique 
$$
\iota^{P,w} \in \Hom\bigl(\cF,{\rm Sky}_P(k)\bigr)
$$
such that $\iota^{P,w}(P)$ takes each $u\in\cF(P)$ to $w(u)\in k$.
\end{proposition}

In computing Ext groups and the Serre functor, the following observations
will be helpful.

\begin{definition}\label{de_specialization}
For $P,P'\in \{A_1,A_2,B_1,B_2,B_3\}$, we say that $P'$ is 
{\em a specialization of $P$} if $P'=P$ or $P=A_i$ for some $i$
and $P'=B_i,B_3$; we also write $P'\le P$, which gives a partial
order on the set $ \{A_1,A_2,B_1,B_2,B_3\}$.
\end{definition}

\begin{proposition}
\label{pr_morphisms_of_special_case_of_k_valued_cosky_and_sky}
For $P,P'$ each equal one of $A_1,A_2,B_1,B_2,B_3$.
We have
$$
\Hom({\rm CoSky}_{P,k},{\rm CoSky}_{P',k})\isom
{\rm CoSky}_{P',k}(P)
$$
which equals $k$ or $0$ according to whether or not $P'$ is a specialization
of $P$.
Furthermore, let $P'$ be a specialization of $P$, and consider the map
$$
\mu \in \Hom({\rm CoSky}_{P,k},{\rm CoSky}_{P',k})
$$
given by $\mu(P)$ takes $1\in \Hom({\rm CoSky}_{P,k}(P)$ to
$\alpha\in k={\rm CoSky}_{P',k}(P)$.  Then
for any $k$-diagram, $\cF$, the map that $\mu$ induces by composition
$$
\Hom({\rm CoSky}_{P',k},\cF) 
\to
\Hom({\rm CoSky}_{P,k},\cF) ,
$$
when equivalently viewed as a map
$$
\cF(P') \to \cF(P),
$$
is the restriction map in $\cF$ from $\cF(P')\to\cF(P)$ multiplied
by $\alpha$.
Similarly, for $P,P'$ each equal one of $A_1,A_2,B_1,B_2,B_3$, we have
$$
\Hom({\rm Sky}_{P,k},{\rm Sky}_{P',k})\isom
{\rm Sky}_{P,k}(P')
$$
which equals $k$ or $0$ according to whether or not $P'$ is a specialization
of $P$;
if 
$$
\mu \in \Hom({\rm Sky}_{P,k},{\rm Sky}_{P',k})
$$
satisfies $(\mu(P'))1=\alpha$, then the map that $\mu$ induces by composition
$$
\Hom(\cF,{\rm Sky}_{P',k})
\to
\Hom(\cF,{\rm Sky}_{P,k}) ,
$$
when equivalently viewed as a map
$$
\cF(P')^* \to \cF(P)^*,
$$
is the dual of restriction map in $\cF$ from $\cF(P')^*\to\cF(P)^*$ multiplied
by $\alpha$.
\end{proposition}
The proof is a straightforward checking of the various cases of $P$ and $P'$
in $\{A_1,A_2,B_1,B_2,B_3\}$, which we leave to the reader.
[The analogous result holds (and is similarly easy to check)
for presheaves (in the sense of \cite{sga4.1})
over any semitopological category; see \cite{friedman_cohomology}.]

\subsection{Two-Term Injective and Projective Resolutions}
\label{su_two_term_resolutions}

In this section we prove Proposition~\ref{pr_two_term}.
The proof will introduce some useful tools to construct
some slightly different resolutions of $\underline k$ and
$\underline k_{/B_1,B_2}$.

\begin{proof}[Proof of Proposition~\ref{pr_two_term}]
For any short exact sequence 
\eqref{eq_short_exact_cF_sequence}, and any
$P=A_1,A_2,B_1,B_2,B_3$,
\begin{equation}\label{eq_short_exact_at_P}
0\to \cF^1(P)\to \cF^2(P)\to \cF^3(P)\to 0
\end{equation} 
is a short exact sequence of $k$-vector spaces.
For any
$k$-vector space, $V$,
setting $\cG={\rm CoSky}_{P,V}$,
\eqref{eq_short_exact_cF_sequence_applied_to_G_first} is equivalent
to the sequence of $k$-vector spaces
$$
0\to \Hom_k\bigl( V,\cF^1(P) \bigr)
\to \Hom_k\bigl( V,\cF^2(P) \bigr)
\to \Hom_k\bigl( V,\cF^3(P) \bigr)
\to 0,
$$
which we easily verify is exact 
by choosing a basis for $V$, which reduces this to the case
$V=k$, which is equivalent to \eqref{eq_short_exact_at_P}.  
Hence any coskyscraper $k$-diagrams is projective.

Similarly for
$\cG={\rm Sky}_{P,V}$,
\eqref{eq_short_exact_cF_sequence_applied_to_G_second} is equivalent to the
sequence
$$
0
\to \Hom_k\bigl(  \cF^3(P),V \bigr)
\to \Hom_k\bigl(  \cF^2(P),V \bigr)
\to \Hom_k\bigl(  \cF^1(P),V \bigr) 
\to 0
$$
which we see is exact by choosing a basis, $X$, for ${\rm Image}(\cF^1(P))$
in $\cF^2(P)$, and then extending this to a basis $X\cup X'$
for all of $\cF^2(P)$.
We then see that each element of $\Hom_k\bigl(  \cF^2(P),V \bigr)$
is determined by one of each of $\Hom_k\bigl(  \cF^i(P),V \bigr)$
for $i=1,3$.
Hence any skyscraper $k$-diagrams is injective.

For the projective resolution of a $k$-diagram, $\cF$, we note that for
each 
$P=A_1,A_2,B_1,B_2,B_3$, the isomorphism
$$
\Hom( {\rm CoSky}_P(\cF(P)) , \cF)  \isom \Hom_k\bigl( \cF(P),\cF(P) \bigr)
$$
(taking $\phi$ to $\phi(P)$) determines a unique map
\begin{equation}\label{eq_phi_P_to_cF}
\phi_{P,\cF}\from {\rm CoSky}_P\bigl( \cF(P) \bigr) \to \cF
\end{equation} 
that corresponds to the identity map of $\Hom_k\bigl( \cF(P),\cF(P) \bigr)$,
i.e., the identity map $\phi_P(P)$ is the identity map
$\cF(P)\to\cF(P)$.
This sets up a morphism $\cP_0\to\cF$, where
$$
\phi\from\cP_0\to\cF,
$$
where
$$
\cP_0 = \bigoplus_P {\rm CoSky}_P\bigl( \cF(P) \bigr), 
\quad
\phi = \bigoplus_P \phi_P.
$$
We see that $\ker(\phi)$ is $0$ on $B_1,B_2,B_3$, and hence
$\ker(\phi)$ is a sum of coskyscrapers at $A_1,A_2$.  Hence we get
a resolution
$$
\ker(\phi)=\cP_1\to\cP_0\to\cF,
$$
where $\cP_1,\cP_0$ are direct sums of coskyscraper $k$-diagrams.

Similarly the maps
$$
\phi^{P,\cF}\from \cF\to {\rm Sky}_P \bigl( \cF(P) \bigr) 
$$
such that $\phi^P(P)$ is the identity, yield an injective map
$\cF\to\cI^0$ whose cokernel at $A_1,A_2$ is $0$, and hence
we get a two-term injective resolution $\cF\to\cI^0\to\cI^1$
where $\cI^0=\oplus_P {\rm Sky}_P \bigl( \cF(P) \bigr)$.
\end{proof}

[Similar remarks hold for the category of presheaves on
any finite categories that are {\em semitopological} in
the sense of \cite{friedman_cohomology}; see
\cite{friedman_memoirs_hnc} for examples of this general
principle in another special case, namely the semitopological
categories used to define a {\em sheaf on a graph}.]

\subsection{Proof of Theorem~\ref{th_cohomology_agrees_with_Ext_groups}}

To prove Theorem~\ref{th_cohomology_agrees_with_Ext_groups}
we rely on the following straightforward calculation.

\begin{lemma}\label{le_projective_resolution_of_underline_k}
The $k$-diagram $\underline k$ has a projective resolution
\begin{equation}\label{eq_projective_resolution_of_underline_k}
0
\to \bigoplus_{i=1}^2{\rm CoSky}_{A_i}(k)
\xrightarrow{\mu_1} \bigoplus_{j=1}^3{\rm CoSky}_{B_j}(k)
\xrightarrow{\mu_0} \underline k \to 0,
\end{equation}
where 
\begin{equation}\label{eq_mu_zero_sum_of_phi_Bs_k}
\mu_0=\bigoplus_{j=1}^3 \phi_{B_j,\underline k}
\end{equation} 
with $\phi_{B_j}$ as in \eqref{eq_phi_P_to_cF}, 
and the map $\mu_1$ is the
as follows: the restriction to $\mu_1$ of ${\rm CoSky}_{A_1}(k)$ 
is determined by the element $(1,0,-1)$ in
$$
\Biggl( \bigoplus_{j=1}^3{\rm CoSky}_{B_j}(k) \Biggr)(A_1)
=
 \bigoplus_{j=1}^3 \bigl( {\rm CoSky}_{B_j}(k) \bigr)(A_1)
=
k\oplus \{0\} \oplus k,
$$
i.e., equals the element $\iota_{A_1,(1,0,-1)}$ as in
Proposition~\ref{pr_special_case_of_k_valued_cosky_and_sky},
and similarly
$\mu_1$ restricted to ${\rm CoSky}(A_2)$ is $\iota_{A_2,(0,1,-1)}$.
\end{lemma}
\begin{proof}
Setting
$$
\cP_0 = \bigoplus_{j=1}^3{\rm CoSky}_{B_j}(k)
$$
we have $\cP_0(B_j)=k$ for $j=1,2,3$, and
$\mu_0$ in \eqref{eq_mu_zero_sum_of_phi_Bs_k}
has $\mu_0(B_j)\from k\to k$ being the identity map.
Furthermore,
$$
\cP_0(A_1)=\bigoplus_{j=1}^3 \bigl( {\rm CoSky}_{B_j}(k) \bigr)(A_1)
= k\oplus \{0\} \oplus k
$$
and $\mu_0(A_1)$ takes $k\oplus \{0\} \oplus k \to \underline k(A_1)=k$
by the map taking $(u_1,0,u_3)$ to $u_1+u_3$.
Similarly $\mu(A_2)$ takes $\cP(A_2)=\{0\}\oplus k \oplus k$ to
$\underline k(A_2)$ by the map taking $(0,u_2,u_3)$ to $u_2+u_3$.
It follows $\mu_0$ is surjective, and that $\cP_1=\ker(\mu_0)$
has $\cP_1(B_j)=0$ for all $j$, and
$$
\cP_1(A_1)=\ker\bigl( (u_1,0,u_3) \mapsto u_1+u_3 \bigr), 
\quad
\cP_1(A_2)=\ker\bigl( (0,u_2,0,u_3) \mapsto u_2+u_3 \bigr).
$$
Hence $\cP_1(A_1)$ is a one dimensional $k$-vector space, spanned by
$(1,0,-1)\in k\oplus\{0\}\oplus k$, and $\cP(A_2)$ similarly,
spanned by $(1,0,-1)$.
Since for ${\rm CoSky}_{A_i,k}$ has value $k$ at $A_i$ and everywhere else
$\{0\}$, $\cP_1$ is the sum of one-dimensional coskyscrapers with
the value $k$ and the maps $\cP_1\to\cP_0$ is as in the statement
of the lemma.
\end{proof}
We remark that in the above lemma, $\cP_1(A_1)$ is a the one-dimensional
kernel of the map $(u_1,0,u_3)\mapsto u_1+u_3$; hence this
kernel is spanned both 
by $\pm(1,0,-1)$; the choice of one over the other
is arbitrary, and similarly the choice of $\mu_1(A_1)$ could
take $1$ to $(1,0,-1)$ or $(-1,0,1)$.  Similarly for $\cP_1(A_2)$.
[The opposite category of $\cC$ above is category in
\cite{friedman_memoirs_hnc} associated to a graph with three vertices,
corresponding to $B_1,B_2,B_3$
and two edges corresponding to $A_1,A_2$, 
and the choice between $\pm(1,0,-1)$ is analogous
to a choice of orientation of the edge corresponding to $A_1$;
similarly for $A_2$.]

\begin{proof}[Proof of Theorem~\ref{th_cohomology_agrees_with_Ext_groups}]
We can compute
${\rm Ext}^i(\underline k,\cF)$ using the projective
resolution
\eqref{eq_projective_resolution_of_underline_k}, which is therefore
the kernel and cokernel of the map that $\mu_1$ induces on
$$
\Hom\left( \bigoplus_{j=1}^3{\rm CoSky}_{B_j}(k), \cF\right) \to
\Hom\left( \bigoplus_{i=1}^2{\rm CoSky}_{A_i}(k), \cF\right),
$$
which is equivalent to a map 
\begin{equation}\label{eq_nu_like_partial_cF}
\nu\from \bigoplus_{j=1}^3 \cF(B_j) \to
\bigoplus_{i=1}^2 \cF(A_i)  .
\end{equation} 
For $i=1,2$ and $j=1,2,3$, let $\alpha_{i,j}\in k$ be given as
\begin{equation}\label{eq_alpha_ij}
\alpha_{1,1}=\alpha_{2,2}=1,
\quad
\alpha_{1,3}=\alpha_{2,3}=-1,
\quad
\alpha_{1,2}=\alpha_{2,1}=0.
\end{equation} 
According to 
Proposition~\ref{pr_morphisms_of_special_case_of_k_valued_cosky_and_sky},
since $\mu_1$ takes
${\rm CoSky}_{A_i,k}$ to ${\rm CoSky}_{B_j,k}$ by
multiplication by $\alpha_{i,j}$,
it follows that $\nu$ maps $\cF(B_j)$ to $\cF(A_i)$
by multiplication by $\alpha_{i,j}$.
But this is precisely the map $\cF(\partial)$.
Hence $\Ext^i(\underline k,\cF)$ for $i=0,1$ are, respectively, the
kernel and cokernel of $\cF(\partial)$.
\end{proof}

\subsection{A Subtlety of Linear Algebra and Zorn's Lemma}
\label{su_linear_algebra_subtlety}

Let $\cL\from U\to V$ be any linear map of (possibly infinite dimensional)
$k$-vector spaces.  Then we get a map:
\begin{equation}\label{eq_cokernel_of_cL_star_to_star_of_kernel_subtlety}
\coker(\cL^*) \to \bigl( \ker(\cL) \bigr)^*,
\end{equation} 
as follows:
any $\ell\in U^*$ gives, by restriction, a map $\ker(U)\to k$, i.e.,
an element of $(\ker(U))^*$, and if $\ell-\ell'\in {\rm Image}(\cL^*)$,
then $\ell,\ell'$ agree on $\ker(U)$.
To know that this map is an isomorphism, one needs to know that
any linear map $\ell\from\ker(\cL)\to k$ extends to a map $U\to k$.
This is true if we assume the axiom of choice or Zorn's Lemma, for then using
Zorn's lemma we can find a subspace $U'\subset U$ such that
$U$ splits as a direct
sum of $\ker(\cL)$ and $U'\in U$, and this
allows us to extend $\ell$ to all of
$U$ (defining $\ell$ to be zero on $U'$, which defines $\ell$ on all of $U$).

[Even if one does not assume the axiom of choice or Zorn's lemma,
\eqref{eq_cokernel_of_cL_star_to_star_of_kernel_subtlety}
will be an isomorphism in certain situations: for example,
if $U$ has a basis and
$\ker(\cL)$ is finite dimensional, then we can
perform a finite basis exchange to have a basis
for $\ker(\cL)$ that extends to a basis for $U$.
Note that this holds for $L=\cM_{W,\mec d}(\partial)$ for any
perfect matching $W$, the values of $\cM_{W,\mec d}$ have a basis
(indexed by $\integers$ or $\integers_{\le d_i}$ for $i=1,2$) and
the kernel of $L$ is finite dimensional.]

On the other hand, the foundations of homological algebra assume 
Zorn's lemma, since one uses Baer's Criterion to show that
$k$ is an injective $k$-module, which assumes Zorn's lemma
(see, e.g., \cite{weibel}, proof of Bear's Criterion~2.3.1, page~39).
In fact, to say that $k$ is injective is precisely to say that
if $A\to B$ is any injectiion, then any map $A\to k$ is the composition
of the map $A\to B$ with some map $B\to k$, which
implies that any linear functional on a subspace of $B$
has an extension to all of $B$.
Hence we will assume Zorn's lemma for the rest of this section.

\subsection{Proof of Theorem~\ref{th_strong_duality}}
\label{su_strong_duality}

\begin{proof}[Proof of Theorem~\ref{th_strong_duality}]
Similar to
\eqref{eq_projective_resolution_of_underline_k},
let us prove that $\underline k_{/B_1,B_2}$ has the injective
resolution
\begin{equation}\label{eq_injective_resolution_of_duality_sheaf}
0 \to \underline k_{/B_1,B_2}
\xrightarrow{\mu^{0}} \bigoplus_{i=1}^2{\rm Sky}_{A_i}(k)
\xrightarrow{\mu^{1}} 
\bigoplus_{j=1}^3{\rm Sky}_{B_j}(k)
\to 0 .
\end{equation} 
First, we take $\mu^{0}$ to be 
$\iota^{A_1,1}\oplus\iota^{A_2,1}$ with notation as in
Proposition~\ref{pr_special_case_of_k_valued_cosky_and_sky}, which yields
an injection
$$
\mu^{0}\from \underline k_{/B_1,B_2} \to \cI^0,
\quad\mbox{where}\quad
\cI^0= \bigoplus_{i=1}^2{\rm Sky}_{A_i,k} ,
$$
where setting $\cI^1=\cI^0/\mu^{0}(\underline k)$, we have
$\cI^1(A_i)=0$ for $i,1,2$, and
$$
\cI^1(B_1)=\cI^1(B_2)=k/\{0\}=k,
\quad
\cI^1(B_3)=k\oplus k/{\rm diag},
$$
where ${\rm diag}$ is the diagonal, i.e., the span of $(1,1)$ in $k\oplus k$.
Hence identifying $k\oplus k/{\rm diag}$ with
the class of $(-1,0)$, we get
$$
\cI^1 \isom \bigoplus_{j=1}^3{\rm Sky}_{B_j}(k),
$$
with the map $\mu^1\from\cI^0\to\cI^1$ being the
map taking ${\rm Sky}_{A_i,k}$ to ${\rm Sky}_{B_j,k}$ being
multiplication by $\alpha_{ij}$, where $\alpha_{i,j}$ are given in
\eqref{eq_alpha_ij}.
Hence for any $\cF$ we have
an exact sequence
$$
0 
\leftarrow \Ext^0(\cF,\underline k) 
\leftarrow
\bigoplus_{i=1}^2 \Hom_k(\cF(A_i),k)
\leftarrow
\bigoplus_{j=1}^3 \Hom_k(\cF(B_i),k)
\leftarrow
\Ext^1(\cF,\underline k) 
\leftarrow  0,
$$
i.e., 
$$
0 
\leftarrow \Ext^0(\cF,\underline k) 
\leftarrow
\bigoplus_{i=1}^2 \cF(A_i)^*
\leftarrow
\bigoplus_{j=1}^3 \cF(B_i)^*
\leftarrow
\Ext^1(\cF,\underline k) 
\leftarrow  0.
$$
But in view of the values of the $\alpha_{i,j}$, the map
\begin{equation}\label{eq_dual_of_nu_partia}
\bigoplus_{i=1}^2 \cF(A_i)^*
\leftarrow
\bigoplus_{j=1}^3 \cF(B_i)^*
\end{equation} 
above is precisely the dual map of $\nu$ in
\eqref{eq_nu_like_partial_cF}.
Hence, in view of the isomorphisms
\eqref{eq_cokernel_to_kernel_map} and
\eqref{eq_cokernel_of_cL_star_to_star_of_kernel_subtlety},
the duals of kernel and cokernel of $\nu$ are 
the cokernel and kernel (respectively) of the map
in \eqref{eq_dual_of_nu_partia}.
This establishes
\eqref{eq_zeroth_cohomology_and_Ext_one}
and hence
\eqref{eq_Serre_type_duality}.

The functoriality in $\cF$ can be verified directly or
by appealing to the functoriality of
\eqref{eq_long_exact_ext_second_var_varying}.
\end{proof}

\subsection{Homological Algebra of $k$-Diagrams, Value-By-Value Evaluation,
and Sheaf Theory}
\label{su_k_diagrams_sheaf_theory}

To apply the machinery of homological algebra, one has to verify that
$k$-diagrams form an {\em abelian category}, and to
prove 
Proposition~\ref{pr_value_by_value}.  As mentioned there, the
reader who is so inclined can verify Proposition~\ref{pr_value_by_value}
from scratch.  However, it is simpler to point out that
(1) this proposition is well known, and (2) the reader familiar with
sheaf theory on topological spaces can view $k$-diagrams as
equivalent to sheaves of $k$-vector spaces on a certain topological
space.  Let us explain both these points.

First, we remark that,
more generally, many convenient notions regarding $k$-diagrams and
their morphisms can be evaluated ``value-by-value,''
just as the notions of kernel, image, and cokernel of
a morphism in
Proposition~\ref{pr_value_by_value}.
Let us give some further examples of these ``value-by-value''
evaluations.

\subsubsection{Examples of Value-By-Value Evaluation}

If $\cF$ is a $k$-diagram, then a 
{\em subdiagram} of $\cF$ refers to any $k$-diagram,
$\cF'$ such that
at each $P=A_1,A_2,B_1,B_2,B_3$,
$\cF'(P)\subset\cF(P)$ is a subspace, and such that each
restriction map, $\cF'(\rho_{ij})$, of $\cF'$ is the restriction
of $\cF(\rho_{ij})$ to the subspace $\cF(B_i)$.
In this case one easily verify that the value-by-value inclusion
of $\cF'$ into $\cF$ gives has a morphism $\phi\from\cF'\to\cF$.
In this case there is a {\em quotient} $\cF/\cF'$, whose
values are $\cF(P)/\cF'(P)$.

The reader familiar with topological sheaf theory will notice
that if we take a value-by-value quotient $\cF/\cF'$ as above,
the result is not generally a sheaf (see, e.g., \cite{hartshorne},
Section~II.1, top of page 65); to get the usual notion of a quotient
sheaf one has to take the extra step of 
{\em sheafifying} the result, i.e., taking the sheaf associated to
the presheaf given by $U\to \cF(U)/cF'(U)$ for open subsets $U$.  
For $k$-diagrams we never need this extra step.

Similarly, for
a morphism of $k$-diagrams $\phi\from\cF\to\cG$,
the {\em kernel}, {\em image}, and {\em cokernel} of $\phi$ to be
the diagrams whose values at each
$P=A_1,A_2,B_1,B_2,B_3$ are the kernel, image, and cokernels
of $\phi(P)$.
Again, in general topological sheaf theory,
the additional step of {\em sheafifying} is needed for images and cokernels
(see, e.g., \cite{hartshorne}, Section~II.1, top of page 64,
just above Proposition/Definition~1.2).

\subsubsection{Why Value-By-Value Evaluation Works}

To see why ``value-by-value'' evaluation works, we appeal to
\cite{sga4.1}, Expos\'e I, Section~3, Proposition~3.1 and
Corollaire~3.2:
consider the category $\cC$, whose objects and non-identity morphisms
are depicted below.
\begin{equation}\label{eq_our_category_for_our_diagrams}
\begin{tikzpicture}[scale=0.40]
\node (B1) at (0,2) {$B_1$};
\node (B2) at (0,-2) {$B_2$};
\node (B3) at (0,0) {$B_3$};
\node (A1) at (3,1) {$A_1$};
\node (A2) at (3,-1) {$A_2$};
\draw [<-] (B1) -- (A1) ;
\draw [<-] (B2) -- (A2) ;
\draw [<-] (B3) -- (A1) ;
\draw [<-] (B3) -- (A2) ;
\node at (1.5,-4){$\cC$};
\end{tikzpicture}
\end{equation} 
Then a $k$-diagram is the same thing as
a contravariant functor from $\cC$ to the category of
$k$-vector spaces, which is called, in \cite{sga4.1}, 
Expos\'e I, a
{\em presheaf} of $k$-vector spaces on $\cC$.
Then Proposition~3.1 and Corollaire~3.2 of
\cite{sga4.1}, Expos\'e I, Section~3 implies that for any
category, $\cC$, the notions of subdiagram, kernel, image, etc.,
agree with the value-by-value evaluation above
(we also refer the reader to
\cite{friedman_memoirs_hnc} 
and \cite{friedman_cohomology} for similar observations with
$\cC$ replaced with categories arising from graphs).

\subsubsection{$k$-diagrams and Sheaf Theory}

To see that ``value-by-value'' evaluation works in $k$-diagrams,
one can alternatively appeal to sheaf theory (e.g., \cite{hartshorne},
Section~II.1) on one particular
topological space.  Let us explain how.

If $X$ is a topological space on a finite set, say that an open
set $U\subset X$ is {\em irreducible} if $U$ is non-empty and
cannot be written as the union of proper open subsets of $X$.
The set of irreducible open subset, ${\rm Irred}(X)$, of $X$ becomes
a category under inclusion.
Each sheaf, $\cF$, of $k$-vector spaces on $X$ restricts to
a presheaf (in the sense of Grothendieck, i.e., a contravariant 
functor from ${\rm Irred}(X)$ to the category of $k$-vector spaces)
of vector spaces on ${\rm Irred}(X)$.
It is not hard to verify (see, e.g., \cite{friedman_cohomology}) that
this functor from the category of sheaves of $k$-vector spaces on $X$
to presheaves of $k$-vector spaces on ${\rm Irred}(X)$ is an
equivalence of categories; in other words, one can reconstruct---up to
isomorphism---a sheaf on $X$ by knowing its restriction to ${\rm Irred}(X)$,
and any presheaf on ${\rm Irred}(X)$ arises as the restriction of
a sheaf on $X$.

Let $X$ be the finite topological space on the points $A_1,A_2,B_1,B_2,B_3$
with a basis of open subsets
$$
\{A_1\}, 
\ \{A_2\},
\ \{B_1,A_1\},
\ \{B_2,A_2\},
\ \{B_3,A_1,A_2\}.
$$
Then on easily sees that ${\rm Irred}(X)$ is precisely the category
$$
\begin{tikzpicture}[scale=0.40]
\node (B1) at (0,2) {$\{B_1,A_1\}$};
\node (B2) at (0,-2) {$\{B_2,A_2\}$};
\node (B3) at (0,0) {$\{B_3,A_1,A_2\}$};
\node (A1) at (6,1) {$\{A_1\}$};
\node (A2) at (6,-1) {$\{A_2\}$};
\draw [<-] (B1) -- (A1) ;
\draw [<-] (B2) -- (A2) ;
\draw [<-] (B3) -- (A1) ;
\draw [<-] (B3) -- (A2) ;
\end{tikzpicture}
$$
(so, for example, $B_1,B_2,B_3$ are closed points, and $\{A_1\},\{A_2\}$
are open subsets, $B_1$ and $B_3$ are specializations of $A_1$, etc.).
Hence ${\rm Irred}(X)$ in this case is the same thing as the
category \eqref{eq_our_category_for_our_diagrams}, which is how
we build our diagrams.

We remark that finite topological spaces the property that each
point, $P$, of the space has a minimal open subset, $U_P$, that contains it.
It follows that the ``stalk at $P$'' of a sheaf is simply its
value at $U_P$.  This is an alternative way to see that
``value-by-value evaluation'' works in all finite topological spaces.

We also remark that if $P,P'$ are points in a topological space, we say
that 
$P'$ is a {\em specialization} of $P$ when
the closure of
$P$ contains $P'$
(see \cite{sga4.1}, Subsection~IV.4.2.2 or 
\cite{hartshorne}, Exercise~II.3.17).
Hence this notion agrees with the notion
in Definition~\ref{de_specialization}.

\subsection{The Usual Skyscraper $k$-Diagrams the Riemann-Roch
Theorem}
\label{su_skyscraper}

If $W\from\integers^2\to\integers$ is a perfect matching,
and $\mec d\in\integers^2$, then we easily check that there
is an exact sequence
\begin{equation}\label{eq_riemann_roch_skyscraper_mimic}
0 \to \cM_{W,\mec d}\to\cM_{W,\mec d+\mec e_i}\to{\rm Sky}_{A_i}(k)
\to 0.
\end{equation} 
Furthermore, $b^1$ of any skyscraper $k$-diagram is $0$ since any skyscraper
$k$-diagram is injective, and we easily check that
$$
b^0(  {\rm Sky}_{A_i}(k) \bigr) = 1.
$$
This therefore mimics the usual short exact sequence in the 
modern formulation of the Riemann-Roch theorem, 
e.g., \cite{hartshorne}, the proof of the Riemann-Roch theorem,
page 296, the sequence
$$
0\to \cL(D)\to \cL(D+P) \to k(P) \to 0.
$$
Hence 
$$
\chi\bigl( \cM_{W,\mec d+\mec e_i} \bigr) =
\chi\bigl( \cM_{W,\mec d} \bigr) +
\chi\bigl( {\rm Sky}_{A_i}(k) \bigr)
=
\chi\bigl( \cM_{W,\mec d} \bigr) + 1,
$$
which is another way of deriving
\eqref{eq_Euler_characteristic_fixed_W_varying_mec_d}.
Furthermore, the full strength of Lemma~\ref{le_codimension_one_cons}
as it applies to the Betti numbers of $\cM_{W,\mec d}$ and
$\cM_{W,\mec d+\mec e_i}$
(as in Corollary~\ref{co_cM_Euler_char}) can also be seen by considering
the long exact sequence that arises from
\eqref{eq_riemann_roch_skyscraper_mimic}, namely
$$
0 
\to H^0 \bigl(\cM_{W,\mec d} \bigr)
\to H^0 \bigl(\cM_{W,\mec d+\mec e_i} \bigr)
\to k
\to H^1 \bigl(\cM_{W,\mec d} \bigr)
\to H^1 \bigl(\cM_{W,\mec d+\mec e_i} \bigr) 
\to 0.
$$

\subsection{$\cO$-Modules and Periodic $k$-Diagrams}

The sheaf theory of algebraic geometry works with sheaves of
$\cO$-modules, where $(X,\cO)$ is a locally ringed space
(see, e.g., \cite{hartshorne}, Sections~II.2, page~72 and~III.2)
We remark that $k$-diagrams are nothing more than $\cO$-modules
over the ringed space $(X,\cO)$ where $X$ is the five-point topological
space described in the previous subsection, and $\cO=\underline k$.

If $W\from\integers^2\to\integers$ is a perfect matching that is also
$r$ periodic---which is the case in the Baker-Norine rank and related
rank functions \cite{baker_norine,amini_manjunath}---then the $k$-diagrams
$\cM_{W,\mec d}$ are $\cO$-modules for a much larger sheaf of rings
(i.e., $k$-diagram of rings) over the same space, namely the diagram
$\cO=\cO_{k,r}$ whose values are for $i=1,2$:
$$
\cO(A_i)=k[x_i,1/x_i], \ \cO(B_i)=k[y_i], 
\ \cO(B_3)=k[v,1/v],
$$
whose restriction maps take $v$ to $x_1^r,x_2^{-r}$ and take $y_i$ to $1/x_i$.
For $r=1$, $(X,\cO)$ is the Riemann sphere and the $\cM_{W,\mec d}$ are
line bundles, but for $r\ge 2$ this is a
much more mysterious space.  [This is not an orbifold; perhaps it is a 
cover of an orbifold reflecting a \v{C}ech cohomology computation or something
related.]
We believe there is a duality theorem akin to Serre duality, involving
the $k$-diagrams 
$\cM_{W,\mec d}$ as $\cO$-modules.
We plan to address this in a future work.
This may shed more light on the $\cM_{W,\mec d}$.  It is also related
to our discussion of Serre functors in the next subsection.

\subsection{The Serre Functor on Chains of $k$-Diagrams}
\label{su_Serre_functor}
\newcommand{\CatkDiag}{\mbox{{\bf $k$-diag}}}

In this section we briefly discuss
so-called {\em Serre functor}, $S$, and show that
$S(\underline k)=\underline k_{/B_1,B_2}[1]$,
which explains
Theorem~\ref{th_strong_duality}.

To motivate our discussion of the Serre functor, notice that
Theorem~\ref{th_strong_duality} yields an
isomorphism
$$
\forall i\in\integers,\quad
\Ext^i(\underline k,\cF)^* \isom 
H^i(\cF)^*  \to \Ext^{1-i}(\cF,\underline k_{/B_1,B_2})
$$
and for any 
$P=A_1,A_2,B_1,B_2,B_3$ we have an isomorphism
$$
\forall i\in\integers,\quad
\Ext^i \bigl({\rm CoSky}_P(k),\cF \bigr)^* 
\to \Ext^{0-i} \bigl(\cF,{\rm Sky}_P(k) \bigr).
$$
It follows that if $\cG$ is a sum $\underline k\oplus {\rm CoSky}_P(k)$,
there is no duality theorem that has a single $k$-diagram associated
to $\cG$ in a duality theorem involving
$$
\Ext^i(\cG,\cF)^*.
$$
So even if one seeks a duality theorem valid only for $k$-diagrams,
one is pretty much forced to express this by working in a larger
context.  The derived category is such a context, and it is a common
tool for expressing duality theorems.

Hence we assume that the reader is familiar with the
derived category (see
\cite{weibel}, Chapter~10, or
\cite{residues}, Chapter~I), and we will describe the so-called
{\em Serre functor} in these terms.
Since each $k$-diagram has a two-term injective and two-term projective
resolution, it is simplest to work in $\cD=\cD^{\rm b}$ of bounded
chains of $k$-diagrams, each of whose values are finite dimensional
$k$-vector spaces\footnote{
  As mentioned in
  \cite{friedman_cohomology}, since our $k$-diagrams involve only a finite
  number of values and restriction maps, to work in homological algebra
  we need to take only finite limits of our diagrams, which are again
  diagrams whose values are finite dimensional.  See also the last 
  paragraph of this section.
  }.
The {\em Serre functor}, $S$,
(see, e.g., \cite{bondal_kapranov,bondal_orlov} 
or \cite{friedman_cohomology}, Section~2.12
and the references therein) is the functor that takes 
an object $\cG\in\cD$ to the functor $S(\cG)\from\cD\to\cD'$
$$
\cF\mapsto \Hom_{\cD}(\cG,\cF)^* 
$$
(all Hom sets are assumed to be finite dimensional $k$-vector spaces);
if $S(\cG)$ is representable, then there is a $\cG'\in\cD$ such that
$$
\Hom_{\cD}(\cG,\cF)^* \isom \Hom_{\cD}(\cF,\cG'),
$$
and we write $S(\cG)=\cG'$; if so then $\cG'$ is uniquely
determined up to unique isomorphism.
The recipe to compute the Serre functor
in \cite{friedman_cohomology} (written ``left-to-right'' or
$!\to *$), in our terms is to observe that
$$
S\bigl( {\rm CoSky}_P(k) \bigr) = {\rm Sky}_P(k),
$$
since if $\cF$ equals
$\cdots\to\cF^{-1}\to \cF^0\to \cF^1\to\cdots$ is any bounded
chain of $k$-diagrams, then since ${\rm CoSky}_P(k)$ is projective, viewing
${\rm CoSky}_P(k)$ as a complex in degree 0 we have
$$
\Hom_{\cD}\bigl( {\rm CoSky}_P(k) , \cF \bigr) =
\Hom_{\cK}\bigl( {\rm CoSky}_P(k) , \cF \bigr) 
= H^0\bigl(\cF(P)\bigr) ,
$$
where $\Hom_\cK$ denotes chain maps modulo homotopy, and $H^0(\cF(P))$
denotes the $0$-th cohomology group of the exact sequence
$$
\cdots\to\cF^{-1}(P) \to \cF^0(P) \to \cF^1(P) \to\cdots
$$
Similarly, since ${\rm Sky}_P(k)$ is injective, we have
$$
\Hom_{\cD}\bigl( \cF, {\rm Sky}_P(k)  \bigr) =
\Hom_{\cK}\bigl( \cF, {\rm Sky}_P(k)  \bigr) =
H^0\bigl(\cF(P)\bigr)^*.
$$
It follows that any chain that if $\cF$ is $0$ outside of degree $0$, and
is a sum of $k$-diagrams of the form ${\rm CoSky}_P(k)$ in degree $0$,
then $S(\cF)$ is represented by exchanging each CoSky with a Sky.
It then follows by induction 
that if $\cG$
is a chain of sums of coskyscraper $k$-diagrams (each of finite dimension), 
then the recipe of exchanging the coskyscrapers with skyscrapers
computes the Serre functor: for the inductive step one can use
Lemma~I.7.2 of \cite{residues}, which states that 
for any complex $\cG$ equal 
$\cdots\to\cG^{-1}\to \cG^0\to \cG^1\to\cdots$ and $n\in\integers$ there
is a distinguished triangle
$$
\tau_{\ge n}(\cG) \to \cG^n \to \tau_{>n}(\cG) \to \tau_{\ge n}(\cG)[1],
$$
and to apply the five-lemma.

Lemma~\ref{le_projective_resolution_of_underline_k} shows that 
$\underline k$ is isomorphic to the chain 
$$
\bigoplus_{i=1}^2{\rm CoSky}_{A_i}(k)
\xrightarrow{\mu_1} \bigoplus_{j=1}^3{\rm CoSky}_{B_j}(k)
$$
with the $A_i$ coskyscrapers in degree $0$.  Hence
$S(\underline k)$ is given by
$$
\bigoplus_{i=1}^2{\rm Sky}_{A_i}(k)
\xrightarrow{\tilde \mu_1} \bigoplus_{j=1}^3{\rm Sky}_{B_j}(k)
$$
with the $A_i$ skyscrapers in degree $0$, and where $\tilde \mu_1$
is $\mu_1$ with each morphism of coskyscrapers replaced by the
corresponding one of skyscrapers.
The proof of Theorem~\ref{th_strong_duality} shows that
$\mu^1$ of 
\eqref{eq_injective_resolution_of_duality_sheaf} is the map
with the coefficients
\eqref{eq_alpha_ij} which are the same as for $\mu_1$;
hence $\tilde\mu_1=\mu^1$.  It follows that
$S(\underline k)$ is isomorphic to the element
of $\cD$ which consists of the single non-zero $k$-diagram
$\underline k_{/B_1,B_2}$ in degree $-1$, i.e.,
the element $\underline k_{/B_1,B_2}[1]$.
Hence
$$
S(\underline k) \isom \underline k_{/B_1,B_2}[1].
$$

% We remark that the Serre functor is a functor, i.e., given a
% morphism $\phi\from\cG_1\to\cG_2$, $\phi$ induces a map 
% \begin{equation}\label{eq_map_yielding_Serre_of_morphism_pre_stage}
% \Hom_{\cD}(\cG_2,\cF) 
% \to 
% \Hom_{\cD}(\cG_1,\cF) 
% \end{equation} 
% via composition with $\phi$, and hence a map
% \begin{equation}\label{eq_map_yielding_Serre_of_morphism}
% \Hom_{\cD}(\cG_1,\cF)^* \to \Hom_{\cD}(\cG_2,\cF)^*,
% \end{equation} 
% and hence a map $S(\phi)\from S(\cG_1)\to S(\cG_2)$, and so if
% $S(\cG_1),S(\cG_2)$ are both representable, then $S(\phi)$ arises
% from a morphism in $\cD$ (by Yoneda's lemma).
% In particular, if $\cG=\cG_1=\cG_2$, and $\phi\from\cG_1\to\cG_2$ for
% some $\alpha\in k$ is
% is $\alpha\,{\rm Id}_{\cG}$, i.e., the identity times $\alpha$, then
% the map in
% \eqref{eq_map_yielding_Serre_of_morphism_pre_stage} is multiplication
% by $\alpha$, and hence so is the map on dual spaces.  Hence
% $$
% S \from \Hom(\cG,\cG) \to \Hom\bigl( S(\cG),S(\cG) \bigr)
% $$
% has 
% $$
% S(\alpha\,{\rm Id}_{\cG} = \alpha {\rm Id}_{S(\cG});
% $$
% in particular
% \begin{equation}\label{eq_Serre_functor_on_Id_underline_k}
% S(\alpha\,{\rm Id}_{\underline k}) =
% \alpha \,{\rm Id}_{\underline k_{/B_1,B_2}[1]} ,
% \end{equation} 
% which thereby, after downshifting, 
% identifies $\alpha\,{\rm Id}_{\underline k}$
% with $\alpha\,{\rm Id}_{\underline k_{/B_1,B_2}}$;
% likely this is related to the identification in
% \eqref{eq_identify_two_hom_sets}
% in Theorem~\ref{th_first_duality_cM}.
% 
% 

We remark that---as mentioned in \cite{friedman_cohomology}---because
$k$-diagrams are build on a diagram with finitely many values and
morphisms, there is no problem in working with $k$-diagrams whose
values are finite dimensional vector spaces.
The problem is that this does not allow us to work with the
$\cM_{W,\mec d}$, and the Hom sets involving the $\cM_{W,\mec d}$ are
not finite dimensional.
One might still be able to work in this larger context and infer that
$S(\underline k)=\underline k_{/B_1,B_2}[1]$ in such a context, as
this equality is essentially the content of
Theorem~\ref{th_strong_duality}.
Finally, we remark that 
the study of $\cO$-modules with $\cO=\cO_{k,r}$ as in the
previous subsection, in the case where $W$ is $r$-periodic,
may allow for a calculation of the
Serre functor in some sense.

			% to (topological and Grothendieck) sheaf theory

% \input{se_clearing}  This is in the thesis, which should be referenced.

% \input{se_conclusion} % this mentions our paper about cM being dualizing

\appendix
\section{Remarks Related to the Definition of a Dual Virtual Vector
Space}
\label{ap_dual_virtual_vec_sp}

In this appendix we make some remarks to indicate why
Definition~\ref{de_dual_virtual_vector_space} may be part of a
larger notion of a ``morphism'' from one virtual $k$-vector space
to another.  However, our remarks do not yield a satisfactory
notion of a morphism that is compatible with our notion
of equivalence of two virtual $k$-vector spaces. 
This appendix is independent of the rest of this paper.

While we limit our discussions to virtual vector spaces, ideally
one would have similar remarks for virtual objects of more general
categories
(such as the category of $k$-Fredholm maps and of $k$-diagrams).

We hope that future work will give a better justification
for Definition~\ref{de_dual_virtual_vector_space}.
Ideally future work would construct a category (or bicategory) of 
virtual vector spaces such that a morphism from
$V_1\ominus V_2$ to $k=k\ominus 0$ is precisely
$V_1^*\ominus V_2^*$.

Note that if $V_1,V_3$ are finite dimensional $k$-vector spaces, then
$V_1\otimes_k V_3$ and $\Hom_k(V_1,V_3)$ are $k$-vector spaces, both
of dimension $\dim(V_1)\dim(V_3)$.  Hence the following definition
maintains this dimension under equivalence.

\begin{definition}\label{de_hom_tensor_virtual}
If $V_1,V_2,V_3,V_4$ are four finiite dimensional $k$-vector spaces, then
we define
$$
\Hom\bigl( V_1\ominus V_2 , V_3 \ominus V_4)
=
\Hom( V_1 , V_3 ) \oplus \Hom(V_2, V_4) 
\ominus
\Hom(V_1,V_4) \ominus \Hom(V_2,V_3)
$$
and
$$
(V_1,V_2) \otimes (V_3,V_4) 
=
( V_1 \otimes V_3 ) \oplus (V_2 \otimes V_4) 
\ominus
(V_1\otimes V_4) \ominus (V_2\otimes V_3) .
$$
\end{definition}
The problem with the above definition is that if
$V_1\ominus V_2$ is equivalent to $V_1'\ominus V_2'$, as virtual
vector spaces, and similarly with $V_3\ominus V_4$
and $V_3'\ominus V_4'$, it is not clear how to relate the two sets
\begin{equation}\label{eq_Vs_and_Vprimes_hom_sets}
\Hom\bigl( V_1\ominus V_2 , V_3 \ominus V_4\bigr), 
\ \Hom\bigl( V_1'\ominus V_2' , V_3' \ominus V_4'\bigr).
\end{equation} 
To make Definition~\ref{de_hom_tensor_virtual} compatible with
equivalence, one might be able to change the Hom sets to get
a category, analogous to how one constructs the derived category
(i.e., where one works with Homs of chain maps by first taking
chain maps modulo homotopy equivalence, and then one ``localizes''
the category so that quasi-isomorphisms become isomorphisms).

However, if we ignore the above difficulties, and merely work with
Definition~\ref{de_hom_tensor_virtual} as is, this definition gives
$$
\Hom( V_1\ominus V_2 , k )
=
\Hom( V_1\ominus V_2 , k\ominus 0 )
=
(V_1^*,V_2^*),
$$
which therefore gives that
Definition~\ref{de_dual_virtual_vector_space} agrees with the usual
definition.  We also note that an ``element'' of $V_1\ominus V_2$
would be an element of
$$
\Hom( k, V_1\ominus V_2 )
=
\Hom( k \ominus 0, V_1 \ominus V_2) =
\Hom(k, V_1) \ominus \Hom(k,V_2)
$$
which suggests that an element of $V_1\ominus V_2$---if such a thing
makes sense---would be a formal difference of an element of $V_1$
``minus'' an element of $V_2$.

We also note that for Hom in Definition~\ref{de_hom_tensor_virtual},
there is a natural composition of
$$
\mu\in \Hom(V_1\ominus V_2,V_3\ominus V_4),
\quad
\nu\in \Hom(V_3\ominus V_4,V_5\ominus V_6),
$$
namely if $\mu_{ij}\from\Hom(V_i,V_j)$ with $i=1,2$ and $j=3,4$,
and $\nu_{j\ell}\from\Hom(V_j,V_\ell)$ with $j=3,4$ and $\ell=5,6$,
then we could set $\nu\circ\mu$ to be the maps
$(\nu\circ\mu)_{i\ell}$ given by
$$
(\nu\circ\mu)_{i\ell} = \sum_{j=3,4} \nu_{j\ell}\circ \mu_{ij}.
$$
Under this composition law, $\Hom$ is associative, and has an 
identity morphism for
$V_1\ominus V_2$, namely ${\rm Id}_{V_1}\oplus{\rm Id}_{V_2}$.

A specific problem with the above definitions is that
the virtual vector $k\ominus k$ is equivalent to $0$,
and yet there is no isomorphism between these virtual vector spaces,
since any morphism in $k\ominus k$ that factors through $0$ must
be the zero map, and hence cannot equal the identity map of
$k\ominus k$.

We remark that an isomorphism of $k$-vector spaces $V_1^*\to V_3$
is the same as a perfect (or non-degenerate) pairing, $V_1\times V_3\to k$,
and such a pairing extends to a map $V_1\otimes V_3\to k$.
Hence, as an alternative
to working with ``dual spaces'' and making sense of what constitutes an
``isomorphism'' 
\begin{equation}\label{eq_dual_of_virtual_vs_to_virtual_vs}
(V_1 \ominus V_2)^* \to V_3\ominus V_4 ,
\end{equation}
one could try to define a reasonable notion of a ``pairing''
and ``perfect pairing'' 
\begin{equation}\label{eq_virtual_vector_space_pairing}
(V_1 \ominus V_2)\times (V_3\ominus V_4 ) \to k.
\end{equation}
In fact, the isomorphism
\eqref{eq_dual_of_virtual_vs_to_virtual_vs}
in Theorem~\ref{th_duality_virtual_two_dimensional} is really
based on a pairing $V_1\times V_3$ and one $V_2\times V_4$.
Therefore, if one accepts that two such pairings should give rise to
a pairing
\eqref{eq_virtual_vector_space_pairing}, then one can avoid any reference
to dual spaces and isomorphism
\eqref{eq_dual_of_virtual_vs_to_virtual_vs}.
Again the link between pairings and tensor products is that in an
ideal setting, a pairing \eqref{eq_virtual_vector_space_pairing} would
``extend'' to a map
$$
(V_1 \ominus V_2)\otimes (V_3\ominus V_4 ) \to k.
$$
Hence one starting point for finding an ideal setting is to
search for a notion of the tensor product of virtual vector spaces,
which may be related to Definition~\ref{de_hom_tensor_virtual} above.

%%%%%%% The other approach below

Another approach to making the formal differences of $k$-vector spaces
(or of $k$-Fredholm maps, or of $k$-diagrams, etc.) is to declare
$\Hom(V_1\ominus V_2,V_3\ominus V_4)$ to equal any morphism
$$
f\from V_0 \oplus V_1 \oplus V_4 \to V_0 \oplus V_2 \oplus V_3,
$$
since equivalence of virtual $k$-vector spaces is the case
when a map as above is an isomorphism.
This notion of morphism does have a composition law, although
it seems that the composition law is not associative 
(hence we may get some type
of $2$-category or bicategory):
namely, if 
$g\in\Hom(V_3\ominus V_4,V_5\ominus V_6)$, i.e., 
$$
g\from 
V_0'\oplus V_3\oplus V_6\to 
V_0'\oplus V_4\oplus V_5 ,
$$
one could add to $f$ the identity map $\id_{V_0'\oplus V_6}$
and get a composition
$$
(V_0'\oplus V_6)\oplus (V_0\oplus V_1\oplus V_4)
\xrightarrow{\id_{V_0'\oplus V_6}\oplus f}
(V_0'\oplus V_6)\oplus (V_0\oplus V_2\oplus V_3)
$$
which can be composed with the morphism
$$
(V_0\oplus V_2)\oplus (V_0'\oplus V_3\oplus V_6)
\xrightarrow{\id_{V_0\oplus V_2}\oplus g}
(V_0\oplus V_2)\oplus (V_0'\oplus V_4\oplus V_5)
$$
by rearranging the direct sums of the domain of
$\id_{V_0\oplus V_2}\oplus g$, which upon rearranging 
direct sums is a map
$$
V_0'' \oplus V_1\oplus V_6 \to V_0'' \oplus V_2\oplus V_5,
\quad\mbox{where $V_0''=V_0\oplus V_0'\oplus V_4$.}
$$
However, it is not clear to us that this notion makes
the Hom set from a virtual $k$-vector space to another having
dimension equal to the product of the two vector spaces,
which seems like a desirable property.

%    Bibliography styles amsplain or harvard are also acceptable.
\providecommand{\bysame}{\leavevmode\hbox to3em{\hrulefill}\thinspace}
\providecommand{\MR}{\relax\ifhmode\unskip\space\fi MR }
\providecommand{\MRhref}[2]{%
  \href{http://www.ams.org/mathscinet-getitem?mr=#1}{#2}
}
\providecommand{\href}[2]{#2}

% \bibliography{bibrefs}

% \bibliography{bibJan16}

\end{document}